\documentclass[oneside,english]{amsart}
\usepackage[T1]{fontenc}
\usepackage[latin9]{inputenc}
\usepackage{mathtools}
\usepackage{enumitem}
\usepackage{amstext}
\usepackage{amsthm}
\usepackage{amssymb}

\usepackage[normalem]{ulem}
\usepackage{hyperref}

\usepackage{verbatim}

\makeatletter
\numberwithin{equation}{section}
\numberwithin{figure}{section}
\theoremstyle{plain}
\newtheorem{thm}{\protect\theoremname}
\newtheorem*{theo}{Theorem}
\theoremstyle{definition}
\newtheorem{defn}[thm]{\protect\definitionname}
\theoremstyle{remark}
\newtheorem*{rem*}{\protect\remarkname}
\theoremstyle{plain}
\newtheorem{fact}[thm]{\protect\factname}
\newtheorem{observation}[equation]{\protect\observationname}
\theoremstyle{remark}
\newtheorem{rem}[thm]{\protect\remarkname}
\theoremstyle{plain}
\newtheorem{lem}[thm]{\protect\lemmaname}
\newlist{casenv}{enumerate}{4}
\setlist[casenv]{leftmargin=*,align=left,widest={iiii}}
\setlist[casenv,1]{label={{\itshape\ \casename} \arabic*.},ref=\arabic*}
\setlist[casenv,2]{label={{\itshape\ \casename} \roman*.},ref=\roman*}
\setlist[casenv,3]{label={{\itshape\ \casename\ \alph*.}},ref=\alph*}
\setlist[casenv,4]{label={{\itshape\ \casename} \arabic*.},ref=\arabic*}
\theoremstyle{plain}
\newtheorem{cor}[thm]{\protect\corollaryname}
\theoremstyle{plain}
\newtheorem{prop}[thm]{\protect\propositionname}
\theoremstyle{plain}
\newtheorem*{cor*}{\protect\corollaryname}
\theoremstyle{plain}
\newtheorem*{lem*}{\protect\lemmaname}

\def\N{\mathbb N}

\newcommand{\Meng}[2]{\left\{#1\mathrel{}\middle|\mathrel{}#2\right\}}
\newcommand{\abs}[1]{\left\lvert#1\right\rvert}

\DeclarePairedDelimiter\floor{\lfloor}{\rfloor}

\usepackage{babel}
\providecommand{\lemmaname}{Lemma}
\providecommand{\propositionname}{Proposition}
\providecommand{\theoremname}{Theorem}

\makeatother

\usepackage{babel}
\providecommand{\casename}{Case}
\providecommand{\corollaryname}{Corollary}
\providecommand{\definitionname}{Definition}
\providecommand{\factname}{Fact}
\providecommand{\observationname}{Observation}
\providecommand{\lemmaname}{Lemma}
\providecommand{\propositionname}{Proposition}
\providecommand{\remarkname}{Remark}
\providecommand{\theoremname}{Theorem}

\begin{document}
	\title{Non-classifiability Of Ergodic Flows Up To Time Change}
	\author{Marlies Gerber$^1$}   
	\thanks{$^1$ Indiana University, Department of Mathematics, Bloomington, IN 47405, USA}
	\author{Philipp Kunde$^2$} 
	\thanks{$^2$ Jagiellonian University in Krakow, Faculty of Mathematics and Computer Science, 30-348 Krakow, Poland. Oregon State University, Department of Mathematics, Corvallis, OR 97331, USA.
		P.K. acknowledges financial support from a DFG Forschungsstipendium under Grant No. 405305501 and a Polonez Bis Grant No. 2021/43/P/ST1/02885.}
	\begin{abstract}A time change of a flow $\{T_t\}$,  ${t\in\mathbb{R}}$, is a reparametrization of the orbits of the flow such that each orbit is mapped to itself by an orientation-preserving homeomorphism of the parameter space. If a flow $\{S_t\}$ is isomorphic to a flow obtained by a reparametrization of a flow $\{T_t\}$, then we say that $\{S_t\}$ and $\{T_t\}$ are isomorphic up to a time change. For ergodic flows $\{S_t\}$ and $\{T_t\}$, Kakutani showed that this happens if and only if the two flows have Kakutani equivalent transformations as cross-sections.
		
		We prove that the Kakutani equivalence relation on ergodic invertible
		measure-preserving transformations of a standard non-atomic probability space is not a Borel set. This shows
		in a precise way that classification of ergodic transformations up to Kakutani equivalence is
		impossible. In particular, our results imply the non-classifiability of ergodic flows up to isomorphism after a time change. Moreover, we obtain anti-classification results under isomorphism for ergodic invertible transformations of a sigma-finite measure space.
		
		We also obtain anti-classification results under Kakutani equivalence for ergodic
		area-preserving smooth diffeomorphisms of the disk, annulus, and 2-torus, as well as real-analytic diffeomorphisms
		of the $2$-torus. Our work generalizes the anti-classification results under isomorphism for ergodic transformations obtained by Foreman, Rudolph, and Weiss.

	\end{abstract}
	
	\maketitle
	
	\insert\footins{\footnotesize - \\
		\textit{2020 Mathematics Subject classification:} Primary: 37A20; Secondary: 37A35, 37A05, 37C40, 03E15\\
		\textit{Key words: } Kakutani equivalence, anti-classification, complete analytic, smooth ergodic theory}
	
	\section{\label{sec:Introduction}Introduction}
	
	A fundamental theme in ergodic theory, going back to J. von Neumann
	\cite{Ne}, is the isomorphism problem: classify measure-preserving
	transformations (MPTs) up to isomorphism. By an MPT we mean a measure-preserving
	automorphism of a standard non-atomic probability space. We let $\mathcal{X}$
	denote the set of all MPTs of such a space. Two automorphisms $S,T\in\mathcal{X}$
	are said to be \emph{isomorphic} (written $S \cong T$) if there exists $\varphi\in\mathcal{X}$
	such that $S\circ\varphi$ and $\varphi\circ T$ agree $\mu$-almost
	everywhere. Since there is a decomposition of each MPT into ergodic
	components (see, for example,
	Chapter 5 of \cite{VO16}), we often consider just the set $\mathcal{E}$
	of ergodic elements of $\mathcal{X}.$ Two well-known successes in the
	classification of MPTs are the Halmos-von Neumann classification
	of ergodic MPTs with pure point spectrum by countable subgroups of
	the unit circle \cite{HN42} and D. Ornstein's classification of Bernoulli
	shifts by their metric entropy \cite{Or70}. Another, lesser known, instance of
		success was obtained by C. Foia\c{s} and \c{S}. Str\u{a}til\u{a} \cite{FS68}, and this was recently generalized by F. Parreau \cite{P23}. Nonetheless, the isomorphism
	problem remains unsolved for general ergodic MPTs. 
	
	In contrast to the positive results in \cite{HN42} and \cite{Or70},
	a series of \emph{anti-classification} results have appeared, starting
	with the work of F. Beleznay and M. Foreman \cite{BF96}. Anti-classification
	results demonstrate the complexity of the isomorphism problem. To
	describe some of this work, we endow $\mathcal{X}$ with the weak
	topology. (Recall that $T_{n}\to T$ in the weak topology if and only
	if $\mu(T_{n}(A)\triangle T(A))\to0$ for every $A\in\mathcal{M}.)$
	This topology is compatible with a complete separable metric and hence
	makes $\mathcal{X}$ into a Polish space. Since $\mathcal{E}$ is
	a (dense) $G_{\delta}$ -set in $\mathcal{X}$, the induced topology
	on $\mathcal{E}$ is Polish as well. In contrast to the Halmos-von
	Neumann classification result, G. Hjorth \cite{Hj01} showed that
	there is no Borel way of associating algebraic invariants to ergodic
	MPTs that completely determines isomorphism. Moreover, Foreman and
	B. Weiss \cite{FW0} proved that there is no generic class (that is,
	a dense $G_{\delta}$ -set) of ergodic MPTs for which there is a
	Borel way of associating a complete algebraic invariant. Hjorth also
	proved that the isomorphism relation $\mathcal{R}$ is itself not
	a Borel set when viewed as a subset of $\mathcal{X}\times\mathcal{X}$.
	However, his proof used nonergodic transformations in an essential
	way. In the case of ergodic MPTs, Foreman, D. Rudolph, and Weiss
	\cite{FRW} proved that $\mathcal{R}\cap(\mathcal{E}\times\mathcal{E})$
	is not Borel. One way to interpret an equivalence relation not being
	a Borel set is to say that there is no inherently countable technique
	for determining whether or not two members of the set are equivalent.
	(See the survey article by Foreman \cite{Fo18} for more details.)
	Thus \cite{FRW} gives a precise formulation of the nonclassifiability
	of ergodic MPTs up to isomorphism.
	
	Recently, Foreman and A. Gorodetski \cite{FGpp} obtained anti-classification
	results in a different context. They considered topological equivalence
	of $C^{\infty}$ diffeomorphisms of a compact manifold $M.$ Two diffeomorphisms
	$f$ and $g$ of $M$ are said to be topologically equivalent if there
	is a homeomorphism $\varphi$ of $M$ such that $f\circ\varphi=\varphi\circ g.$
	In the case dim$(M)\ge2,$ Foreman and Gorodetski proved there is
	no Borel function from the $C^{\infty}$ diffeomorphisms to the reals
	that is a complete invariant for topological equivalence. Moreover,
	if dim$(M)\ge5,$ they proved the analog of the unclassifiability 
	result of \cite{FRW}, that is, the set of topologically equivalent
	pairs of diffeomorphisms is not Borel. 
	
	When the anti-classification results in \cite{FRW} became known,
	the question immediately arose whether these results are also true
	if ``isomorphism'' is replaced by ``Kakutani equivalence''. (See
	\cite{We14} for some history of this question.) In this paper we
	answer this question affirmatively. Our arguments build on the methods
	in \cite{FRW} and also imply the anti-classification result of \cite{FRW}. 
	
	Two ergodic automorphisms $T$ and $S$ of $(\Omega,\mathcal{M},\mu)$ 
	are \emph{Kakutani equivalent} if there exist positive measure subsets $A$ 
	and $B$ of $\Omega$ such that the first return maps $T_A$ and $S_B$ are
	isomorphic. (See Section \ref{subsec:Kak} for further details.) Kakutani equivalence 
	was introduced in \cite{Kt43} in order to study the equivalence of ergodic flows 
	up to isomorphism after a time change. A flow is defined to be a measure-preserving
	action of $\mathbb{R}$ on $(\Omega,\mathcal{M},\mu)$, that is, a collection 
	of measure-preserving maps $\{S_t\}$, for $t\in\mathbb{R}$, such that 
	$(\omega,t)\mapsto S_t(\omega)$ is jointly measurable from 
	$\Omega\times\mathbb{R}$ to $\Omega$ and $S_{t+s}=S_t\circ S_s$, for 
	$t,s\in \mathbb{R}$. Two flows $\{S_t\}$ and $\{T_t\}$ on 
	$(\Omega,\mathcal{M},\mu)$ are said to be \emph{isomorphic} if there exists 
	an invertible measure-preserving map $\varphi:\mathcal{D}\mapsto\mathcal{D}'$,
	where $\mathcal{D}$ and $\mathcal{D}'$ are full-measure subsets of $\Omega$ 
	that are, respectively, $\{S_t\}$ and $\{T_t\}$ invariant, and 
	$\varphi(S_t(\omega))=T_t(\varphi(\omega))$, for  $\omega\in\mathcal{D}$
	and $t\in\mathbb{R}$. It is natural to also allow a reparametrization 
	along the orbits, which is defined to be a jointly measurable map
	$\tau:\Omega\times\mathbb{R}\to\mathbb{R}$ such that $S_{\tau(\cdot,t)}(\cdot)$
	is again a flow on $\Omega$, and for every $\omega\in\mathcal{D}$ the map 
	$t\mapsto\tau(\omega,t)$ is an orientation-preserving 
	homeomorphism of $\mathbb{R}$. If there exists a reparametrization 
	$\tau$ such that $S_{\tau(\cdot,t)}(\cdot)$ is isomorphic to $\{T_t\}$, 
	then the flows $\{S_t\}$ and $\{T_t\}$ are said to be 
	\emph{isomorphic up to a time change}, or \emph{Kakutani equivalent}.
	(If ``orientation-preserving homeomorphism'' were relaxed to ``piecewise continuous
	invertible map'', then all ergodic flows would be isomorphic up to this more
	general type of reparametrization \cite{AmK42,Dye}, and therefore the corresponding 
	equivalence relation is not interesting in the present context.) 
	
	According to \cite{AmK42} and \cite{Kt43}, two ergodic flows are Kakutani equivalent if and only if they have Kakutani 
	equivalent transformations as cross-sections. Therefore our
	anti-classification result for Kakutani equivalence on $\mathcal{E}$ implies the non-classifiability of ergodic flows
	up to isomorphism after a time change.
	
	Since the Kakutani equivalence relation is
	weaker than isomorphism, one might expect classification to be simpler.
	In fact, until the work of J. Feldman \cite{Fe}, it seemed that there
	might be only three Kakutani equivalence classes in $\mathcal{E}$: zero entropy, finite
	positive entropy, and infinite entropy. However, Feldman found examples
	to show that each entropy class (zero, finite positive, and infinite)
	contains at least two non-Kakutani equivalent ergodic MPTs. A. Katok
	\cite{Ka77}, and Ornstein, Rudolph, and Weiss \cite{ORW}, independently,
	extended this to obtain uncountably many non-Kakutani equivalent examples
	in each entropy class.
	
	A precursor to the anti-classification results for isomorphism is
	the observation by Feldman \cite{Fe74} that there does not exist
	a Borel function from $\mathcal{E}$ to the reals such that two transformations
	are isomorphic if and only if the Borel function takes the same value
	at these transformations. His proof is an easy application of a 0-1
	law to the uncountable family of non-Bernoulli K-automorphisms constructed
	by Ornstein and P. Shields \cite{OS73}. It is interesting to note
	that the uncountable family of automorphisms constructed in \cite{ORW}
	can be substituted for the Ornstein-Shields examples to conclude that
	there is no Borel function from $\mathcal{E}$ to the reals whose
	values are complete invariants for Kakutani equivalence. The key common
	feature of the examples in \cite{OS73} and the examples in \cite{ORW}
	is that they are indexed by sequences of zeros and ones in such a
	way that two of the examples are isomorphic (in the case of \cite{OS73})
	or Kakutani equivalent (in the case of \cite{ORW}) if and only if
	the two corresponding sequences of zeros and ones agree except for
	possibly finitely many terms. 
	
	In the context of Kakutani equivalence for flows, M. Ratner \cite{Ra81}
	introduced the Kakutani invariant and used this to prove that for
	distinct integers $k,\ell>0,$ the product of $k$ copies of the horocycle
	flow with itself is not Kakutani equivalent to the product of $\ell$
	copies. More recently, A. Kanigowski, K. Vinhage, and D. Wei \cite{KVW}
	obtained an explicit formula for the Kakutani invariant for unipotent
	flows acting on compact quotients of semisimple Lie groups. It is
	an open question whether the Kakutani invariant is a complete invariant
	in the class of unipotent flows. 
	
	We also note that the analog of the result in \cite{FW0} fails for
	Kakutani equivalence in a trivial way, because the rank one transformations
	form a dense $G_{\delta}$-subset of $\mathcal{E},$ and all rank
	one transformations are zero entropy loosely Bernoulli, and therefore
	in the same Kakutani equivalence class. In case of isomorphism, Foreman,
	Rudolph and Weiss \cite[section 10]{FRW} use a theorem of J. King
	\cite{Ki86} to prove that the collection of pairs $(S,T)$ of rank
	one ergodic transformations such that $S$ and $T$ are isomorphic
	is a Borel subset of $\mathcal{E}\times\mathcal{E}.$ Nonetheless,
	no useful structure theorem to classify rank one transformations up
	to isomorphism is known. 
	
	We prove that the set $\mathcal{S}$ of pairs $(S,T)$ in $\mathcal{E}\times\mathcal{E}$
	such that $S$ and $T$ are Kakutani equivalent is not a Borel set. In fact, our main result is
		\begin{theo}
			The collection 
			\[
			\mathcal{S}\coloneqq \left\{ (S,T):S\text{ and }T\text{ are ergodic and Kakutani equivalent}\right\} \subseteq\mathcal{E}\times\mathcal{E}
			\]
			is a complete analytic set. In particular, it is not Borel.
	\end{theo}
	We follow the general scheme in \cite{FRW}, but with some notable
	differences. As in \cite{FRW} we consider the set, $\text{\ensuremath{\mathcal{T}}}rees,$
	which consists of countable rooted trees with arbitrarily long branches.
	Then, analogously to the approach in \cite{FRW}, we construct a continuous
	function $\Psi:\mathcal{T}\kern-.5mm rees\to\mathcal{E}$ such that $\Psi(\mathcal{T})$
	is Kakutani equivalent to $\Psi(\mathcal{T})^{-1}$ if and only if
	$\mathcal{T}$ has an infinite branch. According to a classical result
	\cite{Kechris}, the collection of $\mathcal{T}\kern-.5mm rees$ with one or
	more infinite branches is a complete analytic subset of $\mathcal{T}\kern-.5mm rees.$
	Since the set $\mathcal{S}$ of Kakutani equivalent pairs is an analytic set (see \cite[section 3.2]{GK4}), it follows that $\mathcal{S}$
	is a complete analytic set and hence not Borel. In fact, by an argument
	similar to that in \cite{FRW}, $\Psi(\mathcal{T})$ is actually \emph{isomorphic}
	to $\Psi(\mathcal{T})^{-1}$ if $\mathcal{T}$ has an infinite branch.
	Therefore any  equivalence relation on $\mathcal{E}$ that is weaker
	than (or equal to) isomorphism and stronger than (or equal to) Kakutani
	equivalence is complete analytic. This includes \emph{even} Kakutani
	equivalence (see Definition \ref{def:EvenEquiv} in Section \ref{subsec:On-even-Kakutani}),
	and $\alpha$-equivalence (defined in \cite{FJR94}). Indeed, we first
	prove our result for even Kakutani equivalence and then use the methods
	of \cite{ORW} to prove that for any $\mathcal{T}\in\mathcal{T}\kern-.5mm rees,$
	$\Psi(\mathcal{T})$ and $\Psi(\mathcal{T})^{-1}$ cannot be Kakutani
	equivalent unless they are evenly Kakutani equivalent. 
	
	Foreman, Rudolph, and Weiss use the elegant method of joinings to
	prove that for $\mathcal{T}\in\mathcal{T}\kern-.5mm rees$ such that $\mathcal{T}$
	has no infinite branch, the corresponding element of $\mathcal{E}$
	is not isomorphic to its inverse. However, as they announced in \cite{FRW06},
	they could also have used a finite coding argument and approximations
	in the $\overline{d}$-metric for this proof. We use the finite coding
	approach and apply the result in \cite{ORW} that for evenly equivalent
	processes $(S,P)$ and $(T,Q),$ there is a finite code defined on
	the $(T,Q)$ process such that the $\overline{f}$-distance between
	the $(S,P)$ process and the process coded from $(T,Q)$ is small. 
	
	As in \cite{FRW}, each level $s$ of a tree $\mathcal{T}\in\mathcal{T}\kern-.5mm rees$
	corresponds to an equivalence relation $\mathcal{Q}_{s}$ on certain
	finite words, and $\mathcal{Q}_{s+1}$ refines $\mathcal{Q}_{s}.$
	One of the main steps in the construction of $\Psi(\mathcal{T})$
	is a procedure for substituting $\mathcal{Q}_{s+1}$ equivalence classes
	into $\mathcal{Q}_{s}$ equivalence classes. Foreman, Rudolph, and
	Weiss use a random substitution method, while we use Feldman patterns
	(defined in Section \ref{sec:Feldman}) to execute a deterministic
	substitution. 
	
	The classification problem (with respect to isomorphism or Kakutani
	equivalence) can also be restricted to the class of smooth ergodic
	diffeomorphisms of a compact manifold $M$ that preserve a smooth
	measure. For dimensions greater than or equal to three, there are
	no known obstacles to realizing an arbitrary ergodic MPT as a diffeomorphism
	of a compact manifold except for the requirement, proved by A. Kushnirenko
	\cite{Ku65}, that the ergodic MPT have finite entropy. Thus it is
	not clear if the restriction to smooth ergodic diffeomorphisms changes
	the classification problem. But even with this restriction, Foreman
	and Weiss \cite{FW3} proved that the isomorphism relation is not
	Borel. Moreover, S. Banerjee and P. Kunde \cite{BK2} showed that
	the isomorphism relation is still not Borel if ``smooth'' is replaced
	by ``real analytic''. We obtain the analogous results for Kakutani
	equivalence both in the smooth and the real analytic settings. 
	
	To obtain their anti-classification results for smooth diffeomorphisms,
	Foreman and Weiss \cite{FW2} introduced a functor $\mathcal{F}$
	from certain odometer-based systems in \cite{FRW} to the so-called
	circular systems, which they showed in \cite{FW1} have a smooth realization
	on $\mathbb{D}^{2}$ and $\mathbb{T}^{2}$ via the untwisted Anosov-Katok
	method \cite{AK70}, while on $\mathbb{T}^{2}$ Banerjee and Kunde \cite{Ba17,BK2} extended
	these results to the real analytic case. We also make use of this
	machinery to obtain smooth and real analytic models of the images
	under $\mathcal{F}$ of the odometer-based systems that we construct
	using the Feldman patterns. The same argument as in \cite{FW2} shows
	that in our construction $\mathcal{F}(\Psi(\mathcal{T}))$ and $\mathcal{F}(\Psi(\mathcal{T})^{-1})$
	are isomorphic if $\mathcal{T}$ has an infinite branch. But the more
	difficult part, showing that $\mathcal{F}(\Psi(\mathcal{T}))$ and
	$\mathcal{F}(\Psi(\mathcal{T})^{-1})$ are \emph{not} Kakutani equivalent
	if $\mathcal{T}$ does \emph{not} have an infinite branch, requires
	a different argument from that in \cite{FW2}. In fact, in \cite{GeKu},
	we obtained an example of two \emph{non}-Kakutani equivalent odometer-based
	systems $S$ and $T$ such that $\mathcal{F}(S)$ and $\mathcal{F}(T)$
	are Kakutani equivalent. However, for the systems $\Psi(\mathcal{T})$
	that we construct in the present paper, if $\Psi(\mathcal{T})$ is
	not Kakutani equivalent to $\Psi(\mathcal{T})^{-1},$ then $\mathcal{F}(\Psi(\mathcal{T}))$
	is not Kakutani equivalent to $\mathcal{F}(\Psi(\mathcal{T})^{-1})$.
	Thus, we can transfer the non-Kakutani equivalence between $\Psi(\mathcal{T})$
	and $\Psi(\mathcal{T})^{-1}$ to non-Kakutani equivalence between
	$\mathcal{F}(\Psi(\mathcal{T}))$ and $\mathcal{F}(\Psi(\mathcal{T})^{-1})$
	if $\mathcal{T}$ does not have an infinite branch. This yields smooth
	and real analytic anti-classification results for isomorphism (already
	proved in \cite{FW3} and \cite{BK2}), Kakutani equivalence, and
	every equivalence relation in between these two. 
	
	Classification questions can also be considered in the context of group actions. For example, E. Gardella and M. Lupini \cite{GL} showed that for every nonamenable countable discrete group $\Gamma$, the relations of isomorphism and orbit equivalence of free ergodic (or weak mixing) measure preserving actions of $\Gamma$ on a standard probability space are not Borel. The analogous question for amenable countable discrete groups and the isomorphism relation is open. 
	
	 Since the submission of the original version of this paper, we have obtained two new applications of the main results. The first application is the non-classifiability of Kolmogorov automorphisms both up to isomorphism and up to Kakutani equivalence. This remains true when ``automorphisms'' are replaced by ``diffeomorphisms'' \cite{GK4}. The second application, which answers a question of M. Foreman, is contained in Section~\ref{sec:Sigma} of the current paper. We prove non-classifiability up to isomorphism of the ergodic automorphisms of a sigma-finite measure space. In both of these results, it is the \emph{non-classifiability up to Kakutani equivalence} in Theorems \ref{thm:CompleteAnalytic} and \ref{thm:Main} that is used in a crucial way to obtain \emph{non-classifiability up to isomorphism} in a different context (Kolmogorov automorphisms instead of ergodic automorphisms, or sigma-finite measures instead of finite measures). 
	
	\section{\label{sec:Preliminaries}Preliminaries}
	
	\subsection{\label{subsec:basicsET}Some terminology from ergodic theory}
	
	A measure space $(\Omega,\mathcal{M},\mu)$ is called
	a \emph{standard non-atomic probability space} if there is a bi-measurable measure-preserving bijection (modulo sets of measure zero) from $(\Omega,\mathcal{M},\mu)$ to $([0,1],\mathcal{B},\lambda)$, where $\mathcal{B}$ is the collection of Borel sets, and $\lambda$ is Lebesgue measure. Thus
	every invertible MPT (also called an \emph{automorphism}) of a standard
	non-atomic probability space is isomorphic to an invertible Lebesgue measure-preserving
	transformation on $[0,1]$. As in the
	introduction, we let $\mathcal{X}$ denote the group of MPTs on $[0,1]$
	endowed with the weak topology, where two MPTs are identified if
	they are equal on sets of full measure. 
	
	While in this model we fix the measure and look at transformations
	preserving that measure, there is also a dual viewpoint of fixing
	the transformation and considering the collection of its invariant
	measures. The transformation that we use is the left shift $sh:\Sigma^{\mathbb{Z}}\to\Sigma^{\mathbb{Z}},$ defined by
	\[
	sh((x_n)_{n=-\infty}^{\infty})=(x_{n+1})_{n=-\infty}^{\infty}
	\] 
	on $(\Sigma^{\mathbb{Z}},\mathcal{C})$,
	where $\Sigma$ is 
	a countable set and $\mathcal{C}$ is the $\sigma$-algebra of Borel sets in $\Sigma^{\mathbb{Z}}.$ 
	We endow the set of shift-invariant measures with the weak$^*$ topology. 
	In our paper, $\Sigma$ is finite and for each shift-invariant measure $\nu$ that we construct, there is also an explicit cutting-and-stacking procedure that yields an ergodic Lebesgue measure-preserving automorphism $T$ of $[0,1]$ such that 
	$\left([0,1],\mathcal{B},\lambda,T\right)$ is isomorphic to $\left(\Sigma^{\mathbb{Z}},\mathcal{C},\nu,sh\right).$
	This is a special case of a more general result of Foreman \cite{Fo}:
	there is a Borel bijection
	$\psi$ between the aperiodic Lebesgue measure-preserving transformations on $[0,1]$ and nonatomic
	shift-invariant measures on $\Sigma^{\mathbb{Z}}$  taking comeager sets to comeager
	sets such that $\left([0,1],\mathcal{B},\lambda,T\right)$ is isomorphic to $\left(\Sigma^{\mathbb{Z}},\mathcal{C},\psi(T),sh\right).$
	In particular,
	the choice of model for measure-preserving transformations is not
	important for the Borel or analytic distinction.
	
	Another important viewpoint is to assign a process to a MPT. For this
	purpose, we recall that a \emph{partition} $P$ of a standard measure
	space $(\Omega,\mathcal{M},\mu)$ is a collection $P=\left\{ c_{\sigma}\right\} _{\sigma\in\Sigma}$
	of subsets $c_{\sigma}\in\mathcal{M}$ with $\mu(c_{\sigma}\cap c_{\sigma'})=0$
	for all $\sigma\neq\sigma'$ and $\mu\left(\bigcup_{\sigma\in\Sigma}c_{\sigma}\right)=1$,
	where $\Sigma$ is a finite set of indices. Each $c_{\sigma}$
	is called an \emph{atom} of the partition $P$. 
	For two
	partitions $P$ and $Q$ we define the join of $P$ and $Q$ to be the partition $P\vee Q=\left\{ c\cap d:c\in P,d\in Q\right\}$. For a sequence of partitions $\{P_{n}\}_{n=1}^{\infty}$ we inductively define $\vee^N_{n=1}P_n$ and we let
		$\vee_{n=1}^{\infty}P_{n}$ be the smallest $\sigma$-algebra containing $\vee^N_{n=1}P_n$ for each $N \in \N$.
	We say that a sequence of partitions $\{P_{n}\}_{n=1}^{\infty}$ is
	a \emph{generating sequence} if 
	$\vee_{n=1}^{\infty}P_{n} = \mathcal{M}$. We also have a standard
	notion of a \emph{distance between two ordered partitions}: If $P=\{c_{i}\}_{i=1}^{N}$
	and $Q=\{d_{i}\}_{i=1}^{N}$ are two ordered partitions with the same number of atoms, then
	we define 
	\[
	d(P,Q)=\sum_{i=1}^{N}\mu(c_{i}\triangle d_{i}),
	\]
	where $\triangle$ denotes the symmetric difference. 
For two partitions $P=\{c_{i}\}_{i=1}^{N}$
		and $Q=\{d_{j}\}_{j=1}^{M}$, we say that $P\subset^{\varepsilon} Q$ if there exists a partition $Q'=\{d'_{i}\}_{i=1}^{N}$ such that each $d_i'$ is a union of atoms of $Q$ and $d(P,Q')<\varepsilon.$
	
	For an automorphism $T$ of $(\Omega,\mathcal{M},\mu)$ and a
	partition $P=\left\{ c_{\sigma}\right\} _{\sigma\in\Sigma}$
	of $\Omega$ we can define the $(T,P)$-name for almost every $x\in\Omega$
	by 
	\[
	(a_n)_{n=-\infty}^{\infty}\in\Sigma^{\mathbb{Z}}\text{ with }T^{i}(x)\in c_{a_{i}}.
	\]
	Note that the $(T,P)$-name of $T(x)$ is the left shift of the $(T,P)$-name of $x$.
	If the partition $P$ is a \emph{generator}
	(that is, $\left\{ T^{n}P\right\} _{n=-\infty}^{\infty}$ is a generating
	sequence),
	then $(X,\mu,T)$ is isomorphic to $(\Sigma^{\mathbb{Z}},\nu,sh)$
	with the measure $\nu$ satisfying $\nu\{(x_n)_{n=-\infty}^{\infty}: x_n=a_n \text{\ for\ } n=j,j+1,...,k\}=\mu(T^{-j}(c_{a_j}))\cap T^{-(j+1)}(c_{a_{j+1}})\cap\cdots\cap T^{-k}(c_{a_k})).$ The pair $(T,P)$ is called a \emph{process}. 
	
	Recall that the \emph{Hamming distance}
	between two strings of symbols $a_{1}\dots a_{n}$ and $b_{1}\dots b_{n}$
	is defined by $\overline{d}\left(a_{1}\dots a_{n},b_{1}\dots b_{n}\right)=\frac{1}{n}|\left\{ i:a_{i}\neq b_{i}\right\} |$.
	It can be extended to infinite words $w=\dots a_{-2}a_{-1}a_{0}a_{1}a_{2}...$
	and $w^{\prime}=\dots b_{-2}b_{-1}b_{0}b_{1}b_{2}\dots$ by $\overline{d}(w,w^{\prime})=\limsup_{n\to\infty}\overline{d}\left(w_{n},w_{n}^{\prime}\right)$,
	where $w_{n}=a_{-n}a_{-n+1}\dots a_{n-1}a_{n}$ and $w_{n}^{\prime}=b_{-n}b_{-n+1}\dots b_{n-1}b_{n}$
	are the truncated words.

	\subsection{Induced maps and Kakutani equivalence} \label{subsec:Kak}
	
	For an automorphism $T$ of $(\Omega,\mathcal{M},\mu)$ and $A \in \mathcal{M}$ with $\mu(A)>0$, let $n_A(x)= \inf \Meng{i\in \mathbb{Z}^+}{T^i(x) \in A}$ be the \emph{first return time} to $A$. Then we define the \emph{induced map} (also called the \emph{first return map}) $T_A : A \to A$ by $T^{n_A(x)}(x)$. Furthermore, the normalized induced measure on $A$ is defined by $\mu_A(E) = \frac{\mu(A\cap E)}{\mu(A)}$ for $E \in \mathcal{M}$. We can now state the definition of Kakutani equivalence for MPTs.
	\begin{defn}\label{defn:Kakutani}
		Two transformations $T:(X,\mu)\to(X,\mu)$ and
		$S:(Y,\nu)\to(Y,\nu)$ are said to be \emph{Kakutani equivalent} if
		there are subsets $A\subseteq X$ with $\mu(A)>0$ and $B\subseteq Y$ with
		$\nu(B)>0$ such that $(T_{A},\mu_A)$ and $(S_{B},\nu_B)$ are isomorphic
		to each other.
	\end{defn}
	Kakutani showed in \cite{Kt43} that $T$ and $S$ are Kakutani equivalent if and only if they are isomorphic to measurable cross-sections of the same ergodic flow. We refer to \cite{ORW} for terminology and further characterizations of Kakutani equivalence.
	
	\subsection{\label{subsec:Symbolic-systems}Symbolic systems}
	
	An \emph{alphabet} is a countable or finite collection of symbols.
	In the following, let $\Sigma$ be a finite alphabet endowed with
	the discrete topology. Then $\Sigma^{\mathbb{Z}}$ with the product
	topology is a separable, totally disconnected and compact space. A
	usual base of the product topology is given by the collection of cylinder
	sets of the form $\left\langle u\right\rangle _{k}=\left\{ f\in\Sigma^{\mathbb{Z}}:f\upharpoonright[k,k+n)=u\right\} $
	for some $k\in\mathbb{Z}$ and finite sequence $u=\sigma_{0}\dots\sigma_{n-1}\in\Sigma^{n}$.
	For $k=0$ we abbreviate this by $\left\langle u\right\rangle $.
	
	The shift map $sh:\Sigma^{\mathbb{Z}}\to\Sigma^{\mathbb{Z}}$
	is a homeomorphism. If $\mu$ is a shift-invariant Borel measure,
	then the measure-preserving dynamical system $\left(\Sigma^{\mathbb{Z}},\mathcal{B},\mu,sh\right)$
	is called a \emph{symbolic system}. The closed support of $\mu$ is
	a shift-invariant subset of $\Sigma^{\mathbb{Z}}$ called a \emph{symbolic
		shift} or \emph{sub-shift}. 
	
	The symbolic shifts that we use are described by specifying a collection of words. A
	word $w$ in $\Sigma$ is a finite sequence of elements of $\Sigma$,
	and we denote its length by $|w|$.
	\begin{defn}
		\label{def:ConstrSeq} A sequence of collections of words $\left(W_{n}\right)_{n\in\mathbb{N}}$,
		where $\mathbb{N}=\left\{ 0,1,2,\dots\right\} $, satisfying the following
		properties is called a \emph{construction sequence}:
		\begin{enumerate}
			\item for every $n\in\mathbb{N}$ all words in $W_{n}$ have the same length
			$h_{n}$,
			\item each $w\in W_{n}$ occurs at least once as a subword of each $w^{\prime}\in W_{n+1}$,
			\item there is a summable sequence $\left(\varepsilon_{n}\right)_{n\in\mathbb{N}}$
			of positive numbers such that for every $n\in\mathbb{N}$, every word
			$w\in W_{n+1}$ can be uniquely parsed into segments $u_{0}w_{1}u_{1}w_{1}\dots w_{l}u_{l}$
			such that each $w_{i}\in W_{n}$, each $u_{i}$ (called spacer or
			boundary) is a word in $\Sigma$ of finite length and for this parsing
			\[
			\frac{\sum_{i=0}^{l}|u_{i}|}{h_{n+1}}<\varepsilon_{n+1}.
			\]
		\end{enumerate}
	\end{defn}
	
	We will often call words in $W_{n}$\emph{ $n$-words} or \emph{$n$-blocks},
	while a general concatenation of symbols from $\Sigma$ is called
	a \emph{string}. A \emph{substring} of a string $x$ is a string of symbols that occur consecutively within $x.$ We also associate a symbolic shift with a construction
	sequence: Let $\mathbb{K}$ be the collection of $x\in\Sigma^{\mathbb{Z}}$
	such that every finite  substring of $x$ occurs 
	as a substring of
	some $w\in W_{n}$. Then $\mathbb{K}$ is a closed shift-invariant
	subset of $\Sigma^{\mathbb{Z}}$ that is compact since $\Sigma$ is finite. 
	\begin{rem*}
		The symbolic shifts built from construction sequences can be realized as
		transformations built by cutting-and-stacking constructions. Spacers are
		not used in our constructions until Section \ref{sec:Transfer}.
	\end{rem*}
	In order to be able to unambiguously parse elements of $\mathbb{K}$
	we will use construction sequences consisting of uniquely readable
	words.
	\begin{defn}
		Let $\Sigma$ be an alphabet and $W$ be a collection of finite words
		in $\Sigma$. 
		Then $W$ is \emph{uniquely readable} if and only if whenever $u,v,w\in W$
		and $uv=pws$ with $p$ and $s$ strings of symbols in $\Sigma$,
		then either $p$ or $s$ is the empty word.
	\end{defn}
	
	Moreover, our $(n+1)$-words will be strongly uniform in the $n$-words
	as defined below. 
	\begin{defn}
		We call a construction sequence $\left(W_{n}\right)_{n\in\mathbb{N}}$ \emph{uniform}
		if there is a summable sequence $(\varepsilon_{n})_{n\in\mathbb{N}}$
		of positive numbers, and for each $n\in\mathbb{N},$ a map $d_{n}:W_{n}\to(0,1)$
		such that for all words $w'\in W_{n+1}$ and $w\in W_{n}$ we have
		\[
		\Big|\frac{r(w,w')}{h_{n+1}/h_{n}}-d_{n}(w)\Big|<\frac{\varepsilon_{n+1}}{ |W_n|},
		\]
		where $r(w,w')$ is the number of occurrences of $w$ in $w'$.
		Moreover, the construction sequence is called\emph{ strongly uniform}
		if for each $n\in\mathbb{N}$ there is a constant $c=c(n)>0$ such that
		for all words $w^{\prime}\in W_{n+1}$ and $w\in W_{n}$ we have $r(w,w')=c$. 
	\end{defn}
	
	In the following we will identify $\mathbb{K}$ with the symbolic
	system $\left(\mathbb{K},sh\right)$. We introduce the following natural
	set $S$ which will be of measure one for measures that we consider.
	\begin{defn}
		Suppose that $\left(W_{n}\right)_{n\in\mathbb{N}}$ is a construction
		sequence for a symbolic system $\mathbb{K}$ with each $W_{n}$ uniquely
		readable. Let $S$ be the collection of $x\in\mathbb{K}$ such that
		there are sequences of natural numbers $\left(a_{n}\right)_{n\in\mathbb{N}}$,
		$\left(b_{n}\right)_{n\in\mathbb{N}}$ going to infinity such that
		for all $m\in\mathbb{N}$ there is $n\in\mathbb{N}$ such that $x\upharpoonright[-a_{m},b_{m})\in W_{n}$.
	\end{defn}
	
	We note that $S$ is a dense shift-invariant $\mathcal{G}_{\delta}$
	subset of $\mathbb{K}$ and we recall the following properties from
	\cite[Lemma 11]{FW1} and \cite[Lemma 2.11]{FW2}.
	\begin{fact}
		\label{fact:MeasureConstrSeq}Fix a construction sequence $\left(W_{n}\right)_{n\in\mathbb{N}}$
		for a symbolic system $\mathbb{K}$ in a finite alphabet $\Sigma$.
		Then:
		\begin{enumerate}
			\item $\mathbb{K}$ is the smallest shift-invariant closed subset of $\Sigma^{\mathbb{Z}}$
			such that for all $n\in\mathbb{N}$ and $w\in W_{n}$, $\mathbb{K}$
			has non-empty intersection with the basic open interval $\left\langle w\right\rangle \subset\Sigma^{\mathbb{Z}}$.
			\item Suppose that $\left(W_{n}\right)_{n\in\mathbb{N}}$ is a uniform construction
			sequence. Then there is a unique non-atomic shift-invariant measure
			$\nu$ on $\mathbb{K}$ concentrating on $S$ and $\nu$ is ergodic.
			\item If $\nu$ is a shift-invariant measure on $\mathbb{K}$ concentrating
			on $S$, then for $\nu$-almost every $s\in S$ there is $N=N(s)\in\mathbb{N}$
			such that for all $n>N$ there are $a_{n}\leq0<b_{n}$ such that $s\upharpoonright[a_{n},b_{n})\in W_{n}$.
		\end{enumerate}
	\end{fact}
	
	Since our symbolic systems $\mathbb{K}\cong\left(\Sigma^{\mathbb{Z}},\mathcal{B},\nu,sh\right)$
	will be built from  uniquely readable uniform construction sequences,
	they will automatically be ergodic. To each symbolic system we will
	also consider its inverse $\mathbb{K}^{-1}$ which stands for $\left(\mathbb{K},sh^{-1}\right)$.
	Since it will often be convenient to have the shifts going in the
	same direction, we also introduce another convention.
	\begin{defn}
		If $w$ is a finite or infinite string, we write $rev(w)$ for the
		reverse string of $w$. In particular, if $x$ is in $\mathbb{K}$
		we define $rev(x)$ by setting $rev(x)(k)=x(-k)$. Then for $A\subseteq\mathbb{K}$
		we define $rev(A)=\left\{ rev(x):x\in A\right\} .$ 
		
		If we explicitly view a finite word $w$ positioned at a location
		interval $[a,b)$, then we take $rev(w)$ to be positioned at the same interval
		$[a,b)$ and we set $rev(w)(k)=w(a+b-(k+1))$. For a collection $W$
		of words, $rev(W)$ is the collection of reverses of words in $W$.
	\end{defn}
	
	Then we introduce the symbolic system $\left(rev(\mathbb{K}),sh\right)$
	as the one built from the construction sequence $\left(rev(W_{n})\right)_{n\in\mathbb{N}}$.
	Clearly, the map sending $x$ to $rev(x)$ is a canonical isomorphism
	between $\left(\mathbb{K},sh^{-1}\right)$ and $\left(rev(\mathbb{K}),sh\right)$.
	We often abbreviate the symbolic system $\left(rev(\mathbb{K}),sh\right)$
	as $rev(\mathbb{K})$. If we have a cutting-and-stacking model of $\left(rev(\mathbb{K}),sh\right)$ where each level (except the top level) of each column is mapped
	to the level above it, then the model for $\left(\mathbb{K},sh^{-1}\right)$ would map each level (except the bottom level) to the level below it. In analogy with 
	$\left(rev(\mathbb{K}),sh\right)$, we could instead reverse the order of the levels in each column and map the levels upward. 
	
	\subsection{\label{subsec:Odometer-Based-Systems}Odometer-Based Systems}
	
	The systems constructed in the first part of the paper to prove our anti-classification
	result for ergodic MPTs will turn out to be so-called odometer-based
	systems. In this subsection we review the definition.
	
	Let $\left(k_{n}\right)_{n\in\mathbb{N}}$ be a sequence of natural
	numbers $k_{n}\geq2$ and 
	\[
	O=\prod_{n\in\mathbb{N}}\left(\mathbb{Z}/k_{n}\mathbb{Z}\right)
	\]
	be the $\left(k_{n}\right)_{n\in\mathbb{N}}$-adic integers. Then
	$O$ has a compact abelian group structure and hence carries a Haar
	measure $\mu$. We define a transformation $T:O\to O$ to be addition
	by $1$ in the $\left(k_{n}\right)_{n\in\mathbb{N}}$-adic integers
	(that is, the map that adds one in $\mathbb{Z}/k_{0}\mathbb{Z}$ and carries
	right). Then $T$ is a $\mu$-preserving invertible transformation,
	called an \emph{odometer transformation}, which is ergodic and has discrete
	spectrum. 
	
	We now define the collection of symbolic systems that have odometer
	systems as their timing mechanism to parse the sequence of symbols constituting a typical element of the system into $n$
		blocks, for $n=1,2,\dots$. See Section~\ref{subsubsec:Odometer} for a definition of the associated factor map from typical points in the symbolic system to $O$, where $k_n$ is as defined below.
	\begin{defn}
		Let $\left(W_{n}\right)_{n\in\mathbb{N}}$ be a uniquely readable
		construction sequence with $W_{0}=\Sigma$ and $W_{n+1}\subseteq\left(W_{n}\right)^{k_{n}}$
		for every $n\in\mathbb{N}$. The associated symbolic shift will be
		called an \emph{odometer-based system}.
	\end{defn}
	
	Thus, odometer-based systems are those built from construction sequences
	$\left(W_{n}\right)_{n\in\mathbb{N}}$ such that the words in $W_{n+1}$
	are concatenations of a fixed number $k_{n}$ of words in $W_{n}$.
	Hence, the words in $W_{n}$ have length 
	\[
	h_{n}=\prod_{i=0}^{n-1}k_{i},
	\]
	if $n>0$, and $h_{0}=1$. Moreover, the spacers in part 3 of Definition
	\ref{def:ConstrSeq} are all the empty words (that is, an odometer-based
	transformation can be built by a cutting-and-stacking construction
	using no spacers).
	\begin{rem*}
		By \cite{FW4}, any finite entropy system
		that has an odometer factor can be represented as an odometer-based
		system.
	\end{rem*}
	
	\subsection{\label{subsec:f-metric}The $\overline{f}$ metric}
	
	Feldman \cite{Fe} introduced a notion of distance, now called $\overline{f}$,
	between strings of symbols. The $\overline{f}$ distance allows more flexibility
		in matching than the $\overline{d}$ distance (defined in Section~\ref{subsec:basicsET}), because we can match symbols in different positions, as long as we preserve
		the ordering of the symbols being matched.
	\begin{defn}
		\label{def:fbar}A \emph{match} between two strings of symbols $a_{1}a_{2}\dots a_{n}$
		and $b_{1}b_{2}\dots b_{m}$ from a given alphabet $\Sigma,$ is a
		collection $\mathcal{M}$ of pairs of indices $(i_{s},j_{s})$, $s=1,\dots,r$
		such that $1\le i_{1}<i_{2}<\cdots<i_{r}\le n$, $1\le j_{1}<j_{2}<\cdots<j_{r}\le m$
		and $a_{i_{s}}=b_{j_{s}}$ for $s=1,2,\dots,r.$ A \emph{pairing}
		is defined the same way as a match, except we drop the requirement
		that $a_{i_{s}}=b_{j_{s}}.$Then 
		\begin{equation}
			\begin{array}{ll}
				\overline{f}(a_{1}a_{2}\dots a_{n},b_{1}b_{2}\dots b_{m})=\hfill\\
				{\displaystyle 1-\frac{2\sup\{|\mathcal{M}|:\mathcal{M}\text{\ is\ a\ match\ between\ }a_{1}a_{2}\cdots a_{n}\text{\ and\ }b_{1}b_{2}\cdots b_{m}\}}{n+m}.}
			\end{array}\label{eq:cl}
		\end{equation}
		We will refer to $\overline{f}(a_{1}a_{2}\cdots a_{n},b_{1}b_{2}\cdots b_{m})$
		as the ``$\overline{f}$-distance'' between $a_{1}a_{2}\cdots a_{n}$
		and $b_{1}b_{2}\cdots b_{m},$ even though $\overline{f}$ does not
		satisfy the triangle inequality unless the strings are all of the
		same length. A match $\mathcal{M}$ is called a\emph{ best possible
			match} if it realizes the supremum in the definition of $\overline{f}$.
		The extension of the definition of $\overline{f}$-distance from finite words to 
		infinite words is analogous to the extension for $\overline{d}$-distance.
		
		If $x=x_{1}x_{2}\cdots x_{n}$ and $y=y_{1}y_{2}\cdots y_{n}$ are
		decompositions of the strings of symbols $x$ and $y$ into substrings
		such that in a given match (or pairing) between $x$ and $y$, all
		matches (or pairings) are between symbols in $x_{i}$ and $y_{i},$
		for $i=1,2,\dots,n,$ then we say that $x_{i}$ \emph{corresponds}
		to $y_{i}$ (and vice versa), under this match (or pairing). If we
		are given a match or pairing between $x=x_{1}x_{2}\cdots x_{n}$ and
		$y,$ then we can decompose $y$ into $y_{1}y_{2}\cdots y_{n}$ so
		that $x_{i}$ corresponds to $y_{i},$ for $i=1,2,\dots,n.$ The choice
		of $y_{1},y_{2},\dots,y_{n}$ is not unique, but this does not matter
		in any of our arguments.
	\end{defn}
	
	\begin{rem*}
		In the proof of Lemma \ref{lem:symbol by block replacement}, we will
		use certain pairings to give an upper bound on $|\mathcal{M}|,$ for
		a best possible match $\mathcal{M}.$
	\end{rem*}
	\begin{rem*}
		Alternatively, one can view a match as an injective order-preserving
		function $\pi:\mathcal{D}(\pi)\subseteq\left\{ 1,\dots,n\right\} \to\mathcal{R}(\pi)\subseteq\left\{ 1,\dots,m\right\} $
		with $a_{i}=b_{\pi(i)}$ for every $i\in\mathcal{D}(\pi)$. Then $\overline{f}\left(a_{1}\dots a_{n},b_{1}\dots b_{m}\right)=1-\max\left\{ \frac{2|\mathcal{D}(\pi)|}{n+m}:\pi\text{ is a match}\right\} $.
	\end{rem*}
	%

	The following simple properties of $\overline{f}$ from  \cite[properties~2.5--2.7]{GeKu},
	that were already used in \cite{Fe} and \cite{ORW}, will appear
	frequently in our arguments. These properties can be proved easily
	by considering the \emph{fit} $1-\bar{f}(a,b)$ between two strings
	$a$ and $b$.
	\begin{fact}
		\label{fact:omit_symbols}Suppose $a$ and $b$ are strings of symbols
		of length $n$ and $m,$ respectively, from an alphabet $\Sigma$.
		If $\tilde{a}$ and $\tilde{b}$ are strings of symbols obtained by
		deleting at most $\lfloor\gamma(n+m)\rfloor$ terms from $a$ and
		$b$ in total, where $0<\gamma<1$, then 
		\begin{equation}
			\overline{f}(a,b)\ge\overline{f}(\tilde{a},\tilde{b})-2\gamma.\label{eq:omit_symbols}
		\end{equation}
		Moreover, if there exists a best possible match between $a$ and
		$b$ such that no term that is deleted from $a$ and $b$ to form
		$\tilde{a}$ and $\tilde{b}$ is matched with a non-deleted term,
		then 
		\begin{equation}
			\overline{f}(a,b)\ge\overline{f}(\tilde{a},\tilde{b})-\gamma.\label{eq:omit_symbols2}
		\end{equation}
		Likewise, if $\tilde{a}$ and $\tilde{b}$ are obtained by adding
		at most $\lfloor\gamma(n+m)\rfloor$ symbols to $a$ and $b$, then
		($\ref{eq:omit_symbols2}$) holds.
	\end{fact}
	
	\begin{fact}
		\label{fact:substring_matching}If $x=x_{1}x_{2}\cdots x_{n}$ and
		$y=y_{1}y_{2}\cdots y_{n}$ are decompositions of the strings of symbols
		$x$ and $y$ into corresponding substrings under a best possible
		$\overline{f}$-match between $x$ and $y,$ that is, one that achieves
		the minimum in the definition of $\overline{f},$ then 
		\[
		\overline{f}(x,y)=\sum_{i=1}^{n}\overline{f}(x_{i},y_{i})v_{i},
		\]
		where 
		\begin{equation}
			v_{i}=\frac{|x_{i}|+|y_{i}|}{|x|+|y|}.\label{eq:substring_matching}
		\end{equation}
	\end{fact}
	
	\begin{fact}
		\label{fact:string_length}If $x$ and $y$ are strings of symbols
		such that $\overline{f}(x,y)\le\gamma,$ for some $0\le\gamma<1,$
		then 
		\begin{equation}
			\left(\frac{1-\gamma}{1+\gamma}\right)|x|\leq|y|\le\left(\frac{1+\gamma}{1-\gamma}\right)|x|.\label{eq:string_length}
		\end{equation}
		
	\end{fact}
	
	We also recall the following elementary fact from \cite[Proposition 2.3, p.85]{ORW}.
	\begin{fact}
		\label{fact:AddSymbol}Suppose $a$ and $b$ are strings of symbols
		of length $n$ and $m,$ respectively, from an alphabet $\Sigma$.
		Let $h$ be a symbol that is not an element in $\Sigma$. Suppose
		$\tilde{a}$ and $\tilde{b}$ are strings of symbols obtained by inserting
		at most $(\beta-1)n$ symbols $h$ in $a$ and at most $(\beta-1)m$
		symbols $h$ in $b$. Then 
		\[
		\overline{f}\left(\tilde{a},\tilde{b}\right)\geq\frac{\overline{f}(a,b)}{\beta}.
		\]
	\end{fact}
	
	The next result provides lower bounds on the $\overline{f}$ distance when
	symbols $a_i$ and $b_j$ are replaced by blocks
	$A_{a_i}$ and $A_{b_j}$. The proof utilizes ideas from \cite[Proposition~1.1, p.79]{ORW}.
	\begin{lem}[Symbol by block replacement]
		\label{lem:symbol by block replacement}Suppose that $A_{a_{1}},A_{a_{2}},\dots,A_{a_{n}}$
		and $A_{b_{1}},A_{b_{2}},\cdots,A_{b_{m}}$ are blocks of symbols
		with each block of length $L.$ Assume that $\alpha\in(0,1/7),R\ge2,$
		and for all substrings $C$ and $D$ consisting of consecutive symbols
		from $A_{a_{i}}$ and $A_{b_{j}},$ respectively, with $|C|,|D|\ge L/R,$
		we have 
		\[
		\overline{f}(C,D)\ge\alpha\text{, if }a_{i}\ne b_{j}.
		\]
		Then 
		\[
		\overline{f}(A_{a_{1}}A_{a_{2}}\cdots A_{a_{n}},A_{b_{1}}A_{b_{2}}\cdots A_{b_{m}})>\alpha\overline{f}(a_{1}a_{2}\dots a_{n},b_{1}b_{2}\dots b_{m})-\frac{1}{R}.
		\]
	\end{lem}

	\begin{proof}
		We may assume $n\le m.$ Suppose $\mathcal{M}$ is a best possible
		match between $A_{a_{1}}A_{a_{2}}\cdots A_{a_{n}}$ and $A_{b_{1}}A_{b_{2}}\cdots A_{b_{m}}.$
		Recall that $|\mathcal{M}|$ denotes the number of pairs of symbols
		in $A_{a_{1}}A_{a_{2}}\cdots A_{a_{n}}$ and $A_{b_{1}}A_{b_{2}}\cdots A_{b_{m}}$
		that are matched by $\mathcal{M}.$ We use the analogous notation
		$|\mathcal{N}|$ for a match $\mathcal{N}$ between the strings $a_{1}a_{2}\cdots a_{n}$
		and $b_{1}b_{2}\cdots b_{m}.$ 
		
		\noindent\emph{Claim. }There exists a match $\mathcal{N}$ (not necessarily
		a best possible match) between the strings $a_{1}a_{2}\dots a_{n}$
		and $b_{1}b_{2}\cdots b_{m}$ such that 
		\begin{equation}
			L(n+m)-2|\mathcal{M}|\ge\alpha\Big\{ L(n+m)-2L|\mathcal{N}|-\frac{3L(n+m)}{R}\Big\}.\label{eq:1}
		\end{equation}
		The lemma will follow from this claim, because 
		\[
		\overline{f}(A_{a_{1}}A_{a_{2}}\cdots A_{a_{n}},A_{b_{1}}A_{b_{2}}\cdots A_{b_{m}})=1-\frac{2|\mathcal{M}|}{L(n+m)}\text{\ and\ }
		\]
		
		\[
		\begin{array}{ccc}
			\frac{\alpha}{L(n+m)}\big\{ L(n+m)-2L|\mathcal{N}|-\frac{3L(n+m)}{R}\big\} & \ge & \alpha\overline{f}(a_{1}a_{2}\cdots a_{n},b_{1}b_{2}\cdots b_{m})-\frac{3\alpha}{R}\\
			& > & \alpha\overline{f}(a_{1}a_{2}\cdots a_{n},b_{1}b_{2}\cdots b_{m})-\frac{1}{R}.
		\end{array}
		\]
		
		\noindent\emph{Proof of Claim. }We will modify $\mathcal{M}$ in
		a series of $i$ steps, for $i=1,2,\dots,n,$ obtaining pairings $\mathcal{M}_{i}$
		between $A_{a_{1}}A_{a_{2}}\cdots A_{a_{n}}$ and $A_{b_{1}}A_{b_{2}}\cdots A_{b_{m}}$
		such that $|\mathcal{M}|\le|\mathcal{M}_{1}|\le\cdots\le$ $|\mathcal{M}_{n-1}|\le|\mathcal{M}_{n}|.$ The positions of the modifications will move from left to right with each step. 
		We will simultaneously construct a sequence of matches $\mathcal{N}_{i}$
		between $a_{1}a_{2}\dots a_{n}$ and $b_{1}b_{2}\dots b_{m}$ such
		that $\mathcal{N}_{1}\subseteq\mathcal{N}_{2}\subseteq\cdots\subseteq\mathcal{N}_{n}.$
		At the end of this construction, we will show that the inequality (\ref{eq:1}) holds with $\mathcal{N}=\mathcal{N}_{n}$
		and $\mathcal{M}$ replaced by $\mathcal{M}_{n}.$ This implies (\ref{eq:1})
		as written. 
		
		We let $\mathcal{M}_{0}=\mathcal{M}$ and $\mathcal{N}_{0}=\emptyset$. In the inductive construction of $\mathcal{M}_{i}$ and $\mathcal{N}_{i}$ (as described below), at most one pair of matching symbols is added to $\mathcal{N}_{i-1}$ to form $\mathcal{N}_{i},$ and each time such an addition occurs, we allow up to $L$ pairs to be in the part of $\mathcal{M}_{i}$ that is newly specified at stage $i$.
		
	For any pairing $\mathcal{K}$ between $A_{a_1}\cdots A_{a_n}$ 
			and $A_{b_1}\cdots A_{b_m}$, and substrings $A$ and $B$ of  
			$A_{a_1}\cdots A_{a_n}$ and $A_{b_1}\cdots A_{b_m}$, 
			respectively,  we let $\mathcal{K}\big| (A,B)$ denote the 
			collection of pairs in $\mathcal{K}$ such that the first 
			symbol in the pair is in $A$ and the second symbol in the pair is in $B$.
		
		Suppose $\mathcal{K}$ is a match between a substring $C$ of a particular $A_{a_i}$ and a substring $D$ of a particular $A_{b_j}$, where $a_i\ne b_j.$ Then there are three possibilities:
			\begin{enumerate}[label=(\Roman*)]
				\item $|C|\ge L/R$ and $|D|\ge L/R.$ Then, by assumption, $\overline{f}(C,D)>\alpha,$ which implies that $2|\mathcal{K}|<\alpha(|C|+|D|).$
				\item One of $C$ and $D$ has length less than $L/R$ and the other has length at least $4L/(3R).$ Then, by Fact 12, $\overline{f}(C,D)>1/7>\alpha,$ and again $2|\mathcal{K}|<\alpha(|C|+|D|).$
				\item Both $C$ and $D$ have length less than $4L/(3R)$ and at least one has length less than $L/R$. This will be called the \emph{short substring case}. In this case we simply use the estimate $|C|+|D|<7L/(3R).$
			\end{enumerate}

		We now define a type of pairing that will occur in our construction. Suppose $C_1$ and $D_1$ are substrings of $A_{a_1}\cdots A_{a_n}$ and $A_{b_1}\cdots A_{b_m}$, respectively, each of length at least $k$. Then a \emph{parallelogram pairing} of length $k$ on $C_1$ and $D_1$ consists of 
			the $k$ pairs obtained by pairing each element of a substring of $C_1$ of length $k$ with an element of a substring of $D_1$ of length $k$.
			
			The following observation, which is easy to verify, is used in Cases 2(ii), 3(ii) and 4(i) below. 
			
			\begin{observation}\label{obs:parallelo}
				If $C=C_1C_2$ and $D=D_1D_2$, where $|C_1|=|D_1|=k$ and $\mathcal{K}$ is a match between $C$ and $D$, then the parallelogram pairing of length $k$ on $C_1$ and $D_1$, together with a best possible match $\mathcal{K}_2$ between $C_2$ and $D_2$, is a pairing that has cardinality at least that of $\mathcal{K}.$ A similar observation holds if we have a parallelogram pairing of the last $k$ symbols in each of $C$ and $D$.
			\end{observation}

	For $k=1,\dots n$, we let $A^{(0)}_k=A_{a_k}$, and we let $B_k^{(0)}$ be strings of symbols such that 
			\[
			\begin{array}{lcc}
				A_{a_1}A_{a_2}\cdots A_{a_n} =  A^{(0)}_1A^{(0)}
				_2\cdots A^{(0)}_n\\
				\text{and}\\
				A_{b_1}A_{b_2}\cdots A_{b_m} =  B^{(0)}_1B^{(0)}
				_2\cdots B^{(0)}_n,\\
			\end{array}
			\]
			where $A_k^{(0)}$ corresponds to $B_k^{(0)}$  under the match $\mathcal{M}_{0}$. \\

	The following conditions will be satisfied in our construction, where we allow $i\in\{1,\dots,n+1\}$ in (1), and $i\in\{1,\dots,n\}$ in (2), (3), (4), and (5):
			
			\begin{enumerate}
				\item For $k=1,\dots,n$, $A_k^{(i-1)}$ and $B_k^{(i-1)}$ are strings of symbols such that  
				\[
				\begin{array}{lcc}
					A_{a_1}A_{a_2}\cdots A_{a_n} =  A^{(i-1)}_1A^{(i-1)}
					_2\cdots A^{(i-1)}_n\\
					\text{and}\\
					A_{b_1}A_{b_2}\cdots A_{b_m} =  B^{(i-1)}_1B^{(i-1)}
					_2\cdots B^{(i-1)}_n,\\
				\end{array}
				\]
				where $A_k^{(i-1)}$ corresponds to $B_k^{(i-1)}$ under the pairing $\mathcal{M}_{i-1}$. \\
				
				\item 
				We have $A_i^{(i-1)}\subseteq A_{a_i}$, and $A_k^{(i-1)}=A_{a_k}$ for $k>i$.  
				\item 
				\[
				\begin{array}{ccc}
					\mathcal{M}_{i-1}\big| \big(A_{i}^{(i-1)}\cdots A_n^{(i-1)},B_{i}^{(i-1)}\cdots B_n^{(i-1)}\big)
					
				\end{array} 
				\]
				is a match, not just a pairing. 
				\item For $i=1,\dots,n,$ if $B_i^{(i-1)}$ starts with a symbol in $A_{b_j}$,
				none of symbols in the matching pairs in $\mathcal{N}_{i-1}$ come from $a_{i}a_{i+1}\dots a_{n}$ or
				$b_{j}b_{j+1}\dots b_{m}.$
				(Therefore these symbols are still available to make a pair to add to $\mathcal{N}_{i-1}$ to form
				$\mathcal{N}_{i}.)$
				
		\end{enumerate}
		
		Suppose $1\le i\le n,$ and $\mathcal{M}_{i-1}$, $A_1^{(i-1)},\dots, A_n^{(i-1)}$, $B_1^{(i-1)},\dots,B_n^{(i-1)},$ and $\mathcal{N}_{i-1}$
		have been constructed so that (1)-(4) are satisfied. We now construct $\mathcal{M}_i,  \mathcal{N}_i$, and $A_k^{(i)}$ and $B_k^{(i)}$, for $k=1,\dots,n$, such that $|\mathcal{M}_{i-1}|\le |\mathcal{M}_i|$, $\mathcal{N}_{i-1}\subseteq \mathcal{N}_{i},$ and condition (1) holds if $i-1$ is replaced by $i$. Moreover, conditions (2)-(4) hold if $i-1$ is replaced by $i$, and $i$ is replaced by $i+1$, provided $1\le i <n$. Furthermore, we will have $A_k^{(i)}=A_k^{(i-1)}$ and $B_k^{(i)}=B_k^{(i-1)}$ for $k=1,\dots, i-1$, and 
			\begin{equation}\label{pairing_change}
				\mathcal{M}_{i}\big|\big(A_1^{(i)}\cdots A_{i-1}^{(i)},B_1^{(i)}\cdots B_{i-1}^{(i)}\big)=\mathcal{M}_{i-1}\big|\big(A_1^{(i-1)}\cdots A_{i-1}^{(i-1)},B_1^{(i-1)}\cdots B_{i-1}^{(i-1)}\big).
		\end{equation}
		
		We consider the following four cases (with various
		subcases). In some cases the $\mathcal{M}_i$ that we construct has to be truncated on the right if we reach the end of the $A_{a_1}\cdots A_{a_n}$ or $A_{b_1}\cdots A_{b_m}$ strings, but our estimate of the upper bound for $\lvert\mathcal{M}_i|(A_i^{(i)},B_i^{(i)})\rvert$ remains valid. 
		
		\noindent{\bf \emph{Case 1. }}$B_i^{ (i-1)}$ contains two or more complete $A_{b_{j}}$'s.
		Then $|B_{i}^{ (i-1)}| { \ge 2L} \ge 2 |A_{ i}^{ (i-1)}|$ and therefore $\overline{f}(A_{ i}^{ (i-1)},B_{i}^{ (i-1)}){ \ge 1/3}  > \alpha$. 
		In this case we take { $\mathcal{N}_i=\mathcal{N}_{i-1}$,} $\mathcal{M}_{i}=\mathcal{M}_{i-1}$, {$ A_k^{(i)} = A_k^{(i-1)},$ and $B_k^{(i)}=B_k^{(i-1)}$, for $k=1,\dots,n.$} {\bf{ Thus $2\big|\mathcal{M}_i|(A_i^{(i)},B_i^{(i)})\big|<\alpha(|A_i^{(i)}|+|B_i^{(i)}|).$}}\\

		\noindent{\bf \emph{Case 2. }} $B_{i}^{ (i-1)}$ consists of one complete $A_{b_{j}}$
		plus either a final string $A_{b_{j-1}}^{\text{fin}}$ in $A_{b_{j-1}}$
		or an initial string $A_{b_{j+1}}^{\text{int}}$ in $A_{b_{j+1}}$, or possibly both, and we are not in Case 1.
		By the induction hypothesis, $b_{j}$,$b_{j+1},$ and $a_{i}$ are
		not in any of the pairs in $\mathcal{N}_{i-1}$, and neither is $b_{j-1}$ if $A_{b_{j-1}}^{\text{fin}}\ne\emptyset.$ \\
		\noindent{\bf \emph{ Case 2(i)}.} At least one of the following holds:  \.(a) $A_{b_{j-1}}^{\text {fin}}\ne \emptyset$ and $a_i=b_{j-1}$, or (b) $a_i=b_j$. If (a) holds, we let $\mathcal{N}_i=\mathcal{N}_{i-1}\cup \{(a_i,b_{j-1})\}$; if (b) holds, but not (a), we let $\mathcal{N}_i=\mathcal{N}_{i-1}\cup \{(a_i,b_j)\}.$ We let $A_k^{(i)}=A_k^{(i-1)}$ and $B_k^{(i)}=B_k^{(i-1)}$, for $k=1,\dots,n.$ Then define $\mathcal{M}_i\big|\big(A_i^{(i)},B_i^{(i)}\big)$ to be a parallelogram pairing of length $|A_i^{(i)}|$ starting with the first symbol in $A_i^{(i)}$ and the first symbol in $B_i^{(i)}.$ This is possible because $|A_i^{(i)}|\le L\le |B_i^{(i)}|.$ We let $\mathcal{M}_i\big|\big(A_k^{(i)},B_k^{(i)}\big)=\mathcal{M}_{i-1}\big|\big(A_k^{(i)},B_k^{(i)}\big)$ for $k\ne i.$ Since the parallelogram pairing is a best possible pairing for $A_i^{(i)}$ and $B_i^{(i)}$, we have $|\mathcal{M}_{i-1}\,| \le |\mathcal{M}_i|.$ (Actually, it is not necessary to use a parallelogram pairing in this case, but it is convenient for consistency with other cases.) {\bf In Case 2(i), $|\mathcal{N}_i|=|\mathcal{N}_{i-1}|+1,$ and there is a parallelogram pairing of length at most $L,$ and no other pairings, in $\mathcal{M}_i\big|\big(A_i^{(i)},B_i^{(i)}\big).$} \\ 
		
		\noindent{\bf \emph{ Case 2(ii)}.} Suppose  Case 2(i) does not hold and 
		$a_{i}=b_{j+1}$. Also assume that $A_{b_{j+1}}^{\text {int}}\ne\emptyset$; if this last assumption is not satisfied,  we can proceed as in Case 2(iii). Let $\mathcal{N}_i=\mathcal{N}_{i-1}\cup \{(a_i,b_{j+1})\}$.  In Case 2(ii) and Case 2(iii), we let
		$A_{i,0}^{(i-1)}, A_{i,1}^{(i-1)}$ and $A_{i,2}^{(i-1)}$ be the subdivision
		of $A_i ^{(i-1)}$ into the substrings corresponding to $A_{b_{j-1}}^{\text{fin}}$,$A_{b_{j}}$, and $A_{b_{j+1}}^{\text {int}},$ respectively, under $\mathcal{M}_{i-1}$. 
		
		Let $P$ be the parallelogram pairing of length $L$ starting with the first symbol in $A_{i,2}^{(i)}$ and the first symbol in $A_{b_{j+1}}$. In this case and other cases below,  define $P_1$ and $P_2$ to be the string of first terms and second terms, respectively, in the pairs in the given $P$. Let $A_{\text {end}}^{(i)}$ be the final string in $A_{a_1}\cdots A_{a_n}$ that starts right after $P_1$ ends, and let $B_{\text {end}}^{(i)}=A_{b_{j+2}}\cdots A_{b_m}$, which starts right after $P_2$ ends. Let  $\mathcal{M}_i|(P_1A_{\text {end}}^{(i)},P_2B_{\text {end}}^{(i)})$ be $P$ together with a best possible match on $(A_{\text {end}}^{(i)}, B_{\text {end}}^{(i)})$. By condition (3) and Observation \ref{obs:parallelo}, $|\mathcal{M}_{i-1}|(P_1A_{\text {end}}^{(i)},P_2B_{\text {end}}^{(i)})|\le |\mathcal{M}_i|(P_1A_{\text {end}}^{(i)},P_2B_{\text {end}}^{(i)})|$. 
		
		Let $A_i^{(i)}$ be the extension of $A_i^{(i-1)}$ to the end of $P_1$, and similarly, let $B_i^{(i)}$ be the extension of $B_i^{(i-1)}$ to the end of $P_2$. Let $A_{i+1}^{(i)}$ be the part of $A_{i+1}^{(i-1)}$ that comes after $P_1$.
			For $k>i+1$, let $A_k^{(i)}=A_k^{(i-1)}$. Then take $B_{i+1}^{(i)}B_{i+2}^{(i)}\cdots B_n^{(i)}$ to be the subdivision of $B_{\text {end}}^{(i-1)}$ that corresponds to the subdivision $A_{i+1}^{(i)}A_{i+2}^{(i)}\cdots A_n^{(i)}$ of $A_{\text {end}}^{(i-1)}$ under $\mathcal{M}_i$. For $k=1,\dots,i-1$, let $A_k^{(i)}=A_k^{(i-1)}$, $B_k^{(i)}=B_k^{(i-1)}$, and $\mathcal{M}_i|(A_k^{(i)},B_k^{(i)})=\mathcal{M}_{i-1}|(A_k^{(i)},B_k^{(i)})$. 
			Furthermore, let $\mathcal{M}_i|(A_{i,0}^{(i-1)}A_{(i,1)}^{(i-1)}, A_{b_{j-1}}^{\text {fin}}A_{b_j})$ = $\mathcal{M}_{i-1}|(A_{i,0}^{(i-1)}A_{i,1}^{(i-1)},A_{b_{j-1}}^{\text {fin}}A_{b_j})$. Overall, we obtain $|\mathcal{M}_{i-1}|\le |\mathcal{M}_i|.$ 
			
			{\bf {In Case 2(ii), we have $|\mathcal{N}_i|=|\mathcal{N}_{i-1}|+1,$ and $\mathcal{M}_i|(A_i^{(i)},B_i^{(i)})$ consists of a parallelogram pairing of length $L$ and matches between two pairs of strings such that the two members of each pair being matched come from blocks with different subscripts.}}\\
		
		\noindent{\bf \emph{Case 2(iii)}.} Case 2(i) does not hold, and
		$a_{i}\ne b_{j+1}.$ Then we keep $A_k^{(i)}=A_k^{(i-1)}$, $ B_k^{(i)}=B_k^{(i-1)}$, for $k=1,\dots,n,$ and $\mathcal{M}_{i}=\mathcal{M}_{i-1}$
			and $\mathcal{N}_{i}=\mathcal{N}_{i-1}.$ Then $\mathcal{M}_i|(A_i^{(i)},B_i^{(i)})$ consists of matches within the following pairs of strings: 
			$(A_{i,0}^{(i)},A_{b_j}^{\text {fin}})$, $(A_{i,1}^{(i)},A_{b_{j-1}}^{(i)})$ 
			and $(A_{i,2}^{(i)},A_{b_{j+1}}^{\text {int}})$. {\bf {In Case 2(iii), we have $|\mathcal{N}_i|=|\mathcal{N}_{i-1}|$ and $\mathcal{M}_i|(A_i^{(i)},B_i^{(i)})$ consists of matches between three pairs of strings such that the two members of each pair (if they are non-empty) come from blocks with different subscripts. At most two pairs (possibly the first and the third pair) are in the short substring case.}}\\
		
		\noindent{\bf \emph{Case 3.}} $B_{i}^{(i-1)}$ contains a final segment $A_{b_{j-1}}^{\text {fin}}$ of $A_{b_{j-1}}$ and an initial segment $A_{b_j}^{\text {int}}$ of $A_{b_j}.$ 
			Let $A_{i,0}^{(i-1)}$ and $A_{i-1}^{(i-1)}$ be strings of symbols corresponding to $A_{b_{j-1}}^{\text {fin}}$ and $A_{b_j}^{\text {int}},$ respectively, under $\mathcal{M}_{i-1}$, such that $A_{i}^{(i-1)}=A_{i,0}^{(i-1)}A_{i,1}^{(i-1)}.$ \\

		\noindent{\bf \emph{Case 3(i)}.} If $a_{i}=b_{j-1},$ then we let $\mathcal{N}_i=\mathcal{N}_{i-1}\cup \{(a_i,b_{j-1})\}$.  We let $A_k^{(i)}=A_k^{(i-1)}$ and $B_k^{(i)}=B_k^{(i-1)}$, for $k=1,\dots,n$. Let $\mathcal{M}_i|(A_i^{(i)},B_i^{(i)})$ be a parallelogram pairing of length $\min(|A_i^{(i)}|,|B_i^{(i)}|)$. For $k\ne i$, 
			let $\mathcal{M}_i|(A_k^{(i)},B_k^{(i)})=\mathcal{M}_{i-1}|(A_k^{(i)},B_k^{(i)}).$ Then $|\mathcal{M}_{i-1}|\le |\mathcal{M}_{i}|.$ {\bf Furthermore $|\mathcal{M}_i|(A_i^{(i)},B_i^{(i)})|\le L$ and $|\mathcal{N}_i|=|\mathcal{N}_{i-1}|+1.$}\\
		
		\noindent{\bf \emph{ Case 3(ii)}.} If $a_{i}\ne b_{j-1}$ and $a_{i}=b_{j},$ then
		let $\mathcal{N}_i=\mathcal{N}_{i-1}\cup \{(a_i,b_j)\}$.
		Let $P$ be the parallelogram pairing of length $L$ starting with the first symbol in $A_{i,1}^{(i)}$ and the first symbol in $A_{b_{j}}$.The definitions of $A_{\text {end}}^{(i)}$, $B_{\text {end}}^{(i)}$, and $\mathcal{M}_i|(P_1A_{\text {end}}^{(i)},P_2B_{\text {end}}^{(i)})$ are the same as in Case 2(ii). By the same argument as in Case 2(ii), again using Observation \ref{obs:parallelo}, 
			$|\mathcal{M}_{i-1}|\le |\mathcal{M}_i|$. {\bf {Furthermore, we have $|\mathcal{N}_i|=|\mathcal{N}_{i-1}|+1,$ and $\mathcal{M}_i|(A_i^{(i)},B_i^{(i)})$
					consists of a parallelogram pairing of length $L$ and a match between one pair of strings from blocks with different subscripts.}} \\

		\noindent{\bf \emph{ Case 3(iii)}.} If $a_{i}\ne b_{j-1}$ and $a_{i}\ne b_{j},$
		then we again have two pairs of strings, $(A_{i,0}^{(i-1)},A_{b_{j-1}}^{\text {fin}})$ and 
			$(A_{i,1}^{(i-1)},A_{b_j}^{\text {int}})$ that are matched under $\mathcal{M}_{i-1}$. We let $\mathcal{N}_i=\mathcal{N}_{i-1}$, 
			$\mathcal{M}_i=\mathcal{M}_{i-1}$, and $A_k^{(i)}=A_k^{(i-1)}$, $B_k^{(i)}=B_k^{(i-1)}$, for $k=1,\dots,n.$ {\bf {Then 
					$|\mathcal{N}_i|=|\mathcal{N}_{i-1}|$ and $\mathcal{M}_i|(A_i^{(i)},B_i^{(i)})$ consists of a match between two pairs of strings
					such that the two members of each pair come from blocks with different subscripts.}} \\
		
		\noindent{\bf \emph{Case 4. }} $B_{i}^{(i-1)}$ is a substring of a single $A_{b_{j}}.$ \\
		
		\noindent{\bf \emph{Case 4(i)}.} If $a_{i}=b_{j},$ then we let $\mathcal{N}_i=\mathcal{N}_{i-1}\cup \{(a_i,b_j)\} $. Let $P$ be a parallelogram pairing of length $L$ starting with the first symbol in $A_i^{(i-1)}$ and the first symbol in $B_i^{(i-1)}$. Let $A_{\text {end}}^{(i)}$ be the final string in $A_{a_1}\cdots A_{a_n}$ starting with the first symbol after the end of $P_1$, and let $B_{\text {end}}^{(i)}$ be the final string in $A_{b_1}\cdots A_{b_m}$ starting with the first symbol after $P_2$. The construction continues as in Case 2(ii).
			{\bf {In Case 4(i), we have $|\mathcal{N}_i|=|\mathcal{N}_{i-1}|+1,$  and $\mathcal{M}_i|(A_i^{(i)},B_i^{(i)})$ consists of a parallelogram pairing of length $L$ and no other pairings.}}\\
		
		\noindent{\bf \emph{Case 4(ii)}.} If $a_i\ne b_j,$ then
		we let $\mathcal{N}_i=\mathcal{N}_{i-1}$, 
			$\mathcal{M}_i=\mathcal{M}_{i-1}$, and $A_k^{(i)}=A_k^{(i-1)}$, $B_k^{(i)}=B_k^{(i-1)}$, for $k=1,\dots,n.$ {\bf {Then 
					$|\mathcal{N}_i|=|\mathcal{N}_{i-1}|$ and $\mathcal{M}_i|(A_i^{(i)},B_i^{(i)})$ consists of a match between a pair of strings
					such that the two members of the pair come from blocks with different subscripts.}} \\

		In summary, after we apply this procedure for $i=1,2,\dots,n,$
		we see that $\cup_{i=1}^{n}\mathcal{M}_i|(A_i^{(i)},B_i^{(i)})$
			consists of at most $|\mathcal{N}_n|$ parallelogram pairings of length at most $L$ each, matches between at most $2n$ pairs of short substrings (as described in (III)), and matches between pairs of strings whose $\overline{f}$ distance is greater than $\alpha$ (either due to Case 1 or condition (I) or (II)). Note that by (\ref{pairing_change}), $\mathcal{M}_n=\cup_{i=1}^n\mathcal{M}_i|(A_i^{(i)},B_i^{(i)}).$ By conditions (I),(II), and (III), and our assumption that $n\le m$, we obtain
		
		\[
		\begin{array}{ccc}
			L(n+m)-2|\mathcal{M}| & \ge & L(n+m)-2|\mathcal{M}_{n}|\\
			& \ge & \alpha\big\{ L(n+m)-2L|\mathcal{N}_{n}|-\frac{14Ln}{3R}\big\}\\
			& \ge & \alpha\big\{ L(n+m)-2L|\mathcal{N}_{n}|-\frac{3L(n+m)}{R}\big\}.
		\end{array}
		\]
	\end{proof}
	
	\subsection{\label{subsec:Basics-in-Descriptive}Some background from descriptive set theory}
	
	To state our results precisely we need some concepts
	explained below. The main tool is the idea of a
	reduction. See \cite{Kechris} and \cite{Fo18} for further information.
	\begin{defn}
		Let $X$ and $Y$ be Polish spaces and $A\subseteq X$, $B\subseteq Y$.
		A function $f:X\to Y$ \emph{reduces} $A$ to $B$ if and only if
		for all $x\in X$: $x\in A$ if and only if $f(x)\in B$. 
		
		Such a function $f$ is called a Borel (respectively, continuous) reduction
		if $f$ is a Borel (respectively, continuous) function.
	\end{defn}
	
	$A$ being reducible to $B$ can be interpreted as saying that $B$
	is at least as complicated as $A$. We note that if $B$ is Borel
	and $f$ is a Borel reduction, then $A$ is also Borel. Equivalently, if $f$ is a Borel reduction of $A$ to $B$ and $A$ is not Borel,
	then $B$ is not Borel.
	
	\begin{defn}\label{dfn:analytic}
		If $X$ is a Polish space and $B\subseteq X$, then $B$ is \emph{analytic}
		if and only if it is the continuous image of a Borel subset of a Polish
		space. Equivalently, there is a Polish space $Y$ and a Borel set
		$C\subseteq X\times Y$ such that $B$ is the $X$-projection of $C$.
	\end{defn}
	
		\begin{defn}
			An analytic subset $A$ of a Polish space $X$ is called \emph{complete analytic} if every analytic set can be continuously reduced to $A$.
		\end{defn}
		Following the interpretation above, a complete analytic set $B$ is said to be at least as
		complicated as each analytic set.
	
	There are analytic sets that are not Borel (see, for example, \cite[section 14]{Kechris}). This implies that a complete analytic set is not Borel. The collection of ill-founded trees, which we describe below, is an example of a complete analytic set, and it is fundamental to our construction, as well as that in \cite{FRW}.
	
	To define the ill-founded trees, we first consider the set $\mathbb{N}^{<\mathbb{N}}$
	of finite sequences of natural numbers. For $\tau \in \mathbb{N}^{<\mathbb{N}},$ 
	let $lh(\tau)$ denote the length of $\tau$. A \emph{tree }is a set $\mathcal{T}\subseteq\mathbb{N}^{<\mathbb{N}}$
	such that if $\tau=\left(\tau_{1},\dots,\tau_{n}\right)\in\mathcal{T}$
	and $\sigma=\left(\tau_{1},\dots,\tau_{s}\right)$ with $s\leq n$
	is an initial segment of $\tau$, then $\sigma\in\mathcal{T}$. If $\sigma$ is an initial segment of $\tau$, then $\sigma$ is a \emph{predecessor}
	of $\tau$ and $\tau$ is a \emph{successor} of $\sigma$. We define
	the level $s$ of a tree $\mathcal{T}$ to be the collection of elements
	of $\mathcal{T}$ that have length $s$. We call a subset $\mathcal{S}$ of a tree $\mathcal{T}$ a \emph{subtree} if $\mathcal{S}$ is a tree. 
	
	\begin{defn}
		A \emph{branch} through a tree $\mathcal{T}$  is a function $f$ into $\mathcal{T}$ with domain either some
		$\{0,\dots,n-1\}$ or $\mathbb{N}$ itself such that $lh(f(s)) = s$, and if $s+1$ is in the domain of $f$, then
		$f(s + 1)$ is an immediate successor of $f(s)$.
		If a tree has an infinite branch (that is, the $f$ above can be taken with domain $\mathbb{N}$), it is called \emph{ill-founded}.
		If it does not have an infinite branch, it is called \emph{well-founded}.
	\end{defn}
	
	We can visualize a tree as in Figure~\ref{fig:tree}.
	
	\begin{figure}
		\centering
		\includegraphics[width=\textwidth]{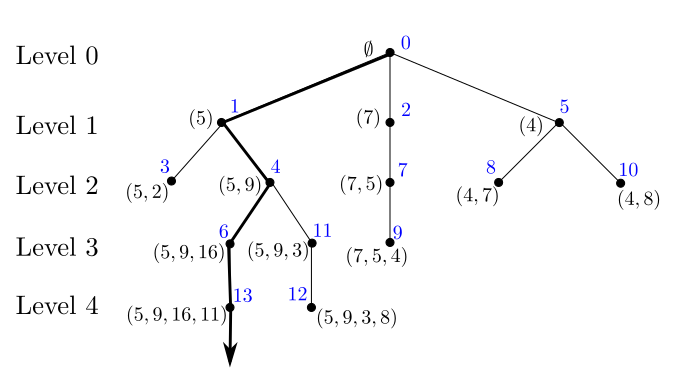}
		\caption{ Visualization of a tree, $\mathcal{T}$. The blue numbers indicate the order in which the nodes appear in the tree $\mathcal{T}$, that is, the blue number $i$ stands for $\sigma_{n_i}$, where $(n_i)_{i\in \N}$ is a subsequence of $\N$ such that $\mathcal{T}=\Meng{\sigma_{n_i}}{i\in \N}$. The bold line represents an infinite branch.}
		\label{fig:tree}
	\end{figure}
	
	In the following, let $\left\{ \sigma_{n}:n\in\mathbb{N}\right\} $
	be an enumeration of $\mathbb{N}^{<\mathbb{N}}$ with the property
	that $\sigma_0$ is the empty sequence and every proper predecessor of $\sigma_{n}$ is some $\sigma_{m}$
	for $m<n$. Under this enumeration subsets $S\subseteq\mathbb{N}^{<\mathbb{N}}$
	can be identified with characteristic functions $\chi_{S}:\mathbb{N}\to\left\{ 0,1\right\} $.
	The collection of such $\chi_{S}$ can be viewed as the members of
	an infinite product space $\left\{ 0,1\right\} ^{\mathbb{N}^{<\mathbb{N}}}$
	homeomorphic to the Cantor space. Here, each function $a:\left\{ \sigma_{m}:m<n\right\} \to\left\{ 0,1\right\} $
	determines a basic open set 
	\[
	\left\langle a\right\rangle =\left\{ \chi \in \left\{ 0,1\right\} ^{\mathbb{N}^{<\mathbb{N}}}:\chi\upharpoonright\left\{ \sigma_{m}:m<n\right\} =a\right\} 
	\]
	and the collection of all such $\left\langle a\right\rangle $ forms
	a basis for the topology. The collection of trees is a closed (hence
	compact) subset of $\left\{ 0,1\right\} ^{\mathbb{N}^{<\mathbb{N}}}$
	in this topology. Moreover, the collection of trees containing arbitrarily
	long finite sequences is a dense $\mathcal{G}_{\delta}$ subset. In
	particular, this collection is a Polish space. We will denote the
	space of trees containing arbitrarily long finite sequences by $\mathcal{T}\kern-.5mm rees$.
	
	Since the topology on the space of trees was introduced via basic
	open sets giving us a finite amount of information about the trees in
	it, we can characterize continuous maps defined on $\mathcal{T}\kern-.5mm rees$
	as follows.
	\begin{fact}
		\label{fact:contTree}Let $Y$ be a topological space. Then a map
		$f:\mathcal{T}\kern-.5mm rees\to Y$ is continuous if and only if for all open
		sets $O\subseteq Y$ and all $\mathcal{T}\in\mathcal{T}\kern-.5mm rees$ with
		$f(\mathcal{T})\in O$ there is $M\in\mathbb{N}$ such that for all
		$\mathcal{T}^{\prime}\in\mathcal{T}\kern-.5mm rees$ we have:
		
		if $\mathcal{T}\cap\left\{ \sigma_{n}:n\leq M\right\} =\mathcal{T}^{\prime}\cap\left\{ \sigma_{n}:n\leq M\right\} $,
		then $f\left(\mathcal{T}^{\prime}\right)\in O$. 
	\end{fact}
	
	As mentioned before, we have the following classical fact (see, for example, \cite[section 27]{Kechris}).
	\begin{fact}
		\label{fact:ill}The collection of ill-founded trees is a complete
		analytic subset of $\mathcal{T}\kern-.5mm rees$. 
	\end{fact}
	
	As described in the next section, we will prove the main results of
	this paper by reducing the collection of ill-founded trees to pairs
	of Kakutani-equivalent ergodic transformations. Then Fact \ref{fact:ill}
	will show that this set of Kakutani-equivalent pairs of transformations
	is also complete analytic.
	
	During our construction the following maps will prove useful.
	\begin{defn}
		\label{def:M-and-s}We define a map $M:\mathcal{T}\kern-.5mm rees\to\mathbb{N}^{\mathbb{N}}$
		by setting $M\left(\mathcal{T}\right)(s)=n$ if and only if $n$ is
		the least number such that $\sigma_{n}\in\mathcal{T}$ and $lh(\sigma_{n})=s$.
		Dually, we also define a map $s:\mathcal{T}\kern-.5mm rees\to\mathbb{N}^{\mathbb{N}}$
		by setting $s\left(\mathcal{T}\right)(n)$ to be the length of the
		longest sequence $\sigma_{m}\in\mathcal{T}$ with $m\leq n$. 
	\end{defn}
	
	\begin{rem*}
		When $\mathcal{T}$ is clear from the context we write $M(s)$ and
		$s(n)$. We also note that $s(n)\leq n$ and that $M$ as well as
		$s$ is a continuous function when we endow $\mathbb{N}$ with the
		discrete topology and $\mathbb{N}^{\mathbb{N}}$ with the product
		topology.
	\end{rem*}
	
	\section{Main Results}
	
	\subsection{Statement of results}
	
	We are ready to state the main result of this paper.
	\begin{thm}
		\label{thm:CompleteAnalytic}The collection 
		\[
		\left\{ (S,T):S\text{ and }T\text{ are ergodic and Kakutani equivalent}\right\} \subseteq\mathcal{E}\times\mathcal{E}
		\]
		is a complete analytic set. In particular, it is not Borel.
	\end{thm}
	
	As we show below, this will follow from the subsequent theorem which is proved in Section \ref{sec:Non-Equiv}.
	\begin{thm}
		\label{thm:Main}There is a continuous one-to-one map 
		\[
		\Psi:\mathcal{T}\kern-.5mm rees\to\mathcal{E}
		\]
		such that for $\mathcal{T}\in\mathcal{T}\kern-.5mm rees$ and $T=\Psi(\mathcal{T})$:
		
		$\mathcal{T}$ has an infinite branch if and only if $T$ and $T^{-1}$
		are Kakutani equivalent.
	\end{thm}
	Using the notions from descriptive set theory in Section \ref{subsec:Basics-in-Descriptive},
	Theorem \ref{thm:Main} can also be expressed as follows.
	\begin{cor}
		\label{cor:Reductiona}There is a continuous one-to-one map 
		\[
		\Psi:\mathcal{T}\kern-.5mm rees\to\mathcal{E}
		\]
		that reduces the collection of ill-founded trees to the collection
		of ergodic transformations $T$ such that $T$ and $T^{-1}$ are Kakutani
		equivalent.
	\end{cor}
	
	\begin{proof}[Proof of Theorem \ref{thm:CompleteAnalytic}]
		The map $i(T)=\left(T,T^{-1}\right)$ is a continuous mapping from
		$\mathcal{E}$ to $\mathcal{E}\times\mathcal{E}$ and reduces $\left\{ T\in\mathcal{E}:T\text{ is Kakutani equivalent to }T^{-1}\right\} $
		to
		\[
		\left\{ (S,T)\in\mathcal{E}\times\mathcal{E}:S\text{ and }T\text{ are Kakutani equivalent}\right\} .
		\]
		Then Theorem \ref{thm:CompleteAnalytic} follows from Fact \ref{fact:ill} and Corollary \ref{cor:Reductiona}.
	\end{proof}
	
	To prove Theorem \ref{thm:Main} we actually show the following stronger
	result.
	\begin{prop}
		\label{prop:criterion}There is a continuous one-to-one map 
		\[
		\Psi:\mathcal{T}\kern-.5mm rees\to\mathcal{E}
		\]
		such that for $\mathcal{T}\in\mathcal{T}\kern-.5mm rees$ and $T=\Psi(\mathcal{T})$:
		\begin{enumerate}
			\item If $\mathcal{T}$ has an infinite branch, then $T$ and $T^{-1}$
			are isomorphic.
			\item If $T$ and $T^{-1}$ are Kakutani equivalent, then $\mathcal{T}$
			has an infinite branch.
		\end{enumerate}
	\end{prop}
	
	This stronger proposition also implies the anti-classification result
	for the isomorphism relation in \cite{FRW}.
	\begin{cor*}
		There is a continuous one-to-one map 
		\[
		\Psi:\mathcal{T}\kern-.5mm rees\to\mathcal{E}
		\]
		that reduces the collection of ill-founded trees to the collection
		of ergodic automorphisms $T$ such that $T$ is isomorphic to $T^{-1}$. Hence, the
		collection 
		\[
		\left\{ (S,T):S\text{ and }T\text{ are ergodic and isomorphic}\right\} \subseteq\mathcal{E}\times\mathcal{E}
		\]
		is a complete analytic set.
	\end{cor*}
	In Section \ref{sec:Transfer} we are able to upgrade these results
	to the setting of diffeomorphisms. 
	\begin{thm}
		\label{thm:smooth}Let $M$ be the disk, annulus or torus with Lebesgue
		measure $\lambda$. Then the collection
		\[
		\left\{ (S,T):S\text{ and }T\text{ are ergodic diffeomorphisms and Kakutani equivalent}\right\} 
		\]
		in $\text{Diff}^{\,\infty}(M,\lambda)\times\text{Diff}^{\,\infty}(M,\lambda)$
		is a complete analytic set and hence not a Borel set with respect
		to the $C^{\infty}$ topology.
	\end{thm}
	
	In case of $M=\mathbb{T}^{2}$ we also obtain anti-classification results in the real-analytic
	category (see Subsection \ref{subsec:DiffeomorphismSpaces} for the definition of the Polish space $\text{Diff}_{\rho}^{\,\omega}(M,\lambda)$).
	\begin{thm}
		\label{thm:analytic}Let $\rho>0$ and $\lambda$ be the Lebesgue
		measure on $\mathbb{T}^{2}$. Then the collection
		\[
		\left\{ (S,T):S\text{ and }T\text{ are ergodic real-analytic diffeomorphisms and Kakutani equivalent}\right\} 
		\]
		in $\text{Diff}_{\rho}^{\,\omega}(\mathbb{T}^{2},\lambda)\times\text{Diff}_{\,\rho}^{\,\omega}(\mathbb{T}^{2},\lambda)$
		is a complete analytic set and hence not a Borel set with respect
		to the $\text{Diff}_{\rho}^{\,\omega}$ topology.
	\end{thm}
	
	 It is natural to equip the space $\text{Diff}^{\,\omega}(\mathbb{T}^{2},\lambda)$ of area-preserving real-analytic diffeomorphisms homotopic to the identity by the direct limit topology \cite[section 2.6]{Krantz}, that is, a set $U \subset  \text{Diff}^{\,\omega}(\mathbb{T}^{2},\lambda)$ is open iff $ U \cap \text{Diff}_{\rho}^{\,\omega}(\mathbb{T}^{2},\lambda)$ is open in $\text{Diff}_{\rho}^{\,\omega}(\mathbb{T}^{2},\lambda)$ for all $\rho>0$. This space is not Polish, but $\text{Diff}_{\rho}^{\,\omega}(\mathbb{T}^{2},\lambda)$ is Borel in $\text{Diff}^{\,\omega}(\mathbb{T}^{2},\lambda)$. The Kakutani equivalence relation on ergodic real-analytic diffeomorphisms of $\mathbb{T}^2$ is not a Borel set, because by Theorem~\ref{thm:analytic} its intersection with the Borel set $\text{Diff}_{\rho}^{\,\omega}(\mathbb{T}^{2},\lambda)\times \text{Diff}_{\rho}^{\,\omega}(\mathbb{T}^{2},\lambda)$ is not Borel.
	
	\subsection{Outline of the proof of Proposition \ref{prop:criterion}}
	
	To construct the reduction $\Psi:\mathcal{T}\kern-.5mm rees\to\mathcal{E}$ we
	build for every $\mathcal{T}\in\mathcal{T}\kern-.5mm rees$ a transformation
	$\Psi(\mathcal{T})$ as a symbolic system with a strongly uniform
	and uniquely readable construction sequence $\left(\mathcal{W}_{n}\left(\mathcal{T}\right)\right)_{n\in\mathbb{N}},$
	which implies its ergodicity.  This
	will be done in such a way that $\mathcal{W}_{n}(\mathcal{T})$ is
	entirely determined by $\mathcal{T}\cap\left\{ \sigma_{m}:m\leq n\right\} $, which, by Fact \ref{fact:contTree} guarantees continuity of $\Psi$ (see Proposition 20 in \cite{FRW} for details).
	(Alternatively, in the cutting-and-stacking model, the $n$th tower
	is determined by $\mathcal{T}\cap\left\{ \sigma_{m}:m\leq n\right\} $,
	and this implies that $\Psi$ is continuous.) 
	Related to the structure of the tree we also specify equivalence relations
	$\mathcal{Q}_{s}^{n}(\mathcal{T})$ on the collections $\mathcal{W}_{n}(\mathcal{T})$
	of $n$-words and group actions on the equivalence classes in $\mathcal{W}_{n}(\mathcal{T})/\mathcal{Q}_{s}^{n}(\mathcal{T})$
	as in \cite{FRW}. These specifications will allow us to prove in
	Section \ref{sec:Isom} that $\Psi(\mathcal{T})$ and $\Psi(\mathcal{T})^{-1}$
	are isomorphic if the tree $\mathcal{T}$ has an infinite branch. 
	
	The harder part is to show that $\Psi(\mathcal{T})$ and $\Psi(\mathcal{T})^{-1}$
	are not Kakutani equivalent if the tree $\mathcal{T}$ does not have
	an infinite branch. For this purpose, we build the $n$-words in our
	construction sequence using specific patterns of blocks (that we call
	Feldman patterns) which originate from Feldman's first example of
	an ergodic zero-entropy automorphism that is not loosely Bernoulli
	\cite{Fe}. As observed in \cite{ORW}, different Feldman patterns
	cannot be matched well in $\overline{f}$ even after a finite coding. We introduce
	these patterns in detail in Section \ref{sec:Feldman}. In our construction
	of $n$-words we apply them repeatedly by substituting Feldman patterns
	of finer equivalence classes into Feldman patterns of coarser classes.
	We present such a substitution step in general in Section \ref{sec:Substitution},
	and then use it in an iterative way to construct the collection $\mathcal{W}_{n}(\mathcal{T})$
	of $n$-words in Section \ref{sec:Construction}. There we also verify
	that $\mathcal{W}_{n}(\mathcal{T})$ satisfies the aforementioned
	specifications. Finally, using techniques from \cite{ORW} based on
	the structure of Feldman patterns, we show in Section \ref{sec:Non-Equiv}
	that $\Psi(\mathcal{T})$ and $\Psi(\mathcal{T})^{-1}$ cannot be
	Kakutani equivalent if the tree $\mathcal{T}$ does not have an infinite
	branch. In the analysis of possible finite codes we face the additional challenge that we do not know the precise words being coded, only their equivalence classes.
	
	We transfer the results to the setting of diffeomorphisms in Section
	\ref{sec:Transfer} and give an outline of this proof at the beginning of that section.
	
	\section{\label{sec:Setup}The Setup}
In this section we present the general setup of our construction with a list of requirements (called \emph{specifications}). We list specifications (E1)-(E3) on construction sequences $\Meng{\mathcal{W}_n(\mathcal{T})}{\sigma_n \in \mathcal{T}}$ for our symbolic systems, specifications (Q4)-(Q6) on equivalence relations $\mathcal{Q}^n_s(\mathcal{T})$ on the collections $\mathcal{W}_n(\mathcal{T})$ of $n$-words, and specifications (A7) and (A8) on group actions on the quotient spaces $\mathcal{W}_n(\mathcal{T})/\mathcal{Q}^n_s(\mathcal{T})$. In Section~\ref{sec:Construction} we will present our actual constructions of words, equivalence relations, and group actions. We will then verify that they satisfy all the requirements from Section~\ref{sec:Setup}. The general setup of our construction is very similar to the one in \cite[Section 4]{FRW} with a few minor changes (for example, specification (Q4) and our specifications on the group actions).

	\subsection{\label{subsec:General-facts}General properties of our construction}
	
	For each $n\in\mathbb{N}$ with $\sigma_{n}\in\mathcal{T}$ we construct
	a set of words $\mathcal{W}_{n}=\mathcal{W}_{n}(\mathcal{T})$ in
	the basic alphabet $\Sigma=\{1,\dots,2^{12}\}$ which
	depend only on $\mathcal{T}\cap\left\{ \sigma_{m}:m\leq n\right\}$, where $\left\{ \sigma_{m}:m\in\mathbb{N}\right\} $
		is an enumeration of $\mathbb{N}^{<\mathbb{N}}$ as described after Definition~\ref{dfn:analytic}.
	To start we set $\mathcal{W}_{0}=\Sigma$. 
	\begin{itemize}
		\item[(E1)]  All words in $\mathcal{W}_{n}$ have the same length $h_{n}$ and
		the cardinality $|\mathcal{W}_{n}|$ is a power of $2$.
		\item[(E2)]  If $\sigma_{m}$ and $\sigma_{n}$ with $m<n$ are consecutive elements of $\mathcal{T}$ (that is, $\sigma_m,\sigma_n \in \mathcal{T}$ and there is no $j$ between $m$ and $n$ with $\sigma_j \in \mathcal{T}$),
		then every word in $\mathcal{W}_{n}$ is built by concatenating words
		in $\mathcal{W}_{m}$. There is $f_{m}\in\mathbb{Z}^{+}$ such that
		every word in $\mathcal{W}_{m}$ occurs in each word of $\mathcal{W}_{n}$
		exactly $f_{m}$ times. This implies $h_n = f_m \cdot \abs{W_m}\cdot h_m$. The number $f_{m}$ will be a product of $\mathfrak{p}_{n}^{2}$
		and powers of $2$, where $\mathfrak{p}_{n}$ is a large prime number chosen
	in \eqref{eq:Pn}.
		\item[(E3)]  If $\sigma_{m}$ and $\sigma_{n}$ with $m<n$ are consecutive elements of $\mathcal{T}$,
		$w=w_{1}\dots w_{h_{n}/h_{m}}\in \mathcal{W}_n$ and $w'=w_{1}^{\prime}\dots w_{h_{n}/h_{m}}^{\prime}\in\mathcal{W}_{n}$, where $w_{i},w_{i}^{\prime}\in\mathcal{W}_{m}$, then for any $k\geq\frac{h_{n}/h_{m}}{2}$ and $1\leq i \leq \frac{h_n}{h_m}-k$, we have $w_{i+1}\dots w_{i+k}\neq w_{1}^{\prime}\dots w_{k}^{\prime}$. 
	\end{itemize}
	\begin{rem}
		In particular, these specifications say that $\left\{ \mathcal{W}_{n}:\sigma_{n}\in\mathcal{T}\right\} $
		is a uniquely readable and strongly uniform construction sequence
		for an odometer-based system. Hence, the corresponding symbolic system
		$\mathbb{K}$ has a unique non-atomic ergodic shift-invariant measure
		$\nu$ by Fact \ref{fact:MeasureConstrSeq}.
	\end{rem}
	
	
	\subsection{\label{subsec:Relations}Equivalence relations and canonical factors}
	
	Related to the structure of the tree we will specify, for each $s \leq s(n)$, an equivalence
	relation $\mathcal{Q}_{s}^{n}$ on our words $\mathcal{W}_{n}$. Recalling Definition~\ref{def:M-and-s}, we note that the case $s\leq s(n)$ corresponds to $n \geq M(s)$.
	We write $Q_{s}^{n}$ for the number of equivalence classes in $\mathcal{Q}_{s}^{n}$
	and enumerate the classes by $\left\{ c_{j}^{(n,s)}:j=1,\dots,Q_{s}^{n}\right\} $.
	In the construction of the factor $\mathbb{K}_s$ described in \ref{subsubsec:factors}, we will identify $\mathcal{W}_{n}/\mathcal{Q}_{s}^{n}$
	with an alphabet denoted by $\left(\mathcal{W}_{n}/\mathcal{Q}_{s}^{n}\right)^{\ast}$
	of $Q_{s}^{n}$ symbols $\left\{ 1,\dots,Q_{s}^{n}\right\} $. 
	
	Each equivalence relation $\mathcal{Q}_{s}^{n}$ will induce an equivalence
	relation on $rev(\mathcal{W}_{n})$, which we will also call $\mathcal{Q}_{s}^{n}$,
	as follows: $rev(w),rev(w')\in rev(\mathcal{W}_{n})$ are equivalent
	with respect to $\mathcal{Q}_{s}^{n}$ if and only if $w,w'\in\mathcal{W}_{n}$
	are equivalent with respect to $\mathcal{Q}_{s}^{n}$. 
	
	To state the specifications on our equivalence relations we recall
	some general notions.
	\begin{defn}
		\label{def:equivrel}Let $X$ be a set and let $\mathcal{Q}$ and $\mathcal{R}$ be equivalence relations on $X$.
		\begin{itemize}
			\item We write $\mathcal{R}\subseteq\mathcal{Q}$ and say that $\mathcal{R}$
			\emph{refines} $\mathcal{Q}$ if considered as sets of ordered pairs
			we have $\mathcal{R}\subseteq\mathcal{Q}$.
			\item We define the \emph{product equivalence relation} $\mathcal{Q}^{n}$
			on $X^{n}$ by setting $x_{0}\dots x_{n-1}\sim x_{0}^{\prime}\dots x_{n-1}^{\prime}$
			if and only if $x_{i}\sim x_{i}^{\prime}$ for all $i=0,\dots,n-1$.
		\end{itemize}
	\end{defn}
	
	Our equivalence relations will satisfy the following specifications (Q4)-(Q6) below. Let $\left(\epsilon_{n}\right)_{n\in\mathbb{N}}$ be a decreasing sequence of positive numbers
	such that
	\begin{equation}
		\sum_{n\in\mathbb{N}}\epsilon_{n} < \infty.\label{eq:Condeps}
	\end{equation}
	To start, we let $\mathcal{Q}_{0}^{0}$ be the equivalence relation
	on $\mathcal{W}_{0}=\Sigma$ which has one equivalence class, that is,
	any two elements of $\Sigma$ are equivalent. 
	\begin{itemize}
		\item[(Q4)]  Suppose that $n=M(s)$. There is a specific number $J_{s(n),n}\in\mathbb{Z}^{+}$
		such that each word in $w_{n}\in\mathcal{W}_{n}$ is a concatenation
		$w_{n}=w_{n,1}\dots w_{n,J}$ of $J=J_{s(n),n}$ strings of equal
		length. Then any two words in the same $\mathcal{Q}_{s}^{n}$ class
		agree with each other except possibly on initial or final strings
		of length at most $\frac{\epsilon_{n}}{2}\frac{h_{n}}{J_{s(n),n}}$
		on the segments $w_{n,i}$ for $i=1,\dots,J_{s(n),n}$.
		\item[(Q5)]  For $n\geq M(s)+1$ we can consider words in $\mathcal{W}_{n}$
		as concatenations of words from $\mathcal{W}_{M(s)}$ and define $\mathcal{Q}_{s}^{n}$
		as the product equivalence relation of $\mathcal{Q}_{s}^{M(s)}$.
		\item[(Q6)]  $\mathcal{Q}_{s+1}^{n}$ refines $\mathcal{Q}_{s}^{n}$ and each
		$\mathcal{Q}_{s}^{n}$ class contains $2^{4e(n)}$ many $\mathcal{Q}_{s+1}^{n}$
		classes for some positive integer $e(n)$ chosen in \eqref{eq:kn}.
	\end{itemize}
	\begin{rem*}
		By (Q5) we can view $\mathcal{W}_{n}/\mathcal{Q}_{s}^{n}$ as sequences
		of elements $\mathcal{W}_{M(s)}/\mathcal{Q}_{s}^{M(s)}$ and similarly
		for $rev(\mathcal{W}_{n})/\mathcal{Q}_{s}^{n}$. In particular,
		it follows that $\mathcal{Q}_{0}^{n}$ is the equivalence relation
		on $\mathcal{W}_{n}$ which has one equivalence class. Moreover, it allows us to regard elements in $\mathcal{W}_{n}/\mathcal{Q}_{s}^{n}$
		for $n\geq M(s)+1$ as sequences of symbols in the alphabet $\left(\mathcal{W}_{M(s)}/\mathcal{Q}_{s}^{M(s)}\right)^{\ast}$. 
	\end{rem*}
	\begin{rem*}
		In case that the exponent is not relevant we will refer to the $\mathcal{Q}_{s}^{n}$
		as $\mathcal{Q}_{s}$. For $w\in\mathcal{W}_{n}$ we write $[w]_{s}$
		for its $\mathcal{Q}_{s}^{n}$ class.
	\end{rem*}
	As in \cite[section 5]{FRW} these specifications are sufficient to
	describe a canonical tower of factors. For this purpose, we consider
	the symbolic shift $\mathbb{K}\subset\Sigma^{\mathbb{Z}}$ with measure
	$\nu$ corresponding to $\Psi(\mathcal{T})$ for a fixed $\mathcal{T}\in\mathcal{T}\kern-.5mm rees$.
	
	\subsubsection{The odometer factor} \label{subsubsec:Odometer}
	
	We start by describing an odometer factor of our transformation which
	will turn out to be the Kronecker factor and will determine a natural
	``timing mechanism'' to detect blocks. 
	
	Suppose that $\sigma_{m}$ and $\sigma_{n}$ with $m<n$ are consecutive
	elements of $\mathcal{T}$ and let $x\in\mathbb{K}$. By unique readability
	there are unique $w,w'\in\mathcal{W}_{n}$ and $k\in\left\{ 0,\dots,h_{n}-1\right\} $
	such that $x\upharpoonright[-k,-k+2h_{n}-1]=ww'$. We call the interval
	of integers $[-k,-k+h_{n}-1]$ the $n$-block of $x$ containing $0$. 
	
	If the $m$-block of $x$ containing $0$ is given by $[-k',-k'+h_{m}-1]$,
	then there is a unique $\pi(x,n)$ such that $-k'=-k+\pi(x,n)h_{m}$.
	We can view $\pi(x,n)$ as an element in $\mathbb{Z}_{h_{n}/h_{m}}$
	and we define  
	\[
	\pi_{0}:\mathbb{K}\to\prod\left\{ \mathbb{Z}_{h_{n}/h_{m}}:\,\sigma_{m}\text{ and }\sigma_{n}\text{ with $m<n$ are consecutive elements of }\mathcal{T}\right\},
	\]
	by $\pi_0(x)=(\pi(x,n))_{\sigma_n\in \mathcal{T}}.$
	Since $\pi(sh(x),\cdot)$ is obtained from $\pi(x,\cdot)$ by adding
	$\bar{1}=(1,0,0,\dots)$ in the odometer $$O_{\mathcal{T}}\coloneqq\prod\left\{ \mathbb{Z}_{h_{n}/h_{m}}:\,\sigma_{m}\text{ and }\sigma_{n}\text{ with $m<n$ are consecutive elements of }\mathcal{T}\right\},$$
	we denote the corresponding odometer transformation by $\mathcal{O}_{\mathcal{T}}$
	and collect the following properties (see \cite[Lemma 22]{FRW}).
	\begin{lem}
		\label{lem:odometer} The map $\pi_{0}$ is a factor map from $\Psi(\mathcal{T})$
		to $\mathcal{O}_{\mathcal{T}}$. If $p>2$ is a prime number, then
		$\mathrm{e}^{2\pi\mathrm{i}/p}$ is an eigenvalue of the unitary operator
		associated with $\mathcal{O}_{\mathcal{T}}$ if and only if $p=\mathfrak{p}_{n}$
		for some $n$ with $\sigma_{n}\in\mathcal{T}$, where $\mathfrak{p}_n$ is the prime number from (E2).
	\end{lem}
	
	By the analogous procedure one associates the same odometer transformation
	$\mathcal{O}_{\mathcal{T}}$ to $\mathbb{K}^{-1}$ and we let $\pi_{0}^{\ast}:\mathbb{K}^{-1}\to\prod\mathbb{Z}_{h_{n}/h_{m}}$
	be the analogous map. We note that $\pi_{0}^{\ast}(rev(x))=i(\pi_{0}(x))$
	for the involution $i$ of $\mathcal{O}_{\mathcal{T}}$ given by $i(x)=-x$.
	
	\subsubsection{The canonical factors } \label{subsubsec:factors}
	
	Now we define a canonical sequence of invariant sub-$\sigma$-algebras
	of the algebra $\mathcal{B}(\mathbb{K})$ of measurable subsets of
	$\mathbb{K}$. 
	
	We let $\mathbb{K}_{0}$ be the odometer factor $O_{\mathcal{T}}$
	introduced above. For each $s\geq1$ a typical $x\in\mathbb{K}$ gives
	a bi-infinite sequence of the classes $\left\{ c_{j}^{(M(s),s)}:j=1,\dots,Q_{s}^{M(s)}\right\} $
	in $\mathcal{Q}_{s}^{M(s)}$ which yields a shift-equivariant map
	from $\mathbb{K}$ to $\left\{ 1,\dots,Q_{s}^{M(s)}\right\} ^{\mathbb{Z}}$.
	To be more precise, we define a map $\tilde{\pi}_{s}:\mathbb{K}\to\left\{ 1,\dots,Q_{s}^{M(s)}\right\} =\left(\mathcal{W}_{M(s)}/\mathcal{Q}_{s}^{M(s)}\right)^{\ast}$
	which assigns to a $x\in\mathbb{K}$ the letter $j$ if the word $w\in\mathcal{W}_{M(s)}$
	on the $M(s)$-block of $x$ containing the position $0$ satisfies
	$[w]_{s}=c_{j}^{(M(s),s)}$. Hereby, we define a 
	map $\pi_{s}:\mathbb{K}\to\mathbb{K}_{s}\coloneqq\left\{ 1,\dots,Q_{s}^{M(s)}\right\} ^{\mathbb{Z}}$
	by letting  $\pi_{s}(x)=\left(\tilde{\pi}_{s}(sh^{k}(x))\right)_{k\in\mathbb{Z}}.$ Note that this map is shift-equivariant, that is, $\pi_s(sh(x))=sh \left(\pi_{s}(x)\right)$. There is an analogous map from $rev(\mathbb{K})$ to $rev(\mathbb{K}_s)$ that we also denote by $\pi_s$.
	
	Informally, the definition of the map $\pi_s$ can be viewed as finding names as described in Section~\ref{subsec:basicsET} with respect to the partition of $\mathcal{W}_{M(s)}$-words into $\mathcal{Q}_{s}^{M(s)}$-classes.
	
	Next, we describe a convenient base for the topology on $\mathbb{K}_{s}$:
	For $n\geq M(s)$, $w\in\mathcal{W}_{n}$ and $0\leq k<h_{n}$, we
	let $\left\langle [w]_{s},k\right\rangle $ be the collection of $x\in\mathbb{K}_{s}$
	such that the position $0$ is at the $k$-th place in the $n$-block
	$B$ of $x$ containing $0$ and if $v\in\left\{ 1,\dots,Q_{s}^{M(s)}\right\} ^{h_{n}}$
	is the word in $x$ at the block $B$, then $v$ is the sequence of
	$\mathcal{Q}_{s}^{M(s)}$ classes given by $[w]_{s}$. Then the collection
	of those $\left\langle [w]_{s},k\right\rangle $ for $\sigma_{n}\in\mathcal{T}$,
	$w\in\mathcal{W}_{n}$ and $0\leq k<h_{n}$ forms a basis for the
	topology of $\mathbb{K}_{s}$ consisting of clopen sets. In particular,
	for $m=M(s)$ a word $w\in\mathcal{W}_{m}$ gives a word of length
	$h_{m}$ in our alphabet $\left(\mathcal{W}_{M(s)}/\mathcal{Q}_{s}^{M(s)}\right)^{\ast}$
	that is the repetition of the same letter. We denote the collection
	of these words by $\left(\mathcal{W}_{m}\right)_{s}^{\ast}$. Then
	for $n>m=M(s)$ with $\sigma_{n}\in\mathcal{T}$ each word in $\mathcal{W}_{n}$
	is a concatenation of $\frac{h_{n}}{h_{m}}$ words in $\mathcal{W}_{m}$
	by specification (E2) and, thus, determines a sequence of $\frac{h_{n}}{h_{m}}$
	many elements of $\left(\mathcal{W}_{m}\right)_{s}^{\ast}$. We let
	$\left(\mathcal{W}_{n}\right)_{s}^{\ast}$ be the collection of words
	in the alphabet $\left(\mathcal{W}_{M(s)}/\mathcal{Q}_{s}^{M(s)}\right)^{\ast}$
	arising this way. Then the sequence $\left((\mathcal{W}_{n})_{s}^{\ast}\right)_{n\geq M(s)}$
	gives a well-defined odometer-based construction sequence for $\mathbb{K}_{s}$
	in the alphabet $\left(\mathcal{W}_{M(s)}/\mathcal{Q}_{s}^{M(s)}\right)^{\ast}$.
	
	We also define the measure $\nu_{s}\coloneqq\pi_{s}^{\ast}\nu$ on
	$\mathbb{K}_{s}$. To be more explicit, with the aid of specifications
	(E2) and (Q6) we see that $\nu_{s}(\left\langle [w]_{s},k\right\rangle)=\frac{1}{h_nQ_{s}^{n}}$ for any $w\in \mathcal{W}_{n}$ and $0\leq k<h_n$.
	
	Finally, we let $\mathcal{H}_{0}$ be the shift-invariant sub-$\sigma$-algebra
	of the $\mathcal{B}(\mathbb{K})$ generated by the collection of $\pi_{0}^{-1}(B)$,
	where $B$ is a basic open set in $O_{\mathcal{T}}$, and we let $\mathcal{H}_{s}$
	be the shift-invariant sub-$\sigma$-algebra of the $\mathcal{B}(\mathbb{K})$
	generated by the collection of $\pi_{s}^{-1}(B)$, where $B$ is a
	basic open set in $\mathbb{\mathbb{K}}_{s}$. Then $\mathcal{H}_{s}$
	is the sub-$\sigma$-algebra determined by the factor map $\pi_{s}$. 
	
	Since the equivalence relation $\mathcal{Q}_{s+1}^{M(s+1)}$ refines
	$\left(\mathcal{Q}_{s}^{M(s)}\right)^{h_{M(s+1)}/h_{M(s)}}$ by specifications
	(E2) and (Q6), we have $\mathcal{H}_{s+1}\supseteq\mathcal{H}_{s}$
	and a continuous factor map $\pi_{s+1,s}:\mathbb{K}_{s+1}\to\mathbb{K}_{s}$.
	To express this one explicitly, we note that a $\mathbb{Z}$-sequence
	of $\mathcal{Q}_{s+1}^{M(s+1)}$ classes determines a sequence of
	$\mathcal{Q}_{s}^{M(s)}$ classes, because a $\mathcal{Q}_{s+1}^{M(s+1)}$
	class is contained in a $\mathcal{Q}_{s}^{M(s+1)}$ class which is
	a $h_{M(s+1)}/h_{M(s)}$-tuple of $\mathcal{Q}_{s}^{M(s)}$ classes. Analogously we define a map from $rev(\mathbb{K}_{s+1})$ to $rev(\mathbb{K}_{s})$ that we also denote by $\pi_{s+1,s}$.
	
	In analogy with \cite[Proposition 23]{FRW} we can prove the following
	statement making use of specification (Q4).
	\begin{lem}
		\label{lem:algebra}$\mathcal{B}\left(\mathbb{K}\right)$ is the smallest
		invariant $\sigma$-algebra that contains $\bigcup_{s}\mathcal{H}_{s}$.
	\end{lem}
	
	\begin{proof}
		It suffices to show that for any $m\in\mathbb{N}$ with $\sigma_m \in \mathcal{T}$ and for any $u\in\mathcal{W}_{m}$
		the basic open set $\left\langle u\right\rangle \subseteq\mathbb{K}$
		can be approximated arbitrarily well in measure by elements of $\bigcup_{s}\mathcal{H}_{s}$.
		For this purpose, we fix $\varepsilon>0$ and choose $n>m$ such that
		$n=M(s)$ for some $s\in\mathbb{N}$ and $\epsilon_{n}+2\frac{h_{m}J_{s(n),n}}{h_{n}}<\varepsilon$.
		Then we let $G\subset\mathbb{K}$ be the collection of $x$ such that
		within the $n$-word $w_{n}\in\mathcal{W}_{n}$ located at the $n$-block
		of $x$ containing the position $0$, $x(0)$ is not among
		the first or last $\left(\frac{\epsilon_{n}}{2}+\frac{h_{m}J_{s(n),n}}{h_{n}}\right)\frac{h_{n}}{J_{s(n),n}}=\frac{\epsilon_{n}}{2}\frac{h_{n}}{J_{s(n),n}}+h_{m}$
		letters of one of the $J_{s(n),n}$ many segments $w_{n,i}$ in specification
		(Q4), that is, the $m$-block of $x$ containing position $0$ belongs to central part of one of the $w_{n,i}$ strings  of the $n$-block. By the ergodic
		theorem and our assumption on $n$ the measure of $G$ is at least
		$1-\varepsilon$. 
		
		Now let $x\in G\cap\left\langle u\right\rangle $ and let $w\in\mathcal{W}_{n}$
		be located at the $n$-block $B=[-k,-k+h_{n})$ of $x$ containing $0$.
		If $c_{j}^{(n,s)}$ is the $\mathcal{Q}_{s}^{M(s)}$ class of $w$,
		then 
		\[
		sh^{-k}(x)\in\pi_{s}^{-1}\left(\left\langle c_{j}^{(n,s)},0\right\rangle \right)\subseteq sh^{-k}\left(\left\langle u\right\rangle \right)
		\]
		by specification (Q4). Hence, $G\cap\left\langle u\right\rangle $
		is a union of shifts of sets of the form $\pi_{s}^{-1}\left(\left\langle c_{j}^{(n,s)},0\right\rangle \right)$
		lying in $\mathcal{H}_{s}$.
	\end{proof}
	As in the proof above, for $n=M(s)$, the set $L_{n}$ of $x$ such
	that within the $n$-word $w_{n}\in\mathcal{W}_{n}$ located at the
	$n$-block of $x$ containing the position $0$ we have that $x(0)$
	is among the first or last $\frac{\epsilon_{n}}{2}\frac{h_{n}}{J_{s(n),n}}$
	letters of one of the $J_{s(n),n}$ many segments $w_{n,i}$ in specification
	(Q4), has measure $\epsilon_{n}$. Let $L$
	be the collection of $x\in\mathbb{K}$ such that there exists $s(x)$ such that for all $s\ge s(x)$,
	if $n=M(s)$ then $x$ does not belong to $L_{n}$, that is, $x(0)$ is
	not among the first or last $\frac{\epsilon_{n}}{2}\frac{h_{n}}{J_{s(n),n}}$
	letters of one of the $J_{s(n),n}$ many segments $w_{n,i}$ within
	the $n$-word $w_{n}\in\mathcal{W}_{n}$ located at the $n$-block
	of $x$ containing the position $0$. Then $L$ has measure one by
	equation (\ref{eq:Condeps}) and the Borel-Cantelli Lemma.
	
	We collect the following properties as in Propositions 24 and 25 of
	\cite{FRW}.
	\begin{lem}
		\label{lem:subalgebra}
		\begin{enumerate}
			\item For all $x\neq y$ belonging to $L$, there is an open set $S\in\bigcup_{s}\mathcal{H}_{s}$
			such that $x\in S$ and $y\notin S.$
			\item For all $s\geq1$, $\mathcal{H}_{s}$ is a strict subalgebra of $\mathcal{H}_{s+1}$.
			Moreover, if $\left\{ \nu_{x}:x\in\mathbb{K}_{s}\right\} $ is the
			disintegration of $\nu_{s+1}$ over $\nu_{s}$, then for $\nu_{s}$-a.e.
			$x$ the measure $\nu_{x}$ is nonatomic. 
		\end{enumerate}
	\end{lem}

	\subsection{\label{subsec:Groups}Groups of involutions}
	
	To approximate conjugacies we will use groups associated to trees
	as in \cite{FRW}. These groups will be direct sums of $\mathbb{Z}_{2}=\mathbb{Z}/2\mathbb{Z}$
	and are called \emph{groups of involutions}. 
	
	If $G=\bigoplus_{i \in I} \mathbb{Z}_2$ with some index set $I$ is a group of involutions and $B=\left\{ r_{i}:i\in I\right\} $
	is a distinguished basis (that is, a particular choice of basis for $G$ considered as a vector space over $\mathbb{Z}_2$), then we call the elements $r_{i}\in B$ 
	\emph{canonical generators} and we have a well-defined notion of \emph{parity}
	for elements in $G$: an element $g\in G$ is called \emph{even} if
	it can be written as the sum of an even number of elements in $B$.
	Otherwise, it is called \emph{odd}. Parity is preserved under homomorphisms
	sending the canonical generators of one group to the canonical generators of the other. Moreover, for an inverse limit system
	of groups of involutions $\left\{ G_{s}:s\in\mathbb{Z}^+\right\} $ (where each group has a set of canonical generators)
	with parity-preserving homomorphisms
	$\rho_{t,s}:G_{t}\to G_{s}$ for $0<s<t$, the elements of the inverse
	limit $\underleftarrow{\lim}G_{s}$ have a well-defined parity. 
	
	For a tree $\mathcal{T}\subset\mathbb{N}^{\mathbb{N}}$ we assign
	to each level $s>0$ a group of involutions $G_{s}\left(\mathcal{T}\right)$
	by taking a sum of copies of $\mathbb{Z}_{2}$ indexed by the nodes
	of $\mathcal{T}$ at level $s$. Let
	\[
	G_s(\mathcal{T})=\bigoplus_{\tau\in\mathcal{T},\,lh(\tau)=s}\left(\mathbb{Z}_{2}\right)_{\tau}.
	\]
	Each node of $\mathcal{T}$ at level $s$ corresponds to an element of $G_s(\mathcal{T})$ that is 1 in 
	the copy of $\mathbb{Z}_2$ for that node and $0$ in the other copies of $\mathbb{Z}_2$.
	We also define $G_{0}(\mathcal{T})$ to be the trivial group. For levels $0<s<t$ of $\mathcal{T}$ we have a canonical homomorphism
	$\rho_{t,s}:G_{t}\left(\mathcal{T}\right)\to G_{s}\left(\mathcal{T}\right)$
	that sends a generator $\tau$ of $G_{t}\left(\mathcal{T}\right)$
	to the unique generator $\sigma$ of $G_{s}\left(\mathcal{T}\right)$
	that is an initial segment of $\tau$. The map $\rho_{t,0}$ is the trivial homomorphism $\rho_{t,0}:G_t(\mathcal{T})\to G_0(\mathcal{T})=\{0\}.$ We denote the inverse limit
	of $\left\langle G_{s}\left(\mathcal{T}\right),\rho_{t,s}:s<t\right\rangle $
	by $G_{\infty}\left(\mathcal{T}\right)$ and we let $\rho_{s}:G_{\infty}\left(\mathcal{T}\right)\to G_{s}\left(\mathcal{T}\right)$
	be the projection map.
	
	Since there is a one-to-one correspondence between the infinite branches
	of $\mathcal{T}$ and infinite sequences $\left(g_{s}\right)_{s\in\mathbb{Z}^+}$
	of generators $g_{s}\in G_{s}\left(\mathcal{T}\right)$ with $\rho_{t,s}\left(g_{t}\right)=g_{s}$
	for $t>s>0$, we obtain the following characterization.
	\begin{fact}[{\cite[Lemma~17]{FRW}}] \label{fact:oddElement}
		Let $\mathcal{T}\subset\mathbb{N}^{\mathbb{N}}$ be a tree. Then $G_{\infty}\left(\mathcal{T}\right)$
		has a nonidentity element of odd parity if and only if $\mathcal{T}$
		is ill-founded (that is, $\mathcal{T}$ has an infinite branch). 
	\end{fact}
	
	In order to make the elements of $G_{\infty}\left(\mathcal{T}\right)$
	correspond to conjugacies, we will build symmetries into our construction
	using the following finite approximations to $G_{\infty}\left(\mathcal{T}\right)$.
	We let $G_{0}^{n}\left(\mathcal{T}\right)$ be the trivial group and
	for $s>0$ we let 
	\[
	G_{s}^{n}\left(\mathcal{T}\right)=\bigoplus\left(\mathbb{Z}_{2}\right)_{\tau}\text{ where the sum is taken over }\tau\in\mathcal{T}\cap\left\{ \sigma_{m}:m\leq n\right\} ,\,lh(\tau)=s.
	\]
	
	When $\mathcal{T}$ is clear from the context, we will frequently
	write $G_{s}^{n}$. We also introduce the finite approximations $\rho_{t,s}^{(n)}:G_{t}^{n}(\mathcal{T})\to G_{s}^{n}(\mathcal{T})$
	to the canonical homomorphisms.
	
	During the course of construction we will define group actions of
	$G_{s}^{n}$ on our quotient spaces $\mathcal{W}_{n}/\mathcal{Q}_{s}^{n}$.
	Here, we will need to control systems of such group actions on the
	refining equivalence relations. For that purpose, the following general
	definitions will prove useful.
	\begin{defn}
		Suppose
		\begin{itemize}
			\item $\mathcal{Q}$ and $\mathcal{R}$ are equivalence relations on a set
			$X$ with $\mathcal{R}$ refining $\mathcal{Q}$,
			\item $G$ and $H$ are groups with $G$ acting on $X/\mathcal{Q}$ and
			$H$ acting on $X/\mathcal{R}$,
			\item $\rho:H\to G$ is a homomorphism.
		\end{itemize}
		Then we say that the $H$ action is \emph{subordinate} to the $G$
		action if for all $x\in X$, whenever $[x]_{\mathcal{R}}\subset[x]_{\mathcal{Q}}$
		we have $h[x]_{\mathcal{R}}\subset\rho(h)[x]_{\mathcal{Q}}$. 
	\end{defn}
	
	\begin{defn}
		If $G$ acts on $X$, then the canonical \emph{diagonal action} of
		$G$ on $X^{n}$ is defined by 
		\[
		g\left(x_{0}x_{1}\dots x_{n-1}\right)=gx_{0}\,gx_{1}\dots gx_{n-1} \ \text{ for any } g \in G.
		\]
		If $G$ is a group of involutions with a collection
		of canonical generators, then we define the \emph{skew diagonal action}
		on $X^{n}$ by setting 
		\[
		g\left(x_{0}x_{1}\dots x_{n-1}\right)= \begin{cases}
			gx_{0}\,gx_{1}\dots gx_{n-1}, & \text{ if $g \in G$ is of even parity,} \\
			gx_{n-1}\,gx_{n-2}\dots gx_{0}, & \text{ if $g\in G$ is of odd parity.}
		\end{cases}
		\]
	\end{defn}

	\begin{rem*}
		Recalling the notion of a product equivalence relation from Definition
		\ref{def:equivrel} we can identify $X^{n}/\mathcal{Q}^{n}$ with
		$\left(X/\mathcal{Q}\right)^{n}$ in an obvious way. Then we can also
		extend an action of $G$ on $X/\mathcal{Q}$ to the diagonal or skew
		diagonal actions on $\left(X/\mathcal{Q}\right)^{n}$ in a straightforward
		way.
	\end{rem*}
	To extend our group actions we will use the following extension lemma
	from \cite[Lemma 46]{FRW}. 
	\begin{lem}
		\label{lem:extension}Let $X$ be a set and $\mathcal{R}\subseteq\mathcal{Q}$
		be equivalence relations on $X$. Suppose that
		\begin{itemize}
			\item $\rho:H\to G\times\mathbb{Z}_{2}$ is a homomorphism,
			\item $G\times\mathbb{Z}_{2}$ acts on $X/\mathcal{Q}$,
			\item $H$ acts on $X/\mathcal{R}$ by a free action subordinate to the
			$G\times\mathbb{Z}_{2}$ action on $X/\mathcal{Q}$,
			\item if $H_{1}=\left\{ h\in H:\rho(h)=(0,i),i\in\left\{ 0,1\right\} \right\} $,
			then every orbit of the $\mathbb{Z}_{2}$ factor of $G\times\mathbb{Z}_{2}$
			contains an even number of $H_{1}$ orbits.
		\end{itemize}
		Then there is a free action of $H\times\mathbb{Z}_{2}$ on $X/\mathcal{R}$
		subordinate to the $G\times\mathbb{Z}_{2}$ action via the map $\rho^{\prime}(h,i)=\rho(h)+(0,i)$.
	\end{lem}
	
	We now list specifications on the group actions.
	\begin{itemize}
		\item[(A7)]  $G_{s}^{n}$ acts freely on $\mathcal{W}_{n}/\mathcal{Q}_{s}^{n}$
		and the $G_{s}^{n}$ action is subordinate to the $G_{s-1}^{n}$ action
		on $\mathcal{W}_{n}/\mathcal{Q}_{s-1}^{n}$ via the canonical homomorphism
		$\rho_{s,s-1}^{(n)}:G_{s}^{n}\to G_{s-1}^{n}$.
		\item[(A8)]  Suppose $M(s)<n$, $\sigma_{m}$ and $\sigma_{n}$ are consecutive
		elements of $\mathcal{T}$ and we view $G_{s}^{n}=G_{s}^{m}\oplus H$.
		Then the action of $G_{s}^{m}$ on $\mathcal{W}_{m}/\mathcal{Q}_{s}^{m}$
		is extended to an action on $\mathcal{W}_{n}/\mathcal{Q}_{s}^{n}$
		by the skew diagonal action. 
	\end{itemize}
	\begin{rem}
		\label{rem:Closed-under-skew} In particular, in the above situation
		with $M(s)<n$ and $\sigma_{m},\sigma_{n}$ consecutive elements of
		$\mathcal{T}$ both specifications together yield that $\mathcal{W}_{n}/\mathcal{Q}_{s}^{n}$
		is closed under the skew diagonal action by $G_{s}^{m}$. Clearly,
		this also holds if we view each element in $\mathcal{W}_{n}/\mathcal{Q}_{s}^{n}$
		as a sequence in the alphabet $\left(\mathcal{W}_{M(s)}/\mathcal{Q}_{s}^{M(s)}\right)^{\ast}$.
	\end{rem}

	\section{\label{sec:Isom}Infinite branches give isomorphisms}
	
	As in \cite{FRW} the specifications on our transformations stated
	in the previous section suffice to build an isomorphism between $\Psi(\mathcal{T})$
	and $\Psi(\mathcal{T})^{-1}$ in case that the tree $\mathcal{T}$
	has an infinite branch.
	\begin{lem}
		\label{lem:siso}Let $s\in\mathbb{N}$ and $g\in G_{s}^{m}$ for some
		$m\in\mathbb{N}$. Suppose that $g$ has odd parity. Then there is 
		a shift-equivariant isomorphism $\eta_{g}:\mathbb{K}_{s}\to rev(\mathbb{K}_{s})$ canonically associated
		to $g$. 
		
		Moreover, if $s'>s$ and $g^{\prime}\in G_{s'}^{n}$ for some $n\geq m$
		with $\rho_{s',s}(g^{\prime})=g$, then $\pi_{s',s}\circ\eta_{g^{\prime}}=\eta_{g}\circ\pi_{s',s}$.
	\end{lem}
	
	\begin{proof}
		We recall from Section \ref{subsec:Relations} that $\left((\mathcal{W}_{m})_{s}^{\ast}\right)_{m\geq M(s)}$
		defines a construction sequence for $\mathbb{K}_{s}$. Similarly,
		$\left(rev\left((\mathcal{W}_{m})_{s}^{\ast}\right)\right)_{m\geq M(s)}$
		is a construction sequence for $rev(\mathbb{K}_{s})$. For
		$n>m$ there is $K\in\mathbb{N}$ such that we can write any element
		$[w]_{s}^{\ast}\in(\mathcal{W}_{n})_{s}^{\ast}$ as $[w]_{s}^{\ast}=[w_{0}]_{s}^{\ast}[w_{1}]_{s}^{\ast}\dots[w_{K-1}]_{s}^{\ast}$
		with $[w_{i}]_{s}^{\ast}\in(\mathcal{W}_{m})_{s}^{\ast}$ for $i=0,\dots,K-1$.
		By specification (A8) the group action is given by the skew diagonal
		action and, hence, we have
		\[
		g\left([w_{0}]_{s}^{\ast}[w_{1}]_{s}^{\ast}\dots[w_{K-1}]_{s}^{\ast}\right)=g[w_{K-1}]_{s}^{\ast}g[w_{K-2}]_{s}^{\ast}\dots g[w_{0}]_{s}^{\ast}.
		\]
		As pointed out in Remark \ref{rem:Closed-under-skew}, $(\mathcal{W}_{n})_{s}^{\ast}$
		is closed under the skew diagonal action. Thus $g[w_{K-1}]_{s}^{\ast}g[w_{K-2}]_{s}^{\ast}\dots g[w_{0}]_{s}^{\ast}\in(\mathcal{W}_{n})_{s}^{\ast}$,
		which implies $g[w_{0}]_{s}^{\ast}g[w_{1}]_{s}^{\ast}\dots g[w_{K-1}]_{s}^{\ast}\in rev\left((\mathcal{W}_{n})_{s}^{\ast}\right)$.
		Hence the map 
		\begin{equation} \label{eq:giso}
			[w_{0}]_{s}^{\ast}[w_{1}]_{s}^{\ast}\dots[w_{K-1}]_{s}^{\ast}\mapsto g[w_{0}]_{s}^{\ast}g[w_{1}]_{s}^{\ast}\dots g[w_{K-1}]_{s}^{\ast}   
		\end{equation}
		is an invertible map from the construction
		sequence for $\mathbb{K}_{s}$ to the construction
		sequence for $rev(\mathbb{K}_{s})$. It can also be interpreted as a shift-equivariant map from cylinder sets in $\mathbb{K}_{s}$ to cylinder sets located in the same position in $rev(\mathbb{K}_{s})$.  This yields the isomorphism in the 
		first assertion.
		
		
		The second assertion follows from specification (A7), which says that the
		action by $g^{\prime}$ is subordinate to the action by $g$ via the
		homomorphism $\rho_{s',s}$. 
	\end{proof}
	
	\begin{rem} \label{rem:etag}
		Let $s\in\mathbb{N}$ and $g\in G_{s}^{m}$ for some
		$m\in\mathbb{N}$. Suppose that $g$ has odd parity. Then $g$ yields
		a shift-equivariant isomorphism $\eta_{g}:\mathbb{K}_{s}\to rev(\mathbb{K}_{s})$ as in Lemma \ref{lem:siso}. For every $n \geq m$ it also induces a map $\mathcal{W}_n/\mathcal{Q}^n_s \to rev(\mathcal{W}_n)/\mathcal{Q}^n_s$ that we denote by $\eta_g$ as well.
	\end{rem}
	
	In the following we call a sequence of isomorphisms $\zeta_{s}$ between
	$\mathbb{K}_{s}$ and $rev(\mathbb{K}_{s})$ \emph{coherent} if $\pi_{s+1,s}\circ\zeta_{s+1}=\zeta_{s}\circ\pi_{s+1,s}$
	for every $s\in\mathbb{N}$.
	\begin{lem}
		\label{lem:iso}Let $\left(\zeta_{s}\right)_{s\in\mathbb{N}}$ be a
		coherent sequence of isomorphisms between $\mathbb{K}_{s}$ and $rev(\mathbb{K}_{s})$.
		Then there is an isomorphism $\zeta:\mathbb{K}\to rev(\mathbb{K})$
		such that $\pi_{s}\circ\zeta=\zeta_{s}\circ\pi_{s}$ for every $s\in\mathbb{N}$.
	\end{lem}
	
	\begin{proof}
		Since the isomorphisms $\zeta_{s}$ cohere, their inverse limit defines
		a measure-preserving isomorphism between the subalgebra of $\mathcal{B}\left(\mathbb{K}\right)$
		generated by $\bigcup_{s}\mathcal{H}_{s}$ and the subalgebra of $\mathcal{B}\left(rev(\mathbb{K})\right)$
		generated by $\bigcup_{s}rev(\mathcal{H}_{s})$. By Lemma \ref{lem:algebra}
		this extends uniquely to a measure-preserving isomorphism $\tilde{\zeta}$
		between $\mathcal{B}\left(\mathbb{K}\right)$ and $\mathcal{B}\left(rev(\mathbb{K})\right)$.
		Then by part (1) of Lemma \ref{lem:subalgebra} we can find sets $D\subset\mathbb{K}$,
		$D'\subset rev(\mathbb{K})$ of measure zero such that $\tilde{\zeta}$
		determines a shift-equivariant isomorphism $\zeta$ between $\mathbb{K}\setminus D$
		and $rev(\mathbb{K})\setminus D'$.
	\end{proof}
	Now we are ready to prove the first half of Proposition \ref{prop:criterion}.
	\begin{proof}[Proof of part (1) in Proposition \ref{prop:criterion}]
		Suppose that $\mathcal{T}\in\mathcal{T}\kern-.5mm rees$ has an infinite branch.
		Then $G_{\infty}(\mathcal{T})$ has an element $g$ of odd parity
		according to Fact \ref{fact:oddElement}. By Lemma \ref{lem:siso} we obtain
		a coherent sequence of isomorphisms $\zeta_s \coloneqq \eta_{\rho_{s}(g)}$ between
		$\mathbb{K}_{s}$ and $rev(\mathbb{K}_{s})$. Hence, Lemma \ref{lem:iso}
		yields an isomorphism between $\mathbb{K}$ and $rev(\mathbb{K})$.
		Since $rev(\mathbb{K})$ is isomorphic to $\mathbb{K}^{-1}$, we conclude
		that $\mathbb{K}\cong\mathbb{K}^{-1}$.
	\end{proof}
	
	\section{\label{sec:Feldman}Feldman Patterns}
	
	In \cite{Fe} Feldman constructed the first example of an ergodic
	zero-entropy automorphism that is not loosely Bernoulli. The construction
	is based on the idea that no pair of the following strings 
	\begin{align*}
		abababab\\
		aabbaabb\\
		aaaabbbb
	\end{align*}
	can be matched very well. In our construction we will concatenate
	variants of his blocks that we call \emph{Feldman patterns}. 
	
	We start with an observation that follows along the lines of Lemma
	6.5 and Remark 6.6 in \cite{GeKu}, whose proofs were inspired by
	\cite[Theorem 4]{Fe} and \cite[Proposition 1.1, p.79]{ORW}.
	\begin{lem}
		\label{lem:DifferentT}Suppose $a_{1},a_{2},\dots,a_{N}$ are distinct
		symbols in an alphabet $\Sigma.$ Let $M,S,T,j,k\in\mathbb{Z}^{+}$,
		$S\geq T$, $M\geq j>k$, and
		\begin{align*}
			B_{j}=\left(a_{1}^{SN^{2j}}a_{2}^{SN^{2j}}\dots a_{N}^{SN^{2j}}\right)^{N^{2(M+1-j)}}, &  & B_{k}= & \left(a_{1}^{TN^{2k}}a_{2}^{TN^{2k}}\dots a_{N}^{TN^{2k}}\right)^{N^{2(M+1-k)}}.
		\end{align*}
		Suppose $B$ and $\overline{B}$ are strings of consecutive symbols
		in $B_{j}$ and $B_{k},$ respectively, where $|B|\ge SN^{2M+2}$.
		Assume that $N\ge20$. Then 
		\[
		\overline{f}(B,\overline{B})>1-\frac{4}{\sqrt{N}}.
		\]
	\end{lem}
	
	\begin{proof}
		By removing fewer than $2SN^{2j}$ symbols from the beginning and
		end of $B,$ we can decompose the remaining part of $B$ into strings
		$C_{1},C_{2},\dots,C_{r}$ each of the form $a_{i}^{SN^{2j}}.$ Since
		$2SN^{2j}\le2\frac{|B|+|\overline{B}|}{N^{2}},$ it follows from Fact
		\ref{fact:omit_symbols} that removing these symbols increases the
		$\overline{f}$ distance between $B$ and $\overline{B}$ by less
		than $\frac{4}{N^{2}}.$ Let $\overline{C}_{1},\overline{C}_{2},\dots,\overline{C}_{r}$
		be the decomposition of $\overline{B}$ into substrings corresponding
		to $C_{1},C_{2},\dots,C_{r}$ under a best possible match between
		$C_{1}C_{2}\cdots C_{r}$ and $\overline{B}.$
		
		Let $i\in\{1,2,\dots,r\}.$
		\begin{casenv}
			\item $|\overline{C}_{i}|<\frac{3}{2\sqrt{N}}|C_{i}|$. Then
			$\overline{f}(C_{i},\overline{C}_{i})>1-\frac{3}{\sqrt{N}}$.
			\item $|\overline{C}_{i}|\ge\frac{3}{2\sqrt{N}}|C_{i}|=\frac{3}{2}SN^{2j-(1/2)}$.
			The length of a cycle $a_{1}^{TN^{2k}}a_{2}^{TN^{2k}}\cdots a_{N}^{TN^{2k}}$
			in $B_{k}$ is at most $TN^{2j-1}.$ Therefore $\overline{C}_{i}$
			contains at least $\lfloor\frac{3\sqrt{N}}{2}\rfloor-1>\sqrt{N}$
			complete cycles. Thus deleting any partial cycles at the beginning
			and end of $\overline{C}_{i}$ would increase the $\overline{f}$
			distance between $C_{i}$ and $\overline{C}_{i}$ by less than $\frac{2}{\sqrt{N}}.$
			On the rest of $\overline{C}_{i}$, only $\frac{1}{N}$ of the symbols
			can match the symbol in $C_{i}.$ Thus $\overline{f}(C_{i},\overline{C}_{i})>1-\frac{2}{\sqrt{N}}-\frac{2}{N}>1-\frac{3}{\sqrt{N}}.$
		\end{casenv}
		Therefore, by Fact \ref{fact:substring_matching}, $\overline{f}(C_{1}C_{2}\cdots C_{r},\overline{C}_{1},\overline{C}_{2}\cdots,\overline{C}_{r})>1-\frac{3}{\sqrt{N}}.$ Hence $\overline{f}(B,\overline{B})>1-\frac{3}{\sqrt{N}}-\frac{4}{N^2}>1-\frac{4}{\sqrt{N}}.$
	\end{proof}
	Let $T,N,M\in\mathbb{Z}^{+}$. A $(T,N,M)$-Feldman pattern in building
	blocks $A_{1},\dots,A_{N}$ of equal length $L$ is one of the strings
	$B_{1},\dots,B_{M}$ that are defined by
	\begin{eqnarray*}
		B_{1}= &  & \left(A_{1}^{TN^{2}}A_{2}^{TN^{2}}\dots A_{N}^{TN^{2}}\right)^{N^{2M}}\\
		B_{2}= &  & \left(A_{1}^{TN^{4}}A_{2}^{TN^{4}}\dots A_{N}^{TN^{4}}\right)^{N^{2M-2}}\\
		\vdots &  & \vdots\\
		B_{M}= &  & \left(A_{1}^{TN^{2M}}A_{2}^{TN^{2M}}\dots A_{N}^{TN^{2M}}\right)^{N^{2}}
	\end{eqnarray*}
	
	Thus $N$ denotes the number of building blocks, $M$ is the number
	of constructed patterns, and $TN^2$ gives the minimum number of consecutive occurrences
	of a building block. We also note that each building block $A_{i}$,
	$1\leq i\leq N$, occurs $TN^{2M+2}$ many times in each pattern.
	Every block $B_{j}$, $1\leq j\leq M$, has total length $TN^{2M+3}L$.
	Moreover, we notice that $B_{j}$ is built with $N^{2(M+1-j)}$ many
	so-called \emph{cycles}: Each cycle winds through all the $N$ building
	blocks.
	
	
	We conclude this analysis with a statement on the $\overline{f}$-distance between
	different $(T,N,M)$-Feldman patterns.
	\begin{prop}[Distance between different Feldman patterns in $\overline{f}$]
		\label{prop:Feldman} Let $N\geq20$, $M\geq2$, and $B_{j}$, $1\leq j\leq M$,
		be the $(T,N,M)$-Feldman patterns in the building blocks $A_{1},\dots,A_{N}$
		of equal length $L$. Assume that $\alpha\in\left(0,\frac{1}{7}\right)$,
		$R\ge 2$, and $\overline{f}(C,D)>\alpha$, for all substrings $C$ and
		$D$ consisting of consecutive symbols from $A_{i_{1}}$ and $A_{i_{2}}$,
		respectively, where $i_{1}\neq i_{2}$, with $|C|,|D|\geq\frac{L}{R}$. 
		
		Then for all $j,k\in\left\{ 1,\dots,M\right\} $, $j\neq k$, and
		all sequences $B$ and $\bar{B}$ of at least $TN^{2M+2}L$ consecutive
		symbols from $B_{j}$ and $B_{k}$, respectively, we have 
		\begin{equation}
			\overline{f}\left(B,\bar{B}\right)\geq\alpha-\frac{4}{\sqrt{N}}-\frac{1}{R}.
		\end{equation}
	\end{prop}
	
	\begin{proof}
		Using Lemma \ref{lem:symbol by block replacement} the statement follows
		from Lemma \ref{lem:DifferentT} with $S=T$.
	\end{proof}
	
	In our iterative construction process, that we present in Sections \ref{sec:Substitution} and \ref{sec:Construction}, we substitute Feldman patterns of finer equivalence classes of words into Feldman patterns of coarser classes. In Section \ref{sec:Non-Equiv} we need estimates on the $\overline{f}$ distance between such different Feldman patterns of equivalence classes even under finite coding. Since we do not know the precise words being coded but just their equivalence classes, the following Coding Lemma will prove useful, especially in Lemma \ref{lem:distDiff} and Lemma \ref{lem:BadCoding0}.
	
	\begin{rem*}
		The Coding Lemma below can be applied in the case $A=B_{k}$ and $B=B_{j}$,
		where $B_{k}$ and $B_{j}$ are as in Lemma \ref{lem:DifferentT}, with $k\ne j,$ except
		in the definition of $B_{j}$, the $a_{i}$'s are replaced by $b_{i}$'s,
		and we assume $T=S.$ In case $j>k,$ each $\Lambda_{im}$ in the
		Coding Lemma is equal to $N^{2(j-k)-1}$ cycles in $B_{k},$ and the
		permutations are the identity, while $ \Gamma_{im}=b_{m}^{TN^{2j}}.$
		If $k>j,$ each substring $\Lambda_{i1}\Lambda_{i2}\cdots\Lambda_{iN}$
		of $B_{k}$ actually consists of repetitions of a single symbol, and
		we again take $\Gamma_{im}=b_{m}^{TN^{2j}}.$ The permutations in
		the Coding Lemma allow for the more general situation that occurs
		in Lemma \ref{lem:Occurrence-Substitutions} and Remark \ref{rem:Substitution-Dagger} below.
	\end{rem*}
	
	\begin{lem}[Coding Lemma]\label{lem:CodingLemma} 
		Let $B_{m \ell}^{(i)}$, $1\le i\le  p,$ $1\le m\le N$, $1\le \ell\le q,$
		and $A_{i\ell}$, $1\le i\le p$, $1\le \ell\le q$, be blocks of symbols,
		with each block of length $L.$ Assume $q\geq N>2^{16}.$ Suppose $R_{1}\ge2,$
		$\alpha\in(0,1/7),$ and for all substrings $C$ and $D$ consisting
		of consecutive symbols in $B_{m\ell}^{(i)}$ and $B_{m'\ell'}^{(i')},$
		respectively, with $|C|,|D|\ge L/R_{1},$ we have $\overline{f}(C,D)\ge\alpha$
		if $m\ne m'.$ For $1\le i\le p$, $1\le m\le N,$ let $\Lambda_{im}$
		be a permutation (depending on $i$ and $m$) of the strings $A_{i1},A_{i2},\dots,A_{iq},$ and
		let $\Gamma_{im}=B_{m1}^{(i)}B_{m2}^{(i)}\cdots B_{mq}^{(i)}.$
		Let 
		\[
		A=\Lambda_{11}\Lambda_{12}\cdots\Lambda_{1N}\Lambda_{21}\Lambda_{22}\cdots\Lambda_{2N}\cdots\Lambda_{p1}\Lambda_{p2}\cdots\Lambda_{pN}
		\]
		and 
		\[
		B=\Gamma   _{11}\Gamma_{12}\cdots\Gamma_{1N}\Gamma_{21}\Gamma_{22}\cdots\Gamma_{2N}\cdots\Gamma_{p1}\Gamma_{p2}\cdots\Gamma_{pN}.
		\]
		Suppose $R_{2}\ge2$ and $\overline{A}$,$\overline{B}$ are strings
		of consecutive symbols in $A$,$B$, respectively, each of length at
		least $pqNL/R_{2}.$ Then 
		\[
		\overline{f}(\overline{A},\overline{B})\ge\alpha\left[\frac{1}{8}-\frac{2}{\sqrt[4]{N}}\right]-\frac{1}{R_{1}}-\frac{4R_{2}}{p}.
		\]
	\end{lem}
	
	\begin{proof}
		We begin by considering, instead of strings of symbols $B_{m\ell}^{(i)}$
		and $A_{i\ell},$ individual symbols, $b_{m\ell}^{(i)},  \text{\ 1\ensuremath{\le i\le p,} } 1\le m\le N, \text{\ 1\ensuremath{\le \ell\le q,}}$
		such that $b_{m\ell}^{( i)}\ne b_{m'\ell'}^{(i')}$ for $m\ne m',$
		and $a_{i{\ell}},$ $1\le i\le p,$ $1\le \ell\le q.$ Let $\lambda_{im}$
		be the permutation of $a_{i1},a_{i2},\dots , a_{iq}$ that is analogous to the permutation $\Lambda_{im}$ of $A_{i1},A_{i2},\dots, A_{iq}$ and let $ \gamma_{m}^{( i)}=b_{m1}^{(i)}b_{m2}^{(i)}\cdots b_{mq}^{(i)}$.
		Let $R_{3}=\sqrt[4]{N}$ and $R_{4}=\sqrt{N}.$
		
		If a substring $\overline{\lambda}_{i m}$ of $\lambda_{im}$
		of length at least $q/R_{3}$ and a substring $\overline{\gamma}^{(i')}_{m'}$
		of $\gamma^{(i')}_{m'}$ satisfy $\overline{f}(\overline{\lambda}_{im},\overline{\gamma}^{(i')}_{m'})\le1/8,$
		then at least 7/9 of the symbols in $\overline{\lambda}_{im}$
		are equal to a symbol in $\{b_{m'1}^{(i')},b_{m'2}^{(i')},\dots,b_{m'q}^{(i')}\}.$ 
		Then at least $7q/(9R_{3})$ symbols among $a_{i1},a_{i2},\dots , a_{iq}$
		are elements of $\{b_{m'1}^{(i')},b_{m'2}^{(i')},\dots,b_{m'q}^{(i')}\}.$ For fixed $i$,
		this can happen for at most $9R_{3}/7$ values of $m'$ (which ones depends on $i$).  For the other
		values of $m',$ and all values of $m$ and $i'$,  $\overline{f}(\overline{\mathcal{\lambda}}_{im},\overline{\gamma}^{(i')}_{m'})>1/8.$
		
		For the moment, we fix $i$ and $i'$, and suppose $\overline{\mathcal{A}}_{i}$ and $\overline{\mathcal{B}}^{({i'})}$ are substrings consisting 
		of consecutive symbols in $\lambda_{i1}\lambda_{i2}\cdots\lambda_{iN}$
		and $\gamma_1^{(i')}\gamma_2^{(i')}\cdots\gamma_N^{(i')}$,
		respectively, both of length at least $qN/R_{4}.$ Then there are
		at least $(N/R_{4})-2$ complete $\lambda_{im}$'s in $\overline{\mathcal{A}}_{i}$
		and at least $(N/R_{4})-2$ complete $\gamma_{m'}^{(i')}$'s
		in $\overline{\mathcal{B}}^{(i')}.$ By Fact \ref{fact:omit_symbols}, if we remove the incomplete strings
		at the beginning and end of $\overline{\mathcal{A}}_{i}$ and $\overline{\mathcal{B}}^{(i')}$
		we increase the $\overline{f}$ distance by at most $4R_{4}/N.$ We
		apply Lemma \ref{lem:symbol by block replacement} with $\alpha=1/8$ to the resulting complete strings
		and then compensate by subtracting $4R_{4}/N.$ Thus 
		\[
		\overline{f}(\overline{\mathcal{A}}_{i},\overline{\mathcal{B}}^{( i')})>\frac{1}{8}\left(1-\frac{9R_{3}/7}{(N/R_{4})-2}\right)-\frac{1}{R_{3}}-\frac{4R_{4}}{N}>\frac{1}{8}-\frac{7/6}{\sqrt[4]{N}}-\frac{4}{\sqrt{N}}=:\beta.
		\]
		If $\mathcal{A}_{i}=\lambda_{i1}\lambda_{i2}\cdots\lambda_{iN}$ and $\mathcal{B}^{(i')}=\gamma_1^{(i')}\gamma_2^{(i')}\cdots\gamma_N^{(i')}$ for $1\le i,i'\le p$, then by Lemma \ref{lem:symbol by block replacement},
		\begin{equation}
			\overline{f}(\mathcal{A}_{n_1}\mathcal{A}_{n_1+1}\cdots\mathcal{A}_{n_1+n_2},\mathcal{B}^{(n_3)}\mathcal{B}^{(n_3+1)}\cdots\mathcal{B}^{( n_3+n_4)})>\beta-\frac{1}{R_{4}},\label{eq:coding}
		\end{equation}
		for $1\le n_1\le n_1+n_2\le p$ and $1\le n_3\le n_3+n_4\le p$.
		
		Suppose $\overline{A}$ and $\overline{B}$ are as in the statement
		of the lemma. We will apply Lemma \ref{lem:symbol by block replacement} again, now regarding $\mathcal{A}_{ n_1}\mathcal{A}_{ n_1+1}\cdots\mathcal{A}_{ n_1+n_2}$
		and $\mathcal{B}^{( n_3)}\mathcal{B}^{(n_3+1)}\cdots\mathcal{B}^{(n_3+n_4)}$
		as indices on the blocks in $\overline{A}$ and $\overline{B}.$ By
		removing any incomplete sequence of blocks at the beginning and end
		of $\overline{A}$ and $\overline{B},$ we may assume that $\overline{A}$
		consists of complete sequences $\Lambda_{i1}\Lambda_{i2}\cdots\Lambda_{iN}$
		and $\overline{B}$ consists of complete sequences $\Gamma_{i'1}\Gamma_{i'2}\cdots\Gamma_{i'N},$
		thereby possibly increasing the $\overline{f}$ distance by at most
		$4R_{2}/{p}.$ By Lemma \ref{lem:symbol by block replacement} and (\ref{eq:coding}), $\overline{f}(\overline{A},\overline{B})>\alpha(\beta-\frac{1}{\sqrt{N}})-\frac{1}{R_{1}}-\frac{4R_{2}}{ p}>\alpha(\frac{1}{8}-\frac{2}{\sqrt[4]{N}})-\frac{1}{R_{1}}-\frac{4R_{2}}{p}.$ 
	\end{proof}
	
	\section{\label{sec:Substitution}A General Substitution Step}
	
	In this section we describe a step in our iteration of substitutions
	in a very general framework. The substitution will have the following
	initial data: 
	\begin{itemize}
		\item An alphabet $\Sigma$ and a collection of words $X\subset\Sigma^{\mathfrak{h}}$
		\item Equivalence relations $\mathcal{P}$ and $\mathcal{R}$ on $X$ with
		$\mathcal{R}$ refining $\mathcal{P}$
		\item Groups of involutions $G$ and $H$ with distinguished generators
		\item A homomorphism $\rho:H\to G$ that preserves the distinguished generators.
		We denote the range of $\rho$ by $G^{\prime}$ and its kernel by
		$H_{0}$ with cardinality $|H_{0}|=2^{t}$ for some $t\in\mathbb{N}$.
		\item A free $G$ action on $X/\mathcal{P}$ and a free $H$ action on $X/\mathcal{R}$
		such that the $H$ action is subordinate to the $G$ action via $\rho$.
		In particular, for each $k\in\mathbb{Z}^{+}$ the skew diagonal actions
		of $G$ on $\left(X/\mathcal{P}\right)^{k}$ and $H$ on $\left(X/\mathcal{R}\right)^{k}$
		are defined. 
		\item There are $N$ different equivalence classes in $X/\mathcal{P}$ denoted
		by $[A_{i}]_{\mathcal{P}}$, $i=1,\dots,N$, where $N=2^{\nu+N'}$
		with $N',\nu\in\mathbb{N}$.
		\item Each equivalence class $[A_{i}]_{\mathcal{P}}$ contains $2^{4e}$
		elements of $X/\mathcal{R}$, where $e\in\mathbb{Z}^{+}$ with $e\geq \max(2,t)$.
		We subdivide these $\mathcal{R}$ classes into tuples as follows.
		Pick an arbitrary class $[A_{1}]_{\mathcal{P}}$ and take a set $\left\{ \left[A_{1,1}\right]_{\mathcal{R}},\dots,\left[A_{1,2^{4e-t}}\right]_{\mathcal{R}}\right\} $
		of $\mathcal{R}$ classes in $[A_{1}]_{\mathcal{P}}$ that intersects
		each orbit of the $H_{0}$ action exactly once. This yields the first
		tuple $\left(\left[A_{1,1}\right]_{\mathcal{R}},\dots,\left[A_{1,2^{4e-t}}\right]_{\mathcal{R}}\right)$.
		Then we obtain $2^{t}-1$ further tuples as the images $\left(h\left[A_{1,1}\right]_{\mathcal{R}},\dots,h\left[A_{1,2^{4e-t}}\right]_{\mathcal{R}}\right)$
		for each $h\in H_{0}\setminus\left\{ \text{id}\right\} $. Hereby,
		we have divided the $2^{4e}$ elements of $X/\mathcal{R}$ in $[A_{1}]_{\mathcal{P}}$
		into $2^{t}$ many tuples such that each tuple intersects each orbit
		of the $H_{0}$ action exactly once and the tuples are images of each
		other under the action by $H_{0}$. In the next step, we choose one
		element $h\in H$ in each of the $H/H_{0}$ cosets and apply each such $h$ to the tuples in $[A_{1}]_{\mathcal{P}}$
		to obtain tuples in $\rho(h)[A_{1}]_{\mathcal{P}}$. If there are any $[A_{j}]_{\mathcal{P}}$'s
		that are not in the $G'$ orbit of $[A_{1}]_{\mathcal{P}}$, then we repeat
		the procedure for such a $[A_{j}]_{\mathcal{P}}$, etc. Our procedure guarantees
		that the choices of elements in the $H_{0}$ orbits are consistent
		between different $[A_{i}]_{\mathcal{P}}$ equivalence classes, so that the
		action of $H$ sends tuples to tuples. We let the $u$-th tuple in $\left[A_i\right]_{\mathcal{P}}$
		be written as 
		\[
		\left(\left[A_{i,u2^{4e-t}+1}\right]_{\mathcal{R}},\left[A_{i,u2^{4e-t}+2}\right]_{\mathcal{R}},\dots,\left[A_{i,(u+1)2^{4e-t}}\right]_{\mathcal{R}}\right),
		\]
		where $u\in\{0,\dots,2^t-1\}.$
		\item For some $R\ge 2$
		and some $\alpha\in\left(0,\frac{1}{8}\right)$
		we have 
		\begin{equation}
			\overline{f}(A,\bar{A})\geq\alpha\label{eq:assumpCoarserClass}
		\end{equation}
		for any substantial substrings $A$ and $\bar{A}$ of at least $\mathfrak{h}/R$
		consecutive symbols in any representatives of two different $\mathcal{P}$-equivalence
		classes, that is, representatives of $\left[A_{i_{1},j_{1}}\right]_{\mathcal{R}}$
		and $\left[A_{i_{2},j_{2}}\right]_{\mathcal{R}}$ for $i_{1}\neq i_{2}$
		and any $j_{1},j_{2}\in\left\{ 1,\dots,2^{4e}\right\} $.
		\item For some $\beta\in(0,\alpha]$ we have 
		\begin{equation}
			\overline{f}(A,\bar{A})\geq\beta\label{eq:assumpFinerClasses}
		\end{equation}
		for any substantial substrings $A$ and $\bar{A}$ of at least $\mathfrak{h}/R$
		consecutive symbols in any representatives of two different $\mathcal{R}$-equivalence
		classes, that is, representatives of $\left[A_{i,j_{1}}\right]_{\mathcal{R}}$
		and $\left[A_{i,j_{2}}\right]_{\mathcal{R}}$, respectively, for $j_{1}\neq j_{2}$.
	\end{itemize}
	\begin{rem*}
			While \eqref{eq:assumpCoarserClass} gives a lower bound on the $\overline{f}$ distance on substantial substrings of different $\mathcal{P}$ classes in $\Omega$, we have a lower bound $\beta\leq \alpha$ for different equivalence classes of the finer relation $\mathcal{R}$ in \eqref{eq:assumpFinerClasses}.
	\end{rem*}
	
	Let $K,\overline{T}\in\mathbb{Z}^{+}$ and $\tilde{R}\ge 2$ be given. The number $\tilde{R}$
	will be used in Proposition~\ref{prop:finer} below to quantify substantial  
	substrings of newly constructed concatenations. Then in Section \ref{sec:Construction} we specify the
	size of $\tilde{R}$ needed for the application of Proposition \ref{prop:finer}. We also suppose that
	there are numbers $M,P,U\in\mathbb{Z}^{+}$, where $\nu\cdot(2M+3)\geq\nu+N'$
	and $U$ is a multiple of $2^{t}$ such that 
	\begin{equation}
		U\geq2\tilde{R}^{2}.\label{eq:U}
	\end{equation}
	Finally, let $\overline{M}\in\mathbb{Z}^{+}$ with
	\begin{equation}
		\overline{M}\geq K\cdot P\cdot U\cdot2^{\nu\cdot(2M+3)}.\label{eq:M2}
	\end{equation}
	Hereby, we define the numbers
	\begin{equation}
		T\coloneqq \overline{T}\cdot2^{(4e-t)\cdot(2\overline{M}+3)},\label{eq:T1}
	\end{equation}
	\begin{equation}
		\overline{U}\coloneqq U\cdot2^{\nu\cdot(2M+3)},\label{eq:U2}
	\end{equation}
	and
	\begin{equation}
		k=U\cdot T\cdot2^{\nu\cdot(2M+3)}.\label{eq:Lstep}
	\end{equation}
	Note that $k$ is a multiple of $N=2^{\nu+N'}$ by our assumption $\nu\cdot(2M+3)\geq\nu+N'$.
	
	Suppose we have a collection $\Omega\subset\left(X/\mathcal{P}\right)^{k}$
	of cardinality $|\Omega|=P$ that satisfies the following properties:
	\begin{itemize}
		\item $\Omega$ is closed under the skew diagonal action of $G$,
		\item each $\omega\in\Omega$ is a concatenation of $U$ many
		different $\left(T,2^{\nu},M\right)$-Feldman patterns  as described in Section \ref{sec:Feldman}, each
		of which is constructed out of a tuple consisting of $2^{\nu}$ many
		$[A_{i}]_{\mathcal{P}}$, 
		\item each $[A_{i}]_{\mathcal{P}}\in X/\mathcal{P}$ occurs exactly $\frac{k}{N}$
		times in each $\omega\in\Omega$.
	\end{itemize}
	Then we construct a collection $S\subset\left(X/\mathcal{R}\right)^{k}$
	of substitution instances of $\Omega$ as follows.
	\begin{enumerate}
		\item We start by choosing a set $\Upsilon\subset\Omega$ that intersects
		each orbit of the action by the group $G^{\prime}$ exactly once.
		\item We construct a collection
		of $\overline{M}$ many different $\left(\overline{T},2^{4e-t},\overline{M}\right)$-Feldman
		patterns, where the tuple of building blocks is to be determined in
		step (6). Note that each such pattern is constructed as a concatenation
		of $\overline{T}2^{(4e-t)\cdot(2\overline{M}+3)}$ many building blocks in total
		which motivates the definition of the number $T$ from above.
		\item By assumption on $\Omega$ we can subdivide each element $r\in\Upsilon$
		as a concatenation of $U2^{\nu\cdot(2M+3)}=\overline{U}$ strings
		of the form $[A_{i}]_{\mathcal{P}}^{T}$ and each $i\in\left\{ 1,\dots,N\right\} $
		occurs exactly $V\coloneqq\frac{1}{N}\overline{U}$ many times in this decomposition. 
		\item For each $r\in\Upsilon$ we choose $K$ different sequences of $\overline{U}$
		concatenations of different $\left(\overline{T},2^{4e-t},\overline{M}\right)$-Feldman
		patterns in ascending order from our collection in the second step
		and enumerate the sequences by $j\in\left\{ 1,\dots,K\right\} $.
		We require that each of the constructed patterns appears at most once,
		that is, we needed to construct at least $|\Upsilon|\cdot K\cdot \overline{U}$
		many different patterns in that step. Since $\abs{\Upsilon}\leq \abs{\Omega}=P$, this explains the condition
		on the number $\overline{M}$ in equation (\ref{eq:M2}). 
		\item We define a sequence $\psi=\left(\psi_{1},\dots,\psi_{V}\right)$ 
		with $V=\frac{1}{N}U2^{\nu\cdot(2M+3)}$, by 
		\[\psi_{v}=u\in\{0,\dots,2^t-1\}, \text{ where } u\equiv v\:\mod\:2^{t}.
		\]
		That is, the sequence $\psi$ cycles through the symbols
		in $\left\{ 0,\dots,2^{t}-1\right\}.$
		Since $U$ was chosen as a multiple of $2^{t}$ and $\frac{1}{N}2^{\nu\cdot(2M+3)}\in\mathbb{Z}$
		by the assumption $\nu\cdot(2M+3)\geq\nu+N'$, each symbol from
		$\left\{ 0,\dots,2^{t}-1\right\} $ occurs the same number $\frac{V}{2^{t}}$
		of times in the sequence $\psi$.
		\item Let $r\in\Upsilon$ and $j\in\left\{ 1,\dots,K\right\} $, and write 
		\[
		r=\left[A_{i(1)}\right]_{\mathcal{P}}^{T}\cdots\left[A_{i(\ell)}\right]_{\mathcal{P}}^{T}\cdots
		\left[A_{i(\overline{U})}\right]_{\mathcal{P}}^{T},
		\]
		where $i(1),\dots,i(\overline{U})\in\{1,\dots,N\}.$ If $i(\ell)=i_0$
		and this is the $m$th occurrence of $i_0$ in the sequence $i(1),\dots,i(\ell)$
		then we let $u=\psi_m$ and substitute a Feldman pattern built with the tuple 
		\[
		\left(\left[A_{i_0,u2^{4e-t}+1}\right]_{\mathcal{R}},\dots,\left[A_{i_0,(u+1)2^{4e-t}}\right]_{\mathcal{R}}\right)
		\]
		into $\left[A_{i_0}\right]_{\mathcal{P}}^{T}.$ 
		The Feldman pattern that is used is the $\ell$th pattern among the $\overline{U}$ 
		patterns previously chosen in step (4) for the given $r$ and $j$. We follow this procedure for each $\ell=1,\dots,\overline{U}$
		to obtain an element $s\in (X/R)^k$. Let $S$ be the collection of such $s\in (X/R)^k$
		obtained for all $r\in\Upsilon$ and $j\in\{1,\dots,K\}.$ Up to this point we have constructed $K$ different substitution instances for each $r\in \Upsilon$. In the later applications in Sections~\ref{subsec:Case-1} and~\ref{subsec:Case-2} the number $K$ will be chosen sufficiently large to produce the required number of substitution instances.
		
	\end{enumerate}
	Using this collection $S\subset\left(X/\mathcal{R}\right)^{k}$ we
	define 
	\[
	\Omega^{\prime}=HS,
	\]
	that is, $\Omega^{\prime}$ is the image of $S$ under the skew diagonal action by $H$.

		\begin{rem*}
			The substitution of Feldman patterns of finer equivalence
			classes into Feldman patterns of coarser classes in step (6) lies at the heart of our construction mechanism. It is visualized in Figure~\ref{fig:fig0}.
		\end{rem*}

	\begin{figure}
		\centering
		\includegraphics[width=\textwidth]{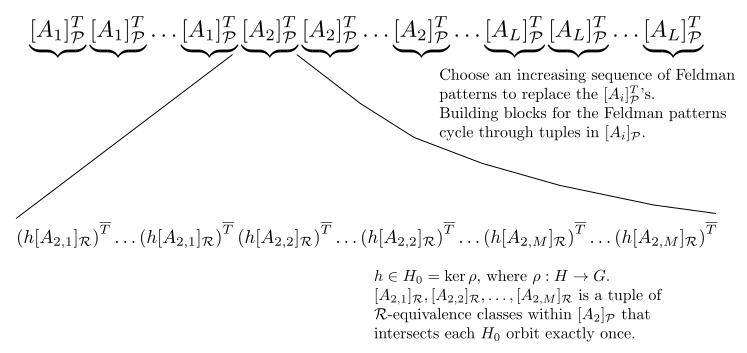}
		\caption{Visualization of the substitution step.}
		\label{fig:fig0}
	\end{figure}
	
	\begin{prop}
		\label{prop:PropertiesCollection}This collection $\Omega^{\prime}\subset\left(X/\mathcal{R}\right)^{k}$
		satisfies the following properties.
		\begin{enumerate}
			\item $\Omega^{\prime}$ is closed under the skew diagonal action by $H$.
			\item For each element in $\omega\in\Omega$ there are $K\cdot|H_{0}|$
			many substitution instances in $\Omega^{\prime}$. 
			\item Each element of $\Omega^{\prime}$ contains each $\left[A_{i,j}\right]_{\mathcal{R}}$
			the same number of times.
		\end{enumerate}
	\end{prop}
	
	\begin{proof}
		The first statement follows from the definition of $\Omega^{\prime}$
		immediately. For part (2) we recall that $\Upsilon\subset\Omega$
		was chosen to intersect each orbit of the action by the group $G^{\prime}$
		exactly once and that for each element in $\Upsilon$ there are $K$
		many substitution instances in $S$. To see the third part we use
		that all the $[A_{i}]_{\mathcal{P}}\in X/\mathcal{P}$ occurred uniformly
		in each $\omega\in\Omega$. By construction of the sequence $\psi$
		in step (5) each tuple $\left(\left[A_{i,a+1}\right]_{\mathcal{R}},\dots,\left[A_{i,a+2^{4e-t}}\right]_{\mathcal{R}}\right)$
		occurs the same number of times as building blocks for $\left(\overline{T},2^{4e-t},\overline{M}\right)$-Feldman
		patterns substituted into strings $[A_{i}]_{\mathcal{P}}^{T}$.
		Since Feldman patterns are uniform in their building blocks, we conclude
		the third statement.
	\end{proof}
	In the following, we estimate the $\overline{f}$ distance of elements
	in $\Omega'$ that are equivalent with respect to the $\mathcal{P}$
	product equivalence relation but are not $\mathcal{R}$-equivalent. This gives a lower bound on the $\overline{f}$ distance of different $\mathcal{R}$ classes in $\Omega'$.
	We conclude that they are still approximately $\beta$ apart from
	each other in $\overline{f}$, where $\beta$ is as in (\ref{eq:assumpFinerClasses}).
	The $\overline{f}$ distance for strings of symbols in $\Sigma_{\mathcal{R}}$ will
	be denoted $\overline{f}_{\Sigma_{\mathcal{R}}}$. When there is no subscript on the
	$\overline{f}$, it is understood that the alphabet is $\Sigma$.
	\begin{prop}
		\label{prop:finer}Let $\omega_{1}^{\prime}$ and $\omega_{2}^{\prime}$
		be two different substitution instances in $\Omega^{\prime}$ of an
		element $\omega\in\Omega$. We consider $\omega_{1}^{\prime},\omega_{2}^{\prime}\in\Omega'\subset\left(X/\mathcal{R}\right)^{k}$
		as strings of length $k$ in the alphabet $\Sigma_{\mathcal{R}}\coloneqq\left\{ \left[A_{i,j}\right]_{\mathcal{R}}:1\leq i\leq N,\,1\leq j\leq2^{4e}\right\} $.
		Then we have 
		\begin{equation}
			\overline{f}_{\Sigma_{\mathcal{R}}}\left(W^{\prime},\overline{W}^{\prime}\right)>1-\frac{4}{2^{e}}-\frac{1}{\tilde{R}}\label{eq:PropFiner1}
		\end{equation}
		for any substrings $W^{\prime}$ and $\overline{W}^{\prime}$ of at
		least $k/\tilde{R}$ consecutive $\Sigma_{\mathcal{R}}$-symbols from
		$\omega_{1}^{\prime}$ and $\omega_{2}^{\prime}$, respectively.
		
		Moreover, we have 
		\begin{equation}
			\overline{f}\left(V,\overline{V}\right)>\beta-\frac{1}{R}-\frac{4}{2^{e}}-\frac{1}{\tilde{R}}\label{eq:PropFiner}
		\end{equation}
		for any substrings $V$, $\overline{V}$ of at least $k\mathfrak{h}/\tilde{R}$
		consecutive symbols in any representatives $W_{1}$ and $W_{2}$ of
		$\omega_{1}^{\prime}$ and $\omega_{2}^{\prime}$, respectively, in
		the alphabet $\Sigma$.
	\end{prop}
	
	\begin{proof}
		Let $W^{\prime}$ and $\overline{W}^{\prime}$ be any substrings of
		at least $k/\tilde{R}$ consecutive $\Sigma_{\mathcal{R}}$-symbols
		from $\omega_{1}^{\prime}$ and $\omega_{2}^{\prime}$, respectively.
		We subdivide both strings into sequences coming from the $\left(\overline{T},2^{4e-t},\overline{M}\right)$-Feldman
		patterns and we recall that their length is $\overline{T}2^{(4e-t)\cdot(2\overline{M}+3)}=T$.
		By adding fewer than $2T$ many $\Sigma_{\mathcal{R}}$-symbols
		to each of $W^{\prime}$ and $\overline{W}^{\prime}$ we can complete
		partial patterns at the beginning and end. Let $W_{aug}^{\prime}$
		and $\overline{W}_{aug}^{\prime}$ be the augmented $W^{\prime}$
		and $\overline{W}^{\prime}$ strings obtained in this way. By Fact
		\ref{fact:omit_symbols} and condition (\ref{eq:U}) we have 
		\begin{equation}
			\overline{f}_{\Sigma_{\mathcal{R}}}\left(W^{\prime},\overline{W}^{\prime}\right)>\overline{f}_{\Sigma_{\mathcal{R}}}\left(W_{aug}^{\prime},\overline{W}_{aug}^{\prime}\right)-\frac{2\tilde{R}}{U2^{v\cdot\left(2M+3\right)}}>\overline{f}_{\Sigma_{\mathcal{R}}}\left(W_{aug}^{\prime},\overline{W}_{aug}^{\prime}\right)-\frac{1}{\tilde{R}}.\label{eq:augment}
		\end{equation}
		We write $W_{aug}^{\prime}$ and $\overline{W}_{aug}^{\prime}$ as
		\begin{alignat*}{2}
			W_{aug}^{\prime}=P_{1}^{(p_{1},a_{1},k_{1})}\dots P_{s}^{(p_{s},a_{s},k_{s})},\;\; &  &  & \overline{W}_{aug}^{\prime}=\overline{P}_{1}^{(q_{1},b_{1},l_{1})}\dots\overline{P}_{t}^{(q_{t},b_{t},l_{t})},
		\end{alignat*}
		where $P_{i}^{(p_{i},a_{i},k_{i})}$, $1\leq i\leq s$, as well as
		$\overline{P}_{j}^{(q_{j},b_{j},l_{j})}$, $1\leq j\leq t$, are complete
		$\left(\overline{T},2^{4e-t},\overline{M}\right)$-Feldman patterns in the $\Sigma_{\mathcal{R}}$-alphabet
		and the superscripts $p_{i},q_{j}\in\left\{ 1,\dots,\overline{M}\right\} $
		indicate which pattern structure is used. Moreover, the other superscripts
		identify the building blocks. Here, $a_{i},b_{j}\in\left\{ 1,\dots,N\right\} $
		refer to the $\mathcal{P}$-equivalence class of the building blocks
		and $k_{i},l_{j}\in\left\{ 1,\dots,2^{4e-t}\right\} $ refer to the
		tuple of $\mathcal{R}$-equivalence classes of building blocks. If we apply
		Lemma \ref{lem:DifferentT} with $N$ replaced by $2^{4e-t}$, we obtain 
		\begin{equation}
			\overline{f}_{\Sigma_{\mathcal{R}}}\left(P,\overline{P}\right)\begin{cases}
				=1, & \text{if }\left(a_{i},k_{i}\right)\neq\left(b_{j},l_{j}\right)\\
				>1-\frac{4}{2^{2e-0.5t}}, & \text{if }\left(a_{i},k_{i}\right)=\left(b_{j},l_{j}\right)\text{ and }p_{i}\neq q_{j}
			\end{cases}\label{eq:casesSymbols}
		\end{equation}
		for any sequences $P$ and $\overline{P}$ of at least $\frac{T}{2^{4e-t}}=\overline{T}2^{(4e-t)\cdot(2\overline{M}+2)}$
		consecutive $\Sigma_{\mathcal{R}}$-symbols in $P_{i}^{(p_{i},a_{i},k_{i})}$
		and $\overline{P}_{j}^{(q_{j},b_{j},l_{j})}$, respectively. 
		
		We examine the two possible cases: $\omega_{1}^{\prime}$ and $\omega_{2}^{\prime}$
		have disjoint $H$ orbits (case 1) or $\omega_{1}^{\prime}=h\omega_{2}^{\prime}$
		for some non-identity element $h\in H_{0}$ (case 2). From the construction we recall that
		in the first case all occurring $\left(\overline{T},2^{4e-t},\overline{M}\right)$-Feldman
		patterns in $\omega_{1}^{\prime}$ and $\omega_{2}^{\prime}$ are
		different from each other. To investigate the second case we recall
		that a tuple of building blocks is mapped to a different tuple under
		the action of a non-identity element $h\in H_{0}$. Hence, identical Feldman patterns
		in $\omega_{1}^{\prime}=h\omega_{2}^{\prime}$ and $\omega_{2}^{\prime}$
		have different tuples of building blocks. 
		
		In both cases we consider a best possible $\overline{f}$ match between
		$W_{aug}^{\prime}$ and $\overline{W}_{aug}^{\prime}$. To simplify
		notation we ignore the superscripts of the occurring $\left(\overline{T},2^{4e-t},\overline{M}\right)$-Feldman
		patterns. Recall that the length of a $\left(\overline{T},2^{4e-t},\overline{M}\right)$-Feldman
			pattern is given by $\overline{T}\cdot2^{(4e-t)\cdot(2\overline{M}+3)}=T$ as defined in \eqref{eq:T1}. We denote the string in $\overline{W}_{aug}^{\prime}$ matched
		with $P_{i}$ by $\tilde{P}_{i}$. If $|\tilde{P}_{i}|<\frac{T}{2^{e}}$
		or $|\tilde{P}_{i}|>2^{e}T$, then 
		\begin{equation}
			\overline{f}_{\Sigma_{\mathcal{R}}}\left(P_{i},\tilde{P}_{i}\right)>1-\frac{2}{2^{e}}.\label{eq:Case1}
		\end{equation}
		Otherwise, we subdivide $\tilde{P}_{i}$ into the strings coming from
		the Feldman patterns in $\overline{W}_{aug}^{\prime}$ and we denote
		the part of $\tilde{P}_{i}$ belonging to $\overline{P}_{j}$ by $\tilde{P}_{i,j}$.
		Note that there are at most $2^{e}+2$ many $j$'s with nonempty $\tilde{P}_{i,j}$
		due to $|\tilde{P}_{i}|\leq2^{e}T$. In the following we ignore
		all $\tilde{P}_{i,j}$ of length $|\tilde{P}_{i,j}|<\frac{T}{2^{2e}}$.
		Hereby, we omit at most $\frac{2T}{2^{2e}}$ many symbols from
		$\tilde{P}_{i}$. We denote the remaining $\tilde{P}_{i}$ string
		by $\tilde{P}_{i,red}$ and we obtain from Fact \ref{fact:omit_symbols}
		that 
		\begin{equation}
			\overline{f}_{\Sigma_{\mathcal{R}}}\left(P_{i},\tilde{P}_{i}\right)>\overline{f}_{\Sigma_{\mathcal{R}}}\left(P_{i},\tilde{P}_{i,red}\right)-\frac{2}{2^{2e}}.\label{eq:TildeReduced}
		\end{equation}
		For all the remaining $j$'s with $|\tilde{P}_{i,j}|\geq\frac{T}{2^{2e}}$
		the corresponding string in $P_{i}$ under the matching is called
		$P_{i,j}$. If $|P_{i,j}|<\frac{T}{2^{4e-t}}$, then 
		\[
		\overline{f}_{\Sigma_{\mathcal{R}}}\left(P_{i,j},\tilde{P}_{i,j}\right)>1-\frac{2|P_{i,j}|}{|\tilde{P}_{i,j}|+|P_{i,j}|}=1-\frac{2}{\frac{|\tilde{P}_{i,j}|}{|P_{i,j}|}+1}>1-\frac{2}{2^{e}}.
		\]
		Otherwise, we use the estimates from formula (\ref{eq:casesSymbols}).
		This yields
		\[
		\overline{f}_{\Sigma_{\mathcal{R}}}\left(P_{i},\tilde{P}_{i}\right)>1-\frac{2}{2^{e}}-\frac{2}{2^{2e}}\geq 1-\frac{3}{2^{e}}
		\]
		by equation (\ref{eq:TildeReduced}). Combining this with equation
		(\ref{eq:Case1}), we obtain 
		\begin{equation}
			\overline{f}_{\Sigma_{\mathcal{R}}}\left(W_{aug}^{\prime},\overline{W}_{aug}^{\prime}\right)>1-\frac{4}{2^{e}}.\label{eq:EstimAug}
		\end{equation}
		With the aid of equation (\ref{eq:augment}) we finish the proof of
		(\ref{eq:PropFiner1}).
		
		Now we turn to the proof of (\ref{eq:PropFiner}) in the alphabet
		$\Sigma$. We treat $V_{aug}$ and $\overline{V}_{aug}$ as concatenations
		of complete $\left(\overline{T},2^{4e-t},\overline{M}\right)$-Feldman patterns
		by adding fewer than $2Th$ many $\Sigma$-symbols to complete
		partial patterns at the beginning and end of $V$ and $\overline{V}$.
		Note that $V_{aug}$ and $\overline{V}_{aug}$ correspond to some
		$W_{aug}^{\prime}$ and $\overline{W}_{aug}^{\prime}$ in the $\Sigma_{\mathcal{R}}$-alphabet
		of the same type as above. By Lemma \ref{lem:symbol by block replacement},
		with $\overline{f}$-distance at least $1-\frac{4}{2^{e}}$ coming from equation (\ref{eq:EstimAug}),
		and our assumptions from equations (\ref{eq:assumpCoarserClass})
		and (\ref{eq:assumpFinerClasses}), we obtain 
		\[
		\overline{f}\left(V_{aug},\overline{V}_{aug}\right)>\left(1-\frac{4}{2^{e}}\right)\cdot\beta-\frac{1}{R}>\beta-\frac{1}{R}-\frac{4}{2^{e}}.
		\]
		Finally, we conclude (\ref{eq:PropFiner}) with the aid of Fact \ref{fact:omit_symbols}
		and equation (\ref{eq:U}).
	\end{proof}
	
	\section{\label{sec:Construction}The Construction Process}
	
	To construct the map $\Psi:\mathcal{T}\kern-.5mm rees\to\mathcal{E}$ we must
	build a construction sequence $\left(\mathcal{W}_{n}\left(\mathcal{T}\right)\right)_{n\in\mathbb{N}}$
	satisfying our specifications for each $\mathcal{T}\in\mathcal{T}\kern-.5mm rees$.
	Moreover, this has to be done in such a way that $\mathcal{W}_{n}\left(\mathcal{T}\right)$
	is entirely determined by $\mathcal{T}\cap\left\{ \sigma_{m}:m\leq n\right\} $,
	that is, 
	\begin{equation}
		\begin{split} & \text{if }\mathcal{T}\cap\left\{ \sigma_{m}:m\leq n\right\} =\mathcal{T}^{\prime}\cap\left\{ \sigma_{m}:m\leq n\right\} ,\\
			& \text{then for all }m\leq n:\:\mathcal{W}_{m}\left(\mathcal{T}\right)=\mathcal{W}_{m}(\mathcal{T}^{\prime}),
		\end{split}
		\label{eq:ContinuityF}
	\end{equation}
	and we guarantee the continuity of our map $\Psi:\mathcal{T}\kern-.5mm rees\to\mathcal{E}$
	by Fact \ref{fact:contTree}. Therefore, we follow \cite{FRW} to
	organize our construction: For each $n$ and for each subtree $\mathcal{S}\subseteq\left\{ \sigma_{m}:m\leq n\right\} $
	and each $\sigma_{m}\in\mathcal{S}$ we build $\mathcal{W}_{m}\left(\mathcal{S}\right)$.
	By induction we want to pass from stage $n-1$ to stage $n$. So,
	we assume that we have constructed $\left(\mathcal{W}_{m}\left(\mathcal{S}\right)\right)_{\sigma_{m}\in\mathcal{S}}$, $\left(\mathcal{Q}^m_s(\mathcal{S})\right)_{\sigma_{m}\in\mathcal{S}}$, and the action of $\left(G^m_s(\mathcal{S})\right)_{\sigma_{m}\in\mathcal{S}}$
	for each subtree $\mathcal{S}\subseteq\left\{ \sigma_{m}:m\leq n-1\right\} $.
	In the inductive step, we now have to construct $\left(\mathcal{W}_{m}\left(\mathcal{S}\right)\right)_{\sigma_{m}\in\mathcal{S}}$
	for each subtree $\mathcal{S}\subseteq\left\{ \sigma_{m}:m\leq n\right\} $. 
	
	Obviously, for those subtrees $\mathcal{S}$ with $\sigma_{n}\notin\mathcal{S}$
	there is nothing to do. So, we have to work on the finitely many subtrees
	$\mathcal{S}$ with $\sigma_{n}\in\mathcal{S}$. We list those in
	arbitrary manner as $\left\{ \mathcal{S}_{1},\dots,\mathcal{S}_{E}\right\} $.
	Assume that $\mathcal{T}$ is the $a$-th subtree on this list and
	that if $a>1$ we have constructed the collections $\mathcal{W}_{n}\left(\mathcal{S}_{i}\right)$
	of $n$-words for all $i\in\left\{ 1,\dots,a-1\right\} $. Let $\mathfrak{P}$
	be the collection of prime numbers occurring in the prime factorization
	of any of the lengths of the words that we have constructed so far.\footnote{In our construction of $\mathcal{W}_{n}\left(\mathcal{T}\right)$ we will choose a prime number
		$\mathfrak{p}_{n}$ larger than the maximum of $\mathfrak{P}$. By Lemma \ref{lem:odometer} we see that our map $\Psi:\mathcal{T}\kern-.5mm rees\to\mathcal{E}$ takes distinct
		trees to transformations with nonisomorphic Kronecker factors. Hence, for $\mathcal{T} \neq \mathcal{T}^{\prime}$ we get that $\Psi(\mathcal{T})$ is not isomorphic to $\Psi(\mathcal{T}^{\prime})$ or $\Psi(\mathcal{T}^{\prime})^{-1}$ (compare with  \cite[Corollary 34]{FRW}). As in \cite{FRW} this is actually not necessary for the proof of the main theorem but we keep this order of induction to parallel the strategy in \cite{FRW} and to allow possible extensions in future work.}
	
	We want to construct $\mathcal{W}_{n}\left(\mathcal{T}\right)$, $\mathcal{Q}_{s}^{n}\left(\mathcal{T}\right)$,
	and the action of $G_{s}^{n}\left(\mathcal{T}\right)$, which we abbreviate by $\mathcal{W}_{n}$,
	$\mathcal{Q}_{s}^{n}$, and $G_{s}^{n}$, respectively. 
	
	For that purpose,
	let $m$ be the largest number less than $n$ such that $\sigma_{m}\in\mathcal{T}$.
	If $m=0$ we have $\mathcal{W}_{0}=\Sigma=\left\{ 1,\dots,2^{12}\right\} $
	(see Remark \ref{rem:BaseCase} for comments on this base case) and
	for $m>0$ our {\bf induction assumption} says that we have $\mathcal{W}_{m}=\mathcal{W}_{m}\left(\mathcal{T}\right)$,
	$\mathcal{Q}_{s}^{m}=\mathcal{Q}_{s}^{m}\left(\mathcal{T}\right)$,
	and $G_{s}^{m}=G_{s}^{m}\left(\mathcal{T}\right)$ satisfying our
	specifications. Furthermore, there are  $\frac{1}{8}>\alpha_{1,m}>\dots>\alpha_{s(m),m}>\beta_{m}>0$ and
	a positive integer $R_{m}$ with
	\begin{equation}
		R_{m}\geq\frac{9}{\beta_{m}} \label{eq:Rassum2}
	\end{equation} 
	such that the following assumptions on $\overline{f}$
	distances hold:
	\begin{itemize}
		\item For every $s\in\left\{ 1,\dots,s(m)\right\} $  we have 
		\begin{equation}
			\overline{f}\left(W,W^{\prime}\right)>\alpha_{s,m}\label{eq:fClasses}
		\end{equation}
		on any substrings $W$ and $W^{\prime}$ of at least $\frac{h_{m}}{R_{m}}$
		consecutive symbols in any words $w,w^{\prime}\in\mathcal{W}_{m}$
		with $[w]_{s}\neq[w^{\prime}]_{s}$. Here, we recall that $[w]_s$ denotes the $\mathcal{Q}^m_s$ class of $w\in \mathcal{W}_m$.
		\item We have 
		\begin{equation}
			\overline{f}\left(W,W^{\prime}\right)>\beta_{m}\label{eq:fBlocks}
		\end{equation}
		on any substrings $W$ and $W^{\prime}$ of at least $\frac{h_{m}}{R_{m}}$
		consecutive symbols in any words $w,w^{\prime}\in\mathcal{W}_{m}$
		with $w\neq w^{\prime}$.
	\end{itemize}

	{\bf Induction step:} We start the construction in stage $n$ by choosing a prime number
	$\mathfrak{p}_{n}$ larger than the maximum of $\mathfrak{P}$ and satisfying
	\begin{equation}
		\mathfrak{p}_{n}>\max\left(4R_{m},2^n,\frac{1}{\epsilon_{n}}\right),\label{eq:Pn}
	\end{equation}
	where we recall the definition of the sequence $\left(\epsilon_{n}\right)_{n\in\mathbb{N}}$
	from (\ref{eq:Condeps}). Then we choose an integer
	$R_{n}$ such that
	\begin{equation}
		R_{n}\geq \frac{40\mathfrak{p}_{n}}{\beta_{m}} \label{eq:Rn}
	\end{equation}
	Moreover, we choose the integer $e(n)>e(m)$ sufficiently large such that 
	\begin{equation}
		2^{e(n)}>\max\left(10R_{n},\max_{s\leq s(n)}|G_{s}^{n}|\right).\label{eq:kn}
	\end{equation}
	Recall that specification (Q6) requires $2^{4e(n)}$ to be the number of $\mathcal{Q}^n_{s+1}$ classes in one $\mathcal{Q}^n_s$ class.
	
	In the two possible cases $s(n)=s(m)$ and $s(n)=s(m)+1$ we now present
	the construction of $n$-words out of $m$-words as well as the extension
	of $G_{s}^{m}$ actions to $G_{s}^{n}$ actions. In case of $s(n)=s(m)+1$
	we also define the new equivalence relation $\mathcal{Q}_{s(n)}^{n}$.
	
	\subsection{\label{subsec:Case-1}Case 1: s(n)=s(m)}
	
	\subsubsection{Construction of $\mathcal{W}_{n}$}
	
	We apply the substitution step described in Section \ref{sec:Substitution}
	$s(m)+1$ many times to construct $\mathcal{W}_{n}/\mathcal{Q}_{s}^{n}$
	for $s=0,\dots,s(m)$ and eventually $\mathcal{W}_{n}$. Adapting
	to the notation in the general substitution step we set $X=\mathcal{W}_{m}$
	and 
	\begin{align*}
		\mathcal{P}_{i}=\mathcal{Q}_{i-1}^{m}, & \ \ \mathcal{R}_{i}=\mathcal{Q}_{i}^{m}, & G_{i}=G_{i-1}^{m}, & \ \ H_{i}=G_{i}^{m}
	\end{align*}
	for the $i$-th application, where $i=1,\dots,s(m)$. In the final
	$(s(m)+1)$-th application we let $\mathcal{P}_{s(m)+1}=\mathcal{Q}_{s(m)}^{m}$,
	$\mathcal{R}_{s(m)+1}=\left\{ (w,w):w\in\mathcal{W}_{m}\right\} $,
	that is, $X/\mathcal{R}_{s(m)+1}=\mathcal{W}_{m}$, $G_{s(m)+1}=G_{s(m)}^{m}$,
	and we let $H_{s(m)+1}$ be the trivial group acting trivially on
	$\mathcal{W}_{m}$. 
	
	To carry out the successive application of the substitution step we now define several parameters which will play the roles of the corresponding letters in our general substitution mechanism from Section~\ref{sec:Substitution}. We
	always set $\tilde{R}=R_{n}$ and introduce $t_{i}\in\mathbb{N}$
	for $i=1,\dots,s(m)+1$ such that $t_{s(m)+1}=0$ and 
	\[
	2^{t_{i}}=|ker(\rho_{i,i-1}^{(m)})|,\text{ for }i=1,\dots,s(m).
	\]
	Here, we recall that $\rho_{i,i-1}^{(m)}:G_{i}^{m}\to G_{i-1}^{m}$
	is the canonical homomorphism defined in Subsection~\ref{subsec:Groups}. Using these numbers we also define
		\begin{align*}
			\nu_1=0, \ \text{ and } \ \nu_i=4e(m)-t_{i-1}, \text{ for } i=2,\dots,s(m)+1. 
	\end{align*}
	Moreover, we define parameters as follows 
	\begin{eqnarray*}
		P_{i}=2^{(i-1)\cdot4e(n)}, \ \text{for }i=1,\dots,s(m)+1,\\
		N_{i}=2^{(i-1)\cdot4e(m)}, \ \text{for }i=1,\dots,s(m)+1.
	\end{eqnarray*}
	Then we have $N_{i}=|\mathcal{W}_{m}/\mathcal{Q}_{i-1}^{m}|$ and the
	number $P_{i}$ will describe $|\mathcal{W}_{n}/\mathcal{Q}_{i-1}^{n}|$.
	We also set $K_{s(m)+1}=2^{4e(n)}$, and choose numbers $K_{i}$,
	$i=1,\dots,s(m)$, such that $K_{i}\cdot|ker(\rho_{i,i-1}^{(m)})|=2^{4e(n)}$
	for $i=1,\dots,s(m)$. We point out that these numbers are chosen
	in this way to always obtain exactly $2^{4e(n)}$ substitution instances
	of each coarser equivalence class (see part (2) of Proposition \ref{prop:PropertiesCollection}) in order to satisfy specification
	property (Q6). Moreover, we choose a number $U_{1}=2^{U_{1}^{\prime}}$
	with 
	\begin{align*}
		U_{1}^{\prime}\geq\max_{i=1,\dots,s(m)}t_{i} &  & \text{ and } &  & U_{1}\geq\max\left(2R_{n}^{2},2^{8e(m)}\right).
	\end{align*}
	In particular, condition (\ref{eq:U}) is satisfied in all construction
	steps. 
	
	By forward recursion we define the parameters $M_{i}$, $i=1,\dots,s(m)+2$,
	as well as $U_{i}$, $i=2,\dots,s(m)+1$: 
	\begin{align*}
		M_{1}=1, &  & M_{2}=K_{1}P_{1}U_{1}, &  & M_{i}=K_{i-1}P_{i-1}U_{i-1}2^{(4e(m)-t_{i-2})\cdot(2M_{i-1}+3)} & \text{ for \ensuremath{i\geq3},}\\
		U_{2}=U_{1}, &  &  &  & U_{i}=U_{i-1}2^{(4e(m)-t_{i-2})\cdot(2M_{i-1}+3)} & \text{ for \ensuremath{i\geq3}}.
	\end{align*}
	Hereby, we define the parameters $T_{i}$, $i=1,\dots,s(m)+2$, by
	backward recursion: 
	\begin{align*}
		T_{s(m)+2}=\mathfrak{p}_{n}^{2}, &  & T_{i}=T_{i+1}\cdot2^{(4e(m)-t_{i})\cdot\left(2M_{i+1}+3\right)} & \text{ for \ensuremath{i\leq s(m)+1}}.
	\end{align*}
	We use the parameters $t=t_{i}$, $\nu=\nu_i$, $K=K_i$, $P=P_i$, $M=M_i$, $\overline{M}=M_{i+1}$, $N=N_i$, $U=U_i$, $\overline{U}=U_{i+1}$, $T=T_i$, $\overline{T}=T_{i+1}$ in the $i$-th application of the general substitution mechanism in our successive construction, where $i=1,\dots,s(m)+1$. From this point of view, we stress that the definitions of $T_{i}$ and $M_{i}$ are directly
	linked to the corresponding conditions in equations (\ref{eq:M2})
	and (\ref{eq:T1}) from the general substitution step. After our $(s(m)+1)$-th application, we conclude
	from equation (\ref{eq:Lstep}) that the number $k$ of concatenations
	working for all substitutions steps is given by 
	\begin{equation}
		k=U_{s(m)+1}\cdot T_{s(m)+1}\cdot2^{(4e(m)-t_{s(m)})\cdot(2M_{s(m)+1}+3)}=U_{s(m)+1}\cdot T_{s(m)}=U_{2}\cdot T_{1},\label{eq:concat}
	\end{equation}
	which implies that the length of $n$-words is given by 
	\[
	h_{n}=U_{s(m)+1}\cdot T_{s(m)+1}\cdot2^{(4e(m)-t_{s(m)})\cdot(2M_{s(m)+1}+3)}\cdot h_{m}=U_{1}\cdot T_{1}\cdot h_{m}.
	\]
	We denote the number $k$ from \eqref{eq:concat} by $k_m$.
	
	After determining all these parameters we describe the successive
	application of the general substitution step. We start by concatenating
	$k=U_{1}\cdot T_{1}$ many times the single element of $X/\mathcal{P}_{1}=\mathcal{W}_{m}/\mathcal{Q}_{0}^{m}$
	which yields $\Omega_{1}=\mathcal{W}_{n}/\mathcal{Q}_{0}^{n}\subseteq\left(\mathcal{W}_{m}/\mathcal{Q}_{0}^{m}\right)^{k}=\left(X/\mathcal{Q}_{0}\right)^{k}$
	of cardinality $|\Omega_{1}|=P_{1}=1$. Note that this element can
	be considered as a repetition of $U_{1}$ many $\left(T_{1},2^{0},1\right)$-Feldman
	patterns. This viewpoint allows the application of our first substitution
	step $i=1$ using the notation from above. More explicitly, the first
	substitution step produces a collection $\Omega_{1}^{\prime}\subset\left(X/\mathcal{R}_{1}\right)^{k}=\left(\mathcal{W}_{m}/\mathcal{Q}_{1}^{m}\right)^{k}$
	of $K_{1}\cdot|ker(\rho_{1,0}^{(m)})|=2^{4e(n)}=P_{2}$ substitution
	instances, which are concatenations of $U_{2}$ many $\left(T_{2},2^{4e(m)-t_{1}},M_{2}\right)$-Feldman
	patterns. On the one hand, $\Omega_{1}^{\prime}$ will serve as our
	collection $\Omega_{2}\subset\left(X/\mathcal{P}_{2}\right)^{k}=\left(\mathcal{W}_{m}/\mathcal{Q}_{1}^{m}\right)^{k}$
	for the application of the substitution step $i=2$. On the other
	hand, $\Omega_{1}^{\prime}$ will turn out to be our collection of
	equivalence classes in $\mathcal{W}_{n}/\mathcal{Q}_{1}^{n}$. Successively,
	we apply the substitution step for $i=2,\dots,s(m)+1$. We visualize this iterative application of the substitution mechanism in Figure~\ref{fig:figSubst}.
	
	\begin{figure}
		\centering
		\includegraphics[width=\textwidth]{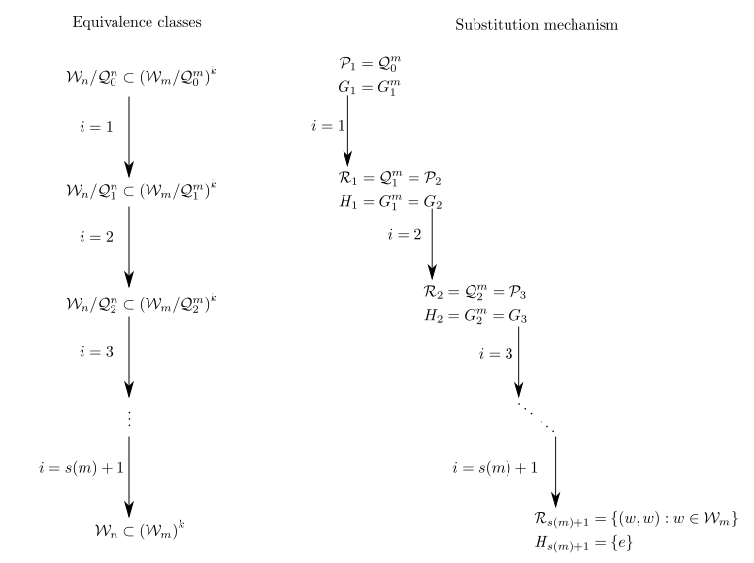}
		\caption{Visualization of the iterative application of the substitution mechanism.}
		\label{fig:figSubst}
	\end{figure}

	\begin{rem}
		\label{rem:s-pattern-type}In correspondence with specification (Q5),
		$\mathcal{Q}_{s}^{n}$ is defined as the product equivalence relation
		on $\mathcal{W}_{n}$ for each $s=1,\dots,s(m)$. Then we see from
		the procedure of the substitution step that for each $w\in\mathcal{W}_{n}$
		its class $[w]_{s}\in\mathcal{W}_{n}/\mathcal{Q}_{s}^{n}$ is a concatenation
		of $U_{s+1}\geq2^{8e(m)}$ different $\left(T_{s+1},2^{4e(m)-t_{s}},M_{s+1}\right)$-Feldman
		patterns (with building blocks out of $\mathcal{W}_{m}/\mathcal{Q}_{s}^{m}$).
		We call these the \emph{$s$-Feldman patterns} and refer to this sequence
		of Feldman patterns as the \emph{$s$-pattern type} of $w$. Using the same enumeration of $\left(T_{s+1},2^{4e(m)-t_{s}},M_{s+1}\right)$-Feldman patterns,
		we also define the $s$-pattern types for words in $rev(\mathcal{W}_{n})$.
	\end{rem}
	
	For such a $s$-Feldman pattern $P_{m}$ we make the following observation.
	\begin{lem}
		\label{lem:Occurrence-Substitutions}Let $w\in\mathcal{W}_{n}$ and
		$[P_{m}]_{s}$ be a $\left(T_{s+1},2^{4e(m)-t_{s}},M_{s+1}\right)$-Feldman
		pattern in $[w]_{s}$ with building blocks from the tuple $\left([A_{1}]_{s},\dots,[A_{2^{4e(m)-t_{s}}}]_{s}\right)$
		in $\mathcal{W}_{m}/\mathcal{Q}_{s}^{m}$. Then, for every $i=1,\dots,2^{4e(m)-t_{s}},$ the repetitions
		$\left(a_{i,j}\right)^{\mathfrak{p}_{n}^{2}}$ of each substitution instance $a_{i,j}$ in $\mathcal{W}_{m}$
		of $[A_{i}]_{s}$ are substituted the same number of times into each occurrence of $[A_i]_s^{T_{s+1}\cdot2^{t_{s+1}}}$ in $[P_m]_s$. This also yields that the repetitions
		$\left(a_{i,j}\right)^{\mathfrak{p}_{n}^{2}}$ of each substitution instance $a_{i,j}$ of each $[A_{i}]_{s}$, $i=1,\dots,2^{4e(m)-t_{s}}$, occur the same number of times in each cycle of $P_{m}$. 
	\end{lem}
	
	\begin{proof}
		By definition of $\left(T_{s+1},2^{4e(m)-t_{s}},M_{s+1}\right)$-Feldman
		patterns we have in each complete cycle that for every $i\in\left\{ 1,\dots,2^{4e(m)-t_{s}}\right\} $
		the maximum repetition 
		\[
		[A_{i}]_{s}^{T_{s+1}\cdot2^{(4e(m)-t_{s})\cdot2j}}
		\]
		occurs exactly once, where $j\in\left\{ 1,\dots,M_{s+1}\right\} $
		labels our Feldman pattern. This maximum repetition contains all the occurrences of the class $[A_i]_s$ in the complete cycle and, hence, will allow us to count the number of occurrences of its substitution instances. Into each $[A_{i}]_{s}^{T_{s+1}}$ we
		substitute a $\left(T_{s+2},2^{4e(m)-t_{s+1}},M_{s+2}\right)$-Feldman
		pattern built out of one of the $2^{t_{s+1}}$ many tuples $\left([A_{i,1}]_{s+1},\dots,[A_{i,2^{4e(m)-t_{s+1}}}]_{s+1}\right)$.
		According to our choice of the tuples in step (5) of the general substitution
		step each of these tuples is used $2^{(4e(m)-t_{s})\cdot2j-t_{s+1}}$
		many times within one maximum repetition $[A_{i}]_{s}^{T_{s+1}\cdot2^{(4e(m)-t_{s})\cdot2j}}$.
		Then we continue inductively for each of the obtained $\left(T_{s+2},2^{4e(m)-t_{s+1}},M_{s+2}\right)$-Feldman
		patterns until we finally insert repetitions of the form $\left(a_{i,j}\right)^{\mathfrak{p}_{n}^{2}}$
		(recall that $T_{s(m)+2}=\mathfrak{p}_{n}^{2}$).
	\end{proof}
	
	\subsubsection{Group actions}
	
	We also need to define the group actions for $G_{s}^{n}$ for $s\leq s(n)$.
	Except for $s_{0}=lh(\sigma_{n})$ we have $G_{s}^{n}=G_{s}^{m}$
	and the group actions are extended via the skew-diagonal actions.
	In the remaining case we have $G_{s_{0}}^{n}=G_{s_{0}}^{m}\oplus\mathbb{Z}_{2}$
	and we use Lemma \ref{lem:extension} to extend the $G_{s_{0}}^{m}$
	action to a $G_{s_{0}}^{n}$ action subordinate to the $G_{s_{0}-1}^{n}$
	action. To be more precise,  in the notation of
	the lemma we set 
	\begin{itemize}
		\item $G\times\mathbb{Z}_{2}\coloneqq G_{s_{0}-1}^{n}=G_{s_{0}-1}^{m}$,
		\item $H\coloneqq G_{s_{0}}^{m}$,
		\item $X/\mathcal{Q}\coloneqq\mathcal{W}_{n}/\mathcal{Q}_{s_{0}-1}^{n}$,
		\item $X/\mathcal{R}\coloneqq\mathcal{W}_{n}/\mathcal{Q}_{s_{0}}^{n}$,
		\item $\rho\coloneqq\rho_{s_{0},s_{0}-1}^{(m)}$.
	\end{itemize}

	\subsubsection{Properties}
	
	By construction the specifications are satisfied. In particular, specification (E3) holds since we obtain words by concatenating different Feldman patterns. 
	
	We also make the
	following observations regarding $\overline{f}$ distances.
	\begin{lem}
		\label{lem:case1f}For every $s\in\left\{ 1,\dots,s(m)\right\} $
		we have 
		\begin{equation}
			\overline{f}\left(W,W^{\prime}\right)>\alpha_{s,m}-\frac{2}{R_{m}}-\frac{1}{R_{n}}\label{eq:Case1fClass}
		\end{equation}
		on any substrings $W$ and $W^{\prime}$ of at least $\frac{h_{n}}{R_{n}}$
		consecutive symbols in any words $w,w^{\prime}\in\mathcal{W}_{n}$
		with $[w]_{s}\neq[w^{\prime}]_{s}$.
		
		Moreover, we have 
		\begin{equation}
			\overline{f}\left(W,W^{\prime}\right)>\beta_{m}-\frac{2}{R_{m}}-\frac{1}{R_{n}}\label{eq:Case1fBlock}
		\end{equation}
		on any substrings $W$ and $W^{\prime}$ of at least $\frac{h_{n}}{R_{n}}$
		consecutive symbols in any words $w,w^{\prime}\in\mathcal{W}_{n}$
		with $w\neq w^{\prime}$.
	\end{lem}
	
	\begin{proof}
		For each $s\in\left\{ 1,\dots,s(m)\right\} $ we use Proposition \ref{prop:finer}
		and assumption (\ref{eq:fClasses}) to obtain the following estimate
		\[
		\overline{f}\left(W,W^{\prime}\right)>\alpha_{s,m}-\frac{1}{R_{m}}-\frac{4}{2^{e(m)}}-\frac{1}{R_{n}}>\alpha_{s,m}-\frac{2}{R_{m}}-\frac{1}{R_{n}}
		\]
		by $2^{e(m)}>10R_{m}$, which follows from equation (\ref{eq:kn})
		on stage $m$ of the induction. Similarly, we use Proposition \ref{prop:finer}
		and assumption (\ref{eq:fBlocks}) to conclude equation (\ref{eq:Case1fBlock}).
	\end{proof}
	The expression on the right hand side in equation (\ref{eq:Case1fClass})
	will give us the definition 
	\begin{equation}
		\alpha_{s,n}\coloneqq\alpha_{s,m}-\frac{2}{R_{m}}-\frac{1}{R_{n}}\label{eq:AlphaCase1}
	\end{equation}
	in this case of the construction. Similarly, equation (\ref{eq:Case1fBlock})
	yields the definition 
	\[
	\beta_{n}\coloneqq\beta_{m}-\frac{2}{R_{m}}-\frac{1}{R_{n}}.
	\]
	We point out that 
	\[
	\beta_{n}>\frac{\beta_{m}}{2}>\frac{10}{R_{n}}>0
	\]
	by the conditions (\ref{eq:Rassum2}) and (\ref{eq:Rn}). This
	also implies $R_{n}>\frac{10}{\beta_{n}}$, that is, the induction assumption
	(\ref{eq:Rassum2}) for the next step is satisfied.
	\begin{rem}
		\label{rem:PrepTransfer1}To prepare the transfer to the setting of
		diffeomorphisms, we consider the constructed words in $\mathcal{W}_{n}$
		as concatenations of $k_{m}$ symbols in the alphabets $\Sigma_{m,s}$
		for $s\leq s(m)$. Here, $\Sigma_{m,s}$ is the alphabet where we
		identify each element in $\mathcal{W}_{m}/\mathcal{Q}_{s}^{m}$ of length
		$h_{m}$ with one symbol in $\left\{ 1,\dots,Q_{s}^{M(s)}\right\} $.
		As in the proof of Lemma \ref{lem:case1f} we use Proposition \ref{prop:finer}
		to obtain 
		\[
		\overline{f}_{\Sigma_{m,s}}\left(W,W^{\prime}\right)>1-\frac{4}{2^{e(m)}}-\frac{1}{R_{n}}
		\]
		on any substrings $W$ and $W^{\prime}$ of at least $\frac{k_{m}}{R_{n}}$
		consecutive $\Sigma_{m,s}$-symbols in any words $w,w^{\prime}\in\mathcal{W}_{n}$
		with $[w]_{s}\neq[w^{\prime}]_{s}$. Similarly, we let $\Sigma_{m}=\{1,\dots, |\mathcal{W}_m|\}$ be the alphabet such that each word in $\mathcal{W}_m$ of length $h_m$ is represented by one symbol in $\Sigma_m$. Then we can consider words
		in $\mathcal{W}_{n}$ as concatenations of $k_{m}$ symbols in the
		alphabets $\Sigma_{m}$.
		As above, we get
		\[
		\overline{f}_{\Sigma_{m}}\left(W,W^{\prime}\right)>1-\frac{4}{2^{e(m)}}-\frac{1}{R_{n}}
		\]
		on any substrings $W$ and $W^{\prime}$ of at least $\frac{k_{m}}{R_{n}}$
		consecutive $\Sigma_{m}$-symbols in any words $w,w^{\prime}\in\mathcal{W}_{n}$
		with $w\neq w^{\prime}$.
	\end{rem}

	\subsection{\label{subsec:Case-2}Case 2: s(n)=s(m)+1}
	
	In this case $\sigma_{n}$ is the only sequence in $\mathcal{T}\cap\left\{ \sigma_{k}:k\leq n\right\} $
	of length $s(n)$. Hence, we have $G_{s(n)}^{n}=\mathbb{Z}_{2}$ and
	we additionally need to define $\mathcal{Q}_{s(n)}^{n}$ as well as
	the action of $\mathbb{Z}_{2}$ on $\mathcal{W}_{n}/\mathcal{Q}_{s(n)}^{n}$. 
	
	Recall that the specification (Q4) on $\mathcal{Q}_{s(n)}^{n}$ is
	given in terms of a number $J_{s(n),n}$. We start the construction by
	choosing $J_{s(n),n}^{\prime}$ sufficiently
		large such that $J_{s(n),n}\coloneqq 2^{J_{s(n),n}^{\prime}}$ satisfies
	\begin{equation}
		J_{s(n),n}>2R_{n}^{2}.\label{eq:J}
	\end{equation}
	
	This time the construction of $\mathcal{W}_{n}$ uses two collections
	$\mathcal{W}^{\dagger}$ and $\mathcal{W}^{\dagger\dagger}$ of concatenation
	of $m$-words. Here, both collections $\mathcal{W}^{\dagger}$ and
	$\mathcal{W}^{\dagger\dagger}$ are obtained by a successive application
	of $s(m)+1$ substitution steps. Their construction parallels the
	one in Case 1 but with different choices of parameters which will play the roles of the corresponding letters in our general substitution mechanism from Section~\ref{sec:Substitution}.
	
	This time we use $\tilde{R}=2J_{s(n),n}^{3}$ in each substitution
	step and we repeat the definition of parameters $t_{i}$, $\nu_i$, $i=1,\dots,s(m)+1$,
	and
	\begin{eqnarray*}
		N_{i}=2^{(i-1)\cdot4e(m)}, &  & \text{for }i=1,\dots,s(m)+1.
	\end{eqnarray*}
	Moreover, we choose a number $U_{1}=2^{U_{1}^{\prime}}$ with $U_{1}^{\prime}\geq\max_{i=1,\dots,s(m)}t_{i}$
	and
	\begin{equation}
		U_{1}\geq\max\left(8J_{s(n),n}^{6},2^{8e(m)}\right).\label{eq:case2U}
	\end{equation}
	
	By forward recursion we define the parameters $M_{i}$, $i=1,\dots,s(m)+2$,
	as well as $U_{i}$, $i=2,\dots,s(m)+1$: 
	\begin{align*}
		M_{1}=1, &  & M_{2}=2^{4e(n)+1}U_{1}, &  & M_{i}=2^{(i+2)\cdot4e(n)}U_{i-1}2^{(4e(m)-t_{i-2})\cdot(2M_{i-1}+3)} & \text{ for \ensuremath{i\geq3},}\\
		U_{2}=U_{1}, &  &  &  & U_{i}=U_{i-1}2^{(4e(m)-t_{i-2})\cdot(2M_{i-1}+3)} & \text{ for \ensuremath{i\geq3}}.
	\end{align*}
	Then for each $i=1,\dots,s(m)+1$ we construct $M_{i+1}$ many $\left(T,2^{4e(m)-t_{i}},M_{i+1}\right)$-Feldman
	patterns, where the parameter $T$ will be chosen later and will differ
	between the two cases $\mathcal{W}^{\dagger}$ and $\mathcal{W}^{\dagger\dagger}$. 
	
	\subsubsection{Construction of $\mathcal{W}^{\dagger\dagger}$} \label{subsubsec:DaggerDagger}
	
	To construct the collection $\mathcal{W}^{\dagger\dagger}$ we choose
	the remaining parameters as follows: We set
	$K_{i}^{\dagger\dagger}=1$ for $i=1,\dots,s(m)$ 
	and
	\begin{equation}
		K_{s(m)+1}^{\dagger\dagger}=2^{8e(n)}.\label{eq:Kdaggerdagger}
	\end{equation}
	Moreover, we define the parameters $T_{i}^{\dagger\dagger}$, $i=1,\dots,s(m)+2$,
	by backward recursion: 
	\begin{align*}
		T_{s(m)+2}^{\dagger\dagger}=\mathfrak{p}_{n}, &  & T_{i}^{\dagger\dagger}=T_{i+1}^{\dagger\dagger}\cdot2^{(4e(m)-t_{i})\cdot\left(2M_{i+1}+3\right)} & \text{ for \ensuremath{i\leq s(m)+1}}.
	\end{align*}
	Finally, we let $P_{1}^{\dagger\dagger}=1$, $P_{2}^{\dagger\dagger}=|G^m_1|$, and $P_{i}^{\dagger\dagger}=|G^m_1|\cdot \prod^{i-1}_{j=2}|ker(\rho_{j,j-1}^{(m)})|$ for
		$i=3,\dots,s(m)+1$. 
		
		Then we apply the general substitution step $s(m)+1$ many times as
		in Case 1. For $i=1,\dots,s(m)+1$, we use the parameters $t=t_{i}$, $\nu=\nu_i$, $K=K_i^{\dagger\dagger}$, $P=P_i^{\dagger\dagger}$, $M=M_i$, $\overline{M}=M_{i+1}$, $N=N_i$, $U=U_i$, $\overline{U}=U_{i+1}$, $T=T_i^{\dagger\dagger}$, $\overline{T}=T_{i+1}^{\dagger\dagger}$ in the $i$-th application of the general substitution mechanism in our successive construction.
	
	In particular, during the $i$-th application
	we use the first at most $K_{i}^{\dagger\dagger}P_{i}^{\dagger\dagger}U_{i}$
	many different $\left(T_{i+1}^{\dagger\dagger},2^{4e(m)-t_{i}},M_{i+1}\right)$-Feldman
	patterns from our collection above in step (4) of the general substitution mechanism. At the end, we obtain a collection
	$\mathcal{W}^{\dagger\dagger}\subset\left(\mathcal{W}_{m}\right)^{k^{\dagger\dagger}}$
	with 
	\begin{align*}
		k^{\dagger\dagger}= & U_{s(m)+1}\cdot T_{s(m)+1}^{\dagger\dagger}\cdot2^{(4e(m)-t_{s(m)})\cdot(2M_{s(m)+1}+3)}\\
		= & U_{s(m)+1}\cdot T_{s(m)+2}^{\dagger\dagger}\cdot2^{8e(m)\cdot(M_{s(m)+2}+M_{s(m)+1}+3)-t_{s(m)}\cdot(2M_{s(m)+1}+3)},
	\end{align*}
	where we use $t_{s(m)+1}=0$, and $|\mathcal{W}^{\dagger\dagger}/\mathcal{Q}_{s}|=P^{\dagger \dagger}_{s+1}$ for $s=1,\dots,s(m)$ as well as $|\mathcal{W}^{\dagger\dagger}|=P^{\dagger \dagger}_{s(m)+1}2^{8e(n)}$.
	\begin{rem}
		\label{rem:fdaggerdagger} We recall that the choice of $U_{1}$ in
		(\ref{eq:case2U}) is made in this way to satisfy condition (\ref{eq:U})
		from the general substitution step for $\tilde{R}=2J_{s(n),n}^{3}$.
		Thus, for the collection $\mathcal{W}^{\dagger\dagger}$ we obtain
		the same statements as in Lemma \ref{lem:case1f} on substantial
		substrings of length at least $\frac{k^{\dagger\dagger}h_{m}}{2J_{s(n),n}^{3}}$.
	\end{rem}

	\subsubsection{Construction of $\mathcal{W}^{\dagger}$}
	
	We continue by constructing the collection $\mathcal{W}^{\dagger}$. As before, the parameters will play the roles of the corresponding letters in our general substitution mechanism from Section~\ref{sec:Substitution}.
	This time we choose the parameters $K_{i}^{\dagger}\in\mathbb{Z}^{+}$,
	$i=1,\dots,s(m)+1$, as follows: 
	\begin{align} \label{eq:Kdagger}
		K_{i}^{\dagger}\cdot|ker(\rho_{i,i-1}^{(m)})| & =2^{4e(n)}\text{ for }i=2,\dots,s(m), &  & K_{s(m)+1}^{\dagger}=2^{4e(n)}.
	\end{align}
	The parameters $T_{i}$, $i=1,\dots,s(m)+2$, are defined by backward
	recursion: 
	\begin{align*}
		T_{s(m)+2}^{\dagger}=\mathfrak{p}_{n}^{2}-\mathfrak{p}_{n}, &  & T_{i}^{\dagger}=T_{i+1}^{\dagger}\cdot2^{(4e(m)-t_{i})\cdot(2M_{i+1}+3)} & \text{ for \ensuremath{i\leq s(m)+1}}.
	\end{align*}
	Once again, we apply the general substitution step $s(m)+1$ many
	times as in Case 1. Here, we define $P_{i}^{\dagger}=2^{(i-1)\cdot4e(n)}$
	for $i=1,\dots,s(m)+1$. In the $i$-th application with $i\in\left\{ 1,\dots,s(m)+1\right\} $
	we use at most $K_{i}^{\dagger}P_{i}^{\dagger}U_{i}$ many different
	$\left(T_{i+1}^{\dagger},2^{4e(m)-t_{i}},M_{i+1}\right)$-Feldman
	patterns from our collection above that have not been used in the
	construction of $\mathcal{W}^{\dagger\dagger}$. In particular, in
	the enumeration of patterns, they come after the patterns used for
	$\mathcal{W}^{\dagger\dagger}$. Note that there are enough patterns
	since 
	\[
	K_{i}^{\dagger}P_{i}^{\dagger}U_{i}+K_{i}^{\dagger\dagger}P_{i}^{\dagger\dagger}U_{i}\leq\left(2^{4e(n)}+K_{i}^{\dagger\dagger}\right)\cdot2^{(i-1)\cdot4e(n)}\cdot U_{i}\leq M_{i+1}.
	\]
	Then we obtain a collection $\mathcal{W}^{\dagger}\subset\left(\mathcal{W}_{m}\right)^{k^{\dagger}}$
	with 
	\begin{align*}
		k^{\dagger}= & U_{s(m)+1}\cdot T_{s(m)+1}^{\dagger}\cdot2^{(4e(m)-t_{s(m)})\cdot(2M_{s(m)+1}+3)}\\
		= & U_{s(m)+1}\cdot T_{s(m)+2}^{\dagger}\cdot2^{8e(m)\cdot(M_{s(m)+2}+M_{s(m)+1}+3)-t_{s(m)}\cdot(2M_{s(m)+1}+3)}
	\end{align*}
	and $|\mathcal{W}^{\dagger}|=2^{(s(m)+1)\cdot e(n)}$ as well as $|\mathcal{W}^{\dagger}/\mathcal{Q}_{s}|=P_{s+1}^{\dagger}$
	for $s=1,\dots,s(m)$.
	\begin{rem}
		\label{rem:fdagger} By the same reasoning as in Remark \ref{rem:fdaggerdagger}
		we obtain for the collection $\mathcal{W}^{\dagger\dagger}$ the same
		statements as in Lemma \ref{lem:case1f} on substantial substrings
		of length at least $\frac{k^{\dagger}h_{m}}{2J_{s(n),n}^{3}}$.
	\end{rem}
	
	We also make the following observation on the $\overline{f}$ distance
	between strings from $\mathcal{W}^{\dagger}$ and $\mathcal{W}^{\dagger\dagger}$.
	\begin{lem}
		\label{lem:differenceDaggers}Let $w\in\mathcal{W}^{\dagger}$ and
		$w'\in\mathcal{W}^{\dagger\dagger}$. For any substrings $W$ and
		$W'$ of $w$ and $w'$, respectively, of lengths at least $\frac{k^{\dagger\dagger}h_{m}}{2J_{s(n),n}^{2}}$
		we have 
		\[
		\overline{f}\left(W,W'\right)>\alpha_{1,m}-\frac{1}{R_{m}}-\frac{4}{2^{e(m)}}-\frac{1}{2J_{s(n),n}^{3}}.
		\]
	\end{lem}
	
	\begin{proof}
		By construction, already in the first application of the substitution
		step the Feldman patterns on the level of the $1$-equivalence classes
		of $w$ and $w'$ differ. In the construction of $\mathcal{W}^{\dagger\dagger}$ in Section~\ref{subsubsec:DaggerDagger}, we use the first at most $K_{1}^{\dagger\dagger}P_{1}^{\dagger\dagger}U_{1}$
			many different $\left(T_{2}^{\dagger\dagger},2^{4e(m)-t_{1}},M_{2}\right)$-Feldman
			patterns from our collection of $(T,2^{4e(m)-t_{1}},M_{2})$-Feldman patterns. Hence, in the enumeration of $(T,2^{4e(m)-t_{1}},M_{2})$-Feldman
		patterns, the patterns used in the construction of $\mathcal{W}^{\dagger\dagger}$
		come before the ones in $\mathcal{W}^{\dagger}$. Since we also have
		$T_{2}^{\dagger\dagger}<T_{2}^{\dagger}$, we can apply Lemma \ref{lem:DifferentT}.
		As in Lemma \ref{lem:case1f} we obtain from Proposition \ref{prop:finer}
		that 
		\[
		\overline{f}(\tilde{W},\tilde{W}')>\alpha_{1,m}-\frac{1}{R_{m}}-\frac{4}{2^{e(m)}}-\frac{1}{2J_{s(n),n}^{3}}
		\]
		on all substrings $\tilde{W}$ and $\tilde{W}'$ of $w$ and $w'$,
		respectively, of lengths at least $\frac{k^{\dagger}h_{m}}{2J_{s(n),n}^{3}}$.
		Since $\frac{k^{\dagger}}{k^{\dagger\dagger}}=\mathfrak{p}_{n}-1<J_{s(n),n},$ we have
			\[
			\frac{k^{\dagger\dagger}h_{m}}{2J_{s(n),n}^{2}}>\frac{k^{\dagger}h_{m}}{2J_{s(n),n}^{3}},
			\]
			and hence our strings $W$ and $W'$ are long enough to conclude the statement.
	\end{proof}
	\begin{rem}
		\label{rem:Substitution-Dagger}As in Remark \ref{rem:s-pattern-type}
		we introduce the notion of $s$-pattern type for strings in $\mathcal{W}^{\dagger}$
		and $rev(\mathcal{W}^{\dagger})$. As in Lemma \ref{lem:Occurrence-Substitutions}
		we make the following observation. Let $w\in\mathcal{W}^{\dagger}$
		and $[P_{m}]_{s}$ be a $\left(T_{s+1}^{\dagger},2^{4e(m)-t_{s}},M_{s+1}\right)$-Feldman
		pattern in $[w]_{s}$ with building blocks from the tuple $\left([A_{1}]_{s},\dots,[A_{2^{4e(m)-t_{s}}}]_{s}\right)$
		in $\mathcal{W}_{m}/\mathcal{Q}_{s}^{m}$. Then in each complete cycle
		of $P_{m}$ and for each substitution instance $a_{i,j}$ in $\mathcal{W}_{m}$
		of each $[A_{i}]_{s}$, $i=1,\dots,2^{4e(m)-t_{s}}$, its repetition
		$\left(a_{i,j}\right)^{\mathfrak{p}_{n}^{2}-\mathfrak{p}_{n}}$ occurs the same number of
		times.
	\end{rem}

	\subsubsection{Construction of $\mathcal{W}_{n}$} \label{subsubsec:Wn}
	
	To construct $\mathcal{W}_{n}$ we recall that $U_{1}$ (and, hence,
	$U_{s(m)+1}$) is a multiple of $2J_{s(n),n}$. Thus, we can divide
	each element $v_{1}\in\mathcal{W}^{\dagger}$ and $v_{2}\in\mathcal{W}^{\dagger\dagger}$
	into $J=J_{s(n),n}$ or $2J$ many pieces of equal length: 
	\[
	v_{1}=v_{1,1}v_{1,2}\dots v_{1,J}\text{ and }v_{2}=v_{2,1}v_{2,2}\dots v_{2,2J}.
	\]
	Then we can concatenate their substrings in the following way:
	\begin{equation}
		v_{1}\ast v_{2}\coloneqq v_{2,1}v_{1,1}v_{2,2}v_{2,3}v_{1,2}v_{2,4}\dots v_{2,2J-1}v_{1,J}v_{2,2J}\label{eq:Star}
	\end{equation}
	
	Using this notation we build $\mathcal{W}_{n}/\mathcal{Q}_{s}^{n}$
	for $0\leq s\leq s(m)$ inductively: 
	\begin{itemize}
		\item Trivially, for $s=0$ we set
		$\mathcal{W}_{n}/\mathcal{Q}_{0}^{n}=\left\{ [v_{1}]_{0}\ast[v_{2}]_{0}\right\} $,
		where $[v_{1}]_{0}$ and $[v_{2}]_{0}$ are the single elements in
		$\mathcal{W}^{\dagger}/\mathcal{Q}_{0}$ and $\mathcal{W}^{\dagger\dagger}/\mathcal{Q}_{0}$,
		respectively.
		\item For $0\leq s<s(m)$ we choose a collection $\Upsilon_s \subseteq \mathcal{W}_{n}/\mathcal{Q}_{s}^{n}$ that intersects each orbit of the action by the group $\rho^{m}_{s+1,s}(G^m_{s+1})$ exactly once. For any element $C\ast D\in\Upsilon_s \subseteq\mathcal{W}_{n}/\mathcal{Q}_{s}^{n}$
		with $C\in\mathcal{W}^{\dagger}/\mathcal{Q}_{s}$, $D\in\mathcal{W}^{\dagger\dagger}/\mathcal{Q}_{s}$
		we take a collection $\mathcal{C}$ of substitution instances of $C$ in $\mathcal{W}^{\dagger}/\mathcal{Q}_{s+1}$ that intersects each orbit of $ker(\rho^{m}_{s+1,s})$ exactly once. We note that $|\mathcal{C}|=K^{\dagger}_{s+1}$ by construction. Moreover, we choose one substitution instance $d\in\mathcal{W}^{\dagger\dagger}/\mathcal{Q}_{s+1}$
		of $D$. For every $c \in \mathcal{C}$ we let $c\ast d$ and its images under the action by $G^m_{s+1}$ belong to $\mathcal{W}_{n}/\mathcal{Q}_{s+1}^{n}$. Notice that for
		any $g\in G_{s+1}^{m}$ and $[w]_{s+1}=c\ast d\in\mathcal{W}_{n}/\mathcal{Q}_{s+1}^{n}$
		we have $g[w]_{s+1}=(gc)\ast(gd)$. Since $\mathcal{W}^{\dagger}/\mathcal{Q}_{s+1}$ as well as $\mathcal{W}^{\dagger\dagger}/\mathcal{Q}_{s+1}$ is closed under the skew diagonal action by construction, we have $gc\in \mathcal{W}^{\dagger}/\mathcal{Q}_{s+1}$ and $gd\in \mathcal{W}^{\dagger\dagger}/\mathcal{Q}_{s+1}$. Altogether there are $2^{4e(n)}$ substitution instances of each element of $\mathcal{W}_{n}/\mathcal{Q}_{s}^{n}$ in $\mathcal{W}_{n}/\mathcal{Q}_{s+1}^{n}$ by our condition $K_{s+1}^{\dagger}\cdot|ker(\rho_{s+1,s}^{(m)})| =2^{4e(n)}$ on $K_{s+1}^{\dagger}$ in (\ref{eq:Kdagger}).
		\item For any element $C\ast D\in\mathcal{W}_{n}/\mathcal{Q}_{s(m)}^{n}$
		with $C\in\mathcal{W}^{\dagger}/\mathcal{Q}_{s(m)}$, $D\in\mathcal{W}^{\dagger\dagger}/\mathcal{Q}_{s(m)}$
		we pair each of the $2^{4e(n)}$ many substitution instances $v_i\in\mathcal{W}^{\dagger}$
		of $C$ with $2^{4e(n)}$ many substitution instances $d_{1}^{(i)},\dots,d_{2^{4e(n)}}^{(i)}\in\mathcal{W}^{\dagger\dagger}$
		of $D$.  Here, we notice that we have $2^{4e(n)}$ many substitution instances $v_i\in\mathcal{W}^{\dagger}$
			of $C$ by $K^{\dagger}_{s(m)+1}=2^{4e(n)}$ from (\ref{eq:Kdagger}), and $2^{8e(n)}$ substitution instances of $D$
		in $\mathcal{W}^{\dagger\dagger}$ by our choice $K_{s(m)+1}^{\dagger\dagger}=2^{8e(n)}$
		in (\ref{eq:Kdaggerdagger}). Then, for such elements of the form $w=v_{i}\ast d_{j}^{(i)},$ we define
		the new equivalence classes 
		\begin{equation} \label{eq:NewEquiRel}
			\left[w\right]_{s(m)+1}=\left\{ v_{i}\ast d_{j}^{(i)}:j=1,\dots,2^{4e(n)}\right\} .
		\end{equation}
	\end{itemize}
	
	Moreover, we set
	\[
	\mathcal{W}_{n}=\left\{ v_{i}\ast d_{j}^{(i)}:v_{i}\in\mathcal{W}^{\dagger},j=1,\dots,2^{4e(n)}\right\} .
	\]
	We note that these definitions fulfill specification (Q6).
	
	\subsubsection{Equivalence relations on $\mathcal{W}_{n}$}
	
	The equivalence classes $\mathcal{W}_{n}/\mathcal{Q}_{s}^{n}$ for
	$s\leq s(m)$ are given by the product equivalence relation. The new
	equivalence relation $\mathcal{Q}_{s(n)}^{n}$ satisfies 
	\[
	v_{1}\ast v_{2}\sim v_{1}^{\prime}\ast v_{2}^{\prime}\text{ if and only if }v_{1}=v_{1}^{\prime}\text{ and }\left[v_{2}\right]_{s(m)}=\left[v_{2}^{\prime}\right]_{s(m)},
	\]
	that is, two words in $\mathcal{W}_{n}$ are equivalent with respect to
	$\mathcal{Q}_{s(n)}^{n}$ if and only if when considering each as a concatenation
	of $J_{s(n),n}$ strings of equal length all their central strings
	(whose concatenation forms an element from $\mathcal{W}^{\dagger}$)
	coincide and their initial and terminal strings (whose concatenation
	forms an element from $\mathcal{W}^{\dagger\dagger}$) are equivalent
	with respect to the product equivalence relation induced from $\mathcal{Q}_{s(m)}^{m}$.
	Hence, by the definition of the equivalence relation $\mathcal{Q}^n_{s(n)}$ in (\ref{eq:NewEquiRel}), for $w=v_{1}\ast v_{2}\in\mathcal{W}_{n},$ we have
	\[
	[w]_{s(n)}=\left\{ v_{1}\ast v_{2}^{\prime}:\left[v_{2}\right]_{s(m)}=\left[v_{2}^{\prime}\right]_{s(m)}\right\}
	\]
	We also observe that 
	\begin{equation}
		\frac{k^{\dagger}}{k^{\dagger}+k^{\dagger\dagger}}=\frac{T_{s(m)+2}^{\dagger}}{T_{s(m)+2}^{\dagger}+T_{s(m)+2}^{\dagger\dagger}}=1-\frac{1}{\mathfrak{p}_{n}}>1-\epsilon_{n}\label{eq:LengthProportion}
	\end{equation}
	by our condition on $\mathfrak{p}_{n}$ in equation (\ref{eq:Pn}). In this way,
	specification (Q4) is satisfied and the equivalence relation $\mathcal{Q}_{s(n)}^{n}$
	refines $\mathcal{Q}_{s(m)}^{n}$. 
	
	\subsubsection{Group actions}
	
	For all $s\leq s(m)$ we have $G_{s}^{n}=G_{s}^{m}$ and the group
	actions are defined via the skew-diagonal action. According to our construction in Section \ref{subsubsec:Wn}, $\mathcal{W}_{n}/\mathcal{Q}_{s}^{n}$
	is closed under the skew diagonal action by $G_{s}^{n}$.
	
	To define the action by $G_{s(n)}^{n}=\mathbb{Z}_{2}$ we apply Lemma
	\ref{lem:extension} with 
	\begin{itemize}
		\item $G\times\mathbb{Z}_{2}\coloneqq G_{s(m)}^{n}=G_{s(m)}^{m}$,
		\item $H$ the trivial group,
		\item $H\times\mathbb{Z}_{2}\coloneqq G_{s(n)}^{n}$,
		\item $X/\mathcal{Q}\coloneqq\mathcal{W}_{n}/\mathcal{Q}_{s(m)}^{n}$,
		\item $X/\mathcal{R}\coloneqq\mathcal{W}_{n}/\mathcal{Q}_{s(n)}^{n}$,
		\item $\rho\coloneqq\rho_{s(n),s(m)}.$
	\end{itemize}
	Hence, we obtain an action of $G_{s(n)}^{n}=\mathbb{Z}_{2}$ on $\mathcal{W}_{n}/\mathcal{Q}_{s(n)}^{n}$
	subordinate to the action of $G_{s(m)}^{m}$.
	
	\subsubsection{Further properties}
	
	In this case we also make observations regarding $\overline{f}$ distances.
	\begin{lem}
		\label{lem:fCase2}
		\begin{enumerate}
			\item For every $s\in\left\{ 1,\dots,s(m)\right\} $ we have 
			\begin{equation}
				\overline{f}\left(W,W^{\prime}\right)>\alpha_{s,m}-\frac{3}{R_{m}}-\frac{2}{R_{n}}\label{eq:Case2fClass}
			\end{equation}
			on any substrings $W$ and $W^{\prime}$ of at least $\frac{h_{n}}{R_{n}}$
			consecutive symbols in any words $w,w^{\prime}\in\mathcal{W}_{n}$
			with $[w]_{s}\neq[w^{\prime}]_{s}$.
			\item In case of $s=s(n)$ we have 
			\begin{equation}
				\overline{f}\left(W,W^{\prime}\right)>\beta_{m}-\frac{3}{R_{m}}-\frac{2}{R_{n}}\label{eq:Case2fClass-1}
			\end{equation}
			on any substrings $W$ and $W^{\prime}$ of at least $\frac{h_{n}}{R_{n}}$
			consecutive symbols in any words $w,w^{\prime}\in\mathcal{W}_{n}$
			with $[w]_{s(n)}\neq[w^{\prime}]_{s(n)}$.
			\item Moreover, we have 
			\begin{equation}
				\overline{f}\left(W,W^{\prime}\right)>\frac{1}{\mathfrak{p}_{n}}\cdot\left(\beta_{m}-\frac{1}{R_{m}}-\frac{4}{2^{e(m)}}-\frac{1}{2J_{s(n),n}^{3}}\right)-\frac{2}{R_{n}}\label{eq:Case2fBlock}
			\end{equation}
			on any substrings $W$ and $W^{\prime}$ of at least $\frac{h_{n}}{R_{n}}$
			consecutive symbols in any words $w,w^{\prime}\in\mathcal{W}_{n}$
			with $w\neq w^{\prime}$.
		\end{enumerate}
	\end{lem}
	
	\begin{proof}
		We decompose the strings $W$ and $W^{\prime}$ into the substrings
		of length $\frac{h_{n}}{J_{s(n),n}}$ according to equation (\ref{eq:Star}).
		To complete partial ones at the beginning and end we add fewer than
		$4\frac{h_{n}}{J_{s(n),n}}$ symbols in total which corresponds to
		a proportion of less than
		\[
		\frac{2R_{n}}{J_{s(n),n}}<\frac{1}{R_{n}}
		\]
		of the total length by equation (\ref{eq:J}). This might increase
		the $\overline{f}$ distance by at most $\frac{1}{R_{n}}$ by Fact
		\ref{fact:omit_symbols}.
		\begin{enumerate}
			\item On each of these substrings of length $\frac{h_{n}}{J_{s(n),n}}$
			the part belonging to $\mathcal{W}^{\dagger}$ has length $\frac{k^{\dagger}h_{m}}{J_{s(n),n}},$
			while the part belonging to $\mathcal{W}^{\dagger\dagger}$ represents
			a proportion of $\frac{1}{\mathfrak{p}_{n}}$ of the length by equation (\ref{eq:LengthProportion}).
			We delete them for the consideration of this subcase, increasing the
			$\overline{f}$ distance by at most $\frac{2}{\mathfrak{p}_{n}}$ according to
			Fact \ref{fact:omit_symbols}. On the remaining part we use Remark
			\ref{rem:fdagger} on the level of the equivalence class $s\in\left\{ 1,\dots,s(m)\right\} $.
			With the aid of conditions (\ref{eq:Pn}) on $\mathfrak{p}_{n}$ and (\ref{eq:kn})
			applied on $e(m),$ we obtain the following estimate: 
			\[
			\overline{f}\left(W,W^{\prime}\right)>\alpha_{s,m}-\frac{1}{R_{m}}-\frac{4}{2^{e(m)}}-\frac{1}{2J_{s(n),n}^{3}}-\frac{2}{\mathfrak{p}_{n}}-\frac{1}{R_{n}}>\alpha_{s,m}-\frac{3}{R_{m}}-\frac{2}{R_{n}}.
			\]
			\item We follow the argument from the previous subcase but apply Remark
			\ref{rem:fdagger} on the level of the actual blocks in $\mathcal{W}^{\dagger}$.
			\item While the strings might coincide on the central parts belonging to
			$\mathcal{W}^{\dagger}$, this time the $\overline{f}$ distance comes
			from the initial and terminal parts of length $\frac{k^{\dagger\dagger}h_{m}}{2J_{s(n),n}}$
			belonging to $\mathcal{W}^{\dagger\dagger}$. To see this, we further
			subdivide the strings of length $\frac{kh_{m}}{J_{s(n),n}}$ into
			pieces of length $\frac{k^{\dagger\dagger}h_{m}}{2J_{s(n),n}}$. Lemma
			\ref{lem:differenceDaggers} allows us to estimate the $\overline{f}$
			distance if one of those pieces lies in $\mathcal{W}^{\dagger}$ and
			the other one lies in $\mathcal{W}^{\dagger\dagger}$. For both pieces
			lying in $\mathcal{W}^{\dagger\dagger}$ we apply Remark \ref{rem:fdaggerdagger}
			for blocks in $\mathcal{W}^{\dagger\dagger}$. Hence, we obtain by
			applying Lemma \ref{lem:symbol by block replacement} with $\overline{f}\geq\frac{1}{\mathfrak{p}_{n}}$
			that
			\begin{align*}
				\overline{f}\left(W,W^{\prime}\right) & >\frac{1}{\mathfrak{p}_{n}}\cdot\left(\beta_{m}-\frac{1}{R_{m}}-\frac{4}{2^{e(m)}}-\frac{1}{2J_{s(n),n}^{3}}\right)-\frac{1}{J_{s(n),n}}-\frac{1}{R_{n}}\\
				& >\frac{1}{\mathfrak{p}_{n}}\cdot\left(\beta_{m}-\frac{1}{R_{m}}-\frac{4}{2^{e(m)}}-\frac{1}{2J_{s(n),n}^{3}}\right)-\frac{2}{R_{n}}.
			\end{align*}
		\end{enumerate}
	\end{proof}
	The expressions on the right hand side in Lemma \ref{lem:fCase2}
	motivate the following definitions. For $s\in\left\{ 1,\dots,s(m)\right\} $
	we set 
	\begin{equation}
		\alpha_{s,n}\coloneqq\alpha_{s,m}-\frac{3}{R_{m}}-\frac{2}{R_{n}},\label{eq:AlphaCase2}
	\end{equation}
	while for $s=s(n)$ we introduce 
	\begin{equation}
		\alpha_{s(n),n}\coloneqq\beta_{m}-\frac{3}{R_{m}}-\frac{2}{R_{n}}.\label{eq:NewAlpha}
	\end{equation}
	We notice that 
	\begin{equation*}
		\alpha_{s(n),n}>\beta_{m}-\frac{1}{3}\beta_{m}-\frac{1}{20\mathfrak{p}_{n}}\beta_{m}\geq\frac{\beta_{m}}{2}>0,
	\end{equation*}
	by equations (\ref{eq:Rassum2}) and (\ref{eq:Rn}). Hence, we get
	\begin{equation}\label{eq:EstimateAlphaR}
		\frac{1}{R_n} < \frac{\alpha_{s(n),n}}{20\mathfrak{p}_n},
	\end{equation}
	by condition (\ref{eq:Rn}). Finally,
	we set
	\[
	\beta_{n}\coloneqq\frac{1}{\mathfrak{p}_{n}}\cdot\left(\beta_{m}-\frac{1}{R_{m}}-\frac{4}{2^{e(m)}}-\frac{1}{2J_{s(n),n}^{3}}\right)-\frac{2}{R_{n}},
	\]
	and we note with the aid of equation (\ref{eq:Rn}) again that
	\[
	\beta_{n}>\frac{1}{\mathfrak{p}_{n}}\cdot\frac{\beta_{m}}{2}-\frac{2}{R_{n}}\geq\frac{9}{20}\cdot\frac{\beta_{m}}{\mathfrak{p}_{n}}\geq\frac{9}{R_{n}},
	\]
	which not only yields $\beta_{n}>0$, but also $R_{n}>\frac{9}{\beta_{n}}$.
	That is, the induction assumption (\ref{eq:Rassum2}) for the next step
	is satisfied.
	\begin{rem}
		\label{rem:PrepTransfer2}As in Remark \ref{rem:PrepTransfer1}
		we prepare the transfer to the setting of diffeomorphisms by considering
		the words in $\mathcal{W}_{n}$ as concatenations of $k_{m}$ symbols
		in the alphabets $\Sigma_{m}$ and $\Sigma_{m,s}$ for $s\leq s(m)$.
		As in the proof of Lemma \ref{lem:fCase2} we use Proposition \ref{prop:finer}
		to obtain 
		\[
		\overline{f}_{\Sigma_{m,s}}\left(W,W^{\prime}\right)>1-\frac{4}{2^{e(m)}}-\frac{1}{R_{n}}-\frac{1}{R_{n}}-\frac{2}{\mathfrak{p}_{n}}>1-\frac{1}{R_{m}}-\frac{2}{R_{n}}
		\]
		on any substrings $W$ and $W^{\prime}$ of at least $\frac{k_{m}}{R_{n}}$
		consecutive $\Sigma_{m,s}$-symbols in any words $w,w^{\prime}\in\mathcal{W}_{n}$
		with $[w]_{s}\neq[w^{\prime}]_{s}$ for $s\leq s(m)$. Similarly,
		the proof of part (3) in Lemma \ref{lem:fCase2} yields 
		\begin{equation}
			\overline{f}_{\Sigma_{m}}\left(W,W^{\prime}\right)>\frac{1}{\mathfrak{p}_{n}}\cdot\left(1-\frac{4}{2^{e(m)}}-\frac{1}{R_{n}}\right)-\frac{2}{R_{n}}>\frac{1}{2\mathfrak{p}_{n}}\label{eq:BaseWords}
		\end{equation}
		on any substrings $W$ and $W^{\prime}$ of at least $\frac{k_{m}}{R_{n}}$
		consecutive $\Sigma_{m}$-symbols in any words $w,w^{\prime}\in\mathcal{W}_{n}$
		with $w\neq w^{\prime}$. Following the proof of part (2) in Lemma
		\ref{lem:fCase2} we can also express the $\overline{f}$-distance
		on substantial substrings of different classes of the new equivalence
		relation $\mathcal{Q}_{s(n)}^{n}$ in the alphabet $\Sigma_{m}$:
		We have 
		\begin{equation}
			\overline{f}_{\Sigma_{m}}\left(W,W^{\prime}\right)>1-\frac{4}{2^{e(m)}}-\frac{1}{R_{n}}-\frac{1}{R_{n}}-\frac{2}{\mathfrak{p}_{n}}>1-\frac{1}{R_{m}}-\frac{2}{R_{n}}\label{eq:BaseClass}
		\end{equation}
		on any substrings $W$ and $W^{\prime}$ of at least $\frac{k_{m}}{R_{n}}$
		consecutive $\Sigma_{m}$-symbols in any words $w,w^{\prime}\in\mathcal{W}_{n}$
		with $[w]_{s(n)}\neq[w^{\prime}]_{s(n)}$.
	\end{rem}
	
	\begin{rem}
		\label{rem:BaseCase}In the base case $m=0$ of our inductive process
		we have $\mathcal{W}_{0}=\left\{ 1,\dots,2^{12}\right\} $, that is, the
		$m$-blocks are actual symbols and $e(0)=3$ by $2^{4e(0)}=2^{12}$. Thus, by the first inequalities in  (\ref{eq:BaseWords})
		and (\ref{eq:BaseClass}), we can set 
		\begin{align*}
			\alpha_{1,n}=\min\left(\frac{1}{2}-\frac{2}{R_{n}}-\frac{2}{\mathfrak{p}_{n}},\frac{1}{9}\right) & , &  & \beta_{n}=\min\left(\frac{1}{\mathfrak{p}_{n}}\cdot\left(\frac{1}{2}-\frac{1}{R_{n}}\right)-\frac{2}{R_{n}},\alpha_{1,n}\right).
		\end{align*}
	\end{rem}

	\subsection{\label{subsec:Concluding-remarks}Conclusion of the construction}
	We recall the definition of the function $M(s)$ from Definition \ref{def:M-and-s}. For every $s \in \mathbb{Z}^+,$ we use equations (\ref{eq:NewAlpha}) and (\ref{eq:AlphaCase2}), as
	well as (\ref{eq:AlphaCase1}), to conclude for every $\ell>M(s)$ with $\sigma_{\ell} \in \mathcal{T}$ that   
	\[
	\alpha_{s,\ell}>\alpha_{s,M(s)}-\frac{3}{R_{M(s)}}-\sum_{k>M(s): \sigma_k \in \mathcal{T}}\frac{5}{R_{k}}.
	\]
	By condition (\ref{eq:Rn}) and $\beta_i < \alpha_{s,M(s)}$ for $i\geq M(s)$, this implies
		\begin{align*}
			\alpha_{s,\ell}& >\alpha_{s,M(s)}-\frac{3}{R_{M(s)}}-\sum_{k>M(s): \sigma_k \in \mathcal{T}}\frac{5\beta_{k-1}}{40\mathfrak{p}_{k}} \\
			& > \alpha_{s,M(s)}-\frac{3}{R_{M(s)}}-\frac{5\alpha_{s,M(s)}}{40}\cdot \sum_{k>M(s): \sigma_k \in \mathcal{T}}\frac{1}{\mathfrak{p}_{k}} \\
			&\geq \alpha_{s,M(s)}-\frac{3}{R_{M(s)}}-\frac{5}{40}\alpha_{s,M(s)},
		\end{align*}
		where we used condition (\ref{eq:Pn}) in the last step. Together with the estimate (\ref{eq:EstimateAlphaR}) we obtain
	\[
	\alpha_{s,\ell}>\alpha_{s,M(s)}-\frac{6\alpha_{s,M(s)}}{40\mathfrak{p}_{M(s)}}-\frac{5}{40}\alpha_{s,M(s)}>\frac{\alpha_{s,M(s)}}{2}.
	\]
	Thus, for every $s\in\mathbb{Z}^+$ we can define
	\[
	\alpha_s \coloneqq \lim_{\substack{\ell \to \infty \\
			\sigma_{\ell}\in \mathcal{T}}} \alpha_{s,\ell}\geq \frac{\alpha_{s,M(s)}}{2},
	\]
	where the limit exists because by ... the sequence is decreasing.
	This yields the following proposition.
	\begin{prop} \label{prop:conclusio}
		We have for every $n\geq M(s)$ with $\sigma_n \in \mathcal{T}$ that
		\begin{equation}
			\overline{f}\left(W,W^{\prime}\right)>\alpha_{s}
		\end{equation}
		on any substrings $W$ and $W^{\prime}$ of at least $\frac{h_{n}}{R_{n}}$
		consecutive symbols in any words $w,w^{\prime}\in\mathcal{W}_{n}$
		with $[w]_{s}\neq[w^{\prime}]_{s}$.
	\end{prop}

	\section{\label{sec:Non-Equiv}Proof of Non-Kakutani Equivalence}
	
	We begin by proving in Proposition \ref{prop:NonIsom} below that for any
	tree $\mathcal{T}\in\mathcal{T}\kern-.5mm rees$ the constructed transformation
	$T=\Psi(\mathcal{T})$ is not isomorphic to a special transformation
	$T^{f}$ for a nontrivial roof function $f:X \to\mathbb{Z}^{+}$. Then, in Subsection \ref{subsec:On-even-Kakutani}, we will use Proposition \ref{prop:NonIsom}
	to show that $\Psi(\mathcal{T})$ and $\Psi(\mathcal{T})^{-1}$
	can only be Kakutani equivalent if they are actually evenly Kakutani equivalent.
	In Subsection \ref{subsec:non-even} we finish the proof by showing
	that for a tree $\mathcal{T}\in\mathcal{T}\kern-.5mm rees$ without an infinite
	branch $\Psi(\mathcal{T})$ and $\Psi(\mathcal{T})^{-1}$ are not
	evenly Kakutani equivalent. 
	
	\subsection{\label{subsec:SpecialTrans}$T^{f}$ is not isomorphic to $T$ for
		nontrivial $f$}
	
	We recall that for an ergodic invertible measure-preserving transformation
	$T:(X,\mu)\to(X,\mu)$ and an integrable function $f:X \to\mathbb{Z}^{+}$
	one defines the \emph{special transformation} (or sometimes called
	\emph{integral transformation}) $T^{f}$ on the set $X^{f}=\left\{ (x,j):x\in X,1\leq j\leq f(x)\right\} $
	by 
	\[
	T^{f}(x,j)=\begin{cases}
		(x,j+1) & \text{if }j<f(x),\\
		(Tx,1) & \text{if }j=f(x).
	\end{cases}
	\]
	The map $T^f$ preserves the probability measure $\mu^f$ that is the restriction to $X^{f}$ of $(\mu\times\lambda)/(\int f\, \mathrm{d}\mu)$ 
	on $X\times\mathbb{Z}^+$, where $\lambda$ is the
	counting measure on $\mathbb{Z}^+.$ Moreover, $T^{f}$ is ergodic with
	respect to $\mu^{f}$ because $T$ is ergodic.
	
	We denote $\left\{ (x,j)\in X^{f}:j>1\right\} $ by $H$. For
	a measurable partition $P=\left\{ C_{1},\dots,C_{N}\right\} $ of
	$X$ we define $P^{f}=\left\{ C_{1},\dots,C_{N},H\right\} $ as the
	corresponding measurable partition of $X^{f}$. Accordingly, if $\left(T^{f}\right)^{i}(x)\in H$, we let $h$ be the letter in the $i$th position of the $(T^f,P^f)$-name of $x$.
	Hence
	the $\left(T^{f},P^{f}\right)$-name is obtained from the $(T,P)$-name
	by inserting  $h$'s. 
	
	In this section we consider a transformation of the form $T=\Psi(\mathcal{T})$
	for any $\mathcal{T}\in\mathcal{T}\kern-.5mm rees$. We recall that $\Psi(\mathcal{T})$
	was built with a uniquely readable and strongly uniform construction
	sequence $\left\{ \mathcal{W}_{n}(\mathcal{T}):\sigma_{n}\in\mathcal{T}\right\} $.
	To simplify notation we enumerate it as $\left\{ \mathcal{W}_{n}\right\} _{n\in\mathbb{N}}$,
	that is, $(n+1)$-words are built by concatenating $n$-words.
	
	Following the approach in \cite[chapter 11]{ORW} (see also \cite[section 4]{Ben}
	for another exposition) we show that for any nontrivial integrable
	function $f:[0,1]\to\mathbb{Z}^{+}$ the special transformation $T^{f}$
	is not isomorphic to $T$. In other words, $T^{f}\cong T$ implies that $f$
	is the constant function equal to $1$. 
	\begin{prop}
		\label{prop:NonIsom}Suppose that $\mathcal{T}\in\mathcal{T}\kern-.5mm rees$, $T=\Psi(\mathcal{T})$,
		and $f:X\to\mathbb{Z}^{+}$ is integrable. If $T^{f}\cong T$, then
		$\int f \, \mathrm{d}\mu=1$. Equivalently, if $A$ is a measurable subset
		of $X$ and $T_{A}\cong T,$ then $\mu(A)=1.$
	\end{prop}
	
	If $T$ and $T^{f}$ were isomorphic by an isomorphism $\Phi:X\to X^{f}$,
	then for every $\tau>0$ there would be $K(\tau)\in\mathbb{N}$ and
	a partition $P(\tau)\subset\vee_{i=-K(\tau)}^{K(\tau)}T^{i}P$ such
	that $|P(\tau)|=|P^{f}|$ and $d(P(\tau),\Phi^{-1}(P^{f}))<\tau$. For the remainder of Subsection \ref{subsec:SpecialTrans}, we assume that such an isomorphism $\Phi$ exists and that $\int f\, \mathrm{d}\mu>1.$ We obtain a contradiction, which proves Proposition \ref{prop:NonIsom}.
	\begin{rem}\label{rem:CodePart}
		Such a partition $P(\tau)$ defines an equivalence relation on $\vee_{i=-K(\tau)}^{K(\tau)}T^{i}P$
		whereby two elements $Q_{1},Q_{2}$ of this partition are equivalent if they are contained in the same element of $P(\tau)$.  For $Q\in \vee_{i=-K(\tau)}^{K(\tau)}T^{i}P$, denote the equivalence class containing $Q$ by $\overline{Q}$. From
		the $(T,P)$-trajectory $\dots a_{-1}a_{0}a_{1}\dots$ of a point
		$x\in X$ one can derive a $(T,\vee_{i=-K(\tau)}^{K(\tau)}T^{i}P)$
		trajectory 
		\[
		\dots\left(a_{-K(\tau)-1}\dots a_{-1}\dots a_{K(\tau)-1}\right)\left(a_{-K(\tau)}\dots a_{0}\dots a_{K(\tau)}\right)\left(a_{-K(\tau)+1}\dots a_{1}\dots a_{K(\tau)+1}\right)\dots
		\]
		with $\left(a_{-K(\tau)+i}\dots a_{i}\dots a_{K(\tau)+i}\right)$
		such that 
		\[
		x\in T^{-(-K(\tau)+i)}(C_{a_{-K(\tau)+i}})\cap\dots\cap T^{-i}(C_{a_{i}})\cap\dots\cap T^{-(K(\tau)+i)}(C_{a_{K(\tau)+i}}).
		\]
		By taking equivalence classes we obtain the $(T,P(\tau))$-trajectory
		\[
		\dots\overline{\left(a_{-K(\tau)-1}\dots a_{-1}\dots a_{K(\tau)-1}\right)}\,\overline{\left(a_{-K(\tau)}\dots a_{0}\dots a_{K(\tau)}\right)}\,\overline{\left(a_{-K(\tau)+1}\dots a_{1}\dots a_{K(\tau)+1}\right)}\dots
		\]
		with $\overline{\left(a_{-K(\tau)+i}\dots a_{i}\dots a_{K(\tau)+i}\right)}$
		such that 
		\[
		x\in\overline{T^{-(-K(\tau)+i)}(C_{a_{-K(\tau)+i}})\cap\dots\cap T^{-i}(C_{a_{i}})\cap\dots\cap T^{-(K(\tau)+i)}(C_{a_{K(\tau)+i}})}.
		\]
		In this way the partition $P(\tau)$ can be identified with a \emph{finite
			code} and we use this term synonymously. 
	\end{rem}
	\begin{defn}
		\label{def:code}A \emph{code} of length $2K+1$ is a function $\phi:\Sigma^{\mathbb{Z}\cap[-K,K]}\to\tilde{\Sigma}$, where $\Sigma,\tilde{\Sigma}$ are alphabets. We denote the length of the code by $|\phi|$.
		Given a code $\phi$ of length $2K+1$ the \emph{stationary code}
		$\bar{\phi}$ on $\Sigma^{\mathbb{Z}}$ determined by $\phi$ is defined
		as $\bar{\phi}(s)$ for any $s\in\Sigma^{\mathbb{Z}}$, where for
		any $l\in\mathbb{Z}$
		\[
		\bar{\phi}(s)(l)=\phi\left(s\upharpoonright[l-K,l+K]\right).
		\]
	\end{defn}
	
	Note that $\overline{\phi}$ is not necessarily injective.
	
	In this setting $\bar{\phi}(s)\upharpoonright[-N,N]$ denotes the
	string of symbols
	\[
	\bar{\phi}(s)(-N)\,\bar{\phi}(s)(-N+1)\,\dots\,\bar{\phi}(s)(N-1)\,\bar{\phi}(s)(N)
	\]
	in $\tilde{\Sigma}^{2N+1}$.
	
	\begin{rem}
		\label{rem:code2}As we will see in Lemma~\ref{lem:ConsistentCode}, if $T$ and $S$ are evenly equivalent and have finite generators $P$ and $Q$, respectively, then there is a code from the $(T,P)$ process to the $(S,Q)$ process such that the coded names match arbitrarily well in $\overline{f}$ with the $(S,Q)$ names. In Section~\ref{subsec:non-even} we will apply Lemma~\ref{lem:ConsistentCode} to our symbolic systems $\mathbb{K}$ and $\mathbb{K}^{-1}$ with generators $P=Q=\left\{ \left\langle \sigma\right\rangle :\sigma\in\Sigma\right\} $. In this setting, a $P(\tau)\subset\lor_{i=-K}^{K}T^{i}P$
			as in Remark \ref{rem:CodePart} corresponds to a code $\phi$ as described
			in Definition \ref{def:code}. 
	\end{rem}
	
	\begin{rem}
		\label{rem:end-effects}There is an ambiguity in applying a code $\phi$
		of length $2K+1$ to blocks of the form $s\upharpoonright[a,b]$:
		It does not make sense to apply it to the initial string $s\upharpoonright[a,a+K-1]$
		and end string $s\upharpoonright[b-K+1,b]$. However, if $b-a$ is
		large with respect to the code length $2K+1$, we can fill in $\phi(s)\upharpoonright[a,a+K-1]$
		and $\phi(s)\upharpoonright[b-K+1,b]$ arbitrarily and it makes a
		negligible difference to the $\overline{d}$ or $\overline{f}$ distance.
		In particular, if $\overline{f}\left(\phi(s),\bar{s}\right)<\varepsilon$,
		then for all large enough $a,b\in\mathbb{N}$, we have 
		\[
		\overline{f}\left(\phi(s)\upharpoonright[-a,b],\bar{s}\upharpoonright[-a,b]\right)<2\varepsilon
		\]
		independently of how we fill in the ambiguous portion. The general phenomenon
		of ambiguity or disagreement at the beginning and end of large strings
		is referred to by the phrase \emph{end effects}.
	\end{rem}

	We show that for  a nontrivial function $f$, there exists $\tau>0$ sufficiently small so that there cannot be a finite code from $(T,P)$ to $(T^{f},P^f)$ that approximates $P^f$ within distance $\tau$. First we investigate
	some properties of $(T^{f},P^{f})$-names. We divide such names into $n$-blocks
	as follows: A $(T^{f},P^{f})$-$n$-block consists of all the symbols
	from the leftmost symbol of a $(T,P)$-$n$-block up to (but not including)
	the leftmost symbol in the next $(T,P)$-$n$-block. In the following
	we denote the $(T^{f},P^{f})$-$n$-block corresponding to a $(T,P)$-$n$-block, $w_{n}$, by $\tilde{w}_{n}$. In general, the $(T^{f},P^{f})$-$n$-blocks
	have varying lengths. The following fact from \cite[Proposition 2.1, p.84]{ORW}
	based on the ergodic theorem shows that a large proportion of them have almost the same length.
	\begin{fact}
		\label{fact:sameLength}Let $T:(X,\mu)\to(X,\mu)$ be an ergodic transformation
		and $f:X\to\mathbb{Z}^{+}$ be integrable. For a measurable subset $F$ of $X$, let $R_F, R_F^f:F\to \mathbb{Z}^+$ denote, respectively, the return time under $T$ to $F$, and the return time under $T^f$ to $F$ considered as a subset of the base of the tower $X^f$. Then for every $\varepsilon>0$
		there is $N(\varepsilon)\in\mathbb{N}$ such that for every measurable subset $F$ of $X$ with $R_{F}\geq N(\varepsilon)$ there is a measurable subset $F'$ of $F$
		such that 
		\begin{align*}
			& \int_{F'}R_{F}^{f}\, \mathrm{d}\mu\geq\int f\,\mathrm{d}\mu-\varepsilon\\
			\text{ and }\; & R_{F}(x)\cdot\left(\int f\,\mathrm{d}\mu-\varepsilon\right)\leq R_{F}^{f}(x)\leq R_{F}(x)\cdot\left(\int f\,\mathrm{d}\mu+\varepsilon\right) & \text{ for all }x\in F'.
		\end{align*}
		
	\end{fact}
	
	In particular, the last line of Fact \ref{fact:sameLength}
	says that if we think of $F$ as the initial points in $n$-blocks,
	then for $n$ large enough, $n$-blocks are expanded nearly uniformly.
	The inequality $\int_{F'}R_{F}^{f}\,\mathrm{d}\mu\geq\int f \,\mathrm{d}\mu-\varepsilon$ implies that
	the density of the set of indices in a typical $(T^{f},P^{f})$-trajectory in blocks not
	expanded within the above bounds is at most $\varepsilon/(\int f\, \mathrm{d}\mu)\le\varepsilon$.
	
	Now for almost every point $x$ we consider its $(T,P)$-name $a(T,P,x)$,
	its $(T,P(\tau))$-name $a(T,P(\tau),x)$, and its $(T,\Phi^{-1}(P^{f}))$-name
	$a\left(T,\Phi^{-1}(P^{f}),x\right)$, which differs from its $(T,P(\tau))$-name on a set of indices
	with density at most $\tau$. Note that the $\left(T,\Phi^{-1}(P^{f})\right)$-name
	of $x$ is identified with the $\left(T^{f},P^{f}\right)$-name
	of $\Phi(x)$. Thus, for any sequence of integers (considered as indices
	labeling positions in names) we can consider the three names
	defined on it.
	
	We recall from our constructions in Section \ref{sec:Construction}
	that the $(n+1)$-words in $\mathcal{W}_{n+1}$ are built as substitution
	instances of the different $1$-equivalence classes in $\mathcal{W}_{n+1}/\mathcal{Q}_{1}^{n+1}$.
	Considered on the level of the $1$-equivalence relation these elements
	in $\mathcal{W}_{n+1}/\mathcal{Q}_{1}^{n+1}$ are concatenations of
	different Feldman patterns with building blocks in $\mathcal{W}_{n}/\mathcal{Q}_{1}^{n}$ by our enumeration of the construction sequence as $\{\mathcal{W}_n\}_{n\in \mathbb{N}}$. In the following we denote these Feldman patterns of $1$-equivalence
	classes of $n$-blocks in the $P$-name of $x$ by $F_{n,i}$, labeled so that $F_{n,0}$ contains the position $0$. Similarly the strings in the $\Phi^{-1}(P^{f})$-name
	are denoted by $\tilde{F}_{n,i}$, where $\tilde{F}_{n,0}$ contains the position $0$. Thus 
	\begin{align*}
		a(T,P,x)= & \dots F_{n,-2}F_{n,-1}F_{n,0}F_{n,1}F_{n,2}\dots\\
		a(T^{f},P^{f},\Phi(x))=a\left(T,\Phi^{-1}(P^{f}),x\right)= & \dots\tilde{F}_{n,-2}\tilde{F}_{n,-1}\tilde{F}_{n,0}\tilde{F}_{n,1}\tilde{F}_{n,2}\dots
	\end{align*}
	We decompose the set $\mathbb{Z}$ of indices of trajectories into
	sets $I_{i}$, where $I_{i}$ is
	the intersection of the set of indices of a $1$-pattern $F_{n,l}$
	in the $(T,P)$-trajectory and of the set of indices of a $1$-pattern
	$\tilde{F}_{n,\tilde{l}}$ in the $(T^{f},P^{f})$-trajectory (see
	Figure \ref{fig:fig1}). For brevity, we denote such a set of indices by $I_{i}=F_{n,l}\cap\tilde{F}_{n,\tilde{l}}$. Moreover, $I_{i}(T,P)$ denotes the $(T,P)$-trajectory
	of $x$ on the indices in $I_{i}$.
	
	If we are in the second case $s(n+1)=s(n)+1$ of the construction
	(see Subsection \ref{subsec:Case-2}), then we delete those parts
	of the $(n+1)$-words coming from the collection $\mathcal{W}^{\dagger\dagger}$. We also delete the corresponding strings in the $(T^f,P^f)$-trajectory.
	Hereby, we neglect a set of indices with density at most $2\epsilon_{n+1}$ by (\ref{eq:LengthProportion}).
	After this, each element in $\mathcal{W}_{n+1}/\mathcal{Q}_{1}^{n+1}$
	is a concatenation of $U_{2}^{(n+1)}$ many different $(T_{2}^{(n+1)},2^{4e(n)-t_{1}},M_{2}^{(n+1)})$-Feldman
	patterns with building blocks out of $\mathcal{W}_{n}/\mathcal{Q}_{1}^{n}$,
	where $U_{2}^{(n+1)}$, $T_{2}^{(n+1)}$, and $M_{2}^{(n+1)}$ are
	the corresponding parameters chosen at our stage $n+1$ of the construction
	in Subsections \ref{subsec:Case-1} and \ref{subsec:Case-2}. Note
	that we had to remove the parts coming from $\mathcal{W}^{\dagger\dagger}$
	to have unique values of $T_{2}^{(n+1)}$ and to guarantee that all
	Feldman patterns have the same length, which we denote by $f_n$. All of this automatically holds
	in the case $s(n+1)=s(n)$ of the construction (see Subsection~\ref{subsec:Case-1}).

	\begin{figure}
		\centering
		\includegraphics[width=\textwidth]{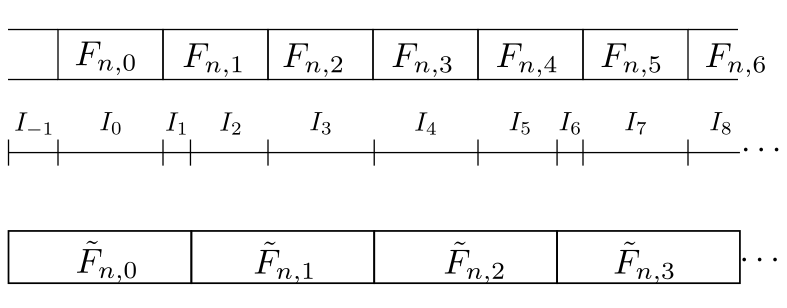}
		\caption{Feldman patterns $F_{n,i}$ of $n$-blocks in the $P$-name of $x$, their corresponding strings $\tilde{F}_{n,i}$ in the $\Phi^{-1}(P^{f})$-name of $x$, and the decomposition of the set $\mathbb{Z}$ of indices into $I_{i}=F_{n,l}\cap\tilde{F}_{n,\tilde{l}}$.}
		\label{fig:fig1}
	\end{figure}
	
	Recall the definition of  $\alpha_1$ from Proposition \ref{prop:conclusio}. We let $$0<\delta< \min \left(\frac{\int f \, \mathrm{d}\mu-1}{3\int f \, \mathrm{d}\mu}, \, \frac{\alpha_1}{160 \int f \, \mathrm{d}\mu}\right),$$
	and let $n$ be sufficiently large such that
	\begin{itemize}
		\item[(n1)]  $\frac{2}{2^{e(n)}}<\delta$, $\epsilon_{n+1}<\delta$, and $\frac{20}{R_n}< \alpha_1$, 
		\item[(n2)]  $2\left(\int f\,\mathrm{d}\mu\right)^2<\min\left(2^{e(n)},U_{2}^{(n+1)},\frac{\alpha_1\mathfrak{p}_{n+1}}{80}\right)$,
		\item[(n3)]  the code length $2K(\tau)+1$ is less than $h_{n}$,
		\item[(n4)]  $h_{n}$ is larger than the number $N(\delta^{3})$
		from Fact \ref{fact:sameLength} for $\delta^{3}$. 
	\end{itemize}
	In particular, $h_{n}>N(\delta^{3})$ implies that the set of indices
	lying in $\left(T,\Phi^{-1}(P^{f})\right)$-$n$-words $\tilde{w}_{n}$,
	that are too long (i.e. $|\tilde{w}_{n}|>\left(\int f\,\mathrm{d}\mu+\delta^{3}\right)h_{n}$)
	or too short (i.e. $|\tilde{w}_{n}|<\left(\int f \, \mathrm{d}\mu-\delta^{3}\right)h_{n}$),
	has density at most $\delta^{3}$. Then at most $\delta$ of the indices are lying in a $1$-pattern
	$\tilde{F}_{n,j}$ that contains a proportion of more than $\delta^{2}$ of indices lying in too long or too short
	$n$-words. Moreover, condition (n4) implies
	that all but a set of indices of density at most $\delta^{3}$ lie
	in $1$-patterns $\tilde{F}_{n,j}$ with $\left(\int f\,\mathrm{d}\mu-\delta^{3}\right)f_{n}\leq|\tilde{F}_{n,j}|\leq\left(\int f\, \mathrm{d}\mu+\delta^{3}\right)f_{n}$.
	Hence, at most $\delta+\delta^3<2\delta$ of indices belong to those ``bad''
	$1$-patterns $\tilde{F}_{n,j}$ and we neglect them until the end of the proof of Proposition~\ref{prop:NonIsom}.
	
	In the first step, we give a lower bound for the $\overline{d}$-distance
	between the strings on $I_{i}$ from the $(T,P(\tau))$-trajectory and the
	$\left(T,\Phi^{-1}(P^{f})\right)$-trajectory of $x$, if the $1$-pattern
	types are different from each other. In fact, in (\ref{eq:estimateCodeAppl1}) we even obtain a lower bound on the $\overline{f}$ distance. The main idea of the proof is as in \cite[Lemma 2.4, p.87]{ORW}, but we face the additional challenge that we do not
		know the precise words being coded, only their equivalence classes. Our proof uses the Coding Lemma \ref{lem:CodingLemma} and relies on the different
	periodicities of maximum repetitions of the same $1$-equivalence
	class of $n$-blocks. 
	\begin{lem}
		\label{lem:distDiff} Suppose for $I_{i}=F_{n,l}\cap\tilde{F}_{n,\tilde{l}}$ that $|I_{i}|\geq\delta f_{n}$ and that $F_{n,l}$ and $\tilde{F}_{n,\tilde{l}}$ have different $1$-pattern types. 
		Then we have 
		\[
		\overline{d}\left(I_{i}\left(T,\Phi^{-1}(P^{f})\right),I_{i}(T,P(\tau))\right)\geq \frac{\alpha_{1}}{20\cdot \int f\,\mathrm{d}\mu},
		\]
		where $\alpha_1$ is defined in Proposition \ref{prop:conclusio}.
	\end{lem}
	
	\begin{proof}
		By the properties of the ``good'' $1$-pattern $\tilde{F}_{n,\tilde{l}}$ as described above, $I_i(T,\Phi^{-1}(P^f))$ is $2\delta$-close in $\overline{f}$ to a string whose $\left(T,\Phi^{-1}(P^{f})\right)$-$n$-words have length $\floor*{ \int f\,\mathrm{d}\mu \cdot h_n } $. Until the end of the proof in inequality~\eqref{eq:estimateCodeAppl1} we identify $I_i(T,\Phi^{-1}(P^f))$ with this string.
		
		We denote the tuple of building blocks for $[F_{n,l}]_{1}$
		by $\left([A_{1}]_{1},\dots,[A_{2^{4e(n)-t_{1}}}]_{1}\right)$ with $[A_i]_1 \in \mathcal{W}_n/\mathcal{Q}^n_1$ and
		the one for $[\tilde{F}_{n,\tilde{l}}]_{1}$ by $([\tilde{A}_{1}]_{1},\dots,[\tilde{A}_{2^{4e(n)-t_{1}}}]_{1})$.
		Furthermore, let $r$ and $\tilde{r}$ be the $1$-pattern types of $F_{n,l}$ and
		$\tilde{F}_{n,\tilde{l}}$, respectively. 
		
		Recalling $\mathfrak{p}_{n+1}^{\prime}=\mathfrak{p}_{n+1}^{2}$
		in Case 1 ($s(n+1)=s(n)$) or $\mathfrak{p}_{n+1}^{\prime}=\mathfrak{p}_{n+1}^{2}-\mathfrak{p}_{n+1}$ in
		Case 2 ($s(n+1)=s(n)+1$) of the construction, we let 
		\begin{align*}
			L & = \mathfrak{p}_{n+1}^{\prime} \cdot h_n, \\
			N &= 2^{4e(n)-t_1}
		\end{align*}
		in the notation of the Coding Lemma \ref{lem:CodingLemma}. Moreover, we subdivide $\tilde{F}_{n,\tilde{l}}$ into substitution instances $\tilde{\Gamma}_{i,m}$ of a maximal repetition $[\tilde{A}_m]^{T_2N^{2\tilde{r}}}_1$. Then we further subdivide $\tilde{\Gamma}_{i,m}$ into 
		\[
		q \coloneqq \floor*{\frac{|\tilde{\Upsilon}_{i,m}|}{T_{2} \cdot N^{2\tilde{r}-1} \cdot h_n}} \cdot \frac{T_{2}}{\mathfrak{p}_{n+1}^{\prime}} \cdot N^{2\tilde{r}-1} = \floor*{ N \cdot \int f\,\mathrm{d}\mu } \cdot \frac{T_{2}}{\mathfrak{p}_{n+1}^{\prime}} \cdot N^{2\tilde{r}-1} 
		\] 
		many strings of $L$ consecutive symbols, where we recall that $T_{2}$ is a multiple of $\mathfrak{p}^{\prime}_{n+1}$. In the notation of the Coding Lemma each of these $L$-strings corresponds to a $B^{(i)}_{m,j}$ and altogether they form $\Gamma_{i,m}=B^{(i)}_{m1}B^{(i)}_{m2}\cdots B^{(i)}_{mq}$. Then we also have
		\[
		p= \floor*{\frac{|F_{n,l}|}{NqL}} = \floor*{\frac{N^{2(M_2-\tilde{r})+3}}{\floor*{N \cdot \int f\,\mathrm{d}\mu}}}.
		\]
		Note that
		\begin{equation} \label{eq:LengthApproxError}
			0 \leq \frac{ |\tilde{\Gamma}_{i,m}| - |\Gamma_{i,m}|}{|\tilde{\Gamma}_{i,m}|} \leq \frac{T_2 \cdot N^{2\tilde{r}-1} \cdot h_n}{T_{2} \cdot N^{2\tilde{r}} \cdot \int f\,\mathrm{d}\mu \cdot h_n} < \frac{1}{N}.
		\end{equation}
		
		Let $B$ and $B^{\prime}$ be any substrings of at least $\frac{L}{\mathfrak{p}_{n+1}}$ consecutive symbols in any $B^{(i)}_{m,j}$ and $B^{(i')}_{m',j'}$, respectively, with $m\neq m'$. We modify $B$ and $B^{\prime}$ by first completing any partial blocks $[\tilde{A}_m]_1$ and $[\tilde{A}_{m'}]_1$, respectively. This can be achieved by adding fewer than $2\cdot \floor*{ \int f\,\mathrm{d}\mu \cdot h_n }$ many symbols to each of $B$ and $B^{\prime}$. We denote the modified strings by $B_{\text{aug}}$ and $B^{\prime}_{\text{aug}}$. By Fact~\ref{fact:omit_symbols}, we have
			\begin{equation*}
				\overline{f}\left(B,B^{\prime}\right)> \overline{f}\left(B_{\text{aug}},B^{\prime}_{\text{aug}}\right)- \frac{2\cdot \floor*{ \int f\,\mathrm{d}\mu \cdot h_n }}{(\mathfrak{p}_{n+1}-1)h_n} >\overline{f}\left(B_{\text{aug}},B^{\prime}_{\text{aug}}\right)- \frac{4\cdot \int f\,\mathrm{d}\mu }{\mathfrak{p}_{n+1}}.
			\end{equation*}
			Combining Fact \ref{fact:AddSymbol}, Lemma \ref{lem:symbol by block replacement}, and Proposition \ref{prop:conclusio} yields 
			\begin{equation*}
				\overline{f}\left(B_{\text{aug}},B^{\prime}_{\text{aug}}\right) > \frac{\alpha_{1}-\frac{1}{R_n}}{\int f\,\mathrm{d}\mu}.
			\end{equation*}
			Altogether, we have 
			\begin{equation} \label{eq:AlphaAppl}
				\overline{f}\left(B,B^{\prime}\right)>\frac{\alpha_{1}-\frac{1}{R_n}}{\int f\,\mathrm{d}\mu} - \frac{4\cdot \int f\,\mathrm{d}\mu }{\mathfrak{p}_{n+1}}>\frac{9\alpha_{1}}{10\cdot \int f\,\mathrm{d}\mu}.
			\end{equation} 
			Here we used conditions (n1) and (n2) in the last estimate. Hence, the assumption of the Coding Lemma is satisfied. Now we examine both possible cases $r<\tilde{r}$ and $r>\tilde{r}$.
		
		In the case $r<\tilde{r}$ the number $T_2N^{2r}$ of maximum repetitions
		of $n$-blocks of the same $1$-class in  $F_{n,l}(T,P)$ is smaller than
		in $\tilde{F}_{n,\tilde{l}}(T,\Phi^{-1}(P^f))$. We subdivide $F_{n,l}(T,P)$ into substitution instances $\tilde{\Lambda}_{i,m}$ of 
		\[
		u_1 \coloneqq \frac{|\Upsilon_{i,m}|}{T_2N^{2r+1}h_n} = \frac{qL}{T_2N^{2r+1}h_n}= \floor*{ N \cdot \int f\,\mathrm{d}\mu } \cdot N^{2(\tilde{r}-r-1)}
		\]
		many consecutive complete cycles in $[F_{n,l}]_1$. Note that
		\begin{equation}\label{eq:LengthApproxError2}
			0 \leq \frac{|F_{n,l}|- pNu_1\cdot T_2N^{2r+1}h_n}{|F_{n,l}|} \leq \frac{pT_2N^{2r+2}h_n}{T_2 N^{2M_2+3}h_n} < \frac{1}{N}.
		\end{equation}
		Lemma \ref{lem:Occurrence-Substitutions} and Remark \ref{rem:Substitution-Dagger}
		yield that for each substitution instance $a_{x,y}$ of each $[A_{x}]_{1}$,
		$x=1,\dots,2^{4e(n)-t_{1}}$, the repetition $\left(a_{x,y}\right)^{\mathfrak{p}_{n+1}^{\prime}}$
		occurs the same number of times in each $\tilde{\Lambda}_{i,m}$. Since the code length $2K(\tau)+1<h_{n}$ and there are $\mathfrak{p}_{n+1}^{\prime}$
		repetitions, we obtain that the coded images of all occurrences of $\left(a_{x,y}\right)^{\mathfrak{p}_{n+1}^{\prime}}$ are $\frac{2}{\mathfrak{p}^{\prime}_{n+1}}$ close to each other in $\overline{d}$. Hence, by making an $\overline{f}$-approximation error of at most $\frac{2}{\mathfrak{p}^{\prime}_{n+1}}$ we can decompose $F_{n,l}(T,P(\tau))$ into permutations $\Lambda_{i,m}$ as in the Coding Lemma.
		
		In the second case $r>\tilde{r}$ we decompose $F_{n,l}(T,P)$ into substitution instances $\tilde{\Lambda}_{i,m}$ of repetitions of the form 
		\[
		[A_x]^{q\mathfrak{p}^{\prime}_{n+1}}_1,  \ x=1,\dots,2^{4e(n)-t_{1}}.
		\]	
		Within the maximum repetition $[A_x]^{T_2N^{2r}}_1$ in $F_{n,l}(T,P)$, there are 
		\[
		u_2 \coloneqq \floor*{ \frac{T_2N^{2r}}{Nq\mathfrak{p}^{\prime}_{n+1}}} = \floor*{ \frac{N^{2(r-\tilde{r})}}{\floor*{ N \cdot \int f\,\mathrm{d}\mu }} }
		\]
		many such $\tilde{\Lambda}_{i,1}\dots \tilde{\Lambda}_{i,N}$. We note that 
		\begin{equation}\label{eq:LengthApproxError3}
			0 \leq \frac{T_2N^{2r}h_n - u_2 Nq\mathfrak{p}^{\prime}_{n+1}h_n}{T_2N^{2r}h_n} \leq \frac{Nq\mathfrak{p}^{\prime}_{n+1}h_n}{T_2N^{2r}h_n} \leq \frac{\int f\,\mathrm{d}\mu}{N^{2(r-\tilde{r})-1}} \leq \frac{\int f\,\mathrm{d}\mu}{N}.
		\end{equation}
		Lemma \ref{lem:Occurrence-Substitutions} and Remark \ref{rem:Substitution-Dagger} yield again that for each substitution instance $a_{x,y}$ of $[A_{x}]_{s}$ the repetition $\left(a_{x,y}\right)^{\mathfrak{p}_{n+1}^{\prime}}$ occurs the same number of times in each $\tilde{\Lambda}_{i,m}$. Again, we can decompose $F_{n,l}(T,P(\tau))$ into permutations $\Lambda_{i,m}$ as in the Coding Lemma by making an $\overline{f}$-approximation error of at most $\frac{2}{\mathfrak{p}^{\prime}_{n+1}}$.
		
		In both cases, we let $A=\Lambda_{1,1}\dots \Lambda_{1,N}\Lambda_{2,1}\dots\Lambda_{2,N}\dots \Lambda_{p,1}\dots \Lambda_{p,N}$ and $B=\Gamma_{1,1}\dots \Gamma_{1,N}\Gamma_{2,1}\dots\Gamma_{2,N}\dots \Gamma_{p,1}\dots \Gamma_{p,N}$ be the strings as in the Coding Lemma. Then the Coding Lemma yields with the aid of (\ref{eq:AlphaAppl}) that
		\[
		\overline{f}(\overline{A},\overline{B})> \frac{9\alpha_{1}}{10\cdot \int f\,\mathrm{d}\mu} \cdot \left(\frac{1}{8}-\frac{2}{2^{e(n)-0.25t_1}}\right)-\frac{1}{\mathfrak{p}_{n+1}}-\frac{4}{N}
		\]
		on any substrings $\overline{A}$ and $\overline{B}$ of at least $pqL$ consecutive symbols in $A$ and $B$, respectively. Since $pqL \leq \frac{1}{N}f_n \leq \delta f_n$ by condition (n1), we can apply this estimate on strings $I_{i}\left(T,\Phi^{-1}(P^{f})\right)$ and $I_{i}(T,P(\tau))$ as in the statement of the lemma. Taking all the aforementioned approximation errors as in \eqref{eq:LengthApproxError}, \eqref{eq:LengthApproxError2}, and \eqref{eq:LengthApproxError3} into account, we conclude
		\begin{equation} \label{eq:estimateCodeAppl1}
			\begin{split}
				& \overline{f}\left(I_{i}\left(T,\Phi^{-1}(P^{f})\right),I_{i}(T,P(\tau))\right) \\
				> & \frac{\alpha_{1}}{10\cdot \int f\,\mathrm{d}\mu} - \frac{1}{\mathfrak{p}_{n+1}}- \frac{6}{N}-\frac{2\cdot \int f\,\mathrm{d}\mu}{N}  -\frac{2}{\mathfrak{p}^{\prime}_{n+1}} - 4\delta \\
				> & \frac{\alpha_{1}}{20\cdot \int f\,\mathrm{d}\mu},
			\end{split}
		\end{equation}
		where we used conditions (n1) and (n2) in the last step.
	\end{proof}
	
	In the second step, we show following \cite[Lemma 2.5, p.89]{ORW} that the $1$-pattern types do not agree
	with each other very often. The proof uses the fact that on average patterns
	in the $T^{f}$ trajectory are expanded by a factor close to $\int f\mathrm{d}\mu$
	and thus only about $\frac{1}{\int f\mathrm{d}\mu}$ of the patterns
	are of the same type.
	\begin{lem}
		\label{lem:densitySame}For almost every point $x$ the set of indices
		in $I_{i}=F_{n,l}\cap\tilde{F}_{n,\tilde{l}}$, for which the $1$-pattern
		types of $F_{n,l}$ and $\tilde{F}_{n,\tilde{l}}$ are the same or $|I_i|<\delta f_n$, has
		density at most 
		\[
		\frac{1}{\int f\, \mathrm{d}\mu}+4\delta.
		\]
	\end{lem}
	
	\begin{proof}
		We decompose $a(T,\Phi^{-1}(P^{f}),x)$ into its $\left(T,\Phi^{-1}(P^{f})\right)$-$(n+1)$-blocks
		\[
		a(T,\Phi^{-1}(P^{f}),x)=\dots\tilde{w}_{n+1,-2}\tilde{w}_{n+1,-1}\tilde{w}_{n+1,0}\tilde{w}_{n+1,1}\tilde{w}_{n+1,2}\dots.
		\]
		We consider its $j$-th word $\tilde{w}_{n+1,j}$ and denote the corresponding
		sets of indices by $\tilde{B}_{n+1,j}$ (that is, $\tilde{B}_{n+1,j}(T,\Phi^{-1}(P^{f}),x)=\tilde{w}_{n+1,j}$).
		Then we further subdivide $\tilde{w}_{n+1,j}$ into its $1$-patterns.
		We denote those by $\tilde{F}_{n,j,k}$ and their set of indices by
		$\tilde{B}_{n+1,j,k}$. By condition (n4) for almost every $x$ for
		all but a set of indices of density at most $\delta$ the occurrent
		word $\tilde{w}_{n+1,j}$ is such that the subset of its indices
		$\tilde{B}_{n+1,j}$ lying in a too long $1$-pattern (that is, $|\tilde{B}_{n+1,j,k}|>\left(\int f\,\mathrm{d}\mu+\delta^{3}\right)\cdot|f_{n}|)$
		or in a too short $1$-pattern ($|\tilde{B}_{n+1,j,k}|<\left(\int f\,\mathrm{d}\mu-\delta^{3}\right)\cdot|f_{n}|)$
		has density at most $\delta$. 
		
		Hence, neglecting a set of indices with density at most $2\delta$
		we can assume that $\left(\int f\,\mathrm{d}\mu-\delta^{3}\right)\cdot|f_{n}|<|\tilde{B}_{n+1,j,k}|<\left(\int f\,\mathrm{d}\mu+\delta^{3}\right)\cdot|f_{n}|$.
		Since $U_{2}^{(n+1)}>\int f\,\mathrm{d}\mu+\delta^{3}$ by our assumption
		(n2), there is at most one $1$-pattern $F_{n,l}$ in $\tilde{B}_{n+1,j,k}(T,P)$
		that has the same $1$-pattern type as $\tilde{F}_{n,j,k}$. Thus,
		in $\tilde{B}_{n+1,j,k}(T,P)$ the set of indices that lie in the
		same $1$-pattern type as $\tilde{F}_{n,j,k}$ has density at most
		\[
		\frac{|f_{n}|}{|\tilde{B}_{n+1,j,k}|}<\frac{1}{\int f\,\mathrm{d}\mu-\delta^{3}}<\frac{1}{\int f\,\mathrm{d}\mu}+\delta^{3}.
		\]
		Taking the neglected indices into account the lemma follows.
	\end{proof}
	Combining both results yields the proof of Proposition \ref{prop:NonIsom}.
	\begin{proof}[Proof of Proposition \ref{prop:NonIsom}]
		For almost every $x\in X$ we obtain from Lemma \ref{lem:distDiff}
		and Lemma \ref{lem:densitySame} that 
		\begin{align*}
			& \overline{d}\left(a\left(T,\Phi^{-1}(P^{f}),x\right),a(T,P(\tau),x\right)\\
			\geq & \frac{\alpha_{1}}{20\cdot \int f\,\mathrm{d}\mu}\cdot\left(1-\frac{1}{\int f\,\mathrm{d}\mu}-10\delta\right),
		\end{align*}
		where we also took into account the indices of density $\delta$ and
		$2\epsilon_{n+1}<2\delta$ that we neglected at the very beginning.
		Letting $\delta\to0$ and $n\to\infty$ we conclude
		\[
		\overline{d}\left(a\left(T,\Phi^{-1}(P^{f}),x\right),a(T,P(\tau),x\right)\geq\frac{\alpha_{1}}{20 \cdot \int f\, \mathrm{d}\mu}\cdot\frac{\int f\, \mathrm{d}\mu-1}{\int f\, \mathrm{d}\mu}.
		\]
		For nontrivial $f$ and $\tau<\frac{\alpha_{1}}{20 \cdot \int f \, \mathrm{d}\mu}\cdot\frac{\int f \, \mathrm{d}\mu-1}{\int f \, \mathrm{d}\mu}$
		this contradicts the assumption that $P(\tau)$ is a finite code with
		$d\left(\Phi^{-1}(P^{f}),P(\tau)\right)<\tau$.
	\end{proof}
	
	\subsection{\label{subsec:On-even-Kakutani}On even Kakutani equivalence}
	
	In this subsection we state several properties of even Kakutani equivalence.
	We start by recalling its definition.
	\begin{defn}
		\label{def:EvenEquiv}Two transformations $S:(X,\mu)\to(X,\mu)$ and
		$T:(Y,\nu)\to(Y,\nu)$ are said to be \emph{evenly equivalent} if
		there are subsets $A\subseteq X$ and $B\subseteq Y$ of equal measure
		$\mu(A)=\nu(B)>0$ such that $S_{A}$ and $T_{B}$ are isomorphic
		to each other. Equivalently, there is $A\subseteq X$ and $f:A\to\mathbb{Z}^{+}$
		with $\mu(A)\cdot\int_{A}f\mathrm{d}\mu_{A}=1$ such that $\left(S_{A}\right)^{f}\cong T$. There is also another equivalent definition in \cite{FieldsteelRudolph}, but we do not make use of it.
	\end{defn}

	The following lemma from \cite{ORW} is related to the second characterization in Definition \ref{def:EvenEquiv}.
	\begin{lem}
		\label{lem:ORW13}Suppose $S:(X,\mu)\to(X,\mu)$ is an ergodic measure-preserving
		automorphism.
		\begin{enumerate}
			\item [(a)]If $f,g:X\to\mathbb{Z}^{+}$ such that $\int g\ d\mu<\int f\ d\mu$,
			then there exists $h:X^{g}\to\mathbb{Z}^{+}$ such that $\int_{X^{g}}h\ d\mu^{g}=(\int f\ d\mu)/(\int g\ d\mu)$
			and $(S^{g})^{h}\cong S^{f}.$ 
			\item [(b)]If $f:A\to\mathbb{Z}^{+}$ and $\mu(A)\int_{A}f\ d\mu_{A}<1,$
			then there exists $A'\subset X$ with $\mu(A')=\mu(A)\int_{A}f\ d\mu_{A}$
			such that $(S_{A})^{f}\cong S_{A'}.$ 
			\item [(c)]If $f:A\to\mathbb{Z}^{+}$ and $\mu(A)\int_{A}f\ d\mu_{A}>1,$
			then there exists $h:X\to\mathbb{Z}^{+}$ with $\int h\ d\mu=\mu(A)\int_{A}f\ d\mu_{A}$
			such that $(S_{A})^{f}\cong S^{h}.$
		\end{enumerate}
	\end{lem}
	
	\begin{proof}
		Part (a) follows from \cite[Lemma 1.3, p. 3]{ORW} and its proof.
		Parts (b) and (c) are stated on p. 91 of \cite{ORW}, and they follow
		easily from (a) and the observation that $S\cong(S_{A})^{r_{A}}$
		and $\int_{A}r_{A}\ d\mu_{A}=1/\mu(A).$
	\end{proof}
	\begin{lem}
		\label{lem:Induce}If $U,V:(X,\mu)\to(X,\mu)$ are ergodic measure-preserving
		automorphisms and $U_{A}\cong V_{B}$ with $0<\mu(A)<\mu(B)\le1,$
		then there exists $C$ with $0<\mu(C)<1$ such that $U_{C}\cong V.$
	\end{lem}
	
	\begin{proof}
		Since $\mu(A)<\mu(B),$ $\int_{A}r_{A}\ d\mu_{A}>\int_{B}r_{B}\ d\mu_{B}.$
		Then $U\cong(U_{A})^{r_{A}}\cong(V_{B})^{\tilde{r}}$ where $\int_{B}\tilde{r}\ d\mu_{B}=\int_{A}r_{A}\ d\mu_{A}>\int_{B}r_{B}\ d\mu_{B}=1/\mu(B).$
		Therefore by Lemma \ref{lem:ORW13}(c), $(V_{B})^{\tilde{r}}\cong V^{h}$
		for some $h:X\to\mathbb{Z}^{+}$ with $\int h\ d\mu>1.$ Therefore
		$U\cong V^{h}$ and consequently there exists $C$ with $\mu(C)=1/(\int h\ d\mu)$
		such that $U_{C}\cong V.$
	\end{proof}
	For the following proposition we use the same proof as on p. 90 of
	\cite{ORW} for the analogous result with $T$ replaced by Feldman's
	example $J.$ 
	\begin{prop}
		\label{prop:only-even}Let $\mathcal{T}\in\mathcal{T}\kern-.5mm rees$ and $T=F(\mathcal{T}).$
		If $(T_{A})^{-1}\cong T_{B}$ , then $\mu(A)=\mu(B).$ That is, if
		$T$ and $T^{-1}$ are Kakutani equivalent, then they are evenly equivalent. 
	\end{prop}
	
	\begin{proof}
		Suppose $(T_{A})^{-1}\cong T_{B}$ where $\mu(A)\ne\mu(B).$ 
		\begin{casenv}
			\item Suppose $\mu(A)<\mu(B).$ Then by Lemma \ref{lem:Induce}, there exists
			$C$ with $0<\mu(C)<1$ such that $(T_{C})^{-1}=(T^{-1})_{C}\cong T.$
			This implies $T_{C}\cong T^{-1}$ and there exists $\widetilde{C}\subset C$
			with $\mu(\widetilde{C})/\mu(C)=\mu(C)$ such that 
			\[
			T_{\widetilde{C}}=(T_{C})_{\widetilde{C}}\cong(T^{-1})_{C}\cong T.
			\]
			Since $0<\mu(\tilde{C})<1,$ this contradicts Proposition \ref{prop:NonIsom}.
			\item Suppose $\mu(A)>\mu(B).$ Then $(T_{B})^{-1}\cong T_{A},$ which is
			the same as Case 1 with $A$ and $B$ switched. Again we have a contradiction
			to Proposition \ref{prop:NonIsom}. 
		\end{casenv}
	\end{proof}

	The first inequality in Lemma \ref{lem:ConsistentCode} below is Proposition 3.2 in \cite[p. 92]{ORW}. The second inequality will serve as an approximate ``consistency
	property'' for a sequence of finite codes between two evenly equivalent
	transformations. It follows from the proof of Lemma 1.3 in \cite[p. 3]{ORW}. 
	
	If we are given an integrable function $f:X \to\mathbb{Z}^{+}$, a partition $P$ of the space $X$ and $A\subset X$
	with $\mu(A)>0$, we use the notation $P_{A}^{f}$ to denote the partition
	$\{A^{f}\setminus A\}\cup\{c\cap A:c\in P\}$ of $A^{f}.$ Moreover,
	if we consider names in the $(S_{A}^{f},P_{A}^{f})$ and $(S_{A}^{g},P_{A}^{g})$
	processes, we regard $A^{f}\setminus A$ and $A^{g}\setminus A$ as
	having different symbols associated to them. In other words, we have
	a match between the $x$-$S_{A}^{f}$-$P_{A}^{f}$ name at time $i$
	and the $x'$-$S_{A}^{g}$-$P_{A}^{g}$ name at time $j$ if and only
	if $S^{i}x,S^{j}x'\in c\cap A$ for some $c\in P.$ 
	\begin{lem}\label{lem:ConsistentCode} 
		Let $S:(X,\mu)\to(X,\mu)$ and $T:(Y,\nu)\to(Y,\nu)$
		be ergodic measure-preserving automorphisms, $Q$ a finite generating
		partition for $T,$ and $P$ a finite partition of $X.$ Suppose that
		$S$ and $T$ are evenly equivalent. Let $(\varepsilon_{\ell})_{\ell=1}^{\infty}$
		be a decreasing sequence of positive numbers such that $0<\varepsilon_{\ell}<1$
		and $\lim_{\ell\to\infty}\varepsilon_{\ell}=0.$ Then there exist
		increasing sequences $(m_{\ell})_{\ell=1}^{\infty}$ and $(K_{\ell})_{\ell=1}^{\infty}$
		of positive integers, sets $F_{\ell}\subset Y$ with $\nu(F_{\ell})>1-\varepsilon_{\ell},$
		and codes $\phi_{\ell}$ of length $2K_{\ell}+1$ from $(T,Q)$-names
		to $(S,P)$-names such that for $y\in F_{\ell}$ we have the following
		properties: There exists $x=x_{\ell}(y)\in X$ such that for all $m\ge m_{\ell},$
		the $(S$-$P$-name of $x)_{0}^{m-1}$ occurs with positive frequency
		in the $(S,P)$ process,
		\[
		\overline{f}(\phi_{\ell}((T\text{-}Q\text{-name of }y)_{0}^{m-1}),(S\text{-}P\text{-name of }x)_{0}^{m-1})<\varepsilon_{\ell},\text{\ for\ }m\ge m_{\ell},
		\]
		and 
		\[
		\overline{f}(\phi_{\ell}((T\text{-}Q\text{-name of }y)_{0}^{m-1})\text{ },\phi_{\ell+1}((T\text{-}Q\text{-name of }y)_{0}^{m-1}))<\varepsilon_{\ell},\text{\ for\ \ensuremath{m\ge m_{\ell+1}}}.
		\]
		
	\end{lem}
	
	\begin{proof}
		Since $S$ and $T$ are evenly equivalent, there are $\widetilde{A}_{1}\subseteq X$
		and $\widetilde{f}_{1}:\widetilde{A}_{1}\to\mathbb{Z}^{+}$ with $\mu(\widetilde{A}_{1})\cdot\int_{\widetilde{A}_{1}}\widetilde{f}_{1}\mathrm{d}\mu_{\widetilde{A}_{1}}=1$
		such that $\left(S_{\widetilde{A}_{1}}\right)^{\widetilde{f}_{1}}\cong T$.
		We choose a set $A_{1}$ with $X\supset A_{1}\supset\widetilde{A}_{1}$
		and $1 > \mu(A_{1})>1-\frac{\varepsilon_{1}}{10}$. Since $\left(\left(S_{A_{1}}\right)_{\widetilde{A}_{1}}\right)^{\widetilde{f}_{1}}\cong T$,
		$S_{A_{1}}$ and $T$ are equivalent. We note that 
		\[
		\mu_{A_{1}}(\widetilde{A}_{1})\cdot\int_{\widetilde{A}_{1}}\widetilde{f}_{1}d\mu_{\widetilde{A}_{1}}=\frac{\mu(\widetilde{A}_{1})}{\mu(A_{1})}\cdot\int_{\widetilde{A}_{1}}\widetilde{f}_{1}d\mu_{\widetilde{A}_{1}}=\frac{1}{\mu(A_{1})}>1.
		\]
		Then Lemma \ref{lem:ORW13} (c) yields the existence of $f_{1}:A_{1}\to\mathbb{Z}^{+}$
		with $\left(S_{A_{1}}\right)^{f_{1}}\cong T$ and $\int_{A_{1}}f_{1}d\mu_{A_{1}}=\frac{1}{\mu(A_{1})}$.
		Let $\theta_{1}:Y\to(A_{1})^{f_{1}}$ be an isomorphism between $T$
		and $\left(S_{A_{1}}\right)^{f_{1}}.$
		
		We let $(A_{\ell})_{\ell=1}^{\infty}$ be a sequence of sets, strictly increasing in measure, such
		that $A_{1}\subset A_{2}\subset A_{3}\subset\cdots\subset X$ with
		$\mu(A_{\ell})>1-\frac{\varepsilon_{\ell}}{10}.$ 
		We will construct, inductively, a sequence of functions $f_{\ell}:A_{\ell}\to \mathbb{Z}$,$\ \ell = 2,3,\dots$ such that $(S_{A_{\ell}})^{f_{\ell}}\cong T$ and $\int_{A_{\ell}}f_{\ell}\ d\mu_{A_{\ell}}=\frac{1}{\mu(A_{\ell})}$. The existence of such functions already follows by the same argument given above in the case of $\ell=1$, and this would be enough to obtain the first $\overline{f}$ inequality in our lemma, but for the second $\overline{f}$ inequality, we need to carefully construct $f_{\ell}$ and a particular isomorphism $\theta_{\ell +1}:(A_{\ell})^{f_{\ell}}\to (A_{\ell +1})^{f_{\ell +1}}$ from $(S_{A_{\ell}})^{f_{\ell}}$ to $(S_{A_{\ell +1}})^{f_{\ell +1}}$ for $\ell=2,3,\dots.$ (A priori, the codes $\phi_{\ell}$ and $\phi_{\ell +1}$ used in the first $\overline{f}$ inequality could be totally unrelated.) To construct $f_{\ell +1}$ and $\theta_{\ell+1}$ we will utilize the method of the proof of Lemma 1.3 in \cite{ORW}.
		
		Let $\ell \ge 1.$ Suppose we are given functions $f_j:A_j\to \mathbb{Z}^+$ with $\int_{A_j}f_j\ d\mu_{A_j}=1/\mu({A_j})$ for $j=1,\dots,\ell$, and $f_1$ is as above. If $\ell >1$ and $j=1,\dots,\ell -1$, also assume that we have constructed isomorphisms $\theta_{j+1}:(A_j)^{f_j}\to (A_{j+1})^{f_{j+1}}$ from $(S_{A_j})^{f_j}$ to   $(S_{A_{j+1}})^{f_{j+1}}$. We will construct a function $f_{\ell+1}:A_{\ell +1}\to\mathbb{Z}^+$ with $\int_{A_{\ell +1}}f_{\ell +1}\ d\mu_{A_{\ell+1}}=1/\mu(A_{\ell+1})$, an isomorphism $\theta_{\ell+1}:(A_{\ell})^{f_{\ell}}\to (A_{\ell+1 })^{f_{\ell+1}}$ from $(S_{A_{\ell}})^{f_{\ell}}$ to $(S_{A_{\ell+1}})^{f_{\ell+1}},$ a positive integer $u_{\ell}$, and a set $D_{\ell}\subset A_{\ell}$ with $\mu(D_{\ell})>1-\frac{\varepsilon_{\ell}}{5}$ such that for $x\in D_{\ell}$ and $m\ge u_{\ell}$, we have 
		\begin{equation}
			\overline{f}(((S_{A_{\ell}})^{f_{\ell}}\text{-}P_{A_{\ell}}^{f_{\ell}}\text{-name of\ }x)_{0}^{m-1},((S_{A_{\ell+1}})^{f_{\ell+1}}\text{-}P_{A_{\ell+1}}^{f_{\ell+1}}\text{-name of\ }\theta_{\ell+1}(x))_{0}^{m-1})<\frac{\varepsilon_{\ell}}{5}.\label{eq:Estimate_2}
		\end{equation}
		
		We can write $S_{A_{\ell+1}}=(S_{A_{\ell}})^{g},$ where $g:A_{\ell}\to\mathbb{Z}^{+}$
			is the return time to $A_{\ell}$ under $S_{A_{\ell+1}}.$ Then $g$ satisfies $\int_{A_{\ell}}g\ d\mu_{A_{\ell}}=(\int_{A_{\ell}}g\ d\mu)/(\mu(A_{\ell}))=\mu(A_{\ell+1})/\mu(A_{\ell}).$
			Since $\int_{A_{\ell}}f_{\ell}\ d\mu_{A_{\ell}}=1/\mu(A_{\ell})>\mu(A_{\ell+1})/\mu(A_{\ell}),$ we have $\int_{A_{\ell}}f_{\ell}\ d\mu_{A_{\ell}}>\int_{A_{\ell}}g\ d\mu_{A_{\ell}}$.
			By the ergodic theorem, for almost every $x\in A_{\ell}$ there exists $N(x)$
			such that 
			\[
			\sum_{i=0}^{n-1}f_{\ell}((S_{A_{\ell}})^{i}(x))>\sum_{i=0}^{n-1}g((S_{A_{\ell}})^{i}(x)),\text{\ for all }n\ge N(x).
			\]
			Thus, we can choose a positive measure set $B\subset A_{\ell}$ and $N\in\mathbb{Z}^{+}$
			such that $N(x)\le N$ for all $x\in B,$ and the return time $r_{B,S_{A_{\ell}}}$
			to $B$ under $S_{A_{\ell}}$ satisfies $r_{B,S_{A_{\ell}}}(x)\ge N$ for all $x\in B.$
			It follows that 
			\[
			\sum_{i=0}^{r_{B,S_{A_{\ell}}}(x)-1}f_{\ell}((S_{A_{\ell}})^{i}(x))>\sum_{i=0}^{r_{B,S_{A_{\ell}}}(x)-1}g((S_{A_{\ell}})^{i}(x)),\ \text{for }x\in B.
			\]
			Now $(S_{A_{\ell}})^{f_{\ell}}\cong(S_B^{r_{B,S_{A_{\ell}}}})^{f_{\ell}}\cong(S_B)^{h_1}$,
			where 
			$$h_1(x)=\sum_{i=0}^{r_{B,S_{A_{\ell}}}(x)-1} f_{\ell}((S_{A_{\ell}})^i(x))\ \text{for }x\in B.$$ Similarly,
			$(S_{A_{\ell}})^g\cong (S_B^{r_B,S_{A_{\ell}}})^g\cong (S_B)^{h_2},$ where
			$$h_2(x)=\sum_{i=0}^{r_{B,S_{A_{\ell}}}(x)-1} g((S_{A_{\ell}})^i(x)),\ \text{for }x\in B.$$ Since
			$h_2(x)<h_1(x)$ for $x\in B$, we can build a tower $((S_B)^{h_2})^{f_{\ell+1}}\cong (S_B)^{h_1}$
			by letting $f_{\ell+1}$ be equal to $h_1(x)-h_2(x)+1$ on the top level above $x$ in the
			$B^{h_2}$ tower, and letting $f_{\ell+1}$ be equal to $1$ elsewhere.
			
			By the induction hypothesis, $\int_{A_{\ell}} f_{\ell}\ d\mu_{A_{\ell}}=1/\mu(A_{\ell})$, or, equivalently,
			$\int_{A_{\ell}}  f_{\ell}\ d\mu = 1$. Thus $1=\mu((A_{\ell})^{f_{\ell}})=\mu(B^{h_1})=\mu((B^{h_2})^{f_{\ell+1}}).$
			On the other hand, $\mu(B^{h_2})=\mu((A_{\ell})^g)=\mu(A_{\ell+1}).$ Hence $1=\mu((B^{h_2})^{f_{\ell+1}})=
			\int_B {h_2}\ d\mu +\int_{B^{h_2}}(f_{\ell+1}-1)\ d\mu = \mu(A_{\ell+1}) + \int_{A_{\ell+1}} (f_{\ell+1}-1)\ d\mu =\int_{A_{\ell+1}} f_{\ell+1}\ d\mu .$
			Therefore $\int_{A_{\ell+1}} f_{\ell+1}\ d\mu =1$ and $\int_{A_{\ell+1}} f_{\ell+1}\ d\mu_{A_{\ell+1}}=1/\mu(A_{\ell+1})$.
			
			To construct the isomorphism $\theta_{\ell+1}:A_{\ell}^{f_{\ell}}\to A_{\ell+1}^{f_{\ell+1}}$ from 
			$(S_{A_{\ell}})^{f_{\ell}}$ to $(S_{A_{\ell+1}})^{f_{\ell+1}}$, we first identify, in the natural way, the tower $B^{r_B,S_{A_{\ell}}}$ with the set $A_{\ell}$, the tower $A_{\ell}^{f_{\ell}}$ with the tower $(B^{r_B,S_{A_{\ell}}})^{f_{\ell}}$, and the tower $A_{\ell+1}^{f_{\ell+1}}$ with the tower $((B^{r_B,S_{A_{\ell}}})^g)^{f_{\ell+1}}.$
			With these identifications, it suffices to define $\theta_{\ell+1}:(B^{r_B,S_{A_{\ell}}})^{f_{\ell}}\to ((B^{r_B,S_{A_{\ell}}})^g)^{f_{\ell+1}}$ so that it is an isomorphism from $((S_B)^{r_B,S_{A_{\ell}}})^{f_{\ell}}$ to $((S_B)^{r_B,S_{A_{\ell}}})^g)^{f_{\ell+1}}$. For $x\in B$, it takes $h_1(x)$ iterations for the first return to $B$ under both $(S_B^{r_B,S_{A_{\ell}}})^{f_{\ell}}$ and $((S_B^{r_B,S_{A_{\ell}}})^g)^{f_{\ell+1}}$, and the point of $B$ at which the first return occurs is the same for both of these transformations. Let $\mathcal{O}_1(x)$ and $\mathcal{O}_2(x)$ denote the orbit segments of length $h_1(x)$ starting at $x\in B$ under, respectively, $(S_B^{r_B,S_{A_{\ell}}})^{f_{\ell}}$ and $((S_B^{r_B,S_{A_{\ell}}})^g)^{f_{\ell+1}}$. Note that $\mathcal{O}_1(x)$ and $\mathcal{O}_2(x)$ 
			contain the same points in $A_{\ell}$, in the same order, but possibly at different times. We let $\theta_{\ell+1}$ be the identity on $B$ and let $\theta_{\ell+1}$ map the orbit segment $\mathcal{O}_1(x)$ to the orbit segment $\mathcal{O}_2(x)$ for each $x\in B$,
			with points of the orbits remaining in consecutive order.
			
			At the points in $A_{\ell},$ the $P_{A_{\ell}}^{f_{\ell}}$-name agrees with the $P_{A_{\ell+1}}^{f_{\ell+1}}$-name.
			Thus, the number of $\overline{f}$ matches between the $P_{A_{\ell}}^{f_{\ell}}$-names
			of the points in $\mathcal{O}_1(x)$ and the $P_{A_{\ell+1}}^{f_{\ell+1}}$-names
			of the points in $\mathcal{O}_2(x)$ is at least the number
			of terms in $A_{\ell}$. Therefore, to estimate
			the $\overline{f}$ distance between the $P_{A_{\ell}}^{f_{\ell}}$-name of $x$
			and the $P_{A_{\ell+1}}^{f_{\ell+1}}$-name of $\theta_2(x)$ from time $0$ to $m-1$
			we divide the orbits into segments that start with a point in $B$
			until right before the next return to $B.$ The number of $\overline{f}$
			matches that we obtain is at least the total number of times that points
			in the segments are in $A_{\ell}.$ If $a>0$ is sufficiently large,
		then the waiting time until the first entry to the set $B$ is less
		than $a$ with high probability, and the last entry to the
		set $B$ occurs after time $m-a$ with high probability. Thus there
		exist $u_{\ell}>>a,$ and a set $D_{\ell}\subset A_{\ell}$ with $\mu(D_{\ell})>1-\frac{\varepsilon_{\ell}}{5}$
		such that for $x\in D_{\ell}$ and $m\ge u_{\ell},$ we have
		(\ref{eq:Estimate_2}).
		
		The functions $f_{\ell}:A_{\ell}\to\mathbb{Z}^{+}$
		satisfy $\left(S_{A_{\ell}}\right)^{f_{\ell}}\cong T$ and $\int_{A_{\ell}}f_{\ell}\ d\mu_{A_{\ell}}=\frac{1}{\mu(A_{\ell})},$ for $\ell=1,2,\dots.$
		We also have $(S_{A_{\ell}})^{r_{\ell}}\cong S,$ where $r_{\ell}=r_{A_{\ell},S}$
		is the return time to $A_{\ell}$ under $S$ and $\int_{A_{\ell}}r_{\ell}\ d\mu_{A_{\ell}}=\frac{1}{\mu(A_{\ell})}.$ 
		For $x\in A_{\ell},$ the $P_{A_{\ell}}^{f_{\ell}}$-name of $x$
		at the $i$th time the orbit under $(S_{A_{\ell}})^{f_{\ell}}$ returns
		to $A_{\ell}$ agrees with the $P_{A_{\ell}}^{r_{\ell}}$-name of
		$x$ at the $i$th time the orbit under $(S_{A_{\ell}})^{r_{\ell}}$
		returns to $A_{\ell}.$ There exist $k_{\ell}$ and a set $C_{\ell}\subset A_{\ell}$
		with $\mu(C_{\ell})>\mu(A_{\ell})-\frac{\varepsilon_{\ell}}{10}>1-\frac{\varepsilon_{\ell}}{5}$
		such that for $x\in C_{\ell}$ and $m\ge k_{\ell},$ both the $(S_{A_{\ell}})^{f_{\ell}}$
		and the $(S_{A_{\ell}})^{r_{\ell}}$ orbits of $x$ are in $A_{\ell}$
		for more than $1-\frac{\varepsilon_{\ell}}{5}$ of the times $0,1,\dots,m-1.$
		Thus, for $x\in C_{\ell}$ and $m\ge k_{\ell},$ we have
		\begin{equation}
			\overline{f}(((S_{A_{\ell}})^{f_{\ell}}\text{-}P_{A_{\ell}}^{f_{\ell}}\text{-name of }x)_{0}^{m-1},(S\text{\text{-}}P\text{-name of }x)_{0}^{m-1})<\frac{\varepsilon_{\ell}}{5}.\label{eq:Estimate_1}
		\end{equation}
		We may assume that any such $S$-$P$-name of length $m$ occurs with positive probability in the $(S,P)$ process. 
		
		Since $Q$ is a generator for $T$ there exists $K_{\ell}$ such that
		\[
		(\theta_{\ell}\circ\cdots\circ\theta_{1})^{-1}(P_{A_{\ell}}^{f_{\ell}})\subset^{\varepsilon_{\ell}/10}\lor_{-K_{\ell}}^{K_{\ell}}T^{j}Q,
		\]
		and this defines a code $\phi_{\ell}$ from $(T,Q)$-names to $(\theta_{\ell}\circ\cdots\circ\theta_1)^{-1}(P_{A_{\ell}}^{f_{\ell}})$-names.
		If $m_{\ell}$ is sufficiently large, in particular, $m_{\ell}\ge\max(k_{\ell},u_{\ell}),$
		then there exists a set $E_{\ell}\subset(\theta_{\ell}\circ\cdots\circ\theta_{1})^{-1}(A_{\ell})$
		such that for $y\in E_{\ell}$ the coded name $\phi_{\ell}(y)$ at
		time $j$ agrees with the actual $(\theta_{\ell}\circ\cdots\circ\theta_{1})^{-1}(P_{A_{\ell}}^{f_{\ell}})$-name
		of $y,$ or equivalently, the $(P_{A_{\ell}}^{f_{\ell}})$-name of
		$(\theta_{\ell}\circ\cdots\circ\theta_{1})(y),$ at time $j$, with
		frequency greater than $1-\frac{\varepsilon_{\ell}}{5}$ among the
		times $j=0,1,\dots,m-1,$ for $m\ge m_{\ell.}$ That is, 
		\begin{equation}
			\overline{d}((\phi_{\ell}(y))_{0}^{m-1},((S_{A_{\ell}})^{f_{\ell}}\text{-}P_{A_{\ell}}^{f_{\ell}}\text{-name of\ }(\theta_{\ell}\circ\cdots\circ\theta_{1})(y))_{0}^{m-1})<\frac{\varepsilon_{\ell}}{5}.\label{eq:Estimate_3}
		\end{equation}

		Finally, let $F_{\ell}=E_{\ell}\cap E_{\ell+1}\cap(\theta_{\ell}\circ\cdots\circ\theta_{1})^{-1}(C_{\ell}\cap D_{\ell})\cap(\theta_{\ell+1}\circ\cdots\circ\theta_{1})^{-1}(D_{\ell+1})$, and let $x_{\ell}(y)=(\theta_{\ell}\circ\cdots\circ\theta_1)(y).$
		The first $\overline{f}$ inequality in the lemma follows from (\ref{eq:Estimate_1})
		and (\ref{eq:Estimate_3}), while the second $\overline{f}$ inequality
		follows from (\ref{eq:Estimate_3}) for
		$\ell$ and $\ell+1,$ and (\ref{eq:Estimate_2}) for $\ell$. 
	\end{proof}

	\begin{rem}
		\label{rem:code} We call such a sequence of codes as in Lemma \ref{lem:ConsistentCode} a sequence of $(\varepsilon_{\ell},K_{\ell})$-finite
		$\overline{f}$ codes from $(T,Q)$ to $(S,P)$.
	\end{rem}
	
	\begin{cor}\label{cor:CodeOnWords}
		Let $0<\gamma<1/2$ and assume the set-up of Lemma \ref{lem:ConsistentCode}, where $T=\Phi(\mathcal{T})$
		for some $\mathcal{T\in\mathcal{T}}rees$ and $Q$ is the partition
		of $Y=\Sigma^{\mathbb{Z}}$ according to the symbol of $\Sigma$ located
		at position $0.$ We also let $X=\Sigma^{\mathbb{Z}}$ and $P=Q,$
		but the measures $\mu$ and $\nu$ may be different. Assume that the
		$\varepsilon_{\ell}$ in Lemma \ref{lem:ConsistentCode} are chosen so that $\Sigma_{i=\ell+1}^{\infty}\varepsilon_{i}<\varepsilon_{\ell}.$
		Then there exists $N(\gamma)\in\mathbb{Z}^{+}$ such that for $N\ge N(\gamma)$
		and $k\in\mathbb{Z}^{+}$ and all sufficiently large $n$ (how large
		depends on $N$,$\text{\ensuremath{\gamma,} and }k)$ there is a collection
		$\mathcal{W}_{n}'\subset\mathcal{W}_{n}$ that includes at least $1-\gamma$
		of the words in $\mathcal{W}_{n}$ and the following condition holds:
		For $w\in\mathcal{W}_{n}'$ there exists $z=z(w)\in X$ such that
		the $(S\text{-}Q\text{-name of }z)_{0}^{h_{n}-1}$ occurs with positive
		frequency in the $(S,Q)$ process and we have 
		\[
		\overline{f}((\phi_{N}(w))_{0}^{h_{n}-1},(S\text{-}Q\text{-name of }z)_{0}^{h_{n}-1})<\gamma,
		\]
		and
		
		\[
		\overline{f}((\phi_{N}(w))_{0}^{h_{n}-1},(\phi_{N+k}(w))_{0}^{h_{n}-1})<\gamma.
		\]
	\end{cor}
	
	\begin{proof}
		Choose $N(\gamma)$ sufficiently large that $\varepsilon_{N(\gamma)-1}<\gamma^{2}/2.$
		Let $N\ge N(\gamma).$ Then choose $n$ sufficiently large so $m_{N+k}<h_{n}/2$.
		Let $F=F_{N}\cap\cdots\cap F_{N+k-1}$ where $F_{N},\dots,F_{N+k-1}$
		are as in Lemma \ref{lem:ConsistentCode}. Then $\nu(F)>1-\varepsilon_{N}-\cdots-\varepsilon_{N+k-1}>1-\varepsilon_{N-1}>1-\gamma^{2}/2.$
		Thus there is a collection $\mathcal{W}_{n}'\subset\mathcal{W}_{n}$
		that includes at least $1-\gamma$ of the words in $\mathcal{W}_{n}$
		such that for $w\in\mathcal{W}_{n}',$ for more than $1-\gamma/2$
		of the indices $i=1,2,\dots,h_{n},$ the $i$th level of column in
		the tower representation of the word $w$ intersects $F.$ For $w=a_{1}a_{2}\cdots a_{h_{n}}\in\mathcal{W}_{n}',$
		there exists $i_{0}=i_{0}(w)$ with $1\le i_{0}\le(\gamma/2)h_{n}$
		such that for some $y=y(w)\in F$ we have $a_{i_{0}}\cdots a_{h_{n}}=(T$-$Q$-name
		of $y)_{0}^{h_{n}-i_{0}-1}$. Let $x=x(y)$ be as in Lemma \ref{lem:ConsistentCode}. Let
		$z=z(y)\in X$ be such that the $(S$-$Q$-name of $z)_{0}^{h_{n}-1}$
		occurs with positive frequency in the $(S,Q)$ process and the $(S$-$Q$-name
		of $z)_{i_{0}}^{h_{n}-1}$ is the same as the $(S$-$Q$-name of $x)_{0}^{h_{n}-i_{0}-1}.$
		Since 
		\[
		\overline{f}(\phi_{N}(a_{i_{0}}\cdots a_{h_{n}}),(S\text{-}Q\text{-name of }x)_{0}^{h_{n}-i_{0}-1})<\varepsilon_{N}<\gamma^{2}
		\]
		and $i_{0}\le(\gamma/2)h_{n},$ it follows that 
		\[
		\overline{f}(\phi_{N}(w),(S\text{-}Q\text{-name of }z)_{0}^{h_{n}-1})<(\gamma/2)+\gamma^{2}<\gamma.
		\]
		By repeated application of the second $\overline{f}$ inequality
		in Lemma \ref{lem:ConsistentCode}, we obtain 
		\[
		\overline{f}(\phi_{N}(a_{i_{0}}\cdots a_{h_{n}}),\phi_{N+k}(a_{i_{0}}\cdots a_{h_{n}}))<\varepsilon_{N}+\cdots\varepsilon_{N+k-1}<\varepsilon_{N-1}<\gamma^{2}/2.
		\]
		Therefore $\overline{f}(\phi_{N}(w),\phi_{N+k}(w))<(\gamma/2)+\gamma^{2}<\gamma.$
	\end{proof}

	\subsection{\label{subsec:non-even}Proof of part (2) in Proposition \ref{prop:criterion} }
	We have seen in Proposition \ref{prop:only-even} that if $\Psi(\mathcal{T})$
	and $\Psi(\mathcal{T})^{-1}$ are Kakutani equivalent, then it can
	only be by an even equivalence. Thus, to prove part (2) in Proposition
	\ref{prop:criterion} it remains to show that under our assumption
	that the tree $\mathcal{T}\in\mathcal{T}\kern-.5mm rees$ does not have an infinite
	branch, $\mathbb{K}:=\Psi(\mathcal{T})$ is not evenly equivalent to
	$\mathbb{K}^{-1}=\Psi(\mathcal{T})^{-1}$.

	As in Section \ref{subsec:SpecialTrans} we enumerate the uniquely
	readable and strongly uniform construction sequence of $\mathbb{K}=\Psi(\mathcal{T})$
	as $\left\{ \mathcal{W}_{n}\right\} _{n\in\mathbb{N}}$, that is,  $n$-words
	are built by concatenating $(n-1)$-words. We also remind the reader of
	the phrase $s$-pattern type from Remark \ref{rem:s-pattern-type}
	as well as Remark \ref{rem:Substitution-Dagger} and that we can decompose an $n$-word $w\in \mathcal{W}_n$ into its $s$-Feldman patterns (if we are in Case 2 of the construction we ignore the strings coming from $\mathcal{W}^{\dagger\dagger}$): $w=P_{n-1,1}\dots P_{n-1,U_{s+1}}$. In the subsequent
	lemma we prove that even under finite coding substantial strings
	of Feldman patterns of different type cannot be matched well. Its proof is partly based on techniques developed in \cite[chapter 12]{ORW} (see also \cite[Proposition 4.12]{Ben}) but new difficulties arise by our iterative substitution of different patterns.
	\begin{lem}
		\label{lem:BadCoding0}Let $s\in\mathbb{N}$ and $\phi$ be a finite
		code from $\mathbb{K}$ to $rev(\mathbb{K})$. Then for sufficiently large $n\in\mathbb{N}$
		and any pair of $s$-Feldman patterns $P_{n-1,\ell}$ in $\mathbb{K}$ and $\overline{P}_{n-1,\overline{\ell}}$ in $rev(\mathbb{K})$ 
		of different pattern type we have 
		\[
		\overline{f}\left(\phi(E),\overline{E}\right)>\frac{\alpha_{s}}{10},
		\]
		for any substrings $E$, $\overline{E}$ of at least $\frac{|P_{n-1,\ell}|}{2^{2e(n-1)}}=\frac{|\overline{P}_{n-1,\overline{\ell}}|}{2^{2e(n-1)}}$
		consecutive symbols in $P_{n-1,\ell}$ and $\overline{P}_{n-1,\overline{\ell}}$, respectively.
	\end{lem}
	
	\begin{proof}
		We recall the definition of $M(s)$ from Definition \ref{def:M-and-s} and choose $n>M(s)+1$ sufficiently large such that $h_{n-1}>2^{4e(n-1)}\cdot|\phi|$ and $R_{n-1}>\frac{110}{\alpha_s}$, where the number $R_{n-1}$ from (\ref{eq:Rassum2}) quantifies substantial substrings of $(n-1)$-words in Proposition~\ref{prop:conclusio}. In particular, end effects are negligible and we can assume $\phi(E)$ and $E$ to have the same length. 
		
		We denote the tuple of building blocks for $[P_{n-1,\ell}]_{s}$ by $\left([A_{1}]_{s},\dots,[A_{2^{4e(n-1)-t_{s}}}]_{s}\right)$
		and the tuple of building blocks for $[\overline{P}_{n-1,\overline{\ell}}]_{s}$ by $\left([\overline{A}_{1}]_{s},\dots,[\overline{A}_{2^{4e(n-1)-t_{s}}}]_{s}\right)$.
		Moreover, let $r$ and $\overline{r}$ denote the types of $s$-Feldman pattern of $P_{n-1,\ell}$ and $\overline{P}_{n-1,\overline{\ell}}$, respectively.
		In the notation of the Coding Lemma \ref{lem:CodingLemma} we let 
		\begin{align*}
			L & = h_{n-1}, \\
			N &=2^{4e(n-1)-t_s}, \\
			p &= N^{2\cdot (M_{s+1}+1-\overline{r})}, \\
			q &= T_{s+1} \cdot 2^{(4e(n-1)-t_{s})\cdot2\overline{r}} = T_{s+1} \cdot N^{2\overline{r}}.
		\end{align*}
		We emphasize that $T_{s+1}$ and thus $q$ are multiples of $\mathfrak{p}^{\prime}_n$, where $\mathfrak{p}_{n}^{\prime}=\mathfrak{p}_{n}^{2}$
		in Case~1 ($s(n)=s(n-1)$) or $\mathfrak{p}_{n}^{\prime}=\mathfrak{p}_{n}^{2}-\mathfrak{p}_{n}$ in
		Case 2 ($s(n)=s(n-1)+1$) of the construction. Furthermore, each $\Gamma_{i,j}$ in the Coding Lemma \ref{lem:CodingLemma} corresponds to a substitution instance of a maximal repetition $[\overline{A}_j]^q_s$ in $\overline{P}_{n-1,\overline{\ell}}$. Since we have $\overline{f}\left(W,W^{\prime}\right)>\alpha_{s}$ on any substrings $W$ and $W^{\prime}$ of at least $\frac{h_{n-1}}{R_{n-1}}$ consecutive symbols in any words $w,w^{\prime}\in\mathcal{W}_{n-1}$
		with $[w]_{s}\neq[w^{\prime}]_{s}$ by Proposition \ref{prop:conclusio}, the assumption of the Coding Lemma is fulfilled. Now we distinguish between the two possible cases. 
		
		In the first instance, let $r<\overline{r}$, that is, the number $T_{s+1}2^{(4e(n-1)-t_{s})\cdot2r}=T_{s+1}N^{2r}$
		of consecutive repetitions of the same $s$-block in $P_{n-1,\ell}$ is smaller than in $\overline{P}_{n-1,\overline{\ell}}$. We subdivide $P_{n-1,\ell}$ into substitution instances $\tilde{\Lambda}_{i,j}$ of $N^{2(\overline{r}-r)-1}$ consecutive complete cycles in $[P_{n-1,\ell}]_s$. Note that $|\tilde{\Lambda}_{i,j}|=|\Gamma_{i,j}|$. Lemma \ref{lem:Occurrence-Substitutions}
		and Remark \ref{rem:Substitution-Dagger} yield that for each substitution instance $a_{x,y}$ of each $[A_{x}]_{s}$, $x=1,\dots,2^{4e(n-1)-t_{s}}$,
		the repetition $\left(a_{x,y}\right)^{\mathfrak{p}_{n}^{\prime}}$ occurs the same number of times in each $\tilde{\Lambda}_{i,j}$. Since $h_{n-1}>2^{4e(n-1)}\cdot|\phi|$, we also see that the images under code $\phi$ of all occurrences of $\left(a_{x,y}\right)^{\mathfrak{p}_{n}^{\prime}}$ are $\frac{2}{\mathfrak{p}^{\prime}_n}$ close to each other in $\overline{d}$. Hence, by making a $\overline{d}$-approximation error of at most $\frac{2}{\mathfrak{p}^{\prime}_n}$ we can decompose $\phi(P_{n-1,\ell})$ into permutations $\Lambda_{i,j}$ as in the Coding Lemma.

		In case of $r> \overline{r}$ we decompose $P_{n-1,\ell}$ into substitution instances $\tilde{\Lambda}_{i,j}$ of repetitions of the form $[A_x]^q_s$, $x=1,\dots,2^{4e(n-1)-t_{s}}$. Lemma \ref{lem:Occurrence-Substitutions} and Remark \ref{rem:Substitution-Dagger} yield again that for each substitution instance $a_{x,y}$ of $[A_{x}]_{s}$ the repetition $\left(a_{x,y}\right)^{\mathfrak{p}_{n}^{\prime}}$ occurs the same number of times in each $\tilde{\Lambda}_{i,j}$. As above, we can decompose $\phi(P_{n-1,\ell})$ into permutations $\Lambda_{i,j}$ as in the Coding Lemma by making a $\overline{d}$-approximation error of at most $\frac{2}{\mathfrak{p}^{\prime}_n}$.
		
		In both cases, the Coding Lemma gives 
		\begin{align*}
			\overline{f}\left(\phi(E),\overline{E}\right)& >\alpha_{s} \cdot \left(\frac{1}{8}-\frac{2}{2^{e(n-1)-0.25t_{s}}}\right)-\frac{1}{R_{n-1}}-\frac{4}{N}-\frac{2}{\mathfrak{p}^{\prime}_n} \\
			& > \frac{\alpha_s-\frac{11}{R_{n-1}}}{9} > \frac{\alpha_s}{10}
		\end{align*}
		for any substrings $E$, $\overline{E}$ of at least $pqL=\frac{|P_{n-1,\ell}|}{N}$ consecutive symbols in $P_{n-1,\ell}$ and $\overline{P}_{n-1,\overline{\ell}}$, respectively. Here, we used conditions (\ref{eq:Pn}) and (\ref{eq:kn}) as well as our assumption $R_{n-1}>\frac{110}{\alpha_s}$. Since $\frac{|P_{n-1,\ell}|}{N}< \frac{|P_{n-1,\ell}|}{2^{2e(n-1)}}$, the conclusion holds for the strings from the statement of this lemma.
	\end{proof}
	
	\begin{lem}
		\label{lem:BadCoding}Let $s\in\mathbb{N}$ and $\phi$ be a finite
		code from $\mathbb{K}$ to $rev(\mathbb{K})$. Then for sufficiently large $n\in\mathbb{N}$
		and any pair $w\in\mathcal{W}_{n}$, $\overline{w}\in rev(\mathcal{W}_{n})$
		of different $s$-pattern type we have 
		\[
		\overline{f}\left(\phi(A),\overline{A}\right)>\frac{\alpha_{s}}{16}
		\]
		for any substrings $A$, $\overline{A}$ of at least $\frac{h_{n}}{2^{e(n-1)}}$
		consecutive symbols in $w$ and $\overline{w}$, respectively.
	\end{lem}
	
	\begin{proof}
		As before, we choose $n>M(s)+1$ sufficiently large such that $h_{n-1}>2^{4e(n-1)}\cdot|\phi|$ and
		end effects are negligible, that is, we can assume $\phi(A)$ and $A$ to
		have the same length. 
		
		If $\overline{f}\left(\phi(A),\overline{A}\right)>\frac{\alpha_{s}}{4}$,
		we do not have to show anything. Otherwise, we have $\overline{f}\left(\phi(A),\overline{A}\right)\leq\frac{1}{32}$
		(recall that $\alpha_{s}<\frac{1}{8}$) and let $\pi$ be a best possible
		$\overline{f}$ matching between $\phi(A)$ and $\overline{A}$. Since
		$\pi$ is a map on the indices, we can also view it as a map between
		$A$ and $\overline{A}$. Since $n>M(s)+1$ and $A$ as well as $\overline{A}$
		has length at least $\frac{h_{n}}{2^{e(n-1)}}$, Remark \ref{rem:s-pattern-type}
		yields that there are at least $2^{6e(n-1)}$ complete $s$-Feldman
		patterns (out of building blocks in $\mathcal{W}_{n-1}/\mathcal{Q}_{s}^{n-1}$)
		in $A$ and $\overline{A}$. Moreover, all of them have the same length
		(if we are in Case 2 of the construction we ignore the strings coming
		from $\mathcal{W}^{\dagger\dagger}$, which might increase the $\overline{f}$
		distance by at most $\frac{2}{\mathfrak{p}_{n}}$ according to (\ref{eq:LengthProportion})
		and Fact \ref{fact:omit_symbols}). We complete partial $s$-patterns
		at the beginning and end of $A$ and $\overline{A}$, which might increase
		the $\overline{f}$ distance by at most $\frac{2}{2^{6e(n-1)}}$ according
		to Fact \ref{fact:omit_symbols} as well. The augmented strings obtained
		in this way are denoted by $A_{aug}$ and $\overline{A}_{aug}$, respectively,
		and we decompose them into the $s$-Feldman patterns:
		\begin{align*}
			A_{aug}=P_{n-1,1}\dots P_{n-1,r} &  &  & \overline{A}_{aug}=\overline{P}_{n-1,1}\dots\overline{P}_{n-1,\overline{r}}.
		\end{align*}
		Since $w$ and $\overline{w}$ are of different $s$-pattern type,
		the occurring Feldman patterns are either disjoint or traversed in
		opposite directions. 
		
		In the next step, we write $A_{aug}$ and $\overline{A}_{aug}$ as
		\begin{align*}
			A_{aug}=E_{1}\dots E_{v} &  &  & \overline{A}_{aug}=\overline{E}_{1}\dots\overline{E}_{v},
		\end{align*}
		where $E_{i}$ and $\overline{E}_{i}$ are maximal substrings such
		that $E_{i}\subseteq P_{n-1,j(i)}$ and $\overline{E}_{i}\subseteq\overline{P}_{n-1,\overline{j}(i)}$
		correspond to each other under $\pi$ (see Figure \ref{fig:fig2} for an illustration of this decomposition). In particular, we have $v\geq2^{6e(n-1)}$.
		
		\begin{figure}
			\centering
			\includegraphics[width=\textwidth]{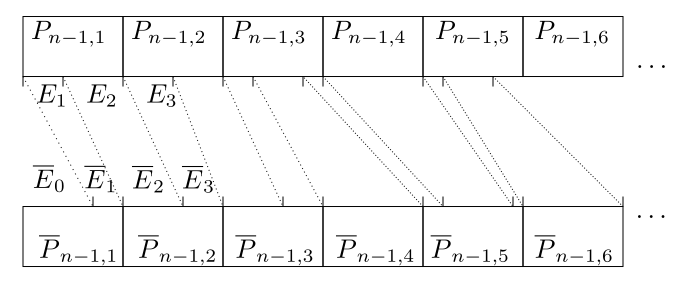}
			\caption{Illustration of the decompositions into $E_{i}\subseteq P_{n-1,j(i)}$ and $\overline{E}_{i}\subseteq\overline{P}_{n-1,\overline{j}(i)}$
				corresponding to each other under $\pi$.}
			\label{fig:fig2}
		\end{figure}
		
		We note that for at most one pair of $E_{i}$ and $\overline{E}_{i}$
		the type of Feldman pattern coincides. We delete this possible pair.
		This might increase the $\overline{f}$ distance by at most $\frac{2}{2^{6e(n-1)}}$
		according to Fact~\ref{fact:omit_symbols} because we have $v\geq2^{6e(n-1)}$.
		Thus, in the following we can assume that all pairs $E_{i}$ and $\overline{E}_{i}$
		lie in Feldman patterns of different type.
		
		If $\overline{f}\left(\phi\left(E_{i}\right),\overline{E}_{i}\right)>\frac{\alpha_{s}}{4}$,
		we stop the investigation of these strings.  In the rest of the proof we investigate the case $\overline{f}\left(\phi\left(E_{i}\right),\overline{E}_{i}\right)\leq \frac{\alpha_{s}}{4}$. Then we have 
		\begin{equation}
			1-\frac{1}{2^{4}}<\frac{1-\frac{1}{2^{5}}}{1+\frac{1}{2^{5}}}\leq\frac{|E_{i}|}{|\overline{E}_{i}|}\leq\frac{1+\frac{1}{2^{5}}}{1-\frac{1}{2^{5}}}<1+\frac{1}{2^{3}}\label{eq:Glength}
		\end{equation}
		by Fact \ref{fact:string_length}. In each $P_{n-1,l}$ there are
		at most two segments $E_{i}$ with $|E_{i}|<\frac{|P_{n-1,l}|}{2^{2e(n-1)}}$
		because a string $\overline{E}\subset \overline{A}_{aug}$ with $\overline{f}(\phi(P_{n-1,l}),\overline{E})\leq \frac{\alpha_s}{4}$  has to lie within
		at most three consecutive patterns in $\overline{A}_{aug}$. For their
		corresponding segments we also get $|\overline{E}_{i}|<\frac{|P_{n-1,l}|}{2^{2e(n-1)-1}}$
		by the estimate in (\ref{eq:Glength}). By the same reasons, there
		are at most two segments $\overline{E}_{i}$ with $|\overline{E}_{i}|<\frac{|P_{n-1,l}|}{2^{2e(n-1)}}$
		in each $\overline{P}_{n-1,l}$ and we can also bound the lengths
		of their corresponding strings. Altogether, the density of symbols
		in those situations is at most $\frac{1}{2^{2e(n-1)-4}}$ and we ignore
		them in the following consideration.
		
		Hence, we consider the case that both $|E_{i}|$ and $|\overline{E}_{i}|$
		are at least $\frac{|P_{n-1,l}|}{2^{2e(n-1)}}$. This is large enough
		such that we can apply Lemma \ref{lem:BadCoding0} and conclude the statement.
	\end{proof}
	As a consequence, we deduce that a well-approximating finite code
	can be identified with an element of the group action, where we recall the notation $\eta_g$ from Remark \ref{rem:etag}. We also recall that $G^n_s \subseteq G_s$ and for every $g \in G_s$ there is $n$ sufficiently large such that $g \in G^n_s$.
	\begin{lem}
		\label{lem:groupelement}Let $s\in\mathbb{N}$, $0<\delta<\frac{1}{4}$,
		$0<\varepsilon<\frac{\alpha_{s}}{200}\delta$, and $\phi$ be a finite
		code from $\mathbb{K}$ to $rev(\mathbb{K})$. Then for $n$ sufficiently
		large and $w\in\mathcal{W}_{n}$ we have that for any string $\overline{A}$
		in $rev(\mathbb{K})$ with $\overline{f}\left(\phi(w),\overline{A}\right)<\varepsilon$
		there must be exactly one $\overline{w}\in rev(\mathcal{W}_{n})$
		with $|\overline{A}\cap\overline{w}|\geq(1-\delta)|\overline{w}|$
		and $\overline{w}$ must be of the form $\left[\overline{w}\right]_{s}=\eta_{g}\left[w\right]_{s}$
		for a unique $g\in G_{s}$, which is necessarily of odd parity.
	\end{lem}
	
	\begin{proof}
		For a start, we choose $n$ large enough such that end effects are
		negligible with respect to the length $h_{n}$ of $n$-words (that is,
		$h_{n}$ has to be much larger than the length of the code). By Fact
		\ref{fact:string_length} we know from $\overline{f}\left(\phi(w),\overline{A}\right)<\varepsilon$
		that 
		\[
		1-3\varepsilon\leq\frac{1-\varepsilon}{1+\varepsilon}\leq\frac{|\overline{A}|}{h_n}\leq\frac{1+\varepsilon}{1-\varepsilon}\leq1+3\varepsilon.
		\]
		Hence, there can be at most one $\overline{w}\in rev(\mathcal{W}_{n})$
		with $|\overline{A}\cap\overline{w}|\geq(1-\delta)|\overline{w}|$.
		Suppose that there is no such $\overline{w}\in rev(\mathcal{W}_{n})$.
		Since we also know $|\overline{A}|\geq(1-3\varepsilon)|\overline{w}|$, there have to be two words $\overline{w}_{1},\overline{w}_{2}\in rev(\mathcal{W}_{n})$
			such that $\overline{A}$ is a substring of $\overline{w}_{1}\overline{w}_{2}$ and 
			\begin{equation}\label{eq:LengthBarAj}
				|\overline{A}\cap\overline{w}_{j}|\geq\left(\delta-3\varepsilon\right)\cdot|\overline{w}_{j}|>\frac{180}{\alpha_{s}}\varepsilon h_{n}\text{ for }j=1,2.
			\end{equation}
		
		We denote $\overline{A}_{j}\coloneqq\overline{A}\cap\overline{w}_{j}$. Let $\pi$ be a best possible match in $\overline{f}$ between $\phi(w)$
		and $\overline{A}$. Then we denote the parts in $\phi(w)$ corresponding
		to $\overline{A}_{j}$ under $\pi$ by $A_{j}$.

			We note that $\overline{f}\left(\phi(w),\overline{A}\right)<\varepsilon$ requires $\overline{f}\left(\phi(A_j),\overline{A}_j\right)<\frac{\alpha_s}{60}$ for both $j\in \{1,2\}$ because otherwise
			\begin{align*}
				\overline{f}\left(\phi(w),\overline{A}\right) 
				= & \frac{|A_1|+|\overline{A}_1|}{|w|+|\overline{A}|}\cdot \overline{f}\left(\phi(A_1),\overline{A}_1\right) + \frac{|A_2|+|\overline{A}_2|}{|w|+|\overline{A}|}\cdot \overline{f}\left(\phi(A_2),\overline{A}_2\right) \\
				\geq & \frac{\frac{180}{\alpha_{s}}\varepsilon h_{n}}{h_n + (1+3\varepsilon)h_n}\cdot \frac{\alpha_s}{60} > \varepsilon 
			\end{align*}
			by Fact~\ref{fact:substring_matching} and our length estimate on $\overline{A}_j$ in~\eqref{eq:LengthBarAj}. Then Fact~\ref{fact:string_length} and inequality~\eqref{eq:LengthBarAj} yield
			\begin{equation}\label{eq:LengthAj}
				|A_j|\geq \frac{1-\frac{\alpha_s}{60}}{1+\frac{\alpha_s}{60}}\cdot |\overline{A}_j| > (1-3\frac{\alpha_s}{60})\cdot \frac{180}{\alpha_{s}}\varepsilon h_{n} > \frac{160}{\alpha_{s}}\varepsilon h_{n}
			\end{equation}
			for both $j\in \{1,2\}$.
			
			We also note that the beginning of $\phi(w)$ is matched with the end of $\overline{w}_{1}$,
			while the end of $\phi(w)$ is matched with the beginning of $\overline{w}_{2}$.
			There are two possible cases: 
			\begin{itemize}
				\item At least one of $\overline{w}_{1},\overline{w}_{2}$
				has a different $s$-pattern type from that of $w$.
				\item Both $\overline{w}_{1}$ and $\overline{w}_{2}$
				have the same $s$-pattern type as $w$. Here, we point out that the
				$s$-Feldman pattern structure of the beginning and end of $w$ are different
				from each other by construction (recall from step (4) in the general substitution
				mechanism in Section \ref{sec:Substitution} that substitution instances
				were constructed as a concatenation of different Feldman patterns).
				In order to have $\overline{f}\left(\phi(w),\overline{A}\right)<\varepsilon$, we need to have $\overline{f}\left(\phi(A_j),\overline{A}_j\right)<\varepsilon$ for at least one $j\in \{1,2\}$. Since the
				$s$-Feldman pattern structure of the beginning and end of $w$ are different
				from each other, our estimate $|A_{j'}|>\frac{160}{\alpha_{s}}\varepsilon h_{n}$ for $j'\neq j$ from \eqref{eq:LengthAj} and $\overline{f}\left(\phi(A_j),\overline{A}_j\right)<\varepsilon$ together imply that $s$-Feldman patterns in $A_j$ and $\overline{A}_j$ matched under $\pi$ cannot be of the same type.
			\end{itemize}
			In both cases, our estimate $|A_{j}|\geq\frac{160}{\alpha_{s}}\varepsilon h_{n}$ from~\eqref{eq:LengthAj} guarantees for sufficiently large $n$ that the strings are long enough 
			so that we can apply Lemmas \ref{lem:BadCoding0} and \ref{lem:BadCoding}. This yields, with the aid of Fact \ref{fact:substring_matching},
		that in both cases
		\[
		\overline{f}\left(\phi(w),\overline{A}\right)\geq\frac{80}{\alpha_{s}}\varepsilon\cdot\frac{\alpha_{s}}{16}=5\varepsilon,
		\]
		which contradicts the assumption $\overline{f}\left(\phi(w),\overline{A}\right)<\varepsilon$. 
		Hence, there must be exactly one $\overline{w}\in rev(\mathcal{W}_{n})$
		with $|\overline{A}\cap\overline{w}|\geq(1-\delta)|\overline{w}|$. 
		
		Lemma \ref{lem:BadCoding} also implies that $w$ and $\overline{w}$
		need to have the same $s$-pattern structure. By construction this
		only happens if their equivalence classes $\left[w\right]_{s}$ and
		$\left[\overline{w}\right]_{s}$ lie on the same orbit of the group
		action by $G_{s}^{n}$ (recall from steps (4) and (6) in the general substitution mechanism
		that substitution instances not lying on the same $G_{s}^{n}$ orbit
		were constructed as a different concatenation of different Feldman
		patterns). Moreover, we have seen in Lemma~\ref{lem:iso} that an element $g\in G_s^n$ of odd parity induces $\eta_g:\mathcal{W}_n/\mathcal{Q}^n_s\to rev(\mathcal{W}_n)/\mathcal{Q}^n_s$ that preserves the order of $s$-Feldman patterns, while an element $g\in G_s^n$ of even parity would induce a map $\mathcal{W}_n/\mathcal{Q}^n_s\to rev(\mathcal{W}_n)/\mathcal{Q}^n_s$ that reverses the order of $s$-Feldman patterns. By the freeness of the group action and Proposition \ref{prop:conclusio}, we conclude that there is a unique $g\in G_{s}$ with $\left[\overline{w}\right]_{s}=\eta_{g}\left[w\right]_{s}$. As explained before, this element $g$ has to be of odd parity.
	\end{proof}
	Assume $\mathbb{K}$ and $\mathbb{K}^{-1}$ are evenly equivalent. Let $(\varepsilon_{\ell})_{\ell\in \N}$ be a sequence of positive reals satisfying $\sum^{\infty}_{i=\ell+1}\varepsilon_i<\varepsilon_{\ell}$ for every $\ell \in \N$.
	Then there exists a sequence of $\left(\varepsilon_{\ell},K_{\ell}\right)$-finite
	$\overline{f}$ codes $\phi_{\ell}$ from $\mathbb{K}$ to $\mathbb{K}^{-1}$ satisfying the properties in Lemma \ref{lem:ConsistentCode}.
	We fix such a sequence of $\left(\varepsilon_{\ell},K_{\ell}\right)$-finite
	$\overline{f}$ codes $\phi_{\ell}$ between $\mathbb{K}$ and $\mathbb{K}^{-1}$. By applying Corollary \ref{cor:CodeOnWords} with $\gamma_s = \frac{\alpha^4_s}{5\cdot 10^9}$, for each $s \in \mathbb{Z}^+$ there is $N(s)\in \mathbb{Z}^+$ such that for every $N \geq N(s)$ and $k\in \mathbb{Z}^+$ there is $n(s,N,k) \in \mathbb{Z}^+$ sufficiently large so that for all $n\geq n(s,N,k)$ Lemma \ref{lem:groupelement} with $\delta_s = \frac{\alpha^3_s}{2 \cdot 10^7}$ and $\phi_{N}$ holds. Moreover, there exists a collection $\mathcal{W}^{\prime}_n$ of $n$-words (that includes at least $1-\gamma_s$ of the $n$-words) satisfying the following properties:
	\begin{itemize}
		\item[(C1)] For $w \in \mathcal{W}^{\prime}_n$ there exists $z = z(w) \in rev(\mathbb{K})$ such that $z\upharpoonright[0,h_n-1]$ occurs with positive frequency in $rev(\mathbb{K})$ and 
		\[
		\overline{f}\left(\phi_{N}(w),z\upharpoonright[0,h_n-1]\right)<\frac{\alpha^4_s}{5\cdot 10^9}.
		\]
		\item[(C2)] For $w \in \mathcal{W}^{\prime}_n$ we have
		\[
		\overline{f}\left(\phi_{N}(w),\phi_{N+k}(w)\right)<\frac{\alpha^4_s}{5\cdot 10^9}.
		\]
	\end{itemize}

	\begin{lem}
		\label{lem:codeGroup}Suppose that $\mathbb{K}$ and $\mathbb{K}^{-1}$ are evenly equivalent and let $s\in\mathbb{N}$. There is a unique $g_{s}\in G_{s}$
		such that for every $N\geq N(s)$ and sufficiently large $n\in\mathbb{N}$
		we have for every $w\in\mathcal{W}_{n}^{\prime}$ that there is $\overline{w}\in rev(\mathcal{W}_{n})$
		with $\left[\overline{w}\right]_{s}=\eta_{g_{s}}\left[w\right]_{s}$
		and 
		\[
		\overline{f}\left(\phi_{N}(w),\overline{w}\right)<\frac{\alpha_{s}}{4},
		\]
		where $\left(\phi_{\ell}\right)_{\ell \in\mathbb{N}}$ is a sequence of $\left(\varepsilon_{\ell},K_{\ell}\right)$-finite
		$\overline{f}$ codes as described above and $\mathcal{W}_{n}^{\prime}$ is the associated collection of $n$-words satisfying properties (C1) and (C2).
		
		Moreover, $g_{s}\in G_{s}$ is of odd parity and the sequence $(g_{s})_{s\in\mathbb{N}}$
		satisfies $g_{s}=\rho_{s+1,s}(g_{s+1})$ for all $s\in\mathbb{N}$.
	\end{lem}
	
	\begin{proof}
		Let $N\geq N(s)$. For $n\geq n(s,N,1)$ we see from property (C1) that for any $w\in \mathcal{W}^{\prime}_n$ there exists $z\in rev(\mathbb{K})$ with 
		\[
		\overline{f}\left(\phi_{N}(w),z\upharpoonright[0,h_n-1]\right)<\gamma_s=\frac{\alpha^4_s}{5\cdot 10^9} < \frac{\alpha_s}{200}\cdot \delta_s.
		\]
		We denote $\overline{A}\coloneqq z\upharpoonright[0,h_n-1]$ in $rev(\mathbb{K})$.
		By Lemma \ref{lem:groupelement} there is exactly one $\overline{w}\in rev(\mathcal{W}_{n})$
		with $|\overline{A}\cap\overline{w}|\geq(1-\delta_{s})|\overline{w}|$
		and $\overline{w}$ must be of the form $\left[\overline{w}\right]_{s}=\eta_{g}\left[w\right]_{s}$
		for a unique $g\in G_{s}$, which is of odd parity. By Fact \ref{fact:omit_symbols}
		we conclude 
		\begin{equation} \label{eq:CodeWord1}
			\overline{f}\left(\phi_{N}(w),\overline{w}\right)<\gamma_s+\delta_{s}<10^{-7}\alpha_{s}^{3}. 
		\end{equation}
		
		To see that one group element of $G_{s}$ is supposed to work for
		$\phi_{N}$ and all words in $\mathcal{W}^{\prime}_{n+1}$ we repeat the argument
		for a word $w\in\mathcal{W}^{\prime}_{n+1}$, which is a concatenation of $n$-words
		by construction. Hence, we obtain a unique $g_{s}\in G_{s}$ and $\overline{w}\in rev(\mathcal{W}_{n+1})$
		with $\left[\overline{w}\right]_{s}=\eta_{g_{s}}\left[w\right]_{s}$
		such that $\overline{f}\left(\phi_{N}(w),\overline{w}\right)<10^{-7}\alpha_{s}^{3}$.
		We decompose $w$ into its $s$-Feldman patterns: $w=P_{1}\dots P_{r}$.
		Moreover, let $\overline{w}=\overline{P}_{1}\dots\overline{P}_{r}$
		be the corresponding decomposition under a best possible $\overline{f}$-match between $\phi_N(w)$ and $\overline{w}$.
		
		According to Fact \ref{fact:substring_matching} the proportion of
		such substrings with $\overline{f}\left(\phi_N(P_{i}),\overline{P}_{i}\right)>10^{-6}\alpha_{s}^{3}$
		could be at most $\frac{1}{10}$. On the other substrings with $\overline{f}\left(\phi_N(P_{i}),\overline{P}_{i}\right)\leq10^{-6}\alpha_{s}^{3}$
		we need to have $\left(1-\frac{2\alpha_{s}^{3}}{10^6}\right)|P_{i}|\leq|\overline{P}_{i}|\leq\left(1+\frac{3\alpha_{s}^{3}}{10^6}\right)|P_{i}|$
		by Fact \ref{fact:string_length}. Then we obtain from Lemma \ref{lem:BadCoding0}
		that there is a substring $\overline{P}_{i}^{\prime}\subseteq\overline{P}_{i}$
		with $|\overline{P}_{i}^{\prime}|\geq\left(1-\frac{\alpha^2_{s}}{5 \cdot 10^4}\right)|\overline{P}_{i}|$
		such that $\overline{P}_{i}^{\prime}$ and $P_{i}$ have the same
		$s$-pattern structure, that is, $[\overline{P}_{i}^{\prime}]_{s}$
		must lie in $\eta_{g_{s}}[P_{i}]_{s}$. We deduce from $\overline{f}\left(\phi_N(P_{i}),\overline{P}_{i}^{\prime}\right)<\frac{\alpha^2_{s}}{4\cdot 10^4}$
		that a proportion of at least $\frac{9}{10}$ of the $n$-blocks $w_{l}$
		occurring as substitution instances in $P_{i}$ have to satisfy that
		there is a string $\overline{A}_{l}$ in $\overline{P}_i^{\prime}$ such that $\overline{f}\left(\phi_{N}(w_{l}),\overline{A}_{l}\right)<\frac{\alpha^2_{s}}{2000}=\frac{\alpha_{s}}{200} \cdot \frac{\alpha_{s}}{10}$. As in the first paragraph of the proof, Lemma \ref{lem:groupelement} implies that there is exactly one $\overline{w}_l\in rev(\mathcal{W}_{n})$
		with $|\overline{A}_l\cap\overline{w}_l|\geq(1-\frac{\alpha_s}{10})|\overline{w}_l|$
		and $\overline{w}_l$ must be of the form $\left[\overline{w}_l\right]_{s}=\eta_{h_l}\left[w_l\right]_{s}$
		for a unique $h_l\in G_{s}$, which is of odd parity. As in the proof of (\ref{eq:CodeWord1}), we also get $\overline{f}\left(\phi_{N}(w_{l}),\overline{w}_{l}\right)<\frac{\alpha_{s}}{4}$. Since the classes	
		within the building tuple of $P_{i}$ have disjoint $G_{s}$ orbits by our selection of the building tuples in the setup of the general substitution mechanism from Section~\ref{sec:Substitution},
		we obtain that a proportion of at least $\frac{9}{10}$ of the $n$-blocks $w_{l}$
		occurring as substitution instances in $P_{i}$ have to satisfy that
		there is $\overline{w}_{l}\in rev(\mathcal{W}_{n})$ with $\left[\overline{w}_{l}\right]_{s}=\eta_{g_{s}}\left[w_{l}\right]_{s}$
		and $\overline{f}\left(\phi_{N}(w_{l}),\overline{w}_{l}\right)<\frac{\alpha_{s}}{4}$.
		Altogether, we conclude that a proportion of at least $\frac{8}{10}$
		of all $n$-blocks $w_{l}$ have to satisfy that there is $\overline{w}_{l}\in rev(\mathcal{W}_{n})$
		with $\left[\overline{w}_{l}\right]_{s}=\eta_{g_{s}}\left[w_{l}\right]_{s}$
		and $\overline{f}\left(\phi_{N}(w_{l}),\overline{w}_{l}\right)<\frac{\alpha_{s}}{4}$.
		
		Suppose now that there is a word $w^{\prime}\in\mathcal{W}^{\prime}_{n+1}$
		such that the argument above gives a different element $h\in G_{s}$
		and $\overline{w}^{\prime}\in rev(\mathcal{W}_{n+1})$ with $\left[\overline{w}^{\prime}\right]_{s}=\eta_{h}\left[w^{\prime}\right]_{s}$
		and $\overline{f}\left(\phi_{N}(w^{\prime}),\overline{w}^{\prime}\right)<\frac{\alpha_{s}}{4}$. Then a proportion of at least $\frac{8}{10}$ of all $n$-blocks $w_{l}$
		would have to satisfy that there is $\overline{w}_{l}\in rev(\mathcal{W}_{n})$
		with $\left[\overline{w}_{l}\right]_{s}=\eta_{h}\left[w_{l}\right]_{s}$
		and $\overline{f}\left(\phi_{N}(w_{l}),\overline{w}_{l}\right)<\frac{\alpha_{s}}{4}$.
		This contradicts $\eta_{g_{s}}\left[w_{l}\right]_{s}\neq\eta_{h}\left[w_{l}\right]_{s}$
		and Proposition \ref{prop:conclusio}. Hence, there is a unique group
		element $g_{s}\in G_{s}$ working for all $(n+1)$-words in $\mathcal{W}^{\prime}_{n+1}$.
		
		In the next step, we check that the same group element $g_{s}\in G_{s}$
		works for all $\phi_{N}$, $N\geq N(s)$. Let $k\in\mathbb{Z}^{+}$.
		For $n \geq n(s,N,k)$ property (C2) guarantees $\overline{f}\left(\phi_{N}(w),\phi_{N+k}(w)\right)<\alpha^4_s$
		for every $w\in\mathcal{W}^{\prime}_{n}$. On the one hand, let $g_{s}$ be
		the element in $G_{s}$ such that for every $w\in\mathcal{W}^{\prime}_{n}$
		there is $\overline{w}\in rev(\mathcal{W}_{n})$ with $\left[\overline{w}\right]_{s}=\eta_{g_{s}}\left[w\right]_{s}$
		and $\overline{f}\left(\phi_{N}(w),\overline{w}\right)<\frac{\alpha_{s}}{4}$.
		On the other hand, let $h_{s}\in G_{s}$ such that for every $w\in\mathcal{W}^{\prime}_{n}$
		there is $w^{\prime}\in rev(\mathcal{W}_{n})$ with $\left[w^{\prime}\right]_{s}=\eta_{h_{s}}\left[w\right]_{s}$
		and $\overline{f}\left(\phi_{N+k}(w),w^{\prime}\right)<\frac{\alpha_{s}}{4}$.
		Assume $g_{s}\neq h_{s}$. Then we have $\left[w^{\prime}\right]_{s}=\eta_{h_{s}}[w]_{s}\neq\eta_{g_{s}}\left[w\right]_{s}=\left[\overline{w}\right]_{s}$
		but 
		\[
		\overline{f}\left(\overline{w},w^{\prime}\right)\leq\overline{f}\left(\overline{w},\phi_{N}(w)\right)+\overline{f}\left(\phi_{N}(w),\phi_{N+k}(w)\right)+\overline{f}\left(\phi_{N+k}(w),w^{\prime}\right)<\frac{2}{3}\alpha_{s},
		\]
		which contradicts Proposition \ref{prop:conclusio}. Thus, we conclude that $g_{s}\in G_{s}$
		works for all $\phi_{N}$, for $N\geq N(s)$.
		
		Suppose that for some $s\in\mathbb{N}$ we have that $g_{s}^{\prime}\coloneqq\rho_{s+1,s}(g_{s+1})\neq g_{s}$.
		Let $N\geq N(s+1)$. Then for $n$ sufficiently large there is $w\in\mathcal{W}^{\prime}_{n}$
		with $\eta_{g_{s}^{\prime}}\left[w\right]_{s}\neq\eta_{g_{s}}\left[w\right]_{s}=\left[\overline{w}_{1}\right]_{s}$,
		where $\overline{w}_{1}\in rev\left(\mathcal{W}_{n}\right)$ is the
		element corresponding to $w$ under the code $\phi_{N}$ with $\overline{f}\left(\phi_{N}(w),\overline{w}_{1}\right)<\frac{\alpha_{s}}{4}$.
		Similarly, let $\overline{w}_{2}\in rev\left(\mathcal{W}_{n}\right)$
		be the element corresponding to $w$ under the code $\phi_{N}$ with
		$\overline{f}\left(\phi_{N}(w),\overline{w}_{2}\right)<\frac{\alpha_{s+1}}{4}$.
		Since the $G_{s+1}^{n}$ action is subordinate to the $G_{s}^{n}$
		action on $\mathcal{W}_{n}/\mathcal{Q}_{s}^{n}$ by specification
		(A7), we see that $\left[\overline{w}_{2}\right]_{s+1}=\eta_{g_{s+1}}\left[w\right]_{s+1}$
		lies in $\eta_{g_{s}^{\prime}}\left[w\right]_{s}$. Hence, $\left[\overline{w}_{2}\right]_{s}\neq\left[\overline{w}_{1}\right]_{s}$,
		which implies $\overline{f}\left(\overline{w}_{1},\overline{w}_{2}\right)>\alpha_{s}$
		by Proposition \ref{prop:conclusio}. On the other hand, we have 
		\[
		\overline{f}\left(\overline{w}_{1},\overline{w}_{2}\right)\leq\overline{f}\left(\overline{w}_{1},\phi_{N}(w)\right)+\overline{f}\left(\phi_{N}(w),\overline{w}_{2}\right)<\frac{\alpha_{s}}{4}+\frac{\alpha_{s+1}}{4}<\frac{\alpha_{s}}{2}.
		\]
		This contradiction shows that $(g_{s})_{s\in\mathbb{N}}$ must satisfy
		$g_{s}=\rho_{s+1,s}(g_{s+1})$ for all $s$.
	\end{proof}

	\begin{proof}[Proof of part (2) in Proposition \ref{prop:criterion}]
		If $\mathbb{K}=\Psi(\mathcal{T})$ and $\mathbb{K}^{-1}$ are Kakutani equivalent, then they have to be evenly equivalent by Proposition \ref{prop:only-even}. Then the conclusion of Lemma \ref{lem:codeGroup} implies that there is a sequence $(g_s)_{s\in \N}$ of group elements $g_{s}\in G_{s}$ of odd parity satisfying $g_{s}=\rho_{s+1,s}(g_{s+1})$ for all $s\in\mathbb{N}$. This results in a nonidentity element of odd parity in $G_{\infty}(\mathcal{T})$. Then Fact~\ref{fact:ill} yields that the
		tree $\mathcal{T}$ has an infinite branch.
	\end{proof}

	\section{\label{sec:Transfer}Transfer to the smooth and real-analytic categories}
	
	In this final section we transfer our anti-classification result for
	Kakutani equivalence to the smooth and real-analytic categories. For
	this purpose, we review the definition of circular systems. These
	symbolic systems can be realized by the Approximation by Conjugation
	method as $C^{\infty}$ diffeomorphisms on compact surfaces admitting
	a nontrivial circle action \cite{FW1} and as real-analytic diffeomorphisms
	on $\mathbb{T}^{2}$ (\cite{Ba17}, \cite{BK2}). We also review the definition of
	the functor $\mathcal{F}$ between odometer-based and circular systems
	introduced in \cite{FW2}. To simplify notation we enumerate construction
	sequences of our odometer-based systems as $\left\{ \mathcal{W}_{n}\right\} _{n\in\mathbb{N}}$,
	that is, $(n+1)$-words are built by concatenating $n$-words. This functor
	$\mathcal{F}$ preserves so-called synchronous and anti-synchronous
	isomorphisms \cite{FW2} but it does not necessarily preserve Kakutani
	equivalence as shown by examples in \cite{GeKu}. In the setting of
	our constructions from Section \ref{sec:Construction} we can prove
	using the techniques from Section \ref{sec:Non-Equiv} that if $\mathcal{T}\in\mathcal{T}\kern-.5mm rees$
	does not have an infinite branch, then $\mathcal{F}\left(\Psi(\mathcal{T})\right)$
	and $\mathcal{F}\left(\Psi(\mathcal{T})^{-1}\right)$ are not Kakutani
	equivalent to each other. This will allow us to conclude the proofs
	of Theorems \ref{thm:smooth} and \ref{thm:analytic} in Subsection
	\ref{subsec:ProofsDiffeos}.
	
	\subsection{Diffeomorphism spaces} \label{subsec:DiffeomorphismSpaces}
	
	Let $M$ be a smooth compact manifold of finite dimension equipped
	with a standard measure $\mu$ determined by a smooth volume element.
	We denote the space of $\mu$-preserving $C^{\infty}$ diffeomorphisms
	on $M$ by $\text{Diff}^{\infty}(M,\mu)$, which inherits a Polish
	topology. 
	
	Following \cite[section 2.3]{BK2} we give a more detailed description
	of the space of Lebesgue measure-preserving real-analytic diffeomorphisms
	on $\mathbb{T}^{2}\coloneqq\mathbb{R}^{2}/\mathbb{Z}^{2}$ that are
	homotopic to the identity. Here, $\lambda$ denotes the standard Lebesgue
	measure on $\mathbb{T}^{2}$.
	
	Any real-analytic diffeomorphism on $\mathbb{T}^{2}$ homotopic to
	the identity admits a lift to a map from $\mathbb{R}^{2}$ to $\mathbb{R}^{2}$
	which has the form 
	\[
	F(x_{1},x_{2})=(x_{1}+f_{1}(x_{1},x_{2}),x_{2}+f_{2}(x_{1},x_{2})),
	\]
	where $f_{i}:\mathbb{R}^{2}\to\mathbb{R}$ are $\mathbb{Z}^{2}$-periodic
	real-analytic functions. Any real-analytic $\mathbb{Z}^{2}$-periodic
	function on $\mathbb{R}^{2}$ can be extended as a holomorphic function
	defined on some open complex neighborhood of $\mathbb{R}^{2}$ in
	$\mathbb{C}^{2}$, where we identify $\mathbb{R}^{2}$ inside $\mathbb{C}^{2}$
	via the natural embedding $(x_{1},x_{2})\mapsto(x_{1}+0 \cdot \mathrm{i},x_{2}+0\cdot \mathrm{i})$.
	For a fixed $\rho>0$ we define the neighborhood 
	\[
	\Omega_{\rho}:=\{(z_{1},z_{2})\in\mathbb{C}^{2}:|\text{Im}(z_{1})|<\rho\text{ and }|\text{Im}(z_{2})|<\rho\},
	\]
	and for a function $f$ defined on this set we let 
	\[
	\|f\|_{\rho}:=\sup_{(z_{1},z_{2})\in\Omega_{\rho}}|f((z_{1},z_{2}))|.
	\]
	We define $C_{\rho}^{\omega}(\mathbb{T}^{2})$ to be the space of
	all $\mathbb{Z}^{2}$-periodic real-analytic functions $f$ on $\mathbb{R}^{2}$
	that extend to a holomorphic function on $\Omega_{\rho}$ and satisfy $\|f\|_{\rho}<\infty$.
	
	Hereby, we define $\text{Diff}_{\rho}^{\,\omega}(\mathbb{T}^{2},\lambda)$
	to be the set of all Lebesgue measure-preserving real-analytic diffeomorphisms
	of $\mathbb{T}^{2}$ homotopic to the identity, whose lift $F$ to
	$\mathbb{R}^{2}$ satisfies $f_{i}\in C_{\rho}^{\omega}(\mathbb{T}^{2})$,
	and we also require that the lift $\tilde{F}(x)=(x_{1}+\tilde{f}_{1}(x),x_{2}+\tilde{f}_{2}(x))$
	of its inverse to $\mathbb{R}^{2}$ satisfies $\tilde{f}_{i}\in C_{\rho}^{\omega}(\mathbb{T}^{2})$.
	Then the metric in $\text{Diff}_{\rho}^{\,\omega}(\mathbb{T}^{2},\lambda)$
	is defined by 
	\begin{align*}
		d_{\rho}(f,g)=\max\{\tilde{d}_{\rho}(f,g),\tilde{d}_{\rho}(f^{-1},g^{-1})\},\text{ where }\tilde{d}_{\rho}(f,g)=\max_{i=1,2}\{\inf_{n\in\mathbb{Z}}\|f_{i}-g_{i}+n\|_{\rho}\}.
	\end{align*}
	We note that if $\{f_{n}\}_{n=1}^{\infty}\subset\text{Diff}_{\rho}^{\,\omega}(\mathbb{T}^{2},\lambda)$
	is a Cauchy sequence in the $d_{\rho}$ metric, then $f_{n}$ converges
	to some $f\in\text{Diff}_{\rho}^{\,\omega}(\mathbb{T}^{2},\lambda)$.
	Thus, this space is Polish.
	
	\subsection{\label{subsec:Circular-systems}Circular systems}
	
	A \emph{circular coefficient sequence} is a sequence of pairs of positive
	integers $\left(k_{n},l_{n}\right)_{n\in\mathbb{N}}$ such that $k_{n}\geq2$
	and $\sum_{n\in\mathbb{N}}\frac{1}{l_{n}}<\infty.$ From these numbers
	we inductively define numbers 
	\[
	q_{n+1}=k_{n}l_{n}q_{n}^{2}
	\]
	and 
	\[
	p_{n+1}=p_{n}k_{n}l_{n}q_{n}+1,
	\]
	where we set $p_{0}=0$ and $q_{0}=1$. Obviously, $p_{n+1}$ and
	$q_{n+1}$ are relatively prime. Moreover, let $\Sigma$ be a non-empty
	finite alphabet (in our case $\Sigma=\left\{ 1,\dots,2^{12}\right\} $)
	and $b,e$ be two additional symbols. Then given a circular coefficient
	sequence $\left(k_{n},l_{n}\right)_{n\in\mathbb{N}}$ we build collections
	of words $\mathcal{W}_{n}^{c}$ in the alphabet $\Sigma\cup\{b,e\}$
	by induction as follows:
	\begin{itemize}
		\item Set $\mathcal{W}_{0}^{c}=\Sigma$.
		\item Having built $\mathcal{W}_{n}^{c}$, we choose a set $P_{n+1}\subseteq\left(\mathcal{W}_{n}^{c}\right)^{k_{n}}$
		of so-called \emph{prewords} and build $\mathcal{W}_{n+1}^{c}$ by
		taking all words of the form 
		\[
		\mathcal{C}_{n}\left(w_{0},w_{1},\dots,w_{k_{n}-1}\right)=\prod_{i=0}^{q_{n}-1}\prod_{j=0}^{k_{n}-1}\left(b^{q_{n}-j_{i}}w_{j}^{l_{n}-1}e^{j_{i}}\right)
		\]
		with $w_{0}\dots w_{k_{n}-1}\in P_{n+1}$. If $n=0$ we take $j_{0}=0$,
		and for $n>0$ we let $j_{i}\in\{0,\dots,q_{n}-1\}$ be such that
		\[
		j_{i}\equiv\left(p_{n}\right)^{-1}i\;\mod q_{n}.
		\]
		We note that each word in $\mathcal{W}_{n+1}^{c}$ has length $k_{n}l_{n}q_{n}^{2}=q_{n+1}$.
	\end{itemize}
	\begin{defn}
		\label{def:circularConstSeq} A construction sequence $\left(\mathcal{W}_{n}^{c}\right)_{n\in\mathbb{N}}$
		will be called \emph{circular} if it is built in this manner using
		the $\mathcal{C}_n$-operators and a circular coefficient sequence, and
		each $P_{n+1}$ is uniquely readable in the alphabet with the words
		from $\mathcal{W}_{n}^{c}$ as letters. (This last property is called
		the \emph{strong readability assumption}.) 
	\end{defn}
	
	\begin{rem*}
		By \cite[Lemma 45]{FW2} each $\mathcal{W}_{n}^{c}$ in a circular
		construction sequence is uniquely readable even if the prewords are
		not uniquely readable. However, the definition of a circular construction
		sequence requires this stronger readability assumption. 
	\end{rem*}
	\begin{defn}
		\label{def:CircularShift}A symbolic shift $\mathbb{K}$ built from
		a circular construction sequence is called a \emph{circular system}.
		For emphasis we will often denote it by $\mathbb{K}^{c}$. 
	\end{defn}
	
	As described in \cite[section 4.3]{FW3} one can give an explicit
	construction sequence $\left\{ rev(\mathcal{W}_{n}^{c})\right\} _{n\in\mathbb{N}}$
	of $(\mathbb{K}^{c})^{-1}$ as 
	\[
	rev(\mathcal{W}_{n+1}^{c})=\left\{ \mathcal{C}_{n}^{r}\left(rev(w_{0}),rev(w_{1}),\dots,rev(w_{k_{n}-1})\right):w_{0}w_{1}\dots w_{k_{n}-1}\in P_{n+1}\right\} 
	\]
	using the operator 
	\begin{equation}
		\mathcal{C}_{n}^{r}\left(w_{0},w_{1},\dots,w_{k_{n}-1}\right)=\prod_{i=0}^{q_{n}-1}\prod_{j=0}^{k_{n}-1}\left(e^{q_{n}-j_{i+1}}w_{k_{n}-j-1}^{l_{n}-1}b^{j_{i+1}}\right).\label{eq:RevCircOper}
	\end{equation}
	
	For a circular system $\mathbb{K}^c$ with circular coefficients $(k_{n},l_{n})_{n\in\mathbb{N}}$
	we also construct a second circular system: Let $\Sigma_{0}=\{\ast\}$, where $\ast$ is an arbitrary symbol,
	and we define a construction sequence 
	\[
	\widetilde{\mathcal{W}}_{0}^{c}:=\Sigma_{0},\qquad\text{and if }\qquad\widetilde{\mathcal{W}}_{n}^{c}:=\{w_{n}\}\quad\text{then}\quad\widetilde{\mathcal{W}}_{n+1}^{c}=\{\mathcal{C}_{n}(w_{n},\ldots,w_{n})\}.
	\]
	We denote the resulting circular system by $\mathcal{K}$. As shown
	in \cite[section 3.4]{FW2} $\mathcal{K}$ is a factor of $\mathbb{K}^c$
	measure theoretically isomorphic to a rotation of the circle. We call
	it the \emph{canonical circle factor}. In \cite[section 4.3]{FW2}
	a specific isomorphism $\natural:\mathcal{K}\to rev(\mathcal{K})$
	is introduced. It is called the \emph{natural map} and will serve
	as a benchmark for understanding maps from $\mathbb{K}^{c}$ to $rev(\mathbb{K}^{c})$
	(see, for example, Definition~\ref{def:synchronous} below).
	
	By \cite[Theorem 60]{FW1} circular systems can be realized as smooth
	diffeomorphisms via the untwisted AbC method provided that the sequence
	$(l_{n})_{n\in\mathbb{N}}$ grows sufficiently fast.
	\begin{lem}
		\label{lem:SmoothRealization}Let $M$ be $\mathbb{D}^{2}$, $\mathbb{A}$,
		or $\mathbb{T}^{2}$ equipped with Lebesgue measure $\lambda$, and
		let $\left\{ \mathcal{W}_{n}^{c}\right\} _{n\in\mathbb{N}}$ be a
		strongly uniform circular construction sequence with circular coefficients
		$\left(k_{n},l_{n}\right)_{n\in\mathbb{N}}$. We suppose that $|\mathcal{W}_{n}^{c}|\to\infty$
		as $n\to\infty$ and that $|\mathcal{W}_{n}^{c}|$ divides $|\mathcal{W}_{n+1}^{c}|$.
		If the sequence $(l_{n})_{n\in\mathbb{N}}$ grows sufficiently fast,
		then there is an area-preserving $C^{\infty}$ diffeomorphism $T$
		on $M$ such that the system $(M,\lambda,T)$ is isomorphic to $(\mathbb{K}^c,\nu,sh)$.
	\end{lem}
	
	The following real-analytic counterpart is proven in \cite[Theorem G]{BK2}.
	\begin{lem}
		\label{lem:Analytic Realization}Let $\rho>0$, and let $\lambda$ be the
		Lebesgue measure on $\mathbb{T}^{2}$. Moreover, let $\left\{ \mathcal{W}_{n}^{c}\right\} _{n\in\mathbb{N}}$
		be a strongly uniform circular construction sequence with circular
		coefficients $\left(k_{n},l_{n}\right)_{n\in\mathbb{N}}$. We suppose
		that $|\mathcal{W}_{n}^{c}|\to\infty$ as $n\to\infty$ and that $|\mathcal{W}_{n}^{c}|$
		divides $|\mathcal{W}_{n+1}^{c}|$. If the sequence $(l_{n})_{n\in\mathbb{N}}$
		grows sufficiently fast, then there is $T\in\text{Diff}_{\rho}^{\,\omega}(\mathbb{T}^{2},\lambda)$
		such that the system $(\mathbb{T}^2,\lambda,T)$ is isomorphic to $(\mathbb{K}^c,\nu,sh)$.
	\end{lem}
	
	We end this subsection by introducing the following subscales for
	a word $w\in\mathcal{W}_{n+1}^{c}$ as in \cite[Subsection 3.3]{FW2}: 
	\begin{itemize}
		\item Subscale $0$ is the scale of the individual powers of $w_{j}\in\mathcal{W}_{n}^{c}$
		of the form $w_{j}^{l-1}$ and each such occurrence of a $w_{j}^{l-1}$
		is called a \emph{$0$-subsection}. 
		\item Subscale $1$ is the scale of each term in the product $\prod_{j=0}^{k_{n}-1}\left(b^{q_{n}-j_{i}}w_{j}^{l_{n}-1}e^{j_{i}}\right)$
		that has the form $\left(b^{q_{n}-j_{i}}w_{j}^{l_{n}-1}e^{j_{i}}\right)$
		and these terms are called \emph{$1$-subsections}. 
		\item Subscale $2$ is the scale of each term of $\prod_{i=0}^{q_{n}-1}\prod_{j=0}^{k_{n}-1}\left(b^{q_{n}-j_{i}}w_{j}^{l_{n}-1}e^{j_{i}}\right)$
		that has the form $\prod_{j=0}^{k_{n}-1}\left(b^{q_{n}-j_{i}}w_{j}^{l_{n}-1}e^{j_{i}}\right)$
		and these terms are called \emph{$2$-subsections}. 
	\end{itemize}
	
	\subsection{\label{subsec:Functor}Categories $\mathcal{OB}$ and $\mathcal{CB}$
		and the functor $\mathcal{F}:\mathcal{OB}\to\mathcal{CB}$}
	
	For a fixed circular coefficient sequence $\left(k_{n},l_{n}\right)_{n\in\mathbb{N}}$
	we consider two categories $\mathcal{OB}$ and $\mathcal{CB}$ whose
	objects are odometer-based and circular systems, respectively. The
	morphisms in these categories are (synchronous and anti-synchronous)
	graph joinings defined below. For this purpose, we recall that a \emph{joining} between two measure preserving systems $(X,\mathcal{B},\mu,T)$
		and $(Y ,\mathcal{C} ,\nu,S)$ is a measure $\rho$ on $X \times Y$ defined on the product $\sigma$-algebra such that $\rho$ is $T \times S$-invariant, $\rho(B \times Y)=\mu(B)$ for each set $B \in \mathcal{B}$, and $\rho(X \times C)=\nu(C)$ for each set $C \in \mathcal{C}$. A joining $\rho$ is a \emph{graph joining} between $X$ and $Y$ if and only if for all $C \in \mathcal{C}$ there is a $B \in \mathcal{B}$ such that $\rho\left((B \times Y) \triangle (X \times C)\right)=0$. We can represent
		the graph joining corresponding to a measure preserving map $\phi:X \to Y$ by $\int (\delta_x \times \delta_{\phi(x)})\; \mathrm{d}\mu(x)$. 
	\begin{defn}
		\label{def:synchronous}If $\mathbb{K}$ is an odometer-based system,
		we denote its odometer factor by $\mathbb{K}^{\pi}$ and let $\pi:\mathbb{K}\to\mathbb{K}^{\pi}$
		be the canonical factor map. Similarly, if $\mathbb{K}^{c}$ is a
		circular system, we let $(\mathbb{K}^{c})^{\pi}$ be the circle factor
		$\mathcal{K}$ and let $\pi:\mathbb{K}^{c}\to\mathcal{K}$ be the
		canonical factor map.
		\begin{enumerate}
			\item Let $\mathbb{K}$ and $\mathbb{L}$ be odometer-based systems with
			the same coefficient sequence and let $\rho$ be a joining between
			$\mathbb{K}$ and $\mathbb{L}^{\pm1}$. Then $\rho$ is called \emph{synchronous}
			if $\rho$ joins $\mathbb{K}$ and $\mathbb{L}$ and the projection
			of $\rho$ to a joining on $\mathbb{K}^{\pi}\times\mathbb{L}^{\pi}$
			is the graph joining determined by the identity map. The joining $\rho$
			is called \emph{anti-synchronous} if $\rho$ joins $\mathbb{K}$ and
			$\mathbb{L}^{-1}$ and the projection of $\rho$ to a joining on $\mathbb{K}^{\pi}\times(\mathbb{L}^{-1})^{\pi}$
			is the graph joining determined by the map $x\mapsto-x$.
			\item Let $\mathbb{K}^{c}$ and $\mathbb{L}^{c}$ be circular systems with
			the same coefficient sequence and let $\rho$ be a joining between
			$\mathbb{K}^{c}$ and $(\mathbb{L}^{c})^{\pm1}$. Then $\rho$ is
			called \emph{synchronous} if $\rho$ joins $\mathbb{K}^{c}$ and $\mathbb{L}^{c}$
			and the projection of $\rho$ to a joining on $\mathcal{K}\times\mathcal{L}$
			is the graph joining determined by the identity map. The joining $\rho$
			is called \emph{anti-synchronous} if $\rho$ joins $\mathbb{K}^{c}$
			and $(\mathbb{L}^{c})^{-1}$ and the projection of $\rho$ to a joining
			on $\mathcal{K}\times\mathcal{L}^{-1}$ is the graph joining determined
			by the map $rev(\cdot)\circ\natural$.
		\end{enumerate}
	\end{defn}
	
	In \cite{FW2} Foreman and Weiss define a functor taking odometer-based
	systems to circular system that preserves the factor and conjugacy
	structure. To review the definition of the functor we fix a circular
	coefficient sequence $\left(k_{n},l_{n}\right)_{n\in\mathbb{N}}$.
	Let $\Sigma$ be a finite alphabet and $\left(\mathcal{W}_{n}\right)_{n\in\mathbb{N}}$
	be a construction sequence for an odometer-based system with coefficients
	$\left(k_{n}\right)_{n\in\mathbb{N}}$. Then we define a circular
	construction sequence $\left(\mathcal{W}_{n}^{c}\right)_{n\in\mathbb{N}}$
	and bijections $c_{n}:\mathcal{W}_{n}\to\mathcal{W}_{n}^{c}$ by induction:
	\begin{itemize}
		\item Let $\mathcal{W}_{0}^{c}=\Sigma$ and $c_{0}$ be the identity map. 
		\item Suppose that $\mathcal{W}_n$, $\mathcal{W}_{n}^{c}$, and $c_{n}$
		have already been defined. Then we define 
		\[
		\mathcal{W}_{n+1}^{c}=\left\{ \mathcal{C}_{n}\left(c_{n}\left(w_{0}\right),c_{n}\left(w_{1}\right),\dots,c_{n}\left(w_{k_{n}-1}\right)\right)\::\:w_{0}w_{1}\dots w_{k_{n}-1}\in\mathcal{W}_{n+1}\right\} 
		\]
		and the map $c_{n+1}$ by setting 
		\[
		c_{n+1}\left(w_{0}w_{1}\dots w_{k_{n}-1}\right)=\mathcal{C}_{n}\left(c_{n}\left(w_{0}\right),c_{n}\left(w_{1}\right),\dots,c_{n}\left(w_{k_{n}-1}\right)\right).
		\]
		In particular, the prewords are 
		\[
		P_{n+1}=\left\{ c_{n}\left(w_{0}\right)c_{n}\left(w_{1}\right)\dots c_{n}\left(w_{k_{n}-1}\right)\::\:w_{0}w_{1}\dots w_{k_{n}-1}\in\mathcal{W}_{n+1}\right\} .
		\]
		
	\end{itemize}
	\begin{defn}
		Suppose that $\mathbb{K}$ is built from a construction sequence $\left(\mathcal{W}_{n}\right)_{n\in\mathbb{N}}$
		and $\mathbb{K}^{c}$ has the circular construction sequence $\left(\mathcal{W}_{n}^{c}\right)_{n\in\mathbb{N}}$
		as constructed above. Then we define a map $\mathcal{F}$ from the
		set of odometer-based systems (viewed as subshifts) to circular systems
		(viewed as subshifts) by 
		\[
		\mathcal{F}\left(\mathbb{K}\right)=\mathbb{K}^{c}.
		\]
	\end{defn}
	
	\begin{rem}
		The map $\mathcal{F}$ is a bijection between odometer-based symbolic
		systems with coefficients $\left(k_{n}\right)_{n\in\mathbb{N}}$ and
		circular symbolic systems with coefficients $\left(k_{n},l_{n}\right)_{n\in\mathbb{N}}$
		that preserves uniformity. Since the construction sequences for our
		odometer-based systems are uniquely readable, the corresponding
		circular construction sequences automatically satisfy the strong
		readability assumption. 
	\end{rem}
	
	In \cite{FW2} it is shown that $\mathcal{F}$ gives an isomorphism
	between the categories $\mathcal{OB}$ and $\mathcal{CB}$. We state
	the following fact which is part of the main result in \cite{FW2}.
	\begin{lem}
		\label{lem:functor}For a fixed circular coefficient sequence $(k_{n},l_{n})_{n\in\mathbb{N}}$
		the categories $\mathcal{OB}$ and $\mathcal{CB}$ are isomorphic
		by the functor $\mathcal{F}$ that takes synchronous isomorphisms
		to synchronous isomorphisms and anti-synchronous isomorphisms to anti-synchronous
		isomorphisms.
	\end{lem}
	
	\begin{rem}
		Using the operator $\mathcal{C}_{n}^{r}$ from equation (\ref{eq:RevCircOper})
		we can give a construction sequence $\left\{ rev(\mathcal{W}_{n}^{c})\right\} _{n\in\mathbb{N}}$
		for $\mathcal{F}\left(\mathbb{K}\right)^{-1}=\left(\mathbb{K}^{c}\right)^{-1}\cong rev(\mathbb{K}^{c})$
		via 
		\begin{align*}
			& rev(\mathcal{W}_{n+1}^{c})=\\
			& \left\{ \mathcal{C}_{n}^{r}\left(rev(c_{n}(w_{0})),rev(c_{n}(w_{1})),\dots,rev(c_{n}(w_{k_{n}-1}))\right):w_{0}w_{1}\dots w_{k_{n}-1}\in\mathcal{W}_{n+1}\right\} .
		\end{align*}
	\end{rem}
	
	We recall that the length of a $n$-block $w$ in the odometer-based
	system is $h_{n}=\prod_{i=0}^{n-1}k_{i}$, if $n>0$, and $h_{0}=1$,
	while the length of a $n$-block in the circular system is $q_{n}$,
	that is, $\lvert c_{n}\left(w\right)\rvert=q_{n}$. Moreover, we will
	use the following map from substrings of the underlying odometer-based
	system to the circular system: 
	\[
	\mathcal{C}_{n,i}\left(w_{s}w_{s+1}\dots w_{t}\right)=\prod_{j=s}^{t}\left(b^{q_{n}-j_{i}}\left(c_{n}\left(w_{j}\right)\right)^{l_{n}-1}e^{j_{i}}\right)
	\]
	for any $0\leq i\leq q_{n}-1$ and $0\leq s\leq t\leq k_{n}-1$. We also have the map  
	\begin{equation*}
		\mathcal{C}^r_{n,i}\left(w_{s}w_{s+1}\dots w_{t}\right)=\prod_{j=0}^{t-s}\left(e^{q_{n}-j_{i+1}}\left(rev\left(c_{n}\left(w_{t-j}\right)\right)\right)^{l_{n}-1}b^{j_{i+1}}\right)
	\end{equation*}
	from substrings of $\mathbb{K}$ into $rev(\mathbb{K}^c)$.
	\begin{rem}
		\label{rem:Propagate}As in \cite[section 4.9]{FW3} we also propagate
		our equivalence relations and group actions to the circular system.
		For all $n\in\mathbb{N}$, $w_{1},w_{2}\in\mathcal{W}_{n}$ we let
		$\left(c_{n}(w_{1}),c_{n}(w_{2})\right)$ be in the relation $(\mathcal{Q}_{s}^{n})^{c}$
		if and only if $(w_{1},w_{2})\in\mathcal{Q}_{s}^{n}$. Moreover, given
		$g\in G_{s}^{n}$, we let $g[c_{n}(w_{1})]_{s}=[c_{n}(w_{2})]_{s}$
		if and only if $g[w_{1}]_{s}=[w_{2}]_{s}$.
	\end{rem}

	\subsection{$\overline{f}$-estimates in the circular system}
	
	We recall that we enumerate construction sequences of our odometer-based
	systems as $\left\{ \mathcal{W}_{n}\right\} _{n\in\mathbb{N}}$ in
	this section, that is, $(n+1)$-words are built by concatenating $n$-words.
	Before we analyze the circular system we foreshadow that we will choose
	the circular coefficients $(l_{n})_{n\in\mathbb{N}}$ as the last
	sequence of parameters in Subsection \ref{subsec:ProofsDiffeos} to
	grow fast enough to allow the realization as diffeomorphisms. We stress
	that none of the other parameters depends on $(l_{n})_{n\in\mathbb{N}}$.
	In particular, we can choose $(l_{n})_{n\in\mathbb{N}}$ to satisfy
	\begin{equation}
		l_{n}\geq\max\left(4R_{n+2}^{2},9l_{n-1}^{2}\right).\label{eq:lCondCircular}
	\end{equation}
	
	For the $\overline{f}$-estimates in the circular system we make use
	of the $\overline{f}$-estimates in the underlying odometer-based
	system. We will also use a sequence $\left(R_{n}^{c}\right)_{n=1}^{\infty}$,
	where $R_{1}^{c}=R_{1}$ (with $R_{1}$ from the sequence $(R_{n})_{n\in\mathbb{N}}$
	in the construction of the odometer-based system in Section \ref{sec:Construction})
	and $R_{n}^{c}=\lfloor\sqrt{l_{n-2}}\cdot k_{n-2}\cdot q_{n-2}^{2}\rfloor$
	for $n\geq2$. We note that for $n\geq2$, 
	\begin{equation}
		\frac{q_{n}}{R_{n}^{c}}=\frac{k_{n-1}\cdot l_{n-1}\cdot\left(k_{n-2}\cdot l_{n-2}\cdot q_{n-2}^{2}\right)\cdot q_{n-1}}{\lfloor\sqrt{l_{n-2}}\cdot k_{n-2}\cdot q_{n-2}^{2}\rfloor}\geq\sqrt{l_{n-2}}\cdot k_{n-1}\cdot l_{n-1}\cdot q_{n-1}.\label{eq:Rcircular}
	\end{equation}
	Hence, for $n\geq2$ a substring of at least $q_{n}/R_{n}^{c}$ consecutive
	symbols in a circular $n$-block contains at least $\sqrt{l_{n-2}}-1$
	complete $2$-subsections, which have length $k_{n-1}l_{n-1}q_{n-1}$
	(recall the notion of a $2$-subsection from the end of Subsection
	\ref{subsec:Circular-systems}). When conducting estimates in Lemma
	\ref{lem:CircularBlock} this will allow us to ignore incomplete $2$-subsections
	at the ends of the substring.
	
	Now we show $\overline{f}$-estimates in the circular system by induction.
	We start with the base case $n=1$.
	\begin{lem}
		\label{lem:CircularBase}Let $w,\overline{w}\in\mathcal{W}_{1}$ and
		$\mathcal{A}$, $\overline{\mathcal{A}}$ be any substrings of at
		least $q_{1}/R_{1}^{c}$ consecutive symbols in $c_{1}(w)$ and $c_{1}(\overline{w})$,
		respectively. 
		\begin{enumerate}
			\item If $[w]_{1}\neq[\overline{w}]_{1}$, then we have 
			\begin{equation}
				\overline{f}\left(\mathcal{A},\overline{\mathcal{A}}\right)>\frac{1}{2}-\frac{3}{R_{1}}-\frac{2}{l_{0}}.\label{eq:CircularBase1}
			\end{equation}
			\item If $w\neq\overline{w}$, then we have 
			\begin{equation}
				\overline{f}\left(\mathcal{A},\overline{\mathcal{A}}\right)>\frac{1}{\mathfrak{p}_{1}}\cdot\left(\frac{1}{2}-\frac{1}{R_{1}}\right)-\frac{3}{R_{1}}-\frac{2}{l_{0}}.\label{eq:CircularBase2}
			\end{equation}
		\end{enumerate}
	\end{lem}
	
	\begin{proof}
		Since $q_{0}=1$, we have 
		\[
		\frac{q_{1}}{R_{1}^{c}}=\frac{k_{0}\cdot l_{0}}{R_{1}}.
		\]
		By adding fewer than $2l_{0}$ symbols to each of $\mathcal{A}$ and
		$\overline{\mathcal{A}}$ we can complete partial strings $bw_{j}^{l_{0}-1}$,
		where $w_{j}\in\mathcal{W}_{0}$, at the beginning and end. The spacer
		$b$ occupies a proportion of $\frac{1}{l_{0}}$ in the augmented
		strings of $\mathcal{A}$ and $\overline{\mathcal{A}}$, respectively.
		We ignore the spacers $b$ and $e$ in the following consideration and denote the
		modified strings by $\mathcal{A}'$ and $\overline{\mathcal{A}}'$.
		By Fact \ref{fact:omit_symbols} we have 
		\[
		\overline{f}\left(\mathcal{A},\overline{\mathcal{A}}\right)>\overline{f}\left(\mathcal{A}',\overline{\mathcal{A}}'\right)-\frac{2R_{1}}{k_{0}}-\frac{2}{l_{0}}>\overline{f}\left(\mathcal{A}',\overline{\mathcal{A}}'\right)-\frac{1}{R_{1}}-\frac{2}{l_{0}}.
		\]
		Since $0$-words are symbols, we obtain from Remark \ref{rem:BaseCase}
		that 
		\[
		\overline{f}\left(\mathcal{A}',\overline{\mathcal{A}}'\right)>\begin{cases}
			\frac{1}{2}-\frac{2}{R_{1}}-\frac{2}{\mathfrak{p}_1}, & \text{if }[w]_{1}\neq[\overline{w}]_{1},\\
			\frac{1}{\mathfrak{p}_{1}}\cdot\left(\frac{1}{2}-\frac{1}{R_{1}}\right)-\frac{2}{R_{1}}, & \text{if }w\neq\overline{w}.
		\end{cases}
		\]
	\end{proof}
	This motivates the definitions 
	\begin{align*}
		\alpha_{1,1}^{c}=\min\left(\frac{1}{2}-\frac{3}{R_{1}}-\frac{2}{\mathfrak{p}_1}-\frac{2}{l_{0}},\frac{1}{9}\right), &  &  & \beta_{1}^{c}=\min\left(\frac{1}{\mathfrak{p}_{1}}\cdot\left(\frac{1}{2}-\frac{1}{R_{1}}\right)-\frac{3}{R_{1}}-\frac{2}{l_{0}},\alpha_{1,1}^{c}\right).
	\end{align*}
	We note that by $R_{1}\geq40\mathfrak{p}_{1}$ from condition (\ref{eq:Rn}),
	we have $\beta_{1}^{c}>\frac{1}{4\mathfrak{p}_{1}}\geq\frac{10}{R_{1}}$. Hence,
	$R_{1}>\frac{10}{\beta_{1}^{c}}$.
	
	\textbf{Induction assumption:} we assume for $n\geq1$ that there are $\frac{1}{8}>\alpha^{c}_{1,n} > \dots > \alpha^{c}_{s(n),n}>\beta_{n}^{c}$ such that in addition to (\ref{eq:Rassum2}) we have
	\begin{equation}
		R_{n}\geq\frac{7}{\beta_{n}^{c}},\label{eq:RAssumCirc2}
	\end{equation}
	and the following estimates on $\overline{f}$ distances hold:
	\begin{itemize}
		\item For every $s \in \{1,\dots, s(n)\}$ we assume that if $w,\overline{w}\in\mathcal{W}_{n}$ with
		$[w]_{s}\neq[\overline{w}]_{s}$, then
		\begin{equation}
			\overline{f}\left(\mathcal{A},\overline{\mathcal{A}}\right)>\alpha_{s,n}^{c}\label{eq:CircularAssump1}
		\end{equation}
		on any substrings $\mathcal{A}$, $\overline{\mathcal{A}}$ of at
		least $q_{n}/R_{n}^{c}$ consecutive symbols in $c_{n}(w)$ and $c_{n}(\overline{w})$,
		respectively.
		\item For $w,\overline{w}\in\mathcal{W}_{n}$ with $w\neq\overline{w}$, we have
		\begin{equation}
			\overline{f}\left(\mathcal{A},\overline{\mathcal{A}}\right)>\beta_{n}^{c}\label{eq:CircularAssump2}
		\end{equation}
		on any substrings $\mathcal{A}$, $\overline{\mathcal{A}}$ of at
		least $q_{n}/R_{n}^{c}$ consecutive symbols in $c_{n}(w)$ and $c_{n}(\overline{w})$,
		respectively.
	\end{itemize}
	We also recall the assumption $\mathfrak{p}_{n+1}>\max\left(4R_{n},2^{n+1},\frac{1}{\epsilon_{n+1}}\right)$
	from (\ref{eq:Pn}). 
	
	At stage $n+1$ of the construction we choose $R_{n+1}$ such that besides (\ref{eq:Rn}) it also satisfies
	\begin{equation}
		R_{n+1}\geq\frac{40\mathfrak{p}_{n+1}}{\beta_{n}^{c}}.\label{eq:RupdateCircular1}
	\end{equation}
	Our assumptions (\ref{eq:CircularAssump1}) and (\ref{eq:CircularAssump2})
	immediately imply $\overline{f}$-estimates for $1$-subsections.
	\begin{lem}
		\label{lem:CircularRepetition}Let $w,\overline{w}\in\mathcal{W}_{n}$,
		$0\leq i_{1},i_{2}<q_{n}$, and $\mathcal{B}$, $\overline{\mathcal{B}}$
		be any substrings of at least $\sqrt{l_{n}}q_{n}$ consecutive symbols
		in $\mathcal{C}_{n,i_{1}}(w)$ and $\mathcal{C}_{n,i_{2}}(\overline{w})$,
		respectively.
		\begin{enumerate}
			\item If $[w]_{s}\neq[\overline{w}]_{s}$ for any $s\leq s(n)$, then we
			have 
			\begin{equation}
				\overline{f}\left(\mathcal{B},\overline{\mathcal{B}}\right)>\alpha_{s,n}^{c}-\frac{1}{R_{n}^{c}}-\frac{2}{\sqrt{l_{n}}}.\label{eq:CircularRep1}
			\end{equation}
			\item If $w\neq\overline{w}$, then we have 
			\begin{equation}
				\overline{f}\left(\mathcal{B},\overline{\mathcal{B}}\right)>\beta_{n}^{c}-\frac{1}{R_{n}^{c}}-\frac{2}{\sqrt{l_{n}}}.\label{eq:CircularRep2}
			\end{equation}
		\end{enumerate}
	\end{lem}
	
	\begin{proof}
		The newly introduced spacers $b$ and $e$ occupy a proportion of
		at most $\frac{q_{n}}{\sqrt{l_{n}}q_{n}}=\frac{1}{\sqrt{l_{n}}}$
		in each $\mathcal{B}$ and $\overline{\mathcal{B}}$. We ignore them
		in the following consideration, which might increase the $\overline{f}$-distance
		by at most $\frac{2}{\sqrt{l_{n}}}$ according to Fact \ref{fact:omit_symbols}.
		Then we apply the symbol by block replacement Lemma~\ref{lem:symbol by block replacement} with $c_n(w)$ and $c_n(\overline{w})$ being represented by different symbols (that is, $\overline{f}=1$)
		and our assumptions (\ref{eq:CircularAssump1}) and (\ref{eq:CircularAssump2}),
		respectively. 
	\end{proof}
	Hereby, we can prove $\overline{f}$-estimates for substantial substrings
	of $2$-subsections.
	\begin{lem}
		\label{lem:CircularStep}Let $w,\overline{w}\in\mathcal{W}_{n+1}$,
		$0\leq i_{1},i_{2}<q_{n}$, and $\mathcal{B}$, $\overline{\mathcal{B}}$
		be any substrings of at least $\frac{k_{n}l_{n}q_{n}}{R_{n+1}}$ consecutive
		symbols in $\mathcal{C}_{n,i_{1}}(w)$ and $\mathcal{C}_{n,i_{2}}(\overline{w})$,
		respectively.
		\begin{enumerate}
			\item If $s(n+1)=s(n)$, then we have 
			\begin{equation}
				\overline{f}\left(\mathcal{B},\overline{\mathcal{B}}\right)>\begin{cases}
					\alpha_{s,n}^{c}-\frac{1}{R_{n}}-\frac{2}{R_{n+1}}-\frac{1}{R_{n}^{c}}-\frac{3}{\sqrt{l_{n}}}, & \text{if }[w]_{s}\neq[\overline{w}]_{s},\\
					\beta_{n}^{c}-\frac{1}{R_{n}}-\frac{2}{R_{n+1}}-\frac{1}{R_{n}^{c}}-\frac{3}{\sqrt{l_{n}}}, & \text{if }w\neq\overline{w}.
				\end{cases}\label{eq:CircularCase1}
			\end{equation}
			\item If $s(n+1)=s(n)+1$, then we have
			\begin{equation}
				\overline{f}\left(\mathcal{B},\overline{\mathcal{B}}\right)>\begin{cases}
					\alpha_{s,n}^{c}-\frac{1}{R_{n}}-\frac{3}{R_{n+1}}-\frac{1}{R_{n}^{c}}-\frac{3}{\sqrt{l_{n}}}, & \text{if }[w]_{s}\neq[\overline{w}]_{s},\\
					\beta_{n}^{c}-\frac{1}{R_{n}}-\frac{3}{R_{n+1}}-\frac{1}{R_{n}^{c}}-\frac{3}{\sqrt{l_{n}}}, & \text{if }[w]_{s(n)+1}\neq[\overline{w}]_{s(n)+1},\\
					\frac{1}{2\mathfrak{p}_{n+1}}\cdot\left(\beta_{n}^{c}-\frac{1}{R_{n}^{c}}-\frac{2}{\sqrt{l_{n}}}\right)-\frac{1}{R_{n+1}}-\frac{1}{\sqrt{l_{n}}}, & \text{if }w\neq\overline{w}.
				\end{cases}\label{eq:CircularCase2}
			\end{equation}
		\end{enumerate}
	\end{lem}
	
	\begin{proof}
		By adding fewer than $2l_{n}q_{n}$ symbols to each of $\mathcal{B}$
		and $\overline{\mathcal{B}}$ we can complete any partial $\mathcal{C}_{n,i_{1}}(w_{j})$
		at the beginning and end of $\mathcal{B}$ and any partial $\mathcal{C}_{n,i_{2}}(\overline{w}_{j})$
		at the beginning and end of $\overline{\mathcal{B}}$, where $w_{j},\overline{w}_{j}\in\mathcal{W}_{n}$.
		Let $\mathcal{B}_{aug}$ and $\overline{\mathcal{B}}_{aug}$ be the
		augmented $\mathcal{B}$ and $\overline{\mathcal{B}}$ strings obtained
		in this way. By Fact \ref{fact:omit_symbols} we have 
		\begin{equation}
			\overline{f}\left(\mathcal{B},\overline{\mathcal{B}}\right)\geq\overline{f}\left(\mathcal{B}_{aug},\overline{\mathcal{B}}_{aug}\right)-\frac{2R_{n+1}}{k_{n}}>\overline{f}\left(\mathcal{B}_{aug},\overline{\mathcal{B}}_{aug}\right)-\frac{1}{R_{n+1}}.\label{eq:CircularAug}
		\end{equation}
		We consider the corresponding strings $B_{aug}$ and $\overline{B}_{aug}$
		in the words $w,\overline{w}\in\mathcal{W}_{n+1}$ expressed in the
		alphabets $\Sigma_{n,s}=\left(\mathcal{W}_{n}/\mathcal{Q}_{s}^{n}\right)^{\ast}$
		or $\Sigma_{n}=\left(\mathcal{W}_{n}\right)^{\ast}$. They have length
		at least $\frac{k_{n}}{R_{n+1}}$. Hence, we can apply Remarks \ref{rem:PrepTransfer1}
		or \ref{rem:PrepTransfer2} depending on whether we are in case 1
		($s(n+1)=s(n)$) or case 2 ($s(n+1)=s(n)+1$) of the construction. 
		
		We examine case 2. If $[w]_{s}\neq[\overline{w}]_{s}$ for $s\leq s(n)$,
		then we have $\overline{f}_{\Sigma_{n,s}}\left(B_{aug},\overline{B}_{aug}\right)\geq1-\frac{1}{R_{n}}-\frac{2}{R_{n+1}}$.
		We apply Lemma \ref{lem:symbol by block replacement} with $\overline{f}\geq1-\frac{1}{R_{n}}-\frac{2}{R_{n+1}}$
		and the estimate from (\ref{eq:CircularRep1}) to obtain 
		\begin{align*}
			\overline{f}\left(\mathcal{B}_{aug},\overline{\mathcal{B}}_{aug}\right) & \geq\left(1-\frac{1}{R_{n}}-\frac{2}{R_{n+1}}\right)\cdot\left(\alpha_{s,n}^{c}-\frac{1}{R_{n}^{c}}-\frac{2}{\sqrt{l_{n}}}\right)-\frac{1}{\sqrt{l_{n}}}\\
			& \geq\alpha_{s,n}^{c}-\frac{1}{R_{n}}-\frac{2}{R_{n+1}}-\frac{1}{R_{n}^{c}}-\frac{3}{\sqrt{l_{n}}}.
		\end{align*}
		Together with (\ref{eq:CircularAug}) we deduce the first inequality
		in (\ref{eq:CircularCase2}). In an analogous manner, we obtain the
		second inequality in (\ref{eq:CircularCase2}) from (\ref{eq:CircularRep2})
		and $\overline{f}_{\Sigma_{n}}\left(B_{aug},\overline{B}_{aug}\right)\geq1-\frac{1}{R_{n}}-\frac{2}{R_{n+1}}$
		following from (\ref{eq:BaseClass}). To see the third inequality,
		we note that $\overline{f}_{\Sigma_{n}}\left(B_{aug},\overline{B}_{aug}\right)\geq\frac{1}{2\mathfrak{p}_{n+1}}$
		by (\ref{eq:BaseWords}) in Remark \ref{rem:PrepTransfer2}. Thus,
		we apply Lemma \ref{lem:symbol by block replacement} and the estimate
		from (\ref{eq:CircularRep2}) to obtain
		\[
		\overline{f}\left(\mathcal{B}_{aug},\overline{\mathcal{B}}_{aug}\right)\geq\frac{1}{2\mathfrak{p}_{n+1}}\cdot\left(\beta_{n}^{c}-\frac{1}{R_{n}^{c}}-\frac{2}{\sqrt{l_{n}}}\right)-\frac{1}{\sqrt{l_{n}}},
		\]
		which yields the claim with the aid of (\ref{eq:CircularAug}).
		
		By the same method we can also examine case 1 and conclude (\ref{eq:CircularCase1}).
	\end{proof}
	As an immediate consequence from the previous lemma we obtain $\overline{f}$-estimates
	for $(n+1)$-words in the circular system. 
	\begin{lem}
		\label{lem:CircularBlock}Let $w,\overline{w}\in\mathcal{W}_{n+1}$
		and $\mathcal{B}$, $\overline{\mathcal{B}}$ be any substrings of
		at least $q_{n+1}/R_{n+1}^{c}$ consecutive symbols in $c_{n+1}(w)$
		and $c_{n+1}(\overline{w})$, respectively. 
		\begin{enumerate}
			\item If $s(n+1)=s(n)$, then we have
			\begin{equation}
				\overline{f}\left(\mathcal{B},\overline{\mathcal{B}}\right)>\begin{cases}
					\alpha_{s,n}^{c}-\frac{2}{R_{n}}-\frac{3}{R_{n+1}}, & \text{if }[w]_{s}\neq[\overline{w}]_{s},\\
					\beta_{n}^{c}-\frac{2}{R_{n}}-\frac{3}{R_{n+1}}, & \text{if }w\neq\overline{w}.
				\end{cases}\label{eq:CircularBlock1}
			\end{equation}
			\item If $s(n+1)=s(n)+1$, then we have
			\begin{equation}
				\overline{f}\left(\mathcal{B},\overline{\mathcal{B}}\right)>\begin{cases}
					\alpha_{s,n}^{c}-\frac{2}{R_{n}}-\frac{4}{R_{n+1}}, & \text{if }[w]_{s}\neq[\overline{w}]_{s},\\
					\beta_{n}^{c}-\frac{2}{R_{n}}-\frac{4}{R_{n+1}}, & \text{if }[w]_{s(n)+1}\neq[\overline{w}]_{s(n)+1},\\
					\frac{1}{2\mathfrak{p}_{n+1}}\cdot\left(\beta_{n}^{c}-\frac{1}{R_{n}^{c}}-\frac{2}{\sqrt{l_{n}}}\right)-\frac{3}{R_{n+1}}, & \text{if }w\neq\overline{w}.
				\end{cases}\label{eq:CircularBlock2}
			\end{equation}
		\end{enumerate}
	\end{lem}
	
	\begin{proof}
		In case of $n>0$ we note that $\mathcal{B}$ and $\overline{\mathcal{B}}$
		have at least the length of $\sqrt{l_{n-1}}$ complete $2$-subsections
		by equation (\ref{eq:Rcircular}). By adding fewer than $2l_{n}k_{n}q_{n}$
		symbols to each of $\mathcal{B}$ and $\overline{\mathcal{B}}$, we
		can complete any partial $2$-subsections at the beginning and end
		of $\mathcal{B}$ and $\overline{\mathcal{B}}$. This change will increase
		the $\overline{f}$ distance between $\mathcal{B}$ and $\overline{\mathcal{B}}$
		by less than $2/\sqrt{l_{n-1}}$ due to Fact \ref{fact:omit_symbols}.
		Then we apply Lemma \ref{lem:symbol by block replacement} with $\overline{f}=1$
		and the corresponding estimates from Lemma \ref{lem:CircularStep}.
		In case of $s(n+1)=s(n)$ this yields
		\[
		\overline{f}\left(\mathcal{B},\overline{\mathcal{B}}\right)>\begin{cases}
			\alpha_{s,n}^{c}-\frac{1}{R_{n}}-\frac{3}{R_{n+1}}-\frac{1}{R_{n}^{c}}-\frac{2}{\sqrt{l_{n-1}}}-\frac{3}{\sqrt{l_{n}}}, & \text{if }[w]_{s}\neq[\overline{w}]_{s},\\
			\beta_{n}^{c}-\frac{1}{R_{n}}-\frac{3}{R_{n+1}}-\frac{1}{R_{n}^{c}}-\frac{2}{\sqrt{l_{n-1}}}-\frac{3}{\sqrt{l_{n}}}, & \text{if }w\neq\overline{w}.
		\end{cases}
		\]
		We use the conditions on the sequence $(l_{n})_{n\in\mathbb{N}}$
		from (\ref{eq:lCondCircular}) to estimate 
		\[
		\frac{1}{R_{n}^{c}}+\frac{2}{\sqrt{l_{n-1}}}+\frac{3}{\sqrt{l_{n}}}\leq\frac{1}{\sqrt{l_{n-2}}}+\frac{3}{\sqrt{l_{n-1}}}\leq\frac{2}{\sqrt{l_{n-2}}}\leq\frac{1}{R_{n}},
		\]
		which implies (\ref{eq:CircularBlock1}). Similarly, we obtain (\ref{eq:CircularBlock2}).
	\end{proof}
	This accomplishes the induction step. In particular, we note with
	the aid of (\ref{eq:RupdateCircular1}) that 
	\[
	\beta_{n+1}^{c}>\frac{\beta_{n}^{c}}{4\mathfrak{p}_{n+1}}-\frac{3}{R_{n+1}}\geq\frac{7\beta_{n}^{c}}{40\mathfrak{p}_{n+1}}\geq\frac{7}{R_{n+1}},
	\]
	that is, $R_{n+1}>\frac{7}{\beta_{n+1}^{c}}$ and the induction assumption
	(\ref{eq:RAssumCirc2}) for the next step is satisfied. Since we have
	\[
	\alpha_{s,M(s)}^{c}=\beta_{M(s)-1}^{c}-\frac{2}{R_{M(s)-1}}-\frac{4}{R_{M(s)}}\geq\beta_{M(s)-1}^{c}-\frac{2}{7}\beta_{M(s)-1}^{c}-\frac{\beta_{M(s)-1}^{c}}{10\mathfrak{p}_{M(s)}}>\frac{\beta_{M(s)-1}^{c}}{2}
	\]
	by (\ref{eq:RAssumCirc2}) and (\ref{eq:RupdateCircular1}), we also obtain
	\[
	\alpha_{s,n}^{c}\geq\alpha_{s,M(s)}^{c}-\frac{2}{R_{M(s)}}-\sum_{j>M(s)}\frac{6}{R_{j}}\geq\alpha_{s,M(s)}^{c}-\frac{\beta_{M(s)-1}^{c}}{20\mathfrak{p}_{M(s)}}-\frac{6\beta_{M(s)-1}^{c}}{40}>\frac{\alpha_{s,M(s)}^{c}}{2}
	\]
	for any $n>M(s)$ by conditions (\ref{eq:RAssumCirc2}) and (\ref{eq:RupdateCircular1}). Hereby, we see
	that for every $s\in\mathbb{N}$ the decreasing sequence $\left(\alpha_{s,n}^{c}\right)_{n\geq M(s)}$
	is bounded from below by $\frac{\alpha_{s,M(s)}^{c}}{2}$, and we denote
	its limit by $\alpha_{s}^{c}$. Altogether, this proves the following proposition which is the counterpart of Proposition~\ref{prop:conclusio} for circular systems.
	\begin{prop}\label{prop:fCircular}
		For every $n\geq M(s)$, we have
		\begin{equation}
			\overline{f}\left(\mathcal{A},\overline{\mathcal{A}}\right)>\alpha_{s}^{c}\label{eq:ConclusioCircular}
		\end{equation}
		on any substrings $\mathcal{A}$, $\overline{\mathcal{A}}$ of at
		least $q_{n}/R_{n}^{c}$ consecutive symbols in $c_{n}(w)$ and $c_{n}(\overline{w})$,
		respectively, for $w,\overline{w}\in\mathcal{W}_{n}$ with $[w]_{s}\neq[\overline{w}]_{s}$.
	\end{prop}

	\subsection{\label{subsec:NonEquivSmooth}Proof of Non-Kakutani Equivalence}
	
	We follow the strategy in Section \ref{sec:Non-Equiv} to show that
	for a tree $\mathcal{T}\in\mathcal{T}\kern-.5mm rees$ without an infinite branch
	$T_{c}=\mathcal{F}(\Psi(\mathcal{T}))$ and $T_{c}^{-1}=\mathcal{F}(\Psi(\mathcal{T}))^{-1}$
	are not Kakutani equivalent.
	\begin{lem}
		Let $\mathcal{T}\in\mathcal{T}\kern-.5mm rees$, $T_{c}=\mathcal{F}(\Psi(\mathcal{T}))$,
		and $f:X\to\mathbb{Z}^{+}$ be integrable. If $T_{c}^{f}\cong T_{c}$,
		then $\int f\mathrm{d}\mu=1$.
	\end{lem}
	
	\begin{proof}
		The proof follows along the lines of Proposition \ref{prop:NonIsom}.
		This time, one decomposes the name $a(T_{c},P_{c},x)$ in the circular
		system into strings of the form $\mathcal{C}_{n,j}\left(F_{n,i}\right)$,
		where $0\leq j<q_{n}$ and $F_{n,i}$ is a Feldman pattern of $1$-equivalence
		classes of $n$-words in the odometer-based system (as considered
		in the proof of Proposition \ref{prop:NonIsom}). 
	\end{proof}
	Hereby, we can show as in Proposition \ref{prop:only-even} that
	if $\mathbb{K}^{c}=\mathcal{F}(\Psi(\mathcal{T}))$ and $(\mathbb{K}^{c})^{-1}=\mathcal{F}(\Psi(\mathcal{T}))^{-1}$
	are Kakutani equivalent, then it can only be by an even equivalence.
	We want to show that $\mathbb{K}^{c}$ and $(\mathbb{K}^{c})^{-1}$ are not evenly
	equivalent if $\mathcal{T}\in\mathcal{T}\kern-.5mm rees$ does not have an infinite
	branch. For this purpose, we prove the following analogues of Lemmas \ref{lem:BadCoding0} and \ref{lem:BadCoding}.
	\begin{lem}
		\label{lem:BadCodingCircular0}Let $s\in\mathbb{N}$ and $\phi$ be a finite
		code from $\mathbb{K}^{c}$ to $rev\left(\mathbb{K}^{c}\right)$. Then for sufficiently large $n\in\mathbb{N}$, $0\leq i_{1},i_{2}<q_{n}$,
		and any pair of $s$-Feldman patterns $P_{n,\ell}$ and $\overline{P}_{n,\overline{\ell}}$ in $\mathbb{K}$  
		of different pattern type we have 
		\[
		\overline{f}\left(\phi(\mathcal{G}),\overline{\mathcal{G}}\right)>\frac{\alpha^c_{s}}{10}
		\]
		for any substrings $\mathcal{G}$, $\overline{\mathcal{G}}$ of at least $\frac{|P_{n,\ell}|\cdot l_n q_n}{2^{2e(n)}h_n}=\frac{|\overline{P}_{n,\overline{\ell}}|\cdot l_n q_n}{2^{2e(n)}h_n}$
		consecutive symbols in $\mathcal{C}_{n,i_{1}}\left(P_{n,\ell}\right)$ and $\mathcal{C}_{n,i_{2}}^{r}\left(\overline{P}_{n,\overline{\ell}}\right)$, respectively.
	\end{lem}
	\begin{proof}
		The proof follows along the lines of Lemma \ref{lem:BadCoding0}.
	\end{proof}
	
	\begin{lem}
		\label{lem:BadCodingCircular}Let $s\in\mathbb{N}$ and $\phi$ be
		a finite code from $\mathbb{K}^{c}$ to $rev\left(\mathbb{K}^{c}\right)$.
		Moreover, let $n\in\mathbb{N}$ be sufficiently large, $0\leq i_{1},i_{2}<q_{n}$,
		and $w,\overline{w}\in\mathcal{W}_{n+1}$ be any pair such that $w$
		and $rev(\overline{w})$ are of different $s$-pattern type. Then
		we have 
		\[
		\overline{f}\left(\phi\left(\mathcal{A}\right),\overline{\mathcal{A}}\right)>\frac{\alpha_{s}^{c}}{16}
		\]
		on any substrings $\mathcal{A}$, $\overline{\mathcal{A}}$ of at
		least $\frac{k_{n}l_{n}q_{n}}{2^{e(n)}}$ consecutive symbols in $\mathcal{C}_{n,i_{1}}(w)$
		and $\mathcal{C}_{n,i_{2}}^{r}\left(\overline{w}\right)$, respectively.
	\end{lem}
	
	\begin{proof}
		Analogously to the proof of Lemma \ref{lem:BadCoding}, we divide $\mathcal{A}$
		and $\overline{\mathcal{A}}$ into circular images $\mathcal{P}_{n,j}=\mathcal{C}_{n,i_{1}}\left(P_{n,j}\right)$
		and $\overline{\mathcal{P}}_{n,j}=\mathcal{C}_{n,i_{2}}^{r}\left(\overline{P}_{n,j}\right)$,
		where $P_{n,j}$ and $\overline{P}_{n,j}$ are $s$-Feldman patterns
		in $w$ and $\overline{w}$, respectively. Then the proof follows
		along the lines of Lemma \ref{lem:BadCoding}. 
	\end{proof}
	We can use Lemmas \ref{lem:BadCodingCircular0} and \ref{lem:BadCodingCircular}
	to identify a well-approximating finite code with an element of the
	group action in a similar way that we used Lemmas \ref{lem:BadCoding0} and \ref{lem:BadCoding} in the proof of
	Lemma \ref{lem:groupelement}.
	\begin{lem}
		\label{lem:groupelementCircular}Let $s\in\mathbb{N}$, $0<\delta<\frac{1}{4}$,
		$0<\varepsilon<\frac{\alpha_{s}^{c}}{200}\delta$, and $\phi$ be
		a finite code from $\mathbb{K}^{c}$ to $rev(\mathbb{K}^{c})$. Then
		for $n$ sufficiently large, any $0\leq i_{1}<q_{n}$, any $w\in\mathcal{W}_{n+1}$,
		and any string $\overline{\mathcal{A}}$ in $rev(\mathbb{K}^{c})$
		with 
		\[
		\overline{f}\left(\phi\left(\mathcal{C}_{n,i_{1}}(w)\right),\overline{\mathcal{A}}\right)<\varepsilon,
		\]
		there must be exactly one string of the form $\mathcal{C}_{n,i_{2}}^{r}(\overline{w})$
		with $|\overline{\mathcal{A}}\cap\mathcal{C}_{n,i_{2}}^{r}(\overline{w})|\geq(1-\delta)k_{n}l_{n}q_{n}$
		for some $0\leq i_{2}<q_{n}$, $\overline{w}\in\mathcal{W}_{n+1}$. Moreover, $\overline{w}$ must be of the form $\left[\overline{w}\right]_{s}=g\left[w\right]_{s}$
		for a unique $g\in G_{s}$, which is necessarily of odd parity.
	\end{lem}
	
	\begin{proof}
		As in the proof of Lemma \ref{lem:groupelement} we neglect end effects
		and conclude that there can be at most one string of the form $\mathcal{C}_{n,i_{2}}^{r}(\overline{w})$
		with $|\overline{\mathcal{A}}\cap\mathcal{C}_{n,i_{2}}^{r}(\overline{w})|\geq(1-\delta)k_{n}l_{n}q_{n}$.
		Suppose that there is no such $\mathcal{C}_{n,i_{2}}^{r}(\overline{w})$,
		but two strings $\mathcal{C}_{n,i_{2,1}}^{r}(\overline{w}_{1})$,
		$\mathcal{C}_{n,i_{2,2}}^{r}(\overline{w}_{2})$ with $\overline{w}_{1},\overline{w}_{2}\in\mathcal{W}_{n+1}$
		and $0\leq i_{2,1},i_{2,2}<q_{n}$ such that $\overline{\mathcal{A}}$
		is a substring of $\mathcal{C}_{n,i_{2,1}}^{r}(\overline{w}_{1})\mathcal{C}_{n,i_{2,2}}^{r}(\overline{w}_{2})$
		and 
		\[
		|\overline{\mathcal{A}}\cap\mathcal{C}_{n,i_{2,j}}^{r}(\overline{w}_{j})|\geq\left(\delta-3\varepsilon\right)\cdot k_{n}l_{n}q_{n}>\frac{180}{\alpha_{s}^{c}}\varepsilon k_{n}l_{n}q_{n}\;\text{ for }j=1,2.
		\]
		We follow the proof of Lemma \ref{lem:groupelement} using Lemmas \ref{lem:BadCodingCircular0} and \ref{lem:BadCodingCircular} to conclude that there must be exactly one $\mathcal{C}_{n,i_{2}}^{r}(\overline{w})$
		with $|\overline{\mathcal{A}}\cap\mathcal{C}_{n,i_{2}}^{r}(\overline{w})|\geq(1-\delta)k_{n}l_{n}q_{n}$. 
		
		Lemma \ref{lem:BadCodingCircular} also implies that $w$ and $rev(\overline{w})$
		need to have the same $s$-pattern structure. As before this only
		happens if $\left[\overline{w}\right]_{s}=g\left[w\right]_{s}$
		for some $g\in G_{s}$ of odd parity.
	\end{proof}
	
	Assume $\mathbb{K}^{c}$ and $(\mathbb{K}^{c})^{-1}$ are evenly
	equivalent. Let $(\varepsilon_{\ell})_{\ell \in \N}$ be a sequence of positive reals satisfying $\sum^{\infty}_{i=\ell +1} \varepsilon_i <\varepsilon_{\ell}$. Then we take a sequence of $(\varepsilon_{\ell},K_{\ell})$-finite $\overline{f}$ codes $\phi^c_{\ell}$ from $\mathbb{K}^c$ to $(\mathbb{K}^{c})^{-1}$ as produced by Lemma \ref{lem:ConsistentCode}. In the following we use $\gamma^c_{s}=\frac{(\alpha_{s}^{c})^{4}}{5 \cdot 10^9}$
	for each $s\in\mathbb{N}$. We apply Corollary \ref{cor:CodeOnWords} with $\gamma^c_s$ on our sequence of codes and obtain $N(s) \in \mathbb{Z}^+$ such that for every $N \geq N(s)$ and $k\in \mathbb{Z}^+$ there is $n(s,N,k)\in \mathbb{Z}^+$ sufficiently large such that for all $n \geq n(s,N,k)$ Lemma \ref{lem:groupelementCircular} with $\delta^c_s=\frac{(\alpha^c_s)^3}{2\cdot 10^7}$ and $\phi_N$ holds and there exists a collection of words $\mathcal{W}^{\prime}_n \subset \mathcal{W}_n$ (that includes at least $1-\sqrt{\gamma^c_s}$ of the $n$-words) such that for every $w\in \mathcal{W}^{\prime}_n$ and $0\leq i<q_{n-1}$ the circular strings $\mathcal{C}_{n-1,i}(w)$ satisfy analogues of properties (C1) and (C2). 
	\begin{lem}
		\label{lem:codeGroupCircular}Suppose $\mathbb{K}^{c}$ and $(\mathbb{K}^{c})^{-1}$ are evenly
		equivalent and let $s\in\mathbb{N}$. There is a unique
		$g_{s}\in G_{s}$ such that for every $N\geq N(s)$ and sufficiently
		large $n\in\mathbb{N}$ we have for every $w\in\mathcal{W}_{n}^{\prime}$
		that there is $\overline{w}\in \mathcal{W}_{n}$ with
		$\left[\overline{w}\right]_{s}=g_{s}\left[w\right]_{s}$
		such that 
		\[
		\overline{f}\left(\phi_{N}^{c}(c_n(w)),rev(c_n(\overline{w}))\right)<\frac{\alpha_{s}^{c}}{4},
		\]
		where $(\phi^c_{\ell})_{\ell}$ is a sequence of codes as described above. Moreover, $g_{s}$ is of odd parity and the sequence $(g_{s})_{s\in\mathbb{N}}$
		satisfies $g_{s}=\rho_{s+1,s}(g_{s+1})$ for all $s\in\mathbb{N}$.
	\end{lem}
	
	\begin{proof}
		Let $w\in\mathcal{W}^{\prime}_{n+1} \subset \mathcal{W}_{n+1}$ and we recall that $c_n(w)=\prod_{i=0}^{q_{n}-1}\mathcal{C}_{n,i}(w)$. As in the proof of Lemma \ref{lem:codeGroup}
		we use property (C1) to choose
		$n$ sufficiently large such that there exists $z \in rev(\mathbb{K}^c)$ with
		\[
		\overline{f}\left(\phi_{N}^{c}\left(\mathcal{C}_{n,i}(w)\right),z\upharpoonright[0,k_nl_nq_n-1]\right)<\gamma^c_s.
		\]
		We denote $\overline{\mathcal{A}}_{i}\coloneqq z\upharpoonright[0,k_nl_nq_n-1]$
		in $rev(\mathbb{K}^{c})$. By Lemma \ref{lem:groupelementCircular}
		there is a string of the form $\mathcal{C}_{n,j_{i}}^{r}(\overline{w})$
		with $|\overline{\mathcal{A}}_{i}\cap\mathcal{C}_{n,j_{i}}^{r}(\overline{w})|\geq(1-\delta^c_{s})k_{n}l_{n}q_{n}$
		for some $0\leq j_{i}<q_{n}$, $\overline{w}\in\mathcal{W}_{n+1}$,
		and $\overline{w}$ must be of the form $\left[\overline{w}\right]_{s}=g\left[w\right]_{s}$
		for a unique $g\in G_{s}$. By Fact \ref{fact:omit_symbols} we conclude
		$\overline{f}\left(\phi_{N}^{c}\left(\mathcal{C}_{n,i}(w)\right),\mathcal{C}_{n,j_{i}}^{r}(\overline{w})\right)<\gamma^c_s+\delta^c_{s}$.
		Since this holds true for every $i\in\left\{ 0,\dots,q_{n}-1\right\} $
		and the newly introduced spacers occupy a proportion of at most $\frac{1}{l_{n}}$,
		we obtain for $n$ sufficiently large that $\overline{f}\left(\phi_{N}^{c}(c_n(w)),\mathcal{C}_{n}^{r}(\overline{w})\right)<\gamma^c_s+\delta^c_{s}+\frac{2}{l_{n}}<10^{-7}(\alpha_{s}^{c})^{3}$.
		To see that the same $g\in G_{s}$
		is supposed to work for all $w\in\mathcal{W}_{n+1}$ we repeat the
		argument as in Lemma \ref{lem:codeGroup} for a word in $\mathcal{W}_{n+2}^c$
		which is a concatenation of $(n+1)$-words by construction. By the
		same arguments as in Lemma \ref{lem:codeGroup} one can also show
		that $g_{s}=\rho_{s+1,s}(g_{s+1})$ for all $s\in\mathbb{N}$.
	\end{proof}
	As in Subsection \ref{subsec:non-even} we conclude that $\mathbb{K}^{c}=\mathcal{F} \circ \Psi(\mathcal{T})$
	and $(\mathbb{K}^{c})^{-1}$ cannot be Kakutani equivalent if $\mathcal{T}$ does not have an infinite branch.

	\subsection{\label{subsec:ProofsDiffeos}Proofs of Theorems \ref{thm:smooth}
		and \ref{thm:analytic}}
	
	Now we can complete the transfer to the setting of diffeomorphisms.
	\begin{proof}[Proof of Theorem \ref{thm:smooth}]
		In the proof of Theorem \ref{thm:Main} we constructed a continuous
		reduction $\Psi:\mathcal{T}\kern-.5mm rees\to\mathcal{E}$ that took values in
		the strongly uniform odometer-based transformations. Under the functor
		$\mathcal{F}$ from Subsection \ref{subsec:Functor} these are transferred
		to strongly uniform circular systems. During the course of the construction
		several parameters with numerical conditions about growth and decay
		rates appear. In particular, at stage $n$ the parameter $l_{n}$
		is chosen last. Hence, we can choose it sufficiently large to guarantee
		condition (\ref{eq:lCondCircular}), as well as 
		the convergence to a $C^{\infty}$ diffeomorphism that is measure-theoretically
		isomorphic to the circular system asserted in Lemma \ref{lem:SmoothRealization}. We denote this realization map from Lemma \ref{lem:SmoothRealization}
		with our choice of the sequence $(l_{n})_{n\in\mathbb{N}}$ by $R$.
		Hereby, we define the map $F^{s}:\mathcal{T}\kern-.5mm rees\to\text{Diff}^{\infty}(M,\lambda)$
		by $F^{s}=R\circ\mathcal{F}\circ\Psi$. We want to show that $F^{s}$
		is a continuous reduction. 
		\begin{lem*}
			$F^{s}:\mathcal{T}\kern-.5mm rees\to\text{Diff}^{\infty}(M,\lambda)$ is continuous.
		\end{lem*}
		\begin{proof}
			Let $T=F^{s}(\mathcal{T})$ for $\mathcal{T}\in\mathcal{T}\kern-.5mm rees$ and
			$U$ be an open neighborhood of $T$ in $\text{Diff}^{\infty}(M,\lambda)$.
			Let $T^{c}=\mathcal{F}\circ\Psi(\mathcal{T})$ be the circular system
			such that $R(T^{c})=T$. By \cite[Proposition 61]{FW1} there is $M\in\mathbb{N}$
			sufficiently large such that for all $S\in\text{Diff}^{\infty}(M,\lambda)$
			in the range of $R$ we have the following property: If $\left(\mathcal{W}_{n}^{c}(T^{c})\right)_{n\leq M}=\left(\mathcal{W}_{n}^{c}(S^{c})\right)_{n\leq M}$,
			then $S\in U$. Here, $S^{c}$ denotes the circular system such that
			$S=R(S^{c})$. Moreover, $\left(\mathcal{W}_{n}(T^{c})\right)_{n\in\mathbb{N}}$
			and $\left(\mathcal{W}_{n}(S^{c})\right)_{n\in\mathbb{N}}$ denote
			the construction sequences of $T^{c}$ and $S^{c}$, respectively.
			We recall from Subsection \ref{subsec:Functor} that $\left(\mathcal{W}_{n}(T^{c})\right)_{n\leq M}$
			is determined by the first $M+1$ members in the construction sequence
			of the odometer based system $F(\mathcal{T})$, that is, it is determined
			by $\left(\mathcal{W}_{n}(\mathcal{T})\right)_{n\leq M}$. By (\ref{eq:ContinuityF})
			there is a basic open set $V\subseteq\mathcal{T}\kern-.5mm rees$ containing
			$\mathcal{T}$ such that for all $\mathcal{S}\in V$ the first $M+1$
			members of the construction sequences of $\Psi(\mathcal{T})$ and
			$\Psi(\mathcal{S})$ are the same, that is, $\left(\mathcal{W}_{n}(\mathcal{T})\right)_{n\leq M}=\left(\mathcal{W}_{n}(\mathcal{S})\right)_{n\leq M}$.
			Then it follows that $F^{s}(\mathcal{S})\in U$ for all $\mathcal{S}\in V$
			which yields the continuity of $F^{s}$ by Fact \ref{fact:contTree}.
		\end{proof}
		\begin{lem*}
			$F^{s}:\mathcal{T}\kern-.5mm rees\to\text{Diff}^{\infty}(M,\lambda)$ reduces
			the collection of ill-founded trees to the collection of ergodic $C^{\infty}$
			diffeomorphisms $T$ such that $T$ and $T^{-1}$ are Kakutani equivalent.
		\end{lem*}
		\begin{proof}
			Since the realization map $R$ maps the circular system to a measure-theoretically
			isomorphic diffeomorphism, it preserves Kakutani equivalence. Hence,
			to see that $F^{s}$ is a reduction, it suffices to check that $\mathcal{F}\circ\Psi$
			is a reduction. Here, we let $\mathbb{K}=\Psi(\mathcal{T})$ and $\mathbb{K}^{c}=\mathcal{F}\circ\Psi(\mathcal{T})$
			for $\mathcal{T}\in\mathcal{T}\kern-.5mm rees$.
			
			Suppose $\mathcal{T}\in\mathcal{T}\kern-.5mm rees$ has an infinite branch. By
			the proof of part (1) of Proposition \ref{prop:criterion} in Section
			\ref{sec:Isom} there is an anti-synchronous isomorphism $\phi:\mathbb{K}\to\mathbb{K}^{-1}$.
			Then we can apply Lemma \ref{lem:functor} on the functor $\mathcal{F}$
			and obtain the existence of an anti-synchronous isomorphism $\phi^{c}:\mathbb{K}^{c}\to(\mathbb{K}^{c})^{-1}$.
			The converse direction follows from Subsection \ref{subsec:NonEquivSmooth},
			where we have shown that if $\mathcal{T}\in\mathcal{T}\kern-.5mm rees$ does
			not have an infinite branch, then $\mathbb{K}^{c}$ and $(\mathbb{K}^{c})^{-1}$
			are not Kakutani equivalent.
		\end{proof}
		Both statements together yield that $F^{s}$ is a continuous reduction
		and we conclude the statement as in the proof of Theorem \ref{thm:CompleteAnalytic}.
	\end{proof}
	In an analogous manner we can also prove Theorem \ref{thm:analytic}.
	\begin{proof}[Proof of Theorem \ref{thm:analytic}]
		This time we let $R$ be the realization map from circular systems
		with fast growing coefficients $(l_{n})_{n\in\mathbb{N}}$ to $\text{Diff}_{\rho}^{\,\omega}(\mathbb{T}^2,\lambda)$
		obtained by Lemma \ref{lem:Analytic Realization}. Then the proof
		follows along the lines of the previous one, using \cite[Proposition 7.10]{BK2} 
		instead of \cite[Proposition 61]{FW1}.
	\end{proof}
	
	\section{An Application: Sigma-Finite Case}\label{sec:Sigma}
		
		We now show non-classifiability of ergodic automorphisms up to isomorphism
		in the case of the sigma-finite measure space $(\mathbb{R},\mathcal{B},m),$
		where $\mathcal{B}$ denotes the Borel subsets of $\mathbb{R}$ and
		$m$ is Lebesgue measure. We will see that this follows easily from
		non-classifiability of ergodic automorphisms \emph{up to Kakutani
			equivalence} in the finite measure case. This is essentially the general
		sigma-finite case, because all standard Borel spaces with a Borel
		non-atomic sigma-finite infinite measure are isomorphic to $(\mathbb{R},\mathcal{B},m)$ \cite{Ku34}.
		
		Let $\tilde{\mathcal{X}}$ be the collection of measure-preserving
		automorphisms of $(\mathbb{R},\mathcal{B},m).$ An element $\tilde{T}$
		of $\tilde{\mathcal{X}}$ is said to be \emph{ergodic} if every $\tilde{T}$-invariant
		set has measure zero or its complement has measure zero. Let $\tilde{\mathcal{E}}\subset\tilde{X}$
		denote the ergodic automorphisms in $\tilde{\mathcal{X}}.$ We endow
		$\tilde{\mathcal{X}}$ with the weak topology, that is, $\tilde{T}_{n}\to\tilde{T}$
		if and only if $m(\tilde{T}_{n}A\Delta\tilde{T}A)\to0$ for every
		$A\in\mathcal{B}$ of finite measure. Note that $\tilde{\mathcal{X}}$
		is a Polish space \cite[Proposition 2.1]{Maitre}. Moreover, as in the
		finite measure case, $\tilde{\mathcal{E}}$ is a dense $G_{\delta}$
		subset of $\tilde{\mathcal{X}}$ \cite{Chokski}. Since $\tilde{\mathcal{E}}$
		is a Borel set, it makes sense to ask whether the isomorphism equivalence
		relation, considered as a subset of $\tilde{\mathcal{E}}\times\tilde{\mathcal{E}}$,
		is Borel. As in the finite measure case, we have the following negative
		answer.

		\begin{thm}
			There exists a continuous map 
			\[
			\tilde{\Psi}:\mathcal{T}rees\to\tilde{\mathcal{E}}
			\]
			such that for every $\mathcal{T}\in\mathcal{T}rees,$ $\tilde{\Psi}(\mathcal{T})$
			is isomorphic to $(\tilde{\Psi}(\mathcal{T}))^{-1}$ if and only if
			$\mathcal{T}$ has an infinite branch. Consequently, 
			\[
			\{(\tilde{S},\tilde{T})\in\tilde{\mathcal{E}}\times\tilde{\mathcal{E}}:\tilde{S}\text{\ is\ isomorphic\ to\ \ensuremath{\tilde{T}\}} }
			\]
			is a complete analytic set. In particular, it is not Borel. 
		\end{thm}

		\begin{proof}
			The map $\tilde{\Psi}$ is obtained from the map $\Psi$
			in Theorem~\ref{thm:Main} as follows. Let $\mathcal{E}$ be the ergodic automorphisms
			of a standard non-atomic probability space $(X,\mu).$ For $T\in\mathcal{E}$
			and a measurable function $f:X\to\mathbb{Z}^{+}$ define $T^{f}$
			on $(X^{f},\mu^{f})$ as in Section~\ref{subsec:SpecialTrans}, but assume that $\int f\ d\mu=\infty.$
			Then there is a natural isomorphism between $(X^{f},\mu^{f})$ and
			$(\mathbb{R},m).$ By applying this isomorphism, automorphisms of
			$(X^{f},\mu^{f})$ may be considered as automorphisms of $(\mathbb{R},m),$
			and $T\in\mathcal{E}$ implies that $T^{f}\in\tilde{\mathcal{E}}$. For each $T=\Psi(\mathcal{T}),$ we will construct a function $f=f_{\mathcal{T}},$
			and we will define $\tilde{\Psi}(\mathcal{T})$ to be $T^{f}.$ 
			
			The definition of $f_{\mathcal{T}}$ will be given in terms of the
			$n$-blocks used in the construction of $\Psi(\mathcal{T}).$ Start
			by choosing a sequence $(a_{n})_{n=1}^{\infty}$ of positive integers
			such that $\sum_{n=1}^{\infty}a_{n}=\infty.$ Let $I_{n}=I_{n}(\mathcal{T})$
			be the set of points $x\in X$ such that position $0$ in the bi-infinite
			string of symbols constituting $x$ is the first or last position within
			an $n$-block. Let $b_{n}=b_{n}(\mathcal{T})$ be the first positive
			integer such that $b_{n}\mu(I_{n})>a_{n}.$ Define $g_{n}=g_{n}(\mathcal{T}):X\to\mathbb{N}$
			by $g_{n}(x)=b_{n}$ if $x\in I_{n},$ and $g_{n}(x)=0$ otherwise.
			Let $f_{\mathcal{T}}=1+\sum_{n=1}^{\infty}g_{n}.$ Since the probability
			of position $0$ being the first position or the last position within
			an $n$-block for infinitely many $n$ is $0,$ $f_{\mathcal{T}}$
			is finite-valued almost everywhere. Moreover, $\int f_{\mathcal{T}}\ d\mu\ge\sum_{n=1}^{\infty}b_{n}\mu(I_{n})\ge\sum_{n=1}^{\infty}a_{n}=\infty.$
			Thus, for $f=f_{\mathcal{T}},$ $(\Psi(\mathcal{T}))^{f}$ can be
			identified with an element of $\tilde{\mathcal{E}}$ in a natural
			way. We define $\tilde{\Psi}(\mathcal{T})$ to be this element of
			$\tilde{\mathcal{E}}.$ The identification can be done in such a way
			that $\tilde{\Psi}$ is continuous. It remains to be shown that $\tilde{\Psi}(\mathcal{T})$
			is isomorphic to $(\tilde{\Psi}(\mathcal{T}))^{-1}$ if and only if
			$\mathcal{T}$ has an infinite branch. 
			
			\emph{{\bf Case 1:} $\mathcal{T}$ has an infinite branch.} Let $T=\Psi(\mathcal{T}).$ In the proof of Theorem~\ref{thm:Main}, a measure-preserving
			automorphism $\phi$ of $X$ is obtained such that $\phi\circ T^{-1}=T\circ\phi.$
			The automorphism $\phi$ takes $n$-blocks to $n$-blocks, possibly
			reversing them. By the symmetry in the definition of $g_{n}=g_{n}(\mathcal{T})$,
			for $f=f_{\mathcal{T}}$ we have $f=f\circ\phi.$ It remains to be
			shown that $T^{f}$ is isomorphic to $(T^{f})^{-1}.$ We define a
			measure-preserving automorphism $\tilde{\phi}$ of $X^{f}$ by $\tilde{\phi}(x,j)=(\phi(x),f(x)-j+1),$
			for $x\in X,$ $j=1,2,\dots,f(x).$ We claim that $\tilde{\phi}\circ(T^{f})^{-1}=T^{f}\circ\tilde{\phi.}$
			We have 
			\begin{align*}
				\tilde{\phi}((T^{f})^{-1}(x,1))&=\tilde{\phi}(T^{-1}x,f(T^{-1}x))=(\phi(T^{-1}x),1)=(T(\phi(x)),1) \\
				&=T^{f}(\phi(x),f(\phi(x)))=T^{f}(\phi(x),f(x))=T^{f}(\tilde{\phi}(x,1));
			\end{align*}
			and for $j=2,\dots,f(x),$ we have 
			\begin{align*}
				\tilde{\phi}((T^{f})^{-1}(x,j))&=\tilde{\phi}(x,j-1)=(\phi(x),f(x)-j+2) \\
				&=T^{f}(\phi(x),f(x)-j+1)=T^{f}(\tilde{\phi}(x,j)).
			\end{align*}
			Therefore $\tilde{\phi}\circ(T^{f})^{-1}=T^{f}\circ\tilde{\phi}.$
			This implies that $T^{f}$ is isomorphic to $(T^{f})^{-1},$ that
			is, $\tilde{\Psi}(\mathcal{T})$ is isomorphic to $(\tilde{\Psi}(\mathcal{T}))^{-1}.$
			
			\emph{{\bf Case 2:} $\mathcal{T}$ does not have an infinite branch.} Let
			$T=\tilde{\Psi}(\mathcal{T}).$ By Theorem~\ref{thm:Main}, $T$ and $T^{-1}$
			are not Kakutani equivalent. Note that $(T^{f})^{-1}$ is isomorphic
			to $(T^{-1})^{g},$ where $g=f\circ T^{-1}.$ More specifically, $\tilde{\varphi}\circ(T^{f})^{-1}=(T^{-1})^{g}\circ\tilde{\varphi},$
			where $\tilde{\varphi}:X^f\to X^g$ is defined by $\tilde{\varphi}(x,1)=(x,1),$
			and $\tilde{\varphi}(x,j)=(Tx,f(x)-j+2)$ if $1<j\le f(x).$ (This
			follows by a calculation similar to the one in Case 1.) Therefore
			Case 2 is reduced to showing that $T^{f}$ is not isomorphic to $(T^{-1})^{g}.$
			Suppose $T^{f}$ is isomorphic to $(T^{-1})^{g}.$ The transformation
			$T$ on $X$ is the first return map of $T^{f}$ to $X.$ Thus there
			exists a subset $Y$ of $\mathbb{R}$ of measure one such that the
			first return map of $(T^{-1})^{g}$ to $Y$ is isomorphic to $T.$
			If $S$ is the first return map of $(T^{-1})^{g}$ to $X\cup Y,$
			then the first return maps of $S$ to $X$ and of $S$ to $Y$ are
			Kakutani equivalent. But the former is isomorphic to $T^{-1},$ while
			the latter is isomorphic to $T.$ This is a contradiction, since $T$
			and $T^{-1}$ are not Kakutani equivalent. 
	\end{proof} 
	
	\subsection*{Acknowledgements}
	
	We thank M. Foreman and B. Weiss for helpful discussions.
	We are especially grateful to Foreman for his encouragement and his detailed answers to 
	our questions regarding his work with D. Rudolph and B. Weiss. We thank the referees for their careful reading of the paper and their very helpful and detailed comments.

\end{document}